\documentclass[11pt]
{amsart}
\usepackage{amssymb,amsmath}
\usepackage{amsmath}
\usepackage{amsfonts}
\usepackage{amssymb}
\usepackage{graphicx}
\usepackage{psfrag,overpic}
\usepackage{fancyhdr}
\usepackage{bm}
\usepackage[a4paper,
margin=1in]
{geometry}

\newtheorem{thm}{Theorem}[section]
\newtheorem{prop}[thm]{Proposition}
\newtheorem{lem}[thm]{Lemma}
\newtheorem{cor}[thm]{Corollary}

\newtheorem{remark}[thm]{Remark}
\newtheorem{definition}[thm]{Definition}
\newtheorem{problem}{Problem}
\newcommand {\qd} {\quad}

\newcommand \pt {\partial}

\newcommand \grad{\nabla}
\newcommand \gam{\gamma}
\newcommand \R{\mathbb{R}}
\newcommand \Om{\Omega}
\newcommand \vphi{\varphi}

\newcommand \Gam{\Gamma}

\newcommand \tx{\text}
\newcommand \bu {{\mathbf u}}
\newcommand \bU {{\bm U}}
\newcommand \be {{\bf e}}
\newcommand \til{\tilde}

\newcommand \p {\prime}

\newcommand \ol{\overline}
\newcommand \mc {\mathcal }
\newcommand \mf {\mathfrak }
\newcommand \Div{{\rm div}}
\newcommand \vtheta {\vartheta}
\newcommand \ra {\to}
\newcommand \ov {\overline}

\newcommand \ngrad {\grad^\perp}
\newcommand \eps {\varepsilon}

\newcommand \n {\mathcal N}
\newcommand \ul {\underline }
\newcommand \ff {{f(0)+f_s}}
\newcommand \qdon {\qd\tx{on}\qd}
\newcommand \qdin {\qd \tx{in}\qd}
\newcommand \qdand {\qd\tx{and}\qd}
\newcommand \fsp {f_s^\p}

\newcommand \rtp {(r,\theta,\vphi)}
\newcommand \bW {{\bm W}}

	\title[3-D axisymmetric transonic shock] { 3-D axisymmetric Transonic shock solutions of the full Euler system in Divergent nozzles}
	\author{Yong Park}
	\address{Yong Park, Department of Mathematics\\
		POSTECH\\
		San 31, Hyojadong, Namgu, Pohang, Gyungbuk, Republic of Korea 37673
	}
	\email{pipablue@postech.ac.kr}

\makeatletter
\@addtoreset{equation}{section}
\makeatother
\numberwithin{equation}{subsection}
\begin{document}
		\begin{abstract}
			We establish the 
			stability of 3-D axisymmetric transonic shock solutions of the steady full Euler system in divergent nozzles under 
			small perturbations of an incoming radial supersonic flow and a constant pressure at the exit of the nozzles. 
			To study 3-D axisymmetric transonic shock solutions of the full Euler system, 
			we use a stream function formulation of the full Euler system for a 3-D axisymmetric flow. 
			We resolve the singularity issue arising in stream function formulations of the full Euler system for a 3-D axisymmetric flow. 
			We develop a new scheme to determine a shock location of a transonic shock solution of the steady full Euler system based on the stream function formulation. 
	    \end{abstract}
         
         \maketitle
        
        \section{Introduction}
        In \cite[Chapter 147]{MR0421279}, authors, using an approximate model, 
        describe 
        a transonic shock phenomenon for a compressible invicid flow of an ideal polytropic gas in a convergent-divergent type nozzle called de Laval nozzle:  If a subsonic flow accelerating as it passes through the convergent part of the nozzle reaches the sonic speed at the throat of the nozzle, then it becomes a supersonic flow right after the throat of the nozzle. It further accelerates as it passes through the divergent part of the nozzle. If an appropriately large exit pressure $p_e$ is imposed 
        at the exit of the nozzle, then at a certain place of the divergent part of the nozzle, a shock front intervenes, the flow is compressed and slowed down to subsonic speed. The position and strength of the shock front are automatically adjusted so that the end pressure at the exit becomes $p_e$. This phenomenon was rigorously studied using
        radial solutions of the full Euler system  in \cite{MR2382380}
        (it was shown that in a divergent nozzle, for given a constant supersonic data on the entrance of the nozzle and an appropriately large constant pressure on the exit of the nozzle, 
        there exists a unique radial transonic shock solution satisfying these conditions).  
        Motivated by this phenomenon, there were many studies on 
        the stability of transonic shock solutions in divergent nozzles (structural stability of radial transonic shock solutions in divergent nozzles under 
        multi-dimensional perturbations of an entrance supersonic data and exit pressure) and related problems. 

        The stability of one-dimensional transonic shock solutions in flat nozzles was first studied. This subject was studied using the potential flow model in \cite{MR1969202,MR2020107,MR2299761,MR2143525,MR2398992} and further studied 
        using the full Euler system in \cite{MR2098096,MR2274487,MR2265619,MR2348768,MR2533922,MR2307052,MR2415074,MR2372813}. These results showed that
        one-dimensional transonic shock solutions in flat nozzles are not stable under a perturbation of a physical boundary condition (supersonic data on the entrance or density, pressure or normal velocity on the exit) and, even if one-dimensional transonic shock solutions in flat nozzles are stable, their shock locations are not uniquely determined unless there exists 
        the assumption that a shock location passes through some point on the wall of the nozzle, as it can be expected from the behavior of one-dimensional transonic shock solutions in flat nozzles (as a shock location changes,
        the value of the subsonic part of an one-dimensional transonic shock solution in a flat nozzle does not change).  
        After that, the stability of radial transonic shock solutions in divergent nozzles was studied. 
        This subject was first studied using the full Euler system in \cite{MR2530157,MR2427405}. 
        In these results, authors, by considering a perturbation of radial transonic shock solutions in divergent nozzles, could show that a shock location is uniquely determined for given an exit pressure without the assumption that a shock location passes through some point of the nozzle but 
        they only had the result under the assumption that the tip angle of the nozzle is sufficiently small. 
        After that, this subject
        without  
        restriction on the tip angle of the nozzle was studied. 
        In \cite{MR2824466}, the authors studied this subject using the non-isentropic potential model introduced in \cite{MR2824466}. And they 
        obtained the stability result for radial transonic shock solutions in divergent nozzles. 
        This subject was also studied using the full Euler system. 
        This 
        for the
        2-D case was done 
        in \cite{MR2420002,MR2576697,MR3005323}. In these papers, the authors had the stability result for radial transonic shock solutions in divergent nozzles. Especially, the authors in \cite{MR3005323} had the result
        for flows having $C^{1,\alpha}$ interior and $C^\alpha$ up to boundary regularity, so that they could consider 
        a general perturbation of a nozzle. 
        This 
        for the
        3-D case for axisymmetric flows with zero angular momentum components was done 
        in \cite{MR2557901}. The authors in this paper also had the same result. 
        This 
        for the general 3-D case was done 
        in \cite{MR3060893,MR3459028}. The authors in \cite{MR3060893,MR3459028} also had the same result but under S-condition introduced in \cite{MR3060893,MR3459028}.
        Recently, this subject for the general 3-D case for flows having some friction term was studied in \cite{YZQfric}.

        In this paper, we study the stability of radial transonic shock solutions in divergent nozzles under small perturbations of an incomming radial supersonic solution and a constant exit pressure using the full Euler system for the 3-D case for axisymmetric flows. We consider axisymmetric flows with non-zero angular momentum components. (This is a difference from \cite{MR2557901}.) We consider a divergent nozzle having no restriction on the tip angle of the nozzle and do not have any assumption on an incomming supersonic solution. 
        
        The main new feature in this paper is to develop a new iteration scheme to determine a shock location for a transonic shock solution of the steady full Euler system in a divergent nozzle
        and resolve the singularity issue arising in the stream function formulations of the full Euler system using an elliptic system approach.

        To deal with the stability of 3-D axisymmetric transonic shock solutions of the full Euler system, we use a stream function formulation for the full Euler system for an axisymmetric flow. This formulation shows the fact that an initial shock position and a shape of a shock location (see the definitions below the proof of Theorem \ref{eulerthm}) are determined in different mechanisms 
        clearly. 
        Based on this formulation and using the fact that the entropy of the downstream subsonic solution of a radial transonic shock solution on a shock location 
        monotonically increases as a shock location moves toward the exit of the nozzle (see Lemma \ref{Slem}), 
        we develop a new scheme to determine a shock location of a transonic shock solution of the full Euler system in a divergent nozzle: 1. Pseudo Free Boundary Problem 
        2. Determination of a shape of a shock location (see below the proof of Theorem \ref{eulerthm}). 
        

        
        In technical part, 
        we resolve the singularity issue arising in stream function formulations of the full Euler system. 
        A stream function formulation when it is formulated by using the Stokes' stream function (see \eqref{V}) has a singularity issue at the axis of symmetry. 
        We resolve this singularity issue by formulating a stream function formulation using the vector potential form of the stream function (see \S \ref{substream}) and solving a singular elliptic equation appearing in this stream function formulation as an elliptic system (see \S \ref{subelliptic}).  
        The stream function formulation formulated by using the vector potential form of the stream function 
        contains a singular elliptic equation. We transform this singular elliptic equation into a form of an elliptic system and, by solving the elliptic system form as an elliptic system, solve  the singular elliptic equation. 
        (We also use this approach 
        to prove the orthogonal completeness of eigenfunction of an associated Legendre problem of type $m=1$ with a general domain (see Lemma \ref{lemlegendre})). 
        Using the stream function formulation formulated by using the vector potential form of the stream function, we obtain the stability result for flows having $C^{1,\alpha}$ interior and $C^\alpha$ up to boundary regularity.

        
        This paper is organized as follows. In Section \ref{secProblemTheorem}, we present definitions and 
        a basic lemma used throughout this paper and introduce our problem and result. 
        In this section, we introduce the stream function formulation used in this paper. 
        In Section \ref{secPseudoFree}, we solve the Pseudo Free Boundary Problem. 
        In this section, we 
        study a linear boundary value problem for a singular elliptic equation and an initial value problem of a transport equation 
        appearing 
        in the Pseudo Free Boundary Problem, and 
        prove the unique existence of solutions of the Pseudo Free Boundary Problem. 
        In Section \ref{secdetermineshock}, we show 
        the existence and uniqueness of transonic shock solutions.
        In Section \ref{secappen}, we present some computations done by using the tensor notation given 
        in \S \ref{subelliptic}. 

        \section{Problem and Theorem}\label{secProblemTheorem}
        \subsection{Preliminary} 
        In this paper, we consider a 3-D steady 
        compressible invicid flow of an ideal polytropic gas. The motion of this flow is governed by the following full Euler system: 
        \begin{align}\label{euler}
        \begin{cases}
        \Div(\rho\bu)=0,\\
        \Div(\rho\bu\otimes\bu+p{\mathbb I})=0,\\
        \Div(\rho\bu B)=0
        \end{cases}
        \end{align}
        where $\rho$, $\bu$ and $p$ are the density, velocity and pressure of the flow, ${\mathbb I}$ is the $3\times 3$ identity matrix and $B$ is the Bernoulli invariant of the flow given by
        \begin{align}
        \label{bernoullidef}B=\frac{|\bu|^2}{2}
        + \frac{ \gam p}{(\gam-1)\rho}
        \end{align}
        for a constant $\gam>1$. 
        Types of the flow are classified by the quantity $M:=\frac{|\bu|}{c}$ called Mach number where $c$ is the sound speed of the flow given by 
        \begin{align*}
        c:=\sqrt{\frac{\gam p}{\rho}}
        \end{align*}
        for an ideal polytropic gas: if $M>1$, then a flow is called supersonic, if $M=1$, then it is called sonic and if $M<1$, then it is called subsonic. 
        It is generally known that types of the system varies depending on the value of $M$.
        If $M>1$, then the system is a hyperbolic system and if $M<1$, then the system is an elliptic-hyperbolic coupled system.
        %
        
        When the flow passes through a domain having a certain geometric structure or satisfies a certain boundary condition, 
        it may have a discontinuity across a surface in the domain in the direction of the flow. 
        Such a discontinuity is called a shock. 
        
        

        A shock solution of \eqref{euler} is defined as follows. 
        \begin{definition}
        	[Shock solution]\label{defshock}
        	Let $\Om$ be an open connected set in $\R^3$. Assume that a $C^1$ surface $\Gam$ in $\Om$ divides $\Om$ into two nonempty disjoint subset $\Om^\pm$ such that $\Om=\Om^+\cup\Gam\cup\Om^-$. 
        	Then a solution $(\rho,\bu,p)$ of \eqref{euler} is called a shock solution of \eqref{euler} with a shock $\Gam$ if 
            \begin{itemize}
            	\item[(i)] $(\rho,\bu,p)$ is in $(C^0(\ol{\Om^\pm})\cap C^1(\Om^\pm))^3$,
            	\item[(ii)] $(\rho,\bu,p)|_{\ol{\Om^-}\cap\Gam}\neq (\rho,\bu,p)|_{\ol{\Om^+}\cap \Gam}$,
            	\item[(iii)] $\rho\bu|_{\ol{\Om^-}} \cdot {\bm \nu}\neq 0$ on $\Gam$,
            	\item [(iv)] and $(\rho,\bu,p)$ satisfies  \eqref{euler} pointwisely in $\Om^\pm$ and the following Rankine-Hugoniot conditions: 
        	\begin{align}\label{RH}
        	[\rho\bu\cdot{\bm \nu}]_{\Gam}=[\bu\cdot{\bm \tau_1}]_{\Gam}=[\bu\cdot{\bm \tau_2}]_{\Gam}=[\rho(\bu\cdot{\bm \nu})^2+p]_{\Gam}=[B]_{\Gam}=0
        	\end{align}
        	where ${\bm \nu}$ is the unit normal vector field on $\Gam$ pointing toward $\Om^+$ and
        	${\bm \tau_i}$ for $i=1,2$ are unit tangent vector fields on $\Gam$ perpendicular to each other at each point on $\Gam$ 
        	and $$[F]_\Gam:=F|_{\ov{\Om^-}}(x)-F|_{\ov{\Om^+}}(x)\qd\tx{for}\qd x\in \Gam.$$
        	\end{itemize}
        \end{definition}
        A shock solution is said to be physically admissible if it satisfies the following entropy condition. 
         \begin{definition}[Entropy condition] 
         	Let $(\rho,\bu,p)$ be a shock solution of \eqref{euler} 
         	defined in Definition \ref{defshock}. Without loss of generality, assume that 
         	$\bu|_{\ol{\Om^-}}\cdot {\bm \nu}>0$ on $\Gam$. 
         	Then $S|_{\ol{\Om^+}\cap \Gam}>S|_{\ol{\Om^-}\cap \Gam}$ where $S:=\frac{p}{\rho^\gam}$  
         	is called the entropy condition. 
        \end{definition}
        \begin{definition}
        	In this paper, we call $S$ the entropy of $(\rho,\bu,p)$.
        \end{definition}
        Using the definition of a shock solution, a transonic shock solution of \eqref{euler} is defined as follows.
        
        \begin{definition}
        A shock solution of \eqref{euler} in Definition \ref{defshock} is called a transonic shock solution if it satisfies $M>1$ in $\ol{\Om^-}$ and $M<1$ in $\ol{\Om^+}$ or the otherway around. 
        \end{definition}
        \begin{remark}
        If a transonic shock solution of \eqref{euler} satisfies $\bu|_{\ol{\Om^-}}\cdot {\bm \nu}>0$ on $\Gam$, $M>1$ in $\ol{\Om^-}$ and $M<1$ in $\ol{\Om^+}$, then it satisfies the entropy condition. 
        \end{remark}

        
        In this paper, we deal with a 3-D axisymmetric transonic shock solution of \eqref{euler}. 
        For precise statement, we define an axisymmetric domain and 
        axisymmetric functions used in this paper. For later use, 
        we present a lemma used to deal with
        regularities of axisymmetric functions.  
        
        
        In this paper, we use the spherical coordinate system $(r,\theta,\vphi)$ given by the relation
        \begin{align}
        \label{xyzrtp}(x,y,z)=(r\sin\theta\cos\vphi,r\sin\theta\sin\vphi,r\cos\theta)
        \end{align}
        where $(x,y,z)$ is the cartesian coordinate system in $\R^3$. 
         The unit vectors in this coordinate system 
         are given by 
        \begin{align*}
        &\be_r=\sin\theta\cos\vphi\be_1+\sin\theta\sin\vphi\be_2+\cos\theta\be_3,\\
        &\be_\theta=\cos\theta\cos\vphi\be_1+\cos\theta\sin\vphi\be_2-\sin\theta\be_3,\\
        &\be_{\vphi}=-\sin\vphi\be_1+\cos\vphi\be_2
        \end{align*} 
        where $\be_i$ for $i=1,2,3$ are the unit vector in the $x$, $y$ and $z$ direction, respectively. 
        
        Using this spherical coordinate system, an axisymmetric domain and axisymmetric functions are defined as follows.  
        \begin{definition}\label{defaxi}
        	Let $\Om\subset\R^3$. $\Om$ is called axisymmetric if $(x,y,z)\in\Om$, then $$(\sqrt{x^2+y^2}\cos\vphi,\sqrt{x^2+y^2}\sin\vphi,z)\in \Om$$ for $\vphi\in [0,2\pi)$.  A function 
        	$f:\Om\ra \R$ is called axisymmetric if $f$ is independent of $\vphi$ as a function of the spherical coordinate system. A vector valued function $\bu:\Om \ra \R^3$ is called axisymmetric 
        	if $u_r=\bu\cdot \be_r$, $u_\theta=\bu\cdot \be_\theta$ and $u_\vphi=\bu\cdot \be_\vphi$ are axisymmetric.  
        \end{definition}
        \begin{definition}
        In this paper, when a velocity field $\bu$ is represented as $\bu=u_r\be_r+u_\theta\be_\theta+u_\vphi\be_\vphi$, $u_\vphi\be_{\vphi}$ is called the angular momentum component of $\bu$. 
        \end{definition}
        
        For later use, we present the following lemma that shows when an axisymmetric function in $C^k$ as a function of the spherical coordinate system is in $C^k$ as a function of the cartesian coordinate system. This lemma is obtained from \cite[Corollary 1]{MR2564196}. 
        \begin{lem}\label{lemaxi}
        	Let $\Om$ be an axisymmetric connected open set in $\R^3$ that does not contain the origin.  Suppose that a function f: $\Omega\rightarrow \mathbb{R}$ is axisymmetric. Then
        	\begin{itemize}
        		\item [(i)] $f$ and $f\be_r$ are in $C^k(\Om)$ for $k\in 0,1,2,\ldots$ if and only if $f$ is in $C^k$ as a function of spherical coordinate system in $\Om$ and $\pt_{\theta}^{2m+1}f=0$ for all $0\le m \le \lfloor \frac{k-1}{2}\rfloor$. 
        		\item [(ii)] $f\be_\theta$ and $f\be_\vphi$ are in $C^k(\Om)$ for $k\in 0,1,2,\ldots$ if and only if $f$ is in $C^k$ as a function of spherical coordinate system in $\Om$ and $\pt_{\theta}^{2m}f=0$ for all $0\le m \le \lfloor \frac{k}{2}\rfloor$. 
        	\end{itemize}
        \end{lem}
        In this paper, we use the same function notation when we represent an axisymmetric function as a function on the cartesian coordinate system or spherical coordinate system. 
        
        \subsection{Radial transonic shock solution}\label{subradial}
        Let $r_0$, $r_1$ and $\theta_1$ be constants such that $0<r_0<r_1$ and $0\le \theta_1<\pi$.
        Define a divergent nozzle by 
        \begin{align}
        \n:=\{(x,y,z)\in \R^3\;|\: r_0<r<r_1,\;0\le \theta<\theta_1 
        \}.
        \end{align}
        To introduce our problem and 
        for our later analysis, we study a radial transonic shock solution of \eqref{euler} in $\n$. 
        
        Fix positive constants $(\rho_{in},u_{in},p_{in})$ 
        satisfying $M_{in}(:=u_{in}/\sqrt{\frac{\gam p_{in}}{\rho_{in}}})>1$. 
        Let $(\bar\rho,\bar u\be_r,\bar p)$ 
        be a radial shock solution of \eqref{euler} in $\n$ with a shock
        $\Gam_t:=\{r=t\}\cap \n$ for some $t\in [r_0,r_1]$ satisfying 
        \begin{align}
        \label{iv}
        (\bar\rho,\bar u\be_r, \bar p)=(\rho_{in},u_{in}\be_r,p_{in})\qdon \Gam_{en}:=\pt\n\cap\{r=r_0,\;0\le\theta<\theta_1\}. 
        \end{align}
        Then 
        $(\bar\rho,\bar u, \bar p)$ 
        is a solution of 
        \begin{align}\label{reuler}
        \begin{cases}
        (r^2\bar\rho\bar u)^\p=0,\\
        \bar\rho \bar u \bar u^\p +\bar p^\p=0,\\
        \bar \rho \bar u \bar B^\p=0
        \end{cases}
        \end{align}
        with 
        \begin{align}
        \label{iiv}
        (\bar \rho,\bar u,\bar p)(r_0)=(\rho_{in},u_{in},p_{in})
        \end{align} 
        in $\ol{D_t^-}$, where $D_t^-:=\{r_0<r<t\}$, $^\p$ is the derivative with respect to $r$ and  $\bar B:=\frac{\bar u^2}{2}+\frac{\gam \bar p}{(\gam-1)\bar \rho}$, and is a solution of 
        \eqref{reuler} 
        with 
        \begin{align}\label{rrh}
        \begin{cases}
        \bar \rho \bar u(t)=\bar\rho\bar u|_{\ol{D_t^-}}(t),\\
        (\bar \rho \bar u^2+\bar p)(t)=(\bar \rho \bar u^2+\bar p)|_{\ol{D_t^-}}(t),\\
        \bar B(t)=
        \bar B|_{\ol{D_t^-}}(t)
        \end{cases}
        \end{align}
        in $\ol{D_t^+}$ where $D_t^+:=\{t<r<r_1\}$. 
        
        By \eqref{reuler} and \eqref{iiv}, 
        a solution 
        $(\bar \rho,\bar u,\bar p)$ of \eqref{reuler} with \eqref{iiv}
        satisfies
        \begin{align}\label{conserv}
        \begin{cases}
        r^2\bar \rho \bar u=m_0,\\
        \bar S=S_{in},\\
        \bar B=B_0
        \end{cases}
        \end{align}
        where 
        $m_0:=r_0^2\rho_{in} u_{in}$, $\bar S:=\frac{\bar p}{\bar \rho^\gam}$, $S_{in}:=\frac{p_{in}}{\rho_{in}^\gam}$ 
        and  $B_0:=\frac{u_{in}^2}{2}+\frac{\gam p_{in}}{(\gam-1)\rho_{in}}$ 
        on the domain where \eqref{reuler} with \eqref{iiv} has a unique solution $(\bar\rho,\bar u,\bar p)$. 
        By this fact, the local unique existence theorem for ODE and the condition that $M_{in}>1$, 
        \eqref{reuler} with \eqref{iiv} has a unique solution $(\bar \rho,\bar u,\bar p)$ satisfying $\bar M>1$ and $\bar M^\p>0$ in $[r_0,r_1]$ where $\bar M:=\bar u/\sqrt{\frac{\gam \bar p}{\bar \rho}}$. 
        Thus, \eqref{rrh} is well-defined for any $t\in [r_0,r_1]$. 
        From \eqref{rrh}, we obtain
        \begin{align}\label{rhodevelocity}
        \bar u(t)=\left.\frac{\bar K}{\bar u}\right|_{\ol{D_t^-}}(t)
        \end{align}
        and
        \begin{align}\label{Sexpression}
        \bar S(t)=(g({\bar M}^2)
        \bar S)|_{\ol{D_t^-}}(t)
        \end{align}
        where $\bar K:=\frac{2(\gam-1)}{\gam+1}
        \bar B$ 
        and 
        \begin{align}
        \label{geqn}g(x):=\frac{1}{\gam+1}(2\gam x-(\gam-1))\left(\frac{\gam-1}{\gam+1}+\frac{2}{\gam+1}\frac{1}{x}\right)^\gam.
        \end{align}
        Using \eqref{rhodevelocity}, the third equations of  \eqref{rrh} and \eqref{conserv} and $\bar M|_{\ol{D_t^-}}(t)>1$, 
        one can check that $(\bar\rho,\bar u,\bar p)$ satisfying \eqref{rrh} satisfies $\bar M(t) 
        <1$.
        By \eqref{reuler}, the first and third equation of \eqref{rrh}, \eqref{conserv} and \eqref{Sexpression}, 
        a solution $(\bar\rho,\bar u,\bar p)$ of \eqref{reuler} with \eqref{rrh} satisfies
        \begin{align}\label{conservsub}
        \begin{cases}
        r^2\bar\rho \bar u=m_0,\\
        \bar S=g({\bar M}|_{\ol{D_t^-}}^2(t))S_{in},\\
        \bar B=B_0
        \end{cases}
        \end{align}
        on the domain where \eqref{reuler} with \eqref{rrh}  has a unique solution $(\bar \rho,\bar u,\bar p)$. By these two facts and the local unique existence theorem for ODE, 
        \eqref{reuler} with \eqref{rrh} has a solution $(\bar \rho,\bar u,\bar p)$ satisfying $\bar M<1$ and $\bar M^\p<0$ in $\ol{D_t^+}$.  
        Therefore, 
        a radial shock solution $(\bar \rho,\bar u\be_r,\bar p)$ uniquely exists in $\n$ for each $t\in [r_0,r_1]$ and it is a radial transonic shock solution. 
        \begin{remark}\label{rmkrhop}
        	By \eqref{reuler} and the fact that a solution $(\bar \rho,\bar u, \bar p)$ of \eqref{reuler} with \eqref{rrh} uniquely exists in $\ol{D_{t}^+}$ satisfying \eqref{conservsub}, we have that 
        	a solution $(\bar \rho,\bar  u,\bar p)$ of \eqref{reuler} with \eqref{rrh} satisfies
        	$$\bar \rho^\p=\frac{2\bar \rho}{r}\frac{\bar  M^2}{1-\bar  M^2}\qdin D_t^+.$$
        	From this fact and $\bar M|_{\ol{D_{t}^+}}<1$ in $\ol{D_{t}^+}$, 
        	we obtain $\bar \rho|_{\ol{D_{t}^+}}^\p>0$ in $\ol{D_{t}^+}$. 
        \end{remark}
        One can see that the values of $(\bar \rho,\bar u,\bar p)|_{\ol{D_t^+}}$ at a fixed location $r$ in $\ol{D_{t}^+}$ are determined by the three conserved quantities in the right-hand sides of the equations in \eqref{conservsub}. 
        This combined with 
        the fact that the conserved quantity for $\bar S$ in $\ol{D_t^+}$ given in \eqref{conservsub} varies depending on $t$ (obtained from \eqref{conservsub} by using $({\bar M}|_{\ol{D_t^-}})^\p>0$ in $\ol{D_t^-}$ for any $t\in [r_0,r_1]$) implies that 
        the values of $(\bar \rho,\bar u,\bar p)|_{\ol{D_t^+}}$ at a fixed location $r$ in $\ol{D_t^+}$ vary depending on $t$. To represent this dependence, we write 
        $(\bar \rho,\bar u,\bar p)|_{\ol{D_t^+}}(r)$ and $\bar S|_{\ol{D_t^+}}(r)$ 
        as
        $(\bar \rho,\bar u,\bar p)|_{\ol{D_t^+}}(r;t)$ and $\bar S|_{\ol{D_t^+}}(r;t)$, respectively.
        
        The conserved quantity for $\bar S$ in $\ol{D_t^+}$ satisfies the following monotonicity. 
        \begin{lem}\label{Slem}
        	Let $r_0$, $r_1$, $t$ be positive constants such that $r_0\le t\le r_1$. Suppose that $(\bar \rho,\bar u,\bar p)$ is as above. Then there holds
        	\begin{align*}
        	\frac{d\bar S|_{\ol{D_t^+}}(t;t)}{dt}>0\qd\tx{for any $t\in [r_0,r_1]$}.
        	\end{align*}
        	\begin{proof}
        		Differentiate $\bar S|_{\ol{D_t^+}}(t;t)$ with respect to $t$. Then we have
        		\begin{align}\label{Sdiff}\frac{d\bar S|_{\ol{D_t^+}}(t;t)}{dt}=\left.\frac{d g(x)}{dx}\right|_{x=(\bar M|_{\ol{D_t^-}})^2(t)}\frac{d(\bar M|_{\ol{D_t^-}})^2(t)}{dt} S_{in}.
        		\end{align}
        		One can easily check that $g(1)=1$ and  $g^\p(x)>0$ for all $x>1$. By this fact, $\bar M|_{\ol{D_t^-}}(t)>1$ and $(\bar M|_{\ol{D_t^-}})^\p(t)>0$ for any $t\in [r_0,r_1]$, we obtain from \eqref{Sdiff} the desired result. 
        	\end{proof}
        \end{lem}
        From Lemma \ref{Slem}, we obtain the following result. 
        \begin{prop}\label{propressure}
        	Suppose that $r_0$, $r_1$, $t$ and $(\bar \rho,\bar u,\bar p)$ are as in Lemma \ref{Slem}. Then for any $t\in [r_0,r_1]$, 
        	$$
        	\frac{d\bar p|_{\ol{D_t^+}}}{dt}(r_1;t)<0.$$
        \end{prop}
        \begin{proof}
        	By the definitions of $\bar B$ and $\bar S$ and the first and third equation of \eqref{conservsub},
        	$$B_0=\left.\left(\frac{1}{2}(\frac{m_0}{r_1^2})^2(\frac{\bar S}{\bar p})^{\frac{2}{\gam}}+\frac{\gam}{\gam-1}{\bar p}^{1-\frac{1}{\gam}}{\bar S}^{\frac{1}{\gam}}\right)\right|_{\ol{D_t^+}}(r_1;t).$$
        	Differentiate this with respect to $t$. Then we get
        	\begin{align}\label{pdiff}
        	\frac{d\bar p|_{\ol{D_t^+}}(r_1;t)}{dt}=-\left.\left(\frac{(\frac{m_0}{r_1^2})^2(\frac{\bar S}{\bar p})^{\frac{2}{\gam}-1}+\frac{\gam}{\gam-1}{\bar p}^{2-\frac{1}{\gam}}{\bar S}^{\frac{1}{\gam}-1}}
        	{-(\frac{m_0}{r_1^2})^2(\frac{\bar S}{\bar p})^{\frac{2}{\gam}}+\gam{\bar p}^{1-\frac{1}{\gam}}{\bar S}^{\frac{1}{\gam}}}\right)\right|_{\ol{D_t^+}}(r_1;t)\frac{d\bar S|_{\ol{D_t^+}}(r_1;t)}{dt}.
        	\end{align} 
        	By $\bar M|_{\ol{D_t^+}}<1$ in $\ol{D_t^+}$, Lemma \ref{Slem} and the second equation of \eqref{conservsub}, we obtain from \eqref{pdiff} 
        	the desired result. 
        \end{proof}
        The above proposition implies that for any given 
        $p_c\in [p_1,p_2]$ where $p_1:=p_0|_{\ol{D_{r_1}^+}}(r_1;r_1)$ and $p_2:=p_0|_{\ol{D_{r_0}^+}}(r_1;r_0)$,
        there is a unique shock location $\Gam_t$ in $\ol{\n}$ such that 
        $(\ol\rho,\ol u\be_r,\ol p)$ satisfies $\ol p|_{\ol{D_{t}^+}}(r_1;t)=p_c$. 
        Hereafter, we fix a constant $p_c\in (p_1,p_2)$ and denote $t\in (r_0,r_1)$ such that 
        a radial transonic shock solution of \eqref{euler} satisfying \eqref{iv} and having a shock location $\Gam_t$ satisfies $p(r_1)=p_c$ by $r_s$. Also, we denote a solution $(\rho,u,p)$ of \eqref{reuler} with \eqref{iiv} and a solution $(\rho,u,p)$ of \eqref{reuler} with \eqref{rrh} for $t=r_s$ by  $(\rho_0^-,u_0^-,p_0^-)$ and  $(\rho_0^+,u_0^+,p_0^+)$, respectively, and denote $\frac{p_0^+}{{\rho_0^+}^\gam}$ by $S_0^+$. 
        By the local unique existence theorem for ODE, there exists a positive constant $\delta_1$ such that $(\rho_0^+,u_0^+,p_0^+)$ uniquely exists in $[r_s-\delta_1,r_s]$ satisfying $M_0^+(=u_0^+/\sqrt{\frac{\gam p_0^+}{\rho_0^+}})<1$. Fix any such $\delta_1$. 

        
        \subsection{Problem}
        Using the radial transonic shock solution given in the previous subsection, we present our problem.
        
        In this paper, we use the following weighted H\"older norm. 
        For a bounded connected open set $\Omega\subset \mathbb{R}^n$, let $\Gam$ be a closed portion of $\partial\Omega$.
        For ${\rm x},{\rm y}\in \Omega$, set $$\delta_{\rm x}:=\tx{dist}({\rm x},\Gam)\qdand \delta_{{\rm x},{\rm y}}:=\tx{min}(\delta_{\rm x},\delta_{\rm y}).$$
        For $k\in \mathbb{R}$, $\alpha\in (0,1)$ and $m\in \mathbb{Z}^+$, we define 
        \begin{align*}
        &||u||^{(k,\Gam)}_{m,0,\Omega}:=\sum_{0\leq |\beta|\leq m}\underset{{\rm x}\in\Omega}{\sup}\delta_{\rm x}^{\max(|\beta|+k,0)}|D^\beta u({\rm x})|\\
        &[u]^{(k,\Gam)}_{m,\alpha,\Omega}:=\sum_{|\beta|=m}\underset{{\rm x},{\rm y}\in\Omega,{\rm x}\neq {\rm y}}{\sup}\delta_{{\rm x},{\rm y}}^{\max(m+\alpha+k,0)}\frac{|D^\beta u({\rm x})-D^\beta u({\rm y})|}{|{\rm x}-{\rm y}|^\alpha}\\
        &||u||^{(k,\Gam)}_{m,\alpha,\Omega}:=||u||^{(k,\Gam)}_{m,0,\Omega}+[u]^{(k,\Gam)}_{m,\alpha,\Omega}
        \end{align*}
        where $D^\beta:=\pt_{x_1}^{\beta_1}\ldots\pt_{x_n}^{\beta_n}$ for a multi-index $\beta=(\beta_1,\ldots,\beta_n)$ with $\beta_i\in {\mathbb Z}^+$ for $i=1,\ldots,n$ and $|\beta|=\sum_{i=1}^n\beta_i$. We denote the completion of a set of smooth functions 
        under $||\cdot||_{m,\alpha,\Om}^{(k,\Gam)}$ norm 
        by $C^{m,\alpha}_{(k,\Gam)}(\Om)$. 
        
        \begin{figure}[htp]
        	\centering
        	\begin{psfrags}
        		\psfrag{Bent}[cc][][0.7][0]{$B=B_0$}
        		\psfrag{M1}[cc][][0.7][0]{$M>1$}
        		\psfrag{M2}[cc][][0.7][0]{$M<1$}
        		\psfrag{sup}[cc][][0.7][0]{$(\rho_-,\bu_-,p_-)$}
        		\psfrag{sub}[cc][][0.7][0]{$(\rho_+,\bu_+,p_+)$}
        		\psfrag{shock}[cc][][0.7][0]{$\Gam_f$}
        		\psfrag{pex}[cc][][0.7][0]{$p=p_{ex}$}
        		\psfrag{n}[cc][][0.7][0]{$\n$}
        		\includegraphics[scale=0.7]{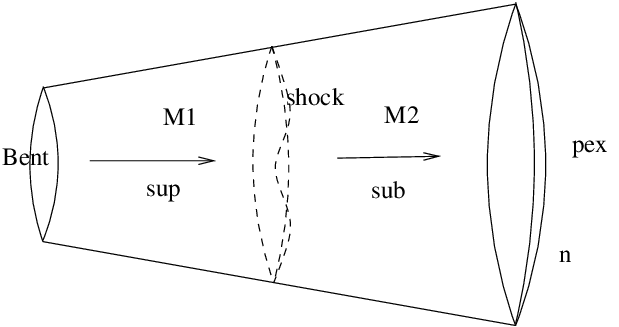}
        	\end{psfrags}
        \end{figure}
        Our problem is given as follows. 
        \begin{problem}[Transonic shock problem]\label{pb1}
        	Given an axisymmetric supersonic solution $(\rho_-,\bu_-,p_-)$ of \eqref{euler} in $\n$ satisfying the slip boundary condition 
        	\begin{align}\label{superslip}
        	\bu_-\cdot {\bm n}_w=0\qdon \Gam_w:=\pt\n\cap\{r_0<r<r_1,\;\theta=\theta_1\}
        	\end{align}
        	where ${\bm n}_w$ is the unit normal vector on $\Gam_w$,
        	\begin{align}
        	\label{B0ent}B=B_0\qdon \Gam_{en}
        	\end{align}
        	and 
        	\begin{align}\label{supesti}
        	||(\rho_-,\bu_-,p_-)-(\rho_0^-,u_0^-\be_r,p_0^-)||_{2,\alpha,\n}\le \sigma
        	\end{align}
        	and an axisymmetric exit pressure $p_{ex}$ on $\Gam_{ex}:=\pt\n \cap \{r=r_1,0\le\theta<\theta_1\}$ satisfying
        	\begin{align}\label{pexesti}
        	||p_{ex}-p_c||_{1,\alpha,\Gam_{ex}}^{(-\alpha,\pt\Gam_{ex})}\le \sigma
        	\end{align}
        	for a positive constant $\sigma$, 
        	find a shock location $\Gam_f:=\n \cap \{r=f(\theta)\}$ and a corresponding 
        	subsonic solution $(\rho_+,\bu_+,p_+)\in (C^0(\ol{\n_f^+})\cap C^1(\n_f^+))^3$ of \eqref{euler}
        	satisfying
        	\begin{itemize}
        		\item [(i)] the system \eqref{euler} in $\n_f^+:=\n\cap\{r>f(\theta)\}$, 
        		\item [(ii)] R-H conditions \eqref{RH} on $\Gam_f$,
        		\item [(iii)] the slip boundary condition 
        		\begin{align}
        		\label{slip}\bu_+\cdot {\bm n}_w=0\qd\tx{on}\qd\Gam_w^+:=\Gam_w\cap \{r>f(\theta_1)\}, 
        		\end{align}
        		\item [(iv)] and the exit pressure condition
        		\begin{align}\label{exitc}
        		p_+=p_{ex}\qd\tx{on}\qd\Gam_{ex}.
        		\end{align}
        	\end{itemize}
        \end{problem} 
        \begin{remark} 
        	It is generally known that a supersonic solution of \eqref{euler} is governed by a hyperbolic system. 
        	We assume that $(\rho_-,\bu_-,p_-)$ in Problem \ref{pb1} exists. 
        \end{remark}
        \begin{remark}
        	To simplify our argument,
        	we assumed in Problem \ref{pb1} that $(\rho_-,\bu_-,p_-)$ satisfies \eqref{B0ent}. 
        	This assumption will be used to reduce \eqref{euler} and \eqref{RH}
        	(see \S \ref{restate}). 
        	The result for Problem \ref{pb1} (Theorem \ref{eulerthm}) does not change if we consider a general perturbation of $(\rho_0^-,u_0^-\be_r,p_0^-)$ in Problem \ref{pb1}. 
        \end{remark}
        We study Problem \ref{pb1} using a stream function formulation of the full Euler system for an axisymmetric flow.
        We introduce a stream function formulation used in this paper in the next subsection. 
        \subsection{Stream function formulation}\label{substream}
        Let $\Om$ be an open simply connected axisymmetric set in $\R^3$. 
        Let $(\rho,\bu)$ be axisymmetric $C^1$ functions in $\Om$ satisfying the first equation of \eqref{euler} and $|\rho\bu|>0$. 
        For such $(\rho,\bu)$, 
        the Stokes'
        stream function for an axisymmetric flow of the full Euler system is defined 
        by
        \begin{align}\label{V}
        V(\rm x)=
        \int_{S_{\rm x}}\rho \bu \cdot {\bm\nu} d S \qd\mbox{for ${\rm x}\in \Om$}
        \end{align}
        where $S_{\rm x}$ is a simply connected $C^1$ surface in $\Om$ whose boundary is a circle centered at $z$-axis, parallel to $xy$-plane and passing through ${\rm x}$, and ${\bm \nu}$ is the unit normal vector on $S_{\rm x}$ pointing outward direction with respect to the cone-like
        domain 
        made by connecting $\pt S_{\rm x}$ and the origin by straight lines. By the first equation of \eqref{euler}, 
        the value of this function at ${\rm x}$ is independent of the choice of $S_{\rm x}$. 
        Since $\pt S_{\rm x}$ is axisymmetric, $V$ is axisymmetric in $\Om$. 
        
        By the first equation of \eqref{euler}, 
        $V$ is a constant on each stream surface of $\rho\bu$ in $\Om$. 
        Here, the stream surfaces of a vector field $\rho\bu$ in $\Om$ by 
        a set of surfaces made by collecting all the streamlines of $\rho\bu$ initiating from a point on a circle in $\Om$ centered at $z$-axis and parallel to $xy$-plane.
        By $|\rho\bu|>0$ in $\Om$, $V$ is a constant on each stream surface of $\rho\bu$ in $\Om$ and $V$ on each different stream surface of vector field $\rho\bu$ in $\Om$ is different from each other. 
        From these facts, we have that if we apply $\ngrad$, where 
        $\ngrad =\frac{1}{2\pi r \sin\theta } \left(\be_r\frac{\pt_{\theta} }{r}-\be_\theta\pt_r \right)$ which satisfies $\ngrad h\cdot\grad h=0$ and $\ngrad h \cdot \be_\vphi=0$ for a scalar function $h$, to $V$, then we have a vector field in $\Om$ tangent to the stream surfaces in $\Om$ 
        and having no $\be_{\vphi}$ component.
        Apply $\ngrad$ to $V$. Then we have
        \begin{align}\label{streamvelocity}
        \ngrad V=\rho u_r \be_r+\rho u_\theta \be_\theta
        \end{align}
        where $u_r=\bu\cdot \be_r$ and $u_\theta=\bu\cdot \be_\theta$.

        Using \eqref{streamvelocity}, we can reformulate the 
        full Euler system for an axisymmetric flow. But if we do this, then there is a singularity issue that can be seen in the relation $||\rho u_r\be_r+\rho u_\theta \be_\theta||_{\alpha,\n}=||\ngrad V||_{\alpha,\n}\not\le C||V||_{1,\alpha,\n}$ for any constant $C$.
        To avoid this issue, we use the following form of the stream function. 
       
        Let $\Phi\be_\vphi$ be an axisymmetric vector field in $\Om$ satisfying  
        \begin{align}\label{LL}
        \oint_{\pt S_{\rm x}} \Phi\be_\vphi\cdot d{\bm r}=\int_{S_{\rm x}}\rho\bu\cdot {\bm\nu} dS
        \end{align}
        where ${\bm r}$ is a parametrization of $\pt S_{\rm x}$ in a counter clockwise direction. Then by the definitions of $\Phi\be_\vphi$ and $V$, 
        \begin{align}
        \label{PhiVdef}\Phi=\frac{V}{2\pi r\sin\theta}
        \end{align}
        It is easily checked that
        \begin{align}\label{PhiV}
        \grad\times (\Phi\be_\vphi)=\ngrad V.
        \end{align} 
        From this relation and \eqref{streamvelocity} or \eqref{LL} directly, we have
        \begin{align}
        \label{Phiu}\grad\times (\Phi\be_\vphi)=\rho u_r\be_r+\rho u_\theta \be_\theta.
        \end{align}
        We call $\Phi\be_{\vphi}$ the vector potential form of the stream function.
       
        We reformulate the full Euler system for an axisymmetric flow using \eqref{Phiu}. 
        For our later analysis, when we reformulate the full Euler system using \eqref{Phiu}, we use the following form of the full Euler system representing the relation between $\grad S$ and $\grad\times \bu$ clearly 
        \begin{align}\label{ceuler}
        \begin{cases}
        \Div(\rho\bu)=0,\\
        (\grad\times\bu)\times \bu=\frac{\rho^{\gam-1}}{\gam-1}\grad S-\grad B,\\
        \rho \bu \cdot \grad B=0
        \end{cases}
        \end{align} 
        which is obtained 
        under the assumption that $(\rho,\bu,p)\in C^1$ and $\rho>0$. 
        
        We assume that $(\rho,\bu,p)$ in \eqref{ceuler} is axisymmetric and  reformulate \eqref{ceuler} using \eqref{Phiu}. 
        Rewrite \eqref{Phiu} as 
        $\bu=\frac{1}{\rho }\grad\times (\Phi\be_\vphi)+u_\vphi \be_\vphi.$
        Substitute this into the second equation of \eqref{ceuler}. 
        Then we obtain 
        \begin{align}\label{curlformulation}
        \left(\grad\times\left(\frac{1}{\rho}\grad\times (\Phi\be_\vphi)+u_\vphi  \be_\vphi\right)\right)\times \left(\frac{1}{\rho}\grad\times (\Phi\be_\vphi)+u_\vphi \be_\vphi\right)=\frac{\rho^{\gam-1}}{\gam-1}\grad S-\grad B.
        \end{align}
        From this equation, we can obtain three equations. From $\be_\vphi$-components of \eqref{curlformulation}, we get
        \begin{align}\label{uphiuphi}
        (\grad \times(u_\vphi \be_\vphi))\times \left(\frac{1}{\rho}\grad\times (\Phi\be_\vphi)\right)=0.
        \end{align}
        Define
        $L:=2\pi r \sin\theta u_\vphi$
        so that 
        $\grad\times(u_\vphi\be_\vphi)=\ngrad L$
        (see \eqref{PhiVdef} and \eqref{PhiV}). With this relation,  write \eqref{uphiuphi} 
        as
        \begin{align*}
        \ngrad L\times  \left(\frac{1}{\rho}\grad\times (\Phi\be_\vphi)\right)=0.
        \end{align*}
        From this equation, we obtain
        \begin{align}\label{Leqn}
        \grad\times (\Phi\be_\vphi)\cdot\grad L=0.
        \end{align}
        From $\be_{\theta}$-components of \eqref{curlformulation}, we have
        \begin{multline*}
        \left(\grad\times\left(\frac{1}{\rho}\grad\times (\Phi\be_\vphi)\right)\right)\times \left(\frac{1}{\rho}\grad\times (\Phi\be_\vphi)\cdot \be_r\right)\be_r+(\grad\times(u_\vphi\be_\vphi)\cdot\be_r)\be_r\times u_\vphi \be_\vphi\\
        =\left(\frac{\rho^{\gam-1}}{\gam-1}\frac{\pt_\theta S}{r}-\frac{\pt_\theta B}{r}\right) \be_\theta.
        \end{multline*}
        With the definition of $L$, rewrite the above equation. Move the rewritten term into the right-hand side of the equation. And then multiply $\be_r\times$ to the resultant equation. Then we obtain
        \begin{multline}\label{Phieqn}
        \left(\frac{1}{\rho}\grad\times (\Phi\be_\vphi)\cdot \be_r\right)\grad\times\left(\frac{1}{\rho}\grad\times (\Phi\be_\vphi)\right) \\
        =\left(\frac{L}{2\pi r\sin\theta}\grad \times\left(\frac{L}{2\pi r\sin\theta}\be_\vphi\right)\cdot\be_r+\frac{\rho^{\gam-1}}{\gam-1}\frac{\pt_\theta S}{r}-\frac{\pt_\theta B}{r}\right)\be_\vphi.
        \end{multline}
        Using the third equation of \eqref{ceuler}, we obtain from $\bu$-components of \eqref{curlformulation}
        \begin{align}\label{Seuler}
        \rho\bu\cdot \grad S=0.
        \end{align}
        With the assumption that $(\rho,\bu,p)$ are axisymmetric and \eqref{Phiu}, this equation can be written as
        \begin{align}\label{Seqn}
        \grad\times (\Phi\be_\vphi)\cdot \grad S=0.
        \end{align}
        Hence, we obtain 
        from \eqref{curlformulation} three equations: \eqref{Leqn}, \eqref{Phieqn} and \eqref{Seqn}. 
        Finally, in the same way that we obtained \eqref{Seqn} from \eqref{Seuler}, we obtain 
        from the third equation of \eqref{ceuler} 
        \begin{align}
        \label{Beqn}\grad\times (\Phi\be_\vphi)\cdot \grad B=0.
        \end{align}
        Combining \eqref{Leqn}, \eqref{Phieqn}, \eqref{Seqn} and \eqref{Beqn}, we have the following stream function formulation of the full Euler system for an axisymmetric flow
        \begin{align}\label{veuler}
        \begin{cases}
        \left(\frac{1}{\rho}\grad\times (\Phi\be_\vphi)\cdot \be_r\right)\grad\times\left(\frac{1}{\rho}\grad\times (\Phi\be_\vphi)\right)\\
        \qd\qd\qd\qd\qd\qd\qd
        =\left(\frac{L}{2\pi r\sin\theta}\grad \times\left(\frac{L}{2\pi r\sin\theta}\be_\vphi\right)\cdot\be_r+\frac{\rho^{\gam-1}}{\gam-1}\frac{\pt_\theta S}{r}-\frac{\pt_\theta B}{r}\right)\be_\vphi,\\
        \grad\times (\Phi\be_\vphi)\cdot\grad L=0,\\
        \grad\times (\Phi\be_\vphi)\cdot \grad S=0,\\
        \grad\times (\Phi\be_\vphi) \cdot \grad B=0.
        \end{cases}
        \end{align}
        Note that the first equation of \eqref{euler} is omitted in \eqref{veuler} because for a 3-D axisymmetric flow, the first equation of \eqref{euler} is reduced to $\Div(\rho u_r\be_r+\rho u_\theta\be_\theta)=0$ and this equation is directly satisfied by $\Div(\grad\times(\Phi\be_{\vphi}))=0$ if $\Phi\in C^2(\Om)$. 
        Also, note that 
        when $(\rho,\bu,p)$ in \eqref{euler} is in a sufficiently small perturbation of $(\rho_0^+,u_0^+\be_r,p_0^+)$ or $(\rho_0^-,u_0^-\be_r,p_0^-)$, 
        then unknowns of \eqref{veuler} can be $(\Phi\be_\vphi,L,S,B)$. 
        This fact (for the first case) is checked via the following lemma. 
        \begin{lem} \label{lemRho} 
        	Let $\Om$ be an axisymmetric connected open subset of $\n_{r_s-\delta_1}^+$.
        	There exist positive constants $\delta_{2,\Om}$ and $\delta_{3,\Om}$ 
        	and a function $\varrho: B_{\delta_{2,\Om},\Om}^{(1)}\ra  B_{\delta_{3,\Om},\Om}^{(2)}$, where
        	\begin{multline*}
        	B_{\delta,\Om}^{(1)}:=\{(\rho u_r\be_r+\rho u_\theta \be_\theta,u_\vphi\be_\vphi,S,B)|\in (C^0(\ol\Om))^4\;|\\\sup_{\Om}\{ |\rho u_r\be_r+\rho u_\theta\be_\theta-\rho_0^+u_0^+\be_r|+|u_\vphi\be_{\vphi}|+|S-S_0^+|+|B-B_0|\}\le \delta\},
        	\end{multline*}
        	and
        	\begin{align*}
        	B_{\delta,\Om}^{(2)}:=\{\rho\in C^0(\ol\Om)\;|\;\sup_{\Om} |\rho-\rho_0^+|\le \delta\},
        	\end{align*}
        	such that 
        	\begin{align}
        	\label{rhodef}\rho^2B-\frac{|\rho u_r\be_r+\rho u_\theta \be_\theta|^2}{2}-\frac{\rho^2}{2}|u_\vphi\be_{\vphi}|^2-\frac{\gam S \rho^{\gam+1}}{\gam-1}=0
        	\end{align}
        	if and only if 
        	\begin{align}\label{rhoex}
        	\rho=\varrho(\rho u_r\be_r+\rho u_\theta \be_\theta,u_\vphi\be_{\vphi},S,B),
        	\end{align}
        	for all $(\rho u_r\be_r+\rho u_\theta \be_\theta,u_\vphi\be_{\vphi},S,B)\in B_{\delta_{2,\Om},\Om}^{(1)}$ and $\rho\in B_{\delta_{3,\Om},\Om}^{(2)}$.
        	\begin{proof}
        		Using the definition of $B$ given in \eqref{bernoullidef}, we define 
        		\begin{multline}
        		\label{bdef}
        		b(\rho,\rho u_r\be_r+\rho u_\theta \be_\theta,u_\vphi\be_\vphi,S,B)=\rho^2B-\frac{|\rho u_r\be_r+\rho u_\theta \be_\theta|^2}{2}-\frac{\rho^2}{2}|u_\vphi\be_\vphi|^2-\frac{\gam S \rho^{\gam+1}}{\gam-1}.
        		\end{multline}
        		Then $b\in C^\infty$ with respect to $(\rho,\rho u_r\be_r+\rho u_\theta \be_\theta,u_\vphi\be_{\vphi},S,B)$,
        		\begin{align*}
        		\pt_{\rho} b(\rho_0^+,\rho_0^+u_0^+\be_r,0,S_0^+,B_0)=\rho_0^+({u_0^+}^2-\gam S_0^+{\rho_0^+}^{\gam-1})<0,
        		\end{align*}
        		and by the third equation of \eqref{conservsub}, 
        		$$b(\rho_0^+,\rho_0^+u_0^+\be_r,0,S_0^+,B_0)=0.$$
        		With these facts, we apply the implicit function theorem to $b(\rho,\rho u_r\be_r+\rho u_\theta \be_\theta,u_\vphi\be_{\vphi},S,B)$. Then we obtain the desired result.  
        	\end{proof}
        \end{lem}
        \begin{remark}
        	Hereafter, $\delta_2$ and $\delta_3$ denote some constants $\delta_{2,\Om}$ and $\delta_{3,\Om}$ in Lemma \ref{lemRho} for $\Om=\n_{r_s-\delta_1}^+$. 
        	Hereafter, $\varrho$ denotes $\varrho$ in Lemma \ref{lemRho} for $\Om=\n_{r_s-\delta_1}^+$. 
        	One can see that if $\Om\subset\n_{r_s-\delta_1}^+$, then for all $(\grad\times(\Phi\be_{\vphi}),\frac{L}{2\pi r \sin\theta}\be_{\vphi},S,B)\in B_{\delta_{2},\Om}^{(1)}$ and $\rho\in B_{\delta_{3},\Om}^{(2)}$, \eqref{rhodef} holds if and only if \eqref{rhoex} holds. 
        \end{remark}
    
        In order to study 3-D axisymmetric transonic shock solution of the full Euler system using the stream function formulation of the full Euler system for an axisymmetric flow above, 
        we reformulate \eqref{RH} with respect to the variables in \eqref{veuler}. 
        
        Assume that $\Gam$ in \eqref{RH} is an axisymmetric $C^1$ surface. 
        Let ${\bm \tau}_2$ and ${\bm \tau}_1$ in \eqref{RH} be $\be_\vphi$ and the unit tangent vector field on $\Gam$ perpendicular to $\be_\vphi$ and satisfying ${\bm \nu}\cdot ({\bm \tau}_1\times \be_\vphi)>0$, respectively, where ${\bm \nu}$ is the unit normal vector field on $\Gam$ pointing toward $\Om^+$. 
        By the definition of $V$ given in \eqref{V}, the first equation of \eqref{RH} is written as
        $[V]_{\Gam}=0$. Rewrite this equation using \eqref{PhiVdef}. Then we have
        $$[\Phi\be_\vphi]_\Gam=0.$$ 
        With \eqref{Phiu}, rewrite the second equation of \eqref{RH}. Then we get
        $$\frac{1}{\rho}\grad\times (\Phi\be_\vphi)\cdot {\bm \tau}_1=\bu_-\cdot {\bm \tau}_1\qdon \Gam.$$
        By the definition of $L$, 
        the third equation of \eqref{RH} can be written as
        $$[L]_{\Gam}=0.$$
        From the second, third and fifth equation of \eqref{RH}, we can obtain
        \begin{align}\label{rh1bernoulli'}
        [B_s]_\Gam=0
        \end{align}
        where $B_s:=\frac{(\bu\cdot{\bm\nu})^2}{2}+\frac{\gam p}{(\gam-1)\rho}$.
        In the same way that \eqref{Sexpression} is obtained from \eqref{rrh}, we obtain from the first and fourth equation of \eqref{RH} and \eqref{rh1bernoulli'} 
        $$S_+=g\left(\left(\frac{\bu_-\cdot {\bm\nu}}{c_-}\right)^2\right)S_-\qd \tx{on}\qd \Gam$$
        where 
        $g(x)$ is a function defined in $\eqref{geqn}$ and variables with  lower indices $\pm$ denote variables in $\Om^\pm$, respectively. 
        Combining these reformulated equations of \eqref{RH} and the fifth equation of \eqref{RH}, we have the following stream function formulation of 
        \eqref{RH}:
        \begin{align}\label{RH2}
        \begin{cases}
        [\Phi\be_\vphi]_\Gam=0,\\
        \frac{1}{\rho}\grad\times (\Phi\be_\vphi)\cdot {\bm \tau}_1=\bu_-\cdot{\bm \tau}_1\qdon \Gam,\\
        [L]_{\Gam}=0,\\
        S_+=g\left(\left(\frac{\bu_-\cdot {\bm\nu}}{c_-}\right)^2\right)S_-\qd \tx{on}\qd \Gam,\\
        [B]_\Gam=0.
        \end{cases}
        \end{align}
        \subsection{Restatement of Problem \ref{pb1} using the stream function formulation and main result} \label{restate} 
        Using the stream function formulation in the previous subsection, we restate Problem \ref{pb1}.
        We first reduce \eqref{euler} and \eqref{RH} in Problem \ref{pb1}.
        
        In Problem \ref{pb1}, we assumed that a supersonic solution $(\rho_-,\bu_-,p_-)$ of \eqref{euler} satisfies \eqref{B0ent}. By this assumption, the third equation of \eqref{euler} in $\n_f^-:=\n\cap \{r<f(\theta)\}$ and $\n_f^+$, the fifth equation of \eqref{RH},  \eqref{superslip} and \eqref{slip}, $(\rho_+,\bu_+,p_+)$ we find in Problem \ref{pb1} must satisfy $B=B_0$ in $\ol{\n_f^+}$. To simplify our argument, we assume that $(\rho_+,\bu_+,p_+)$ in Problem \ref{pb1} satisfies 
        $B=B_0$ in $\ol{\n_f^+}$. Under this assumption, \eqref{euler} and \eqref{RH} that $(\rho_+,\bu_+,p_+)$ in Problem \ref{pb1} satisfies are reduced to the first and second equation of \eqref{euler} and the first, second, third and fourth equation of \eqref{RH}. 
        
        Then we present the  stream function formulations of \eqref{euler} and \eqref{RH} satisfied by $(\rho_+,\bu_+,p_+)$ in Problem \ref{pb1}. 
        By \eqref{veuler} and \eqref{RH2}, the stream function formulations of \eqref{euler} and \eqref{RH} satisfied by $(\rho_+,\bu_+,p_+)$ in Problem \ref{pb1} and reduced by using the assumption that $(\rho_+,\bu_+,p_+)$ satisfies 
        $B=B_0$ in $\ol{\n_f^+}$ are given as 
        \begin{align}
        \label{eqV}
        &\left(\frac{1}{\rho_+}\grad\times (\Phi_+\be_\vphi)\cdot \be_r\right)\grad\times\left(\frac{1}{\rho_+}\grad\times (\Phi_+\be_\vphi)\right)\\
        &\nonumber\qd\qd\qd\qd\qd=\left(\frac{L_+}{2\pi r\sin\theta}\grad \times\left(\frac{L_+}{2\pi r\sin\theta}\be_\vphi\right)\cdot\be_r+\frac{\rho_+^{\gam-1}}{\gam-1}\frac{\pt_\theta S_+}{r}\right)\be_\vphi,\\
        &\label{eqLambda}\grad\times (\Phi_+\be_\vphi)\cdot\grad L_+=0,\\
        &\label{eqS}\grad\times (\Phi_+\be_\vphi)\cdot \grad S_+=0
        \end{align}
        in $\n_f^+$
        and 
        \begin{align}
        &\label{rhV}\Phi_+\be_\vphi=\Phi_-\be_\vphi\qdon \Gam_f,\\
        &\label{rhtau}\frac{1}{\rho_+}\grad\times (\Phi_+\be_\vphi)\cdot {\bm\tau}_f=\bu_-\cdot{\bm\tau}_f\qdon \Gam_f,\\
        &\label{rhLambda}L_+=L_-\qdon \Gam_f,\\
        &\label{rhS}S_+=g\left(\left(\frac{\bu_-\cdot {\bm \nu}_f}{c_-}\right)^2\right)S_-\qd \tx{on}\qd \Gam_f
        \end{align}
        where 
        ${\bm \nu_f}$ is the unit normal vector on $\Gam_f$ pointing toward $\n_f^+$ and ${\bm \tau}_f$ is the unit tangential vector on $\Gam_f$ perpendicular to $\be_\vphi$ and satisfying ${\bm \nu}_f\cdot ({\bm \tau}_f\times \be_\vphi)>0$. 
        Here, $(\Phi_-\be_\vphi,L_-,S_-)$ 
        and  $(\Phi_+\be_\vphi,L_+,S_+)$ are 
        $(\Phi\be_\vphi,L,S)$ 
        given by the definitions of $\Phi\be_{\vphi}$, $L$ and $S$ 
        for $(\rho,\bu,p)=(\rho_-,\bu_-,p_-)$ and $(\rho_+,\bu_+,p_+)$, 
        respectively. 
        
        Determine $\rho_+$ in \eqref{eqV} and \eqref{rhtau}. In Problem \ref{pb1}, we consider the case that $(\rho_-,\bu_-,p_-)$ and $p_{ex}$ are in sufficiently small perturbations of $(\rho_0^-,u_0^-\be_r,p_0^-)$ and $p_c$ 
        so that $(\rho_+,\bu_+,p_+)$ in Problem \ref{pb1} is in a small perturbation of $(\rho_0^+,u_0^+\be_r,p_0^+)$ such that $\rho_+$ is uniquely determined by
        $(\rho_+ u_{+,r}\be_r+\rho_+ u_{+,\theta}\be_{\theta},u_{+,\vphi}\be_{\vphi},S_+,B_0)$ 
        where $u_{+,r}:=\bu_+\cdot \be_r$, $u_{+,\theta}:=\bu_+\cdot \be_{\theta}$ and $u_{+,\vphi}=\bu_+\cdot\be_{\vphi}$ (see Lemma \ref{lemRho}). 
        To find such $(\rho_+,\bu_+,p_+)$ using \eqref{eqV}-\eqref{rhS}, we set
        $\rho_+$ in \eqref{eqV} and \eqref{rhtau} to be
        \begin{align}\label{rhodefinition}
        \rho_+=\varrho(\grad\times(\Phi_+\be_\vphi),\frac{L_+}{2\pi r \sin\theta}\be_\vphi,S_+,B_0)
        \end{align}
        where $\varrho$ is a function given in Lemma \ref{lemRho}.  
        Hereafter, 
        to simplify notation, we write $\varrho(\grad\times(\Phi\be_\vphi),\frac{L}{2\pi r \sin\theta}\be_\vphi,S,B_0)$ as $\varrho(\grad\times(\Phi\be_\vphi),\frac{L}{2\pi r \sin\theta}\be_\vphi,S)$. 

        We will find a subsonic solution of \eqref{euler} using \eqref{eqV}-\eqref{eqS}. To do this,
        we define a subsonic solution of \eqref{eqV}-\eqref{eqS}. 
        Using the definition of subsonic solution of \eqref{euler}, we define a subsonic solution of \eqref{eqV}-\eqref{eqS} by a solution $(\Phi_+\be_{\vphi},L_+,S_+)$  of \eqref{eqV}-\eqref{eqS} satisfying 
        \begin{multline}\label{subdef}|\frac{1}{\varrho(\grad\times(\Phi_+\be_\vphi),\frac{L_+}{2\pi r \sin\theta}\be_{\vphi},S_+)}\grad\times(\Phi_+\be_\vphi)+\frac{L_+}{2\pi r \sin\theta}\be_{\vphi}|^2<\\
        \gam S_+(\varrho(\grad\times(\Phi_+\be_\vphi),\frac{L_+}{2\pi r \sin\theta}\be_{\vphi},S_+))^{\gam-1}.
        \end{multline}

        
        We rewrite the boundary conditions in Problem \ref{pb1} with respect to the variables in \eqref{veuler}. Denote 
        $V$ given by 
        the definition of $V$ for $(\rho,\bu,p)=(\rho_-,\bu_-,p_-)$ and $(\rho_+,\bu_+,p_+)$ by 
        $V_-$ and  
        $V_+$, respectively. 
        By \eqref{streamvelocity}, \eqref{superslip} and \eqref{slip} can be written as $V_-=V_-(r_0,\theta_1)$ on $ \Gam_w$ and $V_+=V_+(f(\theta_1),\theta_1)$ on $\Gam_w^+.$ Rewrite this in the vector potential form. 
        Then we have $
        \Phi_-\be_\vphi=\frac{r_0\Phi_-(r_0,\theta_1)}{r}\be_\vphi$ on  $\Gam_w$
        and
        $
        \Phi_+\be_\vphi=\frac{f(\theta_1)\Phi_+(f(\theta_1),\theta_1)}{r}\be_\vphi$ on $ \Gam_w^+.$ Combine these relations with \eqref{rhV} and the continuity condition for $\Phi\be_{\vphi}$ at $\ol{\Gam_f}\cap \ol{\Gam_w^+}$. Then we obtain 
        \begin{align}
        \label{slipc}
        \Phi_+\be_\vphi=\frac{r_0\Phi_-(r_0,\theta_1)}{r}\be_\vphi\qdon \Gam_w^+.
        \end{align}
        Using the definition of $S$ and 
        \eqref{rhodefinition}, rewrite \eqref{exitc}. Then we have
        \begin{align}
        \label{pressurestream}S_+(\varrho(\grad\times (\Phi_+ \be_\vphi),\frac{L_+}{2\pi r\sin\theta}\be_\vphi,S_+))^\gam=p_{ex}\qdon \Gam_{ex}.
        \end{align}

        Finally, we add some continuity 
        condition for $\Phi_+\be_{\vphi}$ 
        in the restatement of Problem \ref{pb1} using the stream function formulation. 
        By \eqref{euler} and \eqref{slip}, $(\rho_+,\bu_+)$ we find in Prolem \ref{pb1} must satisfy that 
        the total outgoing flux for $\rho_+\bu_+$ on $\Gam_{ex}$ is equal to the total incomming flux for $\rho_+\bu_+$ on any cross section of $\n_f^+$ whose all boundary points intersect with $\Gam_w^+$. Using \eqref{V}, this statement can be expressed as 
        $\lim_{\theta\ra \theta_1}V_+(r_1,\theta)=\lim_{r\ra r_1}V_+(r,\theta_1).$
        Rewrite this in the vector potential form. Then we have 
        \begin{align}
        \label{compc}
        \lim_{\theta\ra \theta_1}(\Phi_+\be_\vphi)(r_1,\theta)=\lim_{r\ra r_1}(\Phi_+\be_\vphi)(r,\theta_1). 
        \end{align}
        Since this condition cannot be achieved from \eqref{pressurestream} (this will be seen in \S \ref{sublinear}), we include this condition in the restatement of Problem \ref{pb1} using the stream function formulation.

        Using the equations and boundary conditions obtained above, Problem \ref{pb1} is restated as follows:
        \begin{problem}\label{pb2}
        	Given an axisymmetric supersonic solution $(\rho_-,\bu_-,p_-)$ of \eqref{euler} and an axisymmetric exit pressure $p_{ex}$ as in Problem \ref{pb1}, find a shock location $\Gam_f=\n\cap \{r=f(\theta)\}$ and a corresponding 
        	subsonic solution $(\Phi_+\be_\vphi,L_+,S_+)$ of \eqref{eqV}-\eqref{eqS} satisfying 
        	\begin{itemize}
        		\item [(i)] the system \eqref{eqV}-\eqref{eqS} in $\n_f^+$,
        		\item [(ii)] the R-H conditions \eqref{rhV}-\eqref{rhS},
        		\item [(iii)] the slip boundary condition \eqref{slipc},
        		\item [(iv)] the exit pressure condition \eqref{pressurestream}, 
        		\item [(v)] the compatibility condition \eqref{compc}. 
        	\end{itemize}
        \end{problem}
        
        Let $S^{2,\theta_1}:=\{(x,y,z)\in \R^3\;|\;
        r=1,\; 0\le \theta<\theta_1 \}$. 
        A function $f$ representing a shock location $\Gam_f$  
        can be considered as a function on $S^{2,\theta_1}$. Using this fact and the stereographic projection from $(0,0,-1)$ onto the plane $z=1$ passing through $S^{2,\theta_1}$, we see that $f$ can be regarded as a function on $\Lambda$ where $\Lambda:=\{(x,y)\in \R^2\;|\; \sqrt{x^2+y^2}< 2\tan\frac{\theta_1}{2}\}$. 
        Thus, $f$ can be regarded as a function on $\Lambda$ or $(0,\theta_1)$. 
        In this paper, we regard $f$ in both ways. 
        To simplify our notation, we use the same function notation when we represent $f$ 
        as a function on $\Lambda$ or $(0,\theta_1)$. 
        
        Hereafter, we denote $\Phi\be_{\vphi}$, $L$ and $V$ given by the definitions of $\Phi\be_{\vphi}$, 
        $L$ and $V$ 
        for $(\rho,\bu,p)=(\rho_0^\pm,u_0^\pm \be_r,p_0^\pm)$ 
        by $\Phi_0^\pm\be_{\vphi}$, $L_0^\pm$ and $V_0^\pm$, 
        respectively.
        To simplify our notation, hereafter, we denote $(\Phi_+\be_\vphi,L_+,S_+)$ by $(\Phi\be_\vphi,L,S)$.
        
        Our result of Problem \ref{pb2}, the main result in this paper, is given as follows. 
        \begin{thm}\label{mainthm}
        	Let $\alpha\in (\frac{2}{3},1)$.
        	There exists a positive constant $\sigma_1$ depending on $(\rho_{in},u_{in},p_{in})$, $p_c$, $\gam$, $r_0$, $r_1$, $\theta_1$ and $\alpha$ such that if $\sigma\in(0,\sigma_1]$, then 
        	Problem \ref{pb2} has a solution 
        	$(f,\Phi\be_\vphi,L,S)$ satisfying the estimate	
        	\begin{multline}\label{mainestimate}
        	||f-r_s||_{2,\alpha,\Lambda}^{(-1-\alpha,\pt\Lambda)}\\
        	+||\grad\times((\Phi-\Phi_0^+)\be_\vphi)||_{1,\alpha,\n_f^+}^{(-\alpha,\Gam_w^+)}+||\frac{L}{2\pi r \sin\theta}\be_\vphi||_{1,\alpha,\n_f^+}^{(-\alpha,\Gam_w^+)}+||S-S_0^+||_{1,\alpha,\n_f^+}^{(-\alpha,\Gam_w^+)}\le C  \sigma
        	\end{multline}
        	where $C$ is a positive constant depending on $(\rho_{in},u_{in},p_{in})$, $p_c$, $\gam$, $r_0$, $r_1$, $\theta_1$ and $\alpha$.
        	Furthermore, this solution is unique in the class of functions satisfying \eqref{mainestimate}. 
        \end{thm}
        Hereafter, we say that a constant depends on the data if a constant depends on $(\rho_{in},u_{in},p_{in})$, $p_c$, $\gam$, $r_0$, $r_1$, $\theta_1$ and $\alpha$.
        
        The following result of Problem \ref{pb1} is obtained from Theorem \ref{mainthm}. 
        \begin{thm}\label{eulerthm} Let $\alpha\in (\frac{2}{3},1)$.
        	There exists a positive constant $\sigma_2$ depending on the data such that if $\sigma\in(0,\sigma_2]$, then 
        	Problem \ref{pb1} has a solution $(f,\rho_+,\bu_+,p_+)$ satisfying the estimate
        	\begin{multline}\label{thmestimate}
        	||f-r_s||_{2,\alpha,\Lambda}^{(-1-\alpha,\pt\Lambda)}\\
        	+||\rho_+-\rho_0^+||_{1,\alpha,\n_f^+}^{(-\alpha,\Gam_w^+)}+||\bu_+ -u_0^+\be_r||_{1,\alpha,\n_f^+}^{(-\alpha,\Gam_w^+)}+||p_+ -p_0^+||_{1,\alpha,\n_f^+}^{(-\alpha,\Gam_w^+)}\le C\sigma
        	\end{multline} 
        	where $C$ is a positive constant depending on the data.
        	Furthermore, this solution is unique in the class of functions satisfying \eqref{thmestimate}. 
        	
        \end{thm}
        \begin{proof}
        1. Let $(f,\Phi\be_{\vphi},L,S)$ be a solution of Problem \ref{pb2} given in Theorem \ref{mainthm} for $\sigma\in(0,\ol{\sigma}_2]$ where $\ol{\sigma}_2$ is a positive constant less than or equal to $\sigma_1$ and to be determined later.
        Then since $(f,\Phi\be_{\vphi},L,S)$ is a solution of Problem \ref{pb2}, $\varrho(\grad\times(\Phi\be_{\vphi}),\frac{L}{2\pi r\sin\theta}\be_{\vphi},S)$ is well-defined in $\ol{\n_f^+}$. 
        Define $\rho_+:=\varrho(\grad\times(\Phi\be_{\vphi}),\frac{L}{2\pi r\sin\theta}\be_{\vphi},S)$, $\bu_+:=\frac{\grad \times(\Phi\be_{\vphi})}{\varrho(\grad\times(\Phi\be_{\vphi}),\frac{L}{2\pi r\sin\theta}\be_{\vphi},S)}$ and $p_+:=S(\varrho(\grad\times(\Phi\be_{\vphi}),\frac{L}{2\pi r\sin\theta}\be_{\vphi},S))^\gam.$
        Then since $(\Phi\be_{\vphi},L,S)$ satisfies \eqref{eqV}-\eqref{eqS} in $\n_f^+$, \eqref{rhV}-\eqref{rhS}, \eqref{slipc}, \eqref{pressurestream}, 
        and 
        $$b(\varrho(\grad\times(\Phi\be_{\vphi}),\frac{L}{2\pi r\sin\theta}\be_{\vphi},S),\grad\times(\Phi\be_{\vphi}),\frac{L}{2\pi r\sin\theta}\be_{\vphi},S,B_0)=0\qdin\ol{\n_f^+}$$ 
        where $b$ is a funtion defined in \eqref{bdef}, $(\rho_+,\bu_+,p_+)$ satisfies the second and third equation of \eqref{euler} in $\n_f^+$, \eqref{RH} on $\Gam_f$, \eqref{slip} and \eqref{exitc}. Furthermore, since 
        $\grad\cdot(\grad\times(\Phi\be_{\vphi}))=0$ and $\grad\cdot (\varrho(\grad\times(\Phi\be_{\vphi}),\frac{L}{2\pi r\sin\theta}\be_{\vphi},S)\frac{L}{2\pi r \sin\theta}\be_{\vphi})=0$ in $\n_f^+$, 
        $(\rho_+,\bu_+,p_+)$ satisfies the first equation of \eqref{euler} in $\n_f^+$. 
        Since 
        $(\Phi\be_\vphi,L,S)$ is a subsonic solution of \eqref{eqV}-\eqref{eqS}, 
        $(\rho_+,\bu_+,p_+)$ is a subsonic solution of \eqref{euler}. 
        Since $(\grad\times(\Phi\be_\vphi),\frac{L}{2\pi r \sin\theta}\be_\vphi,S)\in C^{1,\alpha}_{(-\alpha,\Gam_w^+)}(\n_f^+)$, $(\rho_+,\bu_+,p_+)\in C^{1,\alpha}_{(-\alpha,\Gam_w^+)}(\n_f^+)$. 
        From these facts, we have that $(f,\rho_+,\bu_+,p_+)$ is a solution of Problem \ref{pb1}.
        
        Obtain \eqref{thmestimate}. By Lemma \ref{lemRho}, $\rho_0^+$ can be written as
        $\rho_0^+=\varrho(\grad\times(\Phi_0^+\be_{\vphi}),\frac{L_0^+}{2\pi r \sin\theta}\be_{\vphi}(=0),S_0^+).$
        Using this expression, we write $\rho_+-\rho_0^+$ as
        \begin{multline}\label{rhorho0+}
        \int_0^1\grad \varrho(t(\grad\times(\Phi\be_\vphi),\frac{L}{2\pi r \sin\theta}\be_\vphi,S)+(1-t)(\grad\times(\Phi_0^+\be_\vphi),0,S_0^+))dt\\
        (\grad\times((\Phi-\Phi_0^+)\be_\vphi),\frac{L}{2\pi r \sin\theta}\be_\vphi,S-S_0^+).
        \end{multline}
        Since $b$ is a $C^\infty$ function of $(\rho,\grad\times(\Phi\be_{\vphi}),\frac{L}{2\pi r\sin\theta}\be_{\vphi},S,B)$, $\varrho$ is a $C^\infty$ function of $(\grad\times(\Phi\be_{\vphi}),\frac{L}{2\pi r\sin\theta}\be_{\vphi},S,B)$. With this fact, the fact that $(\rho_0^+,u_0^+\be_r,p_0^+)\in (C^\infty(\ol{\n_f^+}))^3$ and \eqref{mainestimate} satisfied by $(f,\Phi\be_{\vphi},L,S)$ for $\sigma \in (0,\ol{\sigma}_2]$ for $\ol{\sigma}_2\le \sigma_1$, we estimate \eqref{rhorho0+} in $C^{1,\alpha}_{(-\alpha,\Gam_w^+)}(\n_f^+)$. Then we obtain
        \begin{align}\label{rhomrho0+thm}
        ||\rho_+-\rho_0^+||_{1,\alpha,\n_f^+}^{(-\alpha,\Gam_w^+)}
        \le C\sigma 
        \end{align}
        where $C$ is a positive constant depending on the data. 
        By this estimate, there exists a positive constant $\ol{\sigma}_2^{(1)}$ depending on the data such that if $\sigma\in(0, \ol{\sigma}_2^{(1)}]$, then 
        \begin{align}
        \label{rhoupper}\sup_{\n_f^+}|\frac{1}{\rho_+}|\le C
        \end{align}
        where $C$ is a positive constant depending on the data. Take $\ol{\sigma}_2=\min(\sigma_1,\ol{\sigma}_2^{(1)})$ so that \eqref{rhoupper} holds. With \eqref{mainestimate} satisfied by $(f,\Phi\be_{\vphi},L,S)$,  \eqref{rhomrho0+thm} and \eqref{rhoupper}, 
        we estimate $\bu_+ -u_0^+\be_r$ and $p_+ -p_0^+$ in $C^{1,\alpha}_{(-\alpha,\Gam_w^+)}(\n_f^+)$. Then we obtain
        $$||\bu_+ -u_0^+\be_r||_{1,\alpha,\n_f^+}^{(-\alpha,\Gam_w^+)}+||p_+ -p_0^+||_{1,\alpha,\n_f^+}^{(-\alpha,\Gam_w^+)}\le C\sigma$$
        where $C$ is a positive constant depending on the data. Combining this estimate, \eqref{mainestimate} and \eqref{rhomrho0+thm}, we obtain \eqref{thmestimate}. 
        
        2. Assume that for $\sigma\in (0,\ul{\sigma}_2]$ where $\ul{\sigma}_2$ is a positive constant to be determined later, there exist two solutions $(f^i,\rho_+^i,\bu_+^i,p_+^i)$ for $i=1,2$ of Problem \ref{pb1} satisfying the estimate \eqref{thmestimate}. 
        There exists a positive constant $\ul{\sigma}_2^{(1)}$ depending on the data such that if $\sigma\in(0,\ul{\sigma}_2^{(1)}]$, then 1)
        \begin{align}\label{rhoiupper}
        \sup_{\n_{f_i}^+}|\frac{1}{\rho_+^i}|\le C
        \end{align}
        for $i=1,2$ where $C$ is a positive constant depending on the data and 2) $\rho_+^i$ for $i=1,2$ are uniquely determined by $(\rho_+^i(u_{+,r}^i\be_r+u_{+,\theta}^i\be_\theta), u_{+,\vphi}^i\be_{\vphi}, S_+^i,B_0)$ where $u_{+,r}^i:=\bu_+^i\cdot\be_r$,  $u_{+,\theta}^i:=\bu_+^i\cdot\be_\theta$, $u_{+,\vphi}^i:=\bu_+^i\cdot \be_{\vphi}$ and $S_+^i:=\frac{p_+^i}{(\rho_+^i)^\gam}$ for $i=1,2$ (here we used Lemma \ref{lemRho} and the fact that if $(\rho_+^i,\bu_+^i,p_+^i)$ are $C^0(\ol{\n_f^+})\cap C^1(\n_f^+)$ solutions of Problem \ref{pb1}, then $(\rho,\bu,p)=(\rho_+^i,\bu_+^i,p_+^i)$ satisfy $B=B_0$ in $\ol{\n_{f^i}^+}$).
        Take $\ul{\sigma}_2= \ul{\sigma}_2^{(1)}$ so that 1) and 2) hold. 
        Let $(\Phi^i\be_{\vphi},L^i,S^i)$ for $i=1,2$ be $(\Phi\be_{\vphi},L,S)$ given by the definitions of $\Phi$, $L$ and $S$ for  $(\rho,\bu,p)=(\rho_+^i,\bu_+^i,p_+^i)$ for $i=1,2$. 
        Then since  $(f^i,\rho_+^i,\bu_+^i,p_+^i)$ for $i=1,2$ are solutions of Problem \ref{pb1}, $(f^i,\Phi^i\be_{\vphi},L^i,S^i)$ for $i=1,2$ are solutions of Problem \ref{pb2}. Furthermore, since $(f^i,\rho_+^i,\bu_+^i,p_+^i)$ for $i=1,2$ satisfy \eqref{thmestimate}, $(f^i,\Phi^i\be_{\vphi},L^i,S^i)$ for $i=1,2$ satisfy \eqref{mainestimate} with 
        $C$ replaced by 
        $C_1$ for some positive constant $C_1$ depending on the data (here we used 
        \eqref{rhoiupper}).   
        Take $\ul{\sigma}_2=\min (\ul{\sigma}_2^{(1)},\frac{C\sigma_1}{C_1},\sigma_1)$ where $C$ is $C$ in \eqref{mainestimate} so that $(\rho_-,\bu_-,p_-)$ and $p_{ex}$ satisfy \eqref{supesti} and \eqref{pexesti}, respectively, for $\sigma\in (0,\sigma_1]$ and $(f^i,\Phi^i\be_{\vphi},L^i,S^i)$ for $i=1,2$ satisfy \eqref{mainestimate} for $\sigma\in (0,\sigma_1]$. Then  by Theorem \ref{mainthm}, $(f^1,\Phi^1\be_{\vphi},L^1,S^1)=(f^2,\Phi^2\be_{\vphi},L^2,S^2)$.
        This implies
        \begin{align*}
        \rho_+^1(u_{+,r}^1\be_r+u_{+,\theta}^1\be_\theta)=\rho_+^2(u_{+,r}^2\be_r+u_{+,\theta}^2\be_\theta)\qdand u_{+,\vphi}^1\be_{\vphi}=u_{+,\vphi}^2\be_{\vphi}.
        \end{align*} 
        By these relations, $S^1=S^2$ and 2), we have that $\rho_+^1=\rho_+^2$. 
        With this relation, we can conclude that 
        $(\rho_+^1,\bu_+^1,p_+^1)=(\rho_+^2,\bu_+^2,p_+^2)$. 
        Let $\sigma_2=\min(\ol{\sigma}_2,\ul{\sigma}_2)$. 
        This finishes the proof. 
        \end{proof}
        The rest of this paper is devoted to prove Theorem \ref{mainthm}.
        For convenience, 
        we describe our main process of proving Theorem \ref{mainthm} below.  
        
        
        To describe our process of proving 
        Theorem \ref{mainthm}, we define some terminologies. 
        Let $f:[0,\theta_1]\ra \R$ be a function representing an axisymmetric shock location $\Gam_f$. Decompose $f$ into $f(\theta)=f(0)+f_s(\theta)$. Then by $f_s(0)=0$, $f_s$ is uniquely determined by $\fsp$. We call $f(0)$ and $\fsp$ the initial shock position and the shape of a shock location, respectively. 
        
        Using these terminologies, our 
        process of 
        proving 
        Theorem \ref{mainthm} is described as follows. 
        \begin{itemize}
        \item [1.] For given an incomming supersonic solution, an exit pressure and a shape of a shock location $(\rho_-,\bu_-,p_-,p_{ex},\fsp)$ in a small perturbation of  $(\rho_0^-,u_0^-\be_r,p_0^-,p_c,0)$, show that there exists a pair of an initial shock position $f(0)$ and a subsonic solution $(\Phi\be_{\vphi},L,S)$ of \eqref{eqV}-\eqref{eqS} satisfying all the conditions in Problem \ref{pb2} except \eqref{rhtau}, and that this solution is unique in the class of functions in a small perturbation of $(r_s,\Phi_0^+\be_{\vphi},L_0^+,S_0^+)$. 
        
        \item [2.] For given an incomming supersonic solution and an exit pressure as in Step 1 or in a much small perturbation of $(\rho_0^-,u_0^-\be_r,p_0^-,p_c)$ if necessary, show that there exists $\fsp$ in a small perturbation of $0$ as in Step 1 such that $(f(0),\Phi\be_{\vphi},L,S)$ determined by $(\rho_-,\bu_-,p_-,p_{ex},\fsp)$ in Step 1 satisfies \eqref{rhtau}, and that for given $(\rho_-,\bu_-,p_-,p_{ex})$ in a small perturbation of $(\rho_0^-,u_0^-\be_r,p_0^-,p_c)$, a solution $(f,\Phi\be_{\vphi},L,S)$ of Problem \ref{pb2} is unique in the class of functions in a small perturbation of $(r_s,\Phi_0^+\be_{\vphi},L_0^+,S_0^+)$. 
        \end{itemize}
        Once Step 1 and Step 2 are done, then $f=f(0)+f_s$ and $(\Phi\be_{\vphi},L,S)$ obtained through Step 1 and Step 2 satisfies all the conditions in Problem \ref{pb2}. Thus, Theorem \ref{mainthm} is proved if Step 1 and Step 2 are done.  
        We will deal with Step 1 and Step 2 in Section \ref{secPseudoFree} and Section \ref{secdetermineshock}, respectively. 
        
        
        Note that the fact that for transonic shock solutions of the full Euler system, an initial shock position and a shape of a shock location are determined in different mechanisms (an initial shock position is determined by the solvability condition related to the mass conservation law and a shape of a shock location is determined by the R-H conditions) was pointed out in \cite{MR2420002} and the authors in \cite{MR2420002,MR2420002,MR2576697,MR3005323,MR2557901,MR3060893,MR3459028} prove the stability of transonic shock solutions of the full Euler system using 
        iteration schemes based on this fact. 
        In this paper, we also prove the stability of transonic shock solutions of the full Euler system using a scheme based on this fact.  
        But we do this using a different scheme. 
        In our scheme, a non-local elliptic equation appearing in \cite{MR2420002,MR2576697,MR3005323,MR2557901,MR3060893,MR3459028} does not appear.  

        \section{Pseudo Free Boundary Problem}\label{secPseudoFree}
        As a first step to prove Theorem \ref{mainthm}, we will solve the Pseudo Free Boundary Problem below. This problem naturally arises from the requirement that a subsonic solution in Problem \ref{pb2} must satisfy \eqref{compc}. From the  linearized equation of \eqref{pressurestream}, it is seen that 
        an iteration scheme for a fixed boundary problem does not give a subsonic solution satisfying \eqref{compc} in general. Thus, 
        an iteration scheme for a fixed boundary problem is not a proper scheme to find a subsonic solution in Problem \ref{pb2}. 
        To find a subsonic solution satisfying \eqref{compc}, we let $f(0)$ be an unknown to be determined simultaneously with a subsonic solution. Using this variable, we 
        adjust the value of a subsonic solution so that this solution can satisfy \eqref{compc}. 
        The main ingredient for this argument to hold is the monotonicity of the entropy of the downstream subsonic solution of a radial transonic shock solution on a shock location 
        with respect to the shock location. This will be seen in the proof of Proposition \ref{proPseudofree}. 
        

        \begin{problem} [Pseudo Free Boundary Problem]\label{pb3}
        	Given an axisymmetric supersonic solution $(\rho_-,\bu_-,p_-)$ of \eqref{euler} and exit pressure $p_{ex}$ as in Problem \ref{pb1}  
        	and a shape of a shock location $f_s^\p\in C^{1,\alpha}_{(-\alpha,\{\theta=\theta_1\})}((0,\theta_1))$ satisfying $\fsp(\theta)=0$ at $\theta=0, \theta_1$ and
        	\begin{align}
        	\label{fsestimate}||f_s^\p||_{1,\alpha,(0,\theta_1)}^{(-\alpha,\{\theta=\theta_1\})}\le \sigma
        	\end{align}
        	for a sufficiently small $\sigma>0$,
        	find an initial shock position $f(0)$ and a corresponding subsonic solution $(\Phi\be_\vphi,L,S)$ of \eqref{eqV}-\eqref{eqS} satisfying 
        	\begin{align}\tag{A}
        	\begin{split}
        	&\left(\frac{1}{\rho}\grad\times (\Phi\be_\vphi)\cdot \be_r\right)\grad\times\left(\frac{1}{\rho}\grad\times (\Phi\be_\vphi)\right)\\
        	&\qd\qd=\left(\frac{L}{2\pi r\sin\theta}\grad \times\left(\frac{L}{2\pi r\sin\theta}\be_\vphi\right)\cdot\be_r+\frac{\rho^{\gam-1}}{\gam-1}\frac{\pt_\theta S}{r}\right)\be_\vphi\qdin \n_{f(0)+f_s}^+,\\
        	&\Phi\be_\vphi=\Phi_-\be_{\vphi}\qdon \Gam_{f(0)+f_s},\\
        	&\Phi\be_\vphi=\frac{r_0\Phi_-(r_0,\theta_1)}{r}\be_\vphi\qdon \Gam_{w,f(0)+f_s}^+,\\
        	&S\rho^\gam=p_{ex}\qdon \Gam_{ex},\\
        	&\lim_{\theta\ra \theta_1}(\Phi\be_\vphi)(r_1,\theta)=\lim_{r\ra r_1}(\Phi\be_\vphi)(r,\theta_1),
        	\end{split}
        	\end{align}
        	where $\rho=\varrho(\grad\times(\Phi\be_{\vphi}),\frac{L}{2\pi r\sin\theta}\be_{\vphi},S)$,
        	and 
        	\begin{align}\tag{B}
        	\begin{split}
        	&\grad\times(\Phi\be_\vphi)\cdot \grad L=0\qdin \n_{f(0)+f_s}^+,\\
        	&L=L_-\qdon \Gam_{f(0)+f_s},\\
        	&\grad\times(\Phi\be_\vphi)\cdot \grad S=0\qdin \n_{f(0)+f_s}^+,\\
        	&S=g\left(\left(\frac{\bu_-\cdot {\bm \nu}_f}{c_-}\right)^2\right)S_-\qdon \Gam_{f(0)+f_s}
        	\end{split}
        	\end{align}
        	where $f_s$ denotes $\int_0^\theta \fsp$. 
        \end{problem}
        Hereafter, we denote a set of functions in $C^{1,\alpha}_{(-\alpha,\{\theta=\theta_1\})}((0,\theta_1))$ having $0$ value at $\theta=0,\theta_1$ by $C^{1,\alpha}_{(-\alpha,\{\theta=\theta_1\}),0}((0,\theta_1))$. 
        Hereafter, $f_s$ denotes $\int_0^\theta \fsp$. 
        
        Our result of Problem \ref{pb3} is given as follows.
        \begin{prop}
        	\label{proPseudofree} Let $\alpha\in (\frac{2}{3},1)$.
        	There exists a positive constant $\sigma_3$ depending on the data such that if $\sigma \in (0, \sigma_3]$, then Problem \ref{pb3} has a solution $(f(0),\Phi\be_\vphi,L,S)$ satisfying 
        	\begin{multline}\label{Pseudoestimate}
        	|f(0)-r_s|\\+||\grad\times((\Phi-\Phi_0^+)\be_\vphi)||_{1,\alpha,\n_\ff^+}^{(-\alpha,\Gam_w^+)}+||\frac{L}{2\pi r\sin\theta}\be_\vphi||_{1,\alpha,\n_\ff^+}^{(-\alpha,\Gam_w^+)}+||S-S_0^+||_{1,\alpha,\n_\ff^+}^{(-\alpha,\Gam_w^+)}
        	\le C\sigma
        	\end{multline}
        	where $C$ is a positive constant depending on the data. Furthermore, this solution is unique in the class of functions satisfying \eqref{Pseudoestimate}. 
        \end{prop}
        We will prove Proposition \ref{proPseudofree} using a fixed point argument. To do this, 
        we linearize (A) with respect to (A) satisfied by $(\Phi\be_\vphi,L,S)=(\Phi_0^+\be_{\vphi},L_0^+,S_0^+)$
        and reformulate (B) in terms of  $(\Psi\be_{\vphi},A,T)$ where $(\Psi,A,T):=(\Phi-\Phi_0^+,L,S-S_0^+)$. 
        \subsection{Linearization and reformulation of (A) and (B)}\label{sublinear}
        We linearize (A) with respect to (A) satisfied by $(\Phi_0^+\be_{\vphi},L_0^+,S_0^+)$. Since $\rho$ in the first and fourth equation of (A) 
        is given using an 
        implicit relation,
        to obtain the linearized equations of (A), we first linearize $\rho$ with respect to $\rho_0^+$. 

        
         \begin{lem}\label{lemrholinear}
        	Suppose that $f:\Lambda\ra \R$ 
        	is an axisymmetric function in $C^{2,\alpha}_{(-1-\alpha,\pt\Lambda)}(\Lambda)$
        	satisfying 
        	\begin{align}
        	\label{f}||f-r_s||_{2,\alpha,\Lambda}^{(-1-\alpha,\pt\Lambda)}\le \delta_1.
        	\end{align}
        	Also, suppose that $\Psi\be_\vphi:\n_f^+\ra \R^3$, $A: \n_f^+\ra \R$ and $T:\n_f^+\ra \R$ are axisymmetric functions in $C^{2,\alpha}_{(-1-\alpha,\Gam_w^+)}(\n_f^+)$, $C^{1,\alpha}_{(-\alpha,\Gam_w^+)}(\n_f^+)$ and $C^{1,\alpha}_{(-\alpha,\Gam_w^+)}(\n_f^+)$, respectively,  
        	and satisfy 
        	\begin{align}\label{ppdel2}
        	||\grad\times(\Psi\be_\vphi)||^{(-\alpha,\Gam_w^+)}_{1,\alpha,\n_f^+}+||\frac{A}{2\pi r\sin\theta}\be_\vphi||^{(-\alpha,\Gam_w^+)}_{1,\alpha,\n_f^+}+||T||^{(-\alpha,\Gam_w^+)}_{1,\alpha,\n_f^+}\le \delta_2.
        	\end{align}
        	Let $\rho=\varrho(\grad\times(\Phi\be_{\vphi}),\frac{L}{2\pi r \sin\theta}\be_{\vphi},S)$. There holds 
        	\begin{align}\label{rholinear}
        	\rho-\rho_0^+=\frac{\grad\times(\Phi_0^+\be_\vphi)}{\rho_0^+({u_0^+}^2-{c_0^+}^2)}\cdot\grad\times (\Psi\be_\vphi)+\frac{\gam {\rho_0^+}^{\gam}}{(\gam-1)({u_0^+}^2-{c_0^+}^2)}T
        	+g_1(\Psi\be_\vphi,A,T)
        	\end{align}
        	where 
        	\begin{align}\label{g1}
        	&g_1(\Psi\be_\vphi,A,T)=\\
        	&\nonumber\frac{1}{\rho_0^+({u_0^+}^2-{c_0^+}^2)}\left(\int_0^1 \left((2\rho_0^+B_0-\frac{\gam(\gam+1)}{\gam-1}S_0^+{\rho_0^+}^\gam)\right.\right.\\
        	&
        	\nonumber\left.-(2(t\rho+(1-t)\rho_0^+)B_0-\frac{\gam(\gam+1)}{\gam-1}(S_0^++T)(t\rho+(1-t)\rho_0^+)^\gam)\right)dt(\rho-\rho_0^+)\\
        	&\nonumber\left.+\int_0^1(-\grad\times (\Phi_0^+\be_\vphi)+(t\grad\times ((\Phi_0^++\Psi)\be_\vphi)+(1-t)\grad\times( \Phi_0^+\be_\vphi)))dt\cdot \grad\times(\Psi\be_\vphi)\right.\\
        	&\left.+\int_0^1 \frac{{\rho_0^+}^2 tA}{2\pi r\sin\theta}\be_\vphi dt\cdot\frac{A}{2\pi r \sin\theta}\be_\vphi \right).\nonumber
        	\end{align}
        	$g_1(\Psi\be_\vphi,A,T)$ satisfies 
        	\begin{align}\label{g1estimate}
        	||g_1(\Psi\be_\vphi,A,T)||_{1,\alpha,\n_f^+}^{(-\alpha,\Gam_w^+)}\le C(||\grad\times(\Psi\be_\vphi)||_{1,\alpha,\n_f^+}^{(-\alpha,\Gam_w^+)}+||\frac{A}{2\pi r \sin\theta}\be_\vphi||_{1,\alpha,\n_f^+}^{(-\alpha,\Gam_w^+)}+||T||_{1,\alpha,\n_f^+}^{(-\alpha,\Gam_w^+)})^2
        	\end{align}
        	where $C$ is a positive constant depending on $(\rho_0^+,u_0^+,p_0^+)$, $\gam$, $r_s$, $r_1$, $\alpha$, $\delta_1$ and $\delta_2$.  
        	\begin{proof}
        		1. By 
        		the definition of $\varrho$, 
        		$\rho=\varrho(\grad\times(\Phi\be_{\vphi}),\frac{L}{2\pi r \sin\theta}\be_{\vphi},S)$ satisfies 
        		$$b(\rho,\grad\times(\Phi\be_\vphi), \frac{L}{2\pi r\sin\theta}\be_\vphi,S,B_0)=0$$
        		where $b$ is a function defined in \eqref{bdef}. Subtract this equation from $b(\rho_0^+,\grad\times(\Phi_0^+\be_\vphi),0,S_0^+,B_0)=0$ obtained from the third equation of \eqref{conservsub}. And then linearize the resultant equation. 
        		Then we obtain \eqref{rholinear}.

        		2. 
        		With the fact that $(\rho_0^+,u_0^+\be_r,p_0^+)\in (C^\infty(\ol{\n_f^+}))^3$ and \eqref{ppdel2}, 
        		estimate $\rho-\rho_0^+$ in $C^{1,\alpha}_{(-\alpha,\Gam_w^+)}(\n_f^+)$ in the way that we estimated $\rho_+-\rho_0^+$ in the proof of Theorem \ref{eulerthm}. 
        		Then we have
        		\begin{align}
        		\label{rhomrho+}||\rho-\rho_0^+||_{1,\alpha,\n_f^+}^{(-\alpha,\Gam_w^+)}\le C (||\grad\times(\Psi\be_{\vphi})||_{1,\alpha,\n_f^+}^{(-\alpha,\Gam_w^+)}+||\frac{A}{2\pi r \sin\theta}\be_\vphi||_{1,\alpha,\n_f^+}^{(-\alpha,\Gam_w^+)}+||T||_{1,\alpha,\n_f^+}^{(-\alpha,\Gam_w^+)})
        		\end{align}
        		where $C$ is a positive constant depending on $(\rho_0^+,u_0^+,p_0^+)$, $\gam$, $r_s$, $r_1$, $\alpha$, $\delta_1$ and $\delta_2$. 
        		With this estimate, the fact that $(\rho_0^+,u_0^+\be_r,p_0^+)\in (C^\infty(\ol{\n_f^+}))^3$ and \eqref{ppdel2}, estimate \eqref{g1} in $C^{1,\alpha}_{(-\alpha,\Gam_w^+)}(\n_f^+)$. Then we obtain the desired result. 
        	\end{proof}
        \end{lem} 
        Then we linearize (A). 
        
        Subtract \eqref{eqV} satisfied by $(\Phi\be_\vphi,L,S)=(\Phi_0^+\be_{\vphi},L_0^+,S_0^+)$ in $\n_f^+$ from \eqref{eqV} in $\n_f^+$ and then linearize the resultant equation (in this process, we express $\rho-\rho_0^+$ using \eqref{rholinear}). 
        Then we obtain 
        \begin{align}\label{Wlinear}
        \grad\times&\left(\frac{1}{\rho_0^+}(1+\frac{{u_0^+}\be_r\otimes u_0^+\be_r}{{c_0^+}^2-{u_0^+}^2})\grad\times(\Psi\be_\vphi)\right) \\\nonumber&=\frac{{\rho_0^+}^{\gam-1}}{(\gam-1)u_0^+}(1+\frac{\gam {u_0^+}^2}{{c_0^+}^2-{u_0^+}^2})\frac{\pt_\theta T}{r}\be_\vphi+{\bm F}_1(\Psi\be_\vphi,A,T)\qdin \n_f^+
        \\
        &(=:{\bm F}_2(\Psi\be_\vphi,A,T))\nonumber
        \end{align}
        where 
        \begin{align}\label{F1}
        {\bm F}_1&(\Psi\be_\vphi,A,T)=\\
        \nonumber&\grad\times \left(\frac{1}{\rho_0^+}g_1\grad\times (\Phi_0^+\be_\vphi)-\frac{1}{\rho{\rho_0^+}^2}g_2^2\grad\times((\Phi_0^++\Psi)\be_\vphi)+\frac{1}{{\rho_0^+}^2}g_2\grad\times(\Psi\be_\vphi)\right)\\
        &\nonumber-\frac{1}{u_0^+ }\left(-\frac{1}{\rho\rho_0^+}g_2\grad\times((\Phi_0^++\Psi)\be_\vphi)\cdot\be_r+\frac{1}{\rho_0^+}\grad\times(\Psi\be_\vphi)\cdot \be_r\right)\\
        &\nonumber\qd\qd\qd\qd\qd\qd\qd\qd \grad\times\left(-\frac{1}{\rho\rho_0^+}g_2\grad\times((\Phi_0^++\Psi)\be_\vphi)+\frac{1}{\rho_0^+}\grad\times(\Psi\be_\vphi)\right)\\
        &\nonumber+\frac{1}{u_0^+ }\left(\int_0^1(t\rho+(1-t)\rho_0^+)^{\gam-2}dtg_2 \frac{\pt_\theta T}{r}\right)\be_\vphi\\
        &\nonumber+\frac{A}{u_0^+ 2\pi r \sin\theta}\left(\grad\times (\frac{A}{2\pi r \sin\theta}\be_\vphi)\cdot \be_r\right)\be_\vphi
        \end{align}
        where $g_1$ and $g_2$ are $g_1(\Psi\be_\vphi,A,T)$ given in \eqref{g1} and the right-hand side of \eqref{rholinear}, respectively. 
        
        By the definition of $V_0^\pm$ and the first equations of \eqref{conserv} and \eqref{conservsub}, 
        $V_0^+=V_0^-$ on $\Gam_f.$ 
        In the vector potential form, this is written as 
        $\Phi_0^+\be_{\vphi}=\Phi_0^-\be_{\vphi}$ on $\Gam_f$. 
        Subtract this equation from \eqref{rhV}. Then we obtain
        \begin{align}\label{Wf}
        \Psi\be_\vphi=(\Phi_--\Phi_0^-)\be_\vphi\qd\tx{on}\qd \Gam_{f}. 
        \end{align} 
        
        By the definition of $V_0^\pm$ and the first equations of \eqref{conserv} and \eqref{conservsub},  
        $V_0^+=V_0^-(r_0,\theta_1)$ on $\Gam_w^+.$
        In the vector potential form, 
        this is written as $\Phi_0^+\be_\vphi=\frac{r_0\Phi_0^-(r_0,\theta_1)\be_{\vphi}}{r}$ on $\Gam_w^+$. 
        By subtracting this equation 
        from 
        \eqref{slipc}, 
        we obtain
        \begin{align}\label{Wslip}
        \Psi\be_\vphi=\frac{r_0(\Phi_--\Phi_0^-)(r_0,\theta_1)}{r}\be_\vphi\qdon \Gam_w^+.
        \end{align}
        
        Rewrite \eqref{pressurestream} as $\rho=(\frac{p_{ex}}{S})^{\frac{1}{\gam}}$ on $\Gam_{ex}$.
        Subtract this equation from $\rho_0^+=(\frac{p_c}{S_0^+})^{\frac{1}{\gam}}$ on $\Gam_{ex}$.
        And then express $\rho-\rho_0^+$ using \eqref{rholinear}  
        and linearize $(\frac{p_{ex}}{S})^{\frac{1}{\gam}}-(\frac{p_c}{S_0^+})^{\frac{1}{\gam}}$.  
        Multiply $\frac{ {u_0^+}^2-{c_0^+}^2}{u_0^+}r\sin\theta$ and then
        integrate the resultant equation from $0$ to $\theta$. After that, divide $\sin\theta$ and multiply $\be_\vphi$ on both-hand sides of the equation. Then we obtain 
        \begin{multline}\label{Wexit}
        \Psi\be_\vphi=\bigg(\frac{1}{r\sin\theta}\int_0^\theta \bigg(\mf f_0(T,p_{ex})-\frac{\rho_0^+((\gam-1){u_0^+}^2+{c_0^+}^2)}{\gam(\gam-1)u_0^+S_0^+}T\\
        +\mf f_1(\Psi\be_\vphi,A,T) \bigg)r^2\sin\xi d\xi\bigg)\be_\vphi\qdon \Gam_{ex}
        \end{multline}
        where 
        \begin{align} \label{f0}\mf f_0(T,p_{ex})=\frac{{u_0^+}^2-{c_0^+}^2}{u_0^+}\left(\frac{1}{S_0^++T}\right)^{\frac{1}{\gam}}(p_{ex}^{\frac{1}{\gam}}-p_c^{\frac{1}{\gam}})
        \end{align}
        and
        \begin{multline}
        \label{f1}
        \mf f_1(\Psi\be_\vphi,A,T)=
        -\frac{{p_c}^{\frac{1}{\gam}}({u_0^+}^2-{c_0^+}^2)}{\gam u_0^+}\int_0^1 \left(\left(\frac{1}{tS+(1-t)S_0^+}\right)^{\frac{1}{\gam}+1}-\left(\frac{1}{S_0^+}\right)^{\frac{1}{\gam}+1}\right)dt T\\-\frac{{u_0^+}^2-{c_0^+}^2}{u_0^+}g_1. 
        \end{multline}
        
        By the definition of $V_0^\pm$ and the first equation of \eqref{conservsub}, 
        $\lim_{\theta\ra\theta_1}V_0^+(r_1,\theta)=\lim_{r\ra r_1} V_0^+(r,\theta_1)$. 
        In the vector potential form, this is written as $\lim_{\theta\ra \theta_1}(\Phi_0^+\be_{\vphi})(r_1,\theta)=\lim_{r\ra r_1}(\Phi_0^+\be_{\vphi})(r,\theta_1)$. Subtract this equation from \eqref{compc}. Then we obtain
        \begin{align*}
        \lim_{\theta\ra\theta_1}(\Psi\be_\vphi)(r_1,\theta)=\lim_{r\ra r_1} (\Psi\be_\vphi)(r,\theta_1). 
        \end{align*}
        Express this relation using \eqref{Wslip} and \eqref{Wexit}. Then we have
        \begin{multline}\label{Wcompc}
        \frac{1}{r_1\sin\theta_1}\int_0^{\theta_1} \bigg(\mf f_0(T,p_{ex})-\frac{\rho_0^+((\gam-1){u_0^+}^2+{c_0^+}^2)}{\gam(\gam-1)u_0^+S_0^+}T\\+\mf f_1(\Psi\be_\vphi,A,T) \bigg)\bigg|_{r=r_1}r_1^2\sin\xi d\xi=\frac{r_0(\Phi_--\Phi_0^-)(r_0,\theta_1)}{r_1}.
        \end{multline}
        
        Combining \eqref{Wlinear}, \eqref{Wf}, \eqref{Wslip}, \eqref{Wexit} and \eqref{Wcompc}, we have the following linearized equations of (A):
        \begin{align}\tag{A$^\prime$}
        \begin{split}
        &\grad\times\left(\frac{1}{\rho_0^+}(1+\frac{{u_0^+}\be_r\otimes u_0^+\be_r}{{c_0^+}^2-{u_0^+}^2})\grad\times(\Psi\be_\vphi)\right) \\&=\frac{{\rho_0^+}^{\gam-1}}{(\gam-1)u_0^+}(1+\frac{\gam {u_0^+}^2}{{c_0^+}^2-{u_0^+}^2})\frac{\pt_\theta T}{r}\be_\vphi+{\bm F}_1(\Psi\be_\vphi,A,T)\qdin \n_{f(0)+f_s}^+,\\
        &\Psi\be_\vphi=(\Phi_--\Phi_0^-)\be_\vphi\qd\tx{on}\qd \Gam_{f(0)+f_s},\\
        &\Psi\be_{\vphi}=\frac{r_0(\Phi_--\Phi_0^-)(r_0,\theta_1)}{r}\be_{\vphi}\qdon \Gam_{w,f(0)+f_s}^+,\\
        &\Psi\be_\vphi=\bigg(\frac{1}{r\sin\theta}\int_0^\theta \bigg(\mf f_0(T,p_{ex})-\frac{\rho_0^+((\gam-1){u_0^+}^2+{c_0^+}^2)}{\gam(\gam-1)u_0^+S_0^+}T\\
        &\qd\qd\qd\qd\qd\qd\qd\qd\qd\qd\qd\qd\qd\qd\qd+\mf f_1(\Psi\be_\vphi,A,T) \bigg)r^2\sin\xi d\xi\bigg)\be_\vphi\qdon \Gam_{ex},\\
        &\frac{1}{r_1\sin\theta_1}\int_0^{\theta_1} \bigg(\mf f_0(T,p_{ex})-\frac{\rho_0^+((\gam-1){u_0^+}^2+{c_0^+}^2)}{\gam(\gam-1)u_0^+S_0^+}T\\&\qd\qd\qd\qd\qd\qd\qd+\mf f_1(\Psi\be_\vphi,A,T) \bigg)\bigg|_{r=r_1}r_1^2\sin\xi d\xi=\frac{r_0(\Phi_--\Phi_0^-)(r_0,\theta_1)}{r_1}.
        \end{split}
        \end{align}
        
        For later use, we present the following estimates of ${\bm F}_1(\Psi\be_\vphi,A,T)$ and $\mf f_1(\Psi\be_\vphi,A,T)$. 
        \begin{lem}\label{F1f1estimate}
        	Let $\delta_4$ be a positive constant $\le \delta_2$ 
        	such that for $f$ as in Lemma \ref{lemrholinear}, if $(\Psi\be_{\vphi},A,T)$ satisfies 
        	\begin{align}\label{delta4}
        		||\grad\times(\Psi\be_\vphi)||^{(-\alpha,\Gam_w^+)}_{1,\alpha,\n_f^+}+||\frac{A}{2\pi r\sin\theta}\be_\vphi||^{(-\alpha,\Gam_w^+)}_{1,\alpha,\n_f^+}+||T||^{(-\alpha,\Gam_w^+)}_{1,\alpha,\n_f^+}\le \delta_4,
        	\end{align} 
            then 
        	\begin{align*}
        		\sup_{\n_f^+}|\frac{1}{\varrho(\grad\times(\Phi\be_{\vphi}),\frac{L}{2\pi r \sin\theta}\be_{\vphi},S)}|
        		\le C\qdand \sup_{\n_f^+}|\frac{1}{S}|\le C
        	\end{align*}
        	where $C$s are positive constants depending on $(\rho_0^+,u_0^+,p_0^+)$, $\gam$, $r_s$, $r_1$ 
        	and $\delta_4$. 
        	Suppose that $f$ is as in Lemma \ref{lemrholinear}.
        	Also, suppose that $(\Psi\be_{\vphi},A,T)$ are as in Lemma \ref{lemrholinear}
        	and satisfy \eqref{delta4}. 
            Then there hold
        	\begin{align}\label{F1estimate}
        	||{\bm F}_1(\Psi\be_\vphi,A,T)||_{\alpha,\n_f^+}^{(1-\alpha,\Gam_w^+)}
        	\le C (||\Psi\be_\vphi||_{2,\alpha,\n_f^+}^{(-1-\alpha,\Gam_w^+)}+||\frac{A}{2\pi r \sin\theta}\be_\vphi||_{1,\alpha,\n_f^+}^{(-\alpha,\Gam_w^+)}+||T||_{1,\alpha,\n_f^+}^{(-\alpha,\Gam_w^+)})^2
        	\end{align}
        	and
        	\begin{align}\label{f1estimate}
        	||\mf f_1(\Psi\be_\vphi,A,T)||_{1,\alpha,\Gam_{ex}}^{(-\alpha,\pt \Gam_{ex})}
        	\le C
        	(||\Psi\be_\vphi||_{2,\alpha,\Gam_{ex}}^{(-1-\alpha,\pt \Gam_{ex})}+||\frac{A}{2\pi r \sin\theta}\be_\vphi|||_{1,\alpha,\Gam_{ex}}^{(-\alpha,\pt\Gam_{ex})}+||T||_{1,\alpha,\Gam_{ex}}^{(-\alpha,\pt \Gam_{ex})})^2
        	\end{align}
        	where $Cs$ are positive constants depending on  $(\rho_0^+,u_0^+,p_0^+)$, $\gam$, $r_s$, $r_1$, $\alpha$, $\delta_1$, $\delta_2$ and $\delta_4$.
        	\begin{proof} 
        		With \eqref{g1estimate}, \eqref{rhomrho+}, \eqref{delta4} and the fact that if an axisymmetric vector field ${\bm a}$ on an axisymmetric connected open set $\Om$ is in $C^k(\Om)$, then $\grad\times{\bm a}\in C^{k-1}(\Om)$ and ${\bm a}\cdot \be_r\in C^k(\Om)$ (the second one is obtained from Lemma \ref{lemaxi}), 
        		we estimate \eqref{F1} 
        		and \eqref{f1} in $C^{\alpha}_{(1-\alpha,\Gam_w^+)}(\n_f^+)$, 
        		and $C^{1,\alpha}_{(-\alpha,\pt \Gam_{ex})}(\Gam_{ex})$, respectively. Then we obtain the desired result. 
        	\end{proof}
        \end{lem}
        Next, we reformulate (B) in terms of $(\Psi\be_\vphi,A,T)$. 
        With the facts that $L_0^+=L_0^-=0$ in $\n_f^+$ and $S_0^+=(g({M_0^-}^2))(r_s)S_{in}$ in $\n_f^+$ (see \eqref{conservsub}), we reformulate 
        (B) 
        in terms of $(\Psi\be_\vphi,A,T)$. Then we obtain 
        \begin{align}\tag{B$^\p$}
        \begin{split}
        &\grad\times((\Phi_0^++\Psi)\be_\vphi)\cdot \grad A=0\qdin \n_{f(0)+f_s}^+,\\
        &A=A_{en,f(0)+f_s}\qdon \Gam_{f(0)+f_s},\\
        &\grad\times((\Phi_0^++\Psi)\be_\vphi)\cdot \grad T=0\qdin \n_{f(0)+f_s}^+,\\
        &T=T_{en,f(0)+f_s}\qdon \Gam_{f(0)+f_s}
        \end{split}
        \end{align}
        where 
        \begin{align}\label{Aentdef}
        A_{en,f(0)+f_s}:=L_-\qdon \Gam_{f(0)+f_s}
        \end{align}
        and 
        \begin{align}\label{Tentdef}
        T_{en,f(0)+f_s}:=
        g\left(\left(\frac{\bu_-\cdot {\bm\nu}_{f(0)+f_s}}{c_-}\right)^2\right)S_-
        -\left(g({M_0^-}^2)\right)(r_s)S_{in}\qdon \Gam_{f(0)+f_s}.
        \end{align}
        
        For later use, we present the following estimate of $T_{en,f(0)+f_s}$. 
        \begin{lem}\label{LemTenestimate} 
        	Let $f(0)$ and $\fsp$  be a constant and a function in $C^{1,\alpha}_{(-\alpha,\{\theta=\theta_1\})}((0,\theta_1))$, respectively.
        	Let $\delta_5$ be a positive constant  
        	such that if $(\rho_-,\bu_-,p_-)$ satisfies \eqref{supesti} 
        	for $\sigma\in (0,\delta_5]$, then 
        	\begin{align*}
        		\sup_{\n}|\frac{1}{c_-}|\le C
        	\end{align*}
        	where $C$ is a positive constant depending on $(\rho_0^-,u_0^-,p_0^-)$, $\gam$, $r_0$, $r_1$ and $\delta_5$.
        	Suppose that $f=f(0)+f_s$ satisfies \eqref{f}. 
        	Also, suppose that $(\rho_-,\bu_-,p_-)$ is an axisymmetric supersonic solution of \eqref{euler} in $\n$ satisfying
        	\eqref{supesti} for $\sigma\in (0,\delta_5]$.
        	Then there holds
        	\begin{align}\label{Tenesti}
        	||T_{en,f}||_{1,\alpha,\Gam_f}^{(-\alpha,\pt\Gam_f)}\le
        	C(|f(0)-r_s|+||\fsp||_{1,\alpha,(0,\theta_1)}^{(-\alpha,\{\theta=\theta_1\})}) 
        	+C\sigma
        	\end{align}
        	where $C$s are positive constants depending on $(\rho_0^-,u_0^-,p_0^-)$, $\gam$, $r_0$, $r_s$, $r_1$,  $\theta_1$, $\alpha$, $\delta_1$ and $\delta_5$. 
        	\begin{proof}
        		In this proof, $C$s denote positive constants depending on the whole or a part of $(\rho_0^-,u_0^-,p_0^-)$, $\gam$, $r_0$, $r_s$, $r_1$,  $\theta_1$, $\alpha$, $\delta_1$ and $\delta_5$. Each $C$ in different situations differs from each other.

        		By the fact that $(\rho_-,\bu_-,p_-)$ and $f=f(0)+f_s$ are axisymmetric,  $T_{en,f}$ defined in \eqref{Tentdef} can be regarded as a function of $\theta$. As a function of $\theta$, $T_{en,f}$ can be written as
        		\begin{align}
        		\nonumber T_{en,f}
        		&=\int_0^1(g({M_0^-}^2))^\p(t f(\theta)+(1-t)r_s)S_{in}dt (f(\theta)-r_s)
        		\\
        		&\;\nonumber+\left(g\left(\left(\frac{\bu_-\cdot {\bm\nu}_f(\theta)}{c_-}\right)^2\right)S_-\right)(f(\theta),\theta)-\left(g({M_0^-}^2)\right)(f(\theta),\theta)S_{in}
        		\\&\nonumber=:(a)+(b).
        		\end{align}
        		To estimate $||T_{en,f}||_{1,\alpha,\Gam_f}^{(-\alpha,\pt\Gam_f)}$,
        		we estimate $(a)$ and $(b)$ in $C^{1,\alpha}_{(-\alpha,\{\theta=\theta_1\})}((0,\theta_1))$, respectively. 
        		
        		Since an estimate of $(a)$ in $C^{1,\alpha}_{(-\alpha,\{\theta=\theta_1\})}((0,\theta_1))$ is obtained directly: 
        		\begin{align}\label{Ten1}
        		||(a)||_{1,\alpha,(0,\theta_1)}^{(-\alpha,\{\theta=\theta_1\})}\le C
        		(|f(0)-r_s|+||f_s^\p||_{1,\alpha,(0,\theta_1)}^{(-\alpha,\{\theta=\theta_1\})})
        		\end{align}
        		where we used $||f_s||_{1,\alpha,(0,\theta_1)}^{(-\alpha,\{\theta=\theta_1\})}\le C||f_s^\p||_{1,\alpha,(0,\theta_1)}^{(-\alpha,\{\theta=\theta_1\})}$, we only estimate $(b)$ in $C^{1,\alpha}_{(-\alpha,\{\theta=\theta_1\})}((0,\theta_1))$.



        		
        		Estimate of $(b)$ in $C^{1,\alpha}_{(-\alpha,\{\theta=\theta_1\})}((0,\theta_1))$:
        		
        		Decompose $(b)$ into two parts:
        		\begin{align*}
        		\left(g\left(\left(\frac{\bu_-\cdot \be_r}{c_-}\right)^2\right)S_-\right)(f(\theta),\theta)-\left(g({M_0^-}^2)\right)(f(\theta),\theta)S_{in}=:(b)_1
        		\end{align*}
        		and
        		\begin{align*}
        		\left(g\left(\left(\frac{\bu_-\cdot {\bm\nu}_f}{c_-}\right)^2\right)S_-\right)(f(\theta),\theta)-\left(g\left(\left(\frac{\bu_-\cdot \be_r}{c_-}\right)^2\right)S_-\right)(f(\theta),\theta)=:(b)_2.
        		\end{align*}
        		With \eqref{supesti} for $\sigma\in (0,\delta_5]$ and 
        		\eqref{f}, we estimate $(b)_1$ in   $C^{1,\alpha}_{(-\alpha,\{\theta=\theta_1\})}((0,\theta_1))$. 
        		Then we obtain 
        		\begin{align}\label{Ten2}
        		||(b)_1||_{1,\alpha,(0,\theta_1)}^{(-\alpha,\{\theta=\theta_1\})}\le C\sigma.
        		\end{align}
        		Write $(b)_2$ as 
        		\begin{align}\label{gprime}
        		\int_0^1 g^\p\left(\left(\frac{\bu_-}{c_-}\cdot(t{\bm \nu}_f+(1-t)\be_r)\right)^2\right)2\left(\frac{\bu_-}{c_-}\cdot (t{\bm \nu}_f+(1-t)\be_r)\right)S_-dt \frac{\bu_-}{c_-}\cdot({\bm \nu}_f-\be_r).
        		\end{align}
        		By ${\bm \nu_f}=\frac{\be_r-\frac{f^\p}{f}\be_\theta}{\sqrt{1+(\frac{f^\p}{f})^2}}$, ${\bm \nu}_f-\be_r$ can be written as
        		\begin{align*}
        		{\bm \nu}_f-\be_r
        		=\left(\int_0^1-\frac 1 2 \left(1+t\left(\frac{f^\p}{f}\right)^2\right)^{-\frac 3 2}dt\left(\frac{f^\p}{f}\right)^2\right)\be_r-\frac{\frac{f^\p}{f}}{\sqrt{1+\left(\frac{f^\p}{f}\right)^2}}\be_{\theta}.
        		\end{align*}
        		Substitute this expression of ${\bm \nu}_f-\be_r$ into ${\bm \nu}_f-\be_r$ in \eqref{gprime} and then 
        		estimate \eqref{gprime} 
        		in $C^{1,\alpha}_{(-\alpha,\{\theta=\theta_1\})}((0,\theta_1))$ with \eqref{supesti} for $\sigma\in (0,\delta_5]$ and 
        		\eqref{f}. Then we obtain
        		\begin{align*}
        		||(b)_2||_{1,\alpha,(0,\theta_1)}^{(-\alpha,\{\theta=\theta_1\})}\le C (||f_s^\p||_{1,\alpha,(0,\theta_1)}^{(-\alpha,\{\theta=\theta_1\})})^2+C\delta_5||f_s^\p||_{1,\alpha,(0,\theta_1)}^{(-\alpha,\{\theta=\theta_1\})}.
        		\end{align*}
        		Using the fact that $||f_s^\p||_{1,\alpha,(0,\theta_1)}^{(-\alpha,(0,\theta_1))}\le C||f-r_s||_{2,\alpha,\Lambda}^{(-1-\alpha,\pt\Lambda)}\le C\delta_1$, 
        		we get from this estimate
        		\begin{align}\label{Ten5}
        		||(b)_2||_{1,\alpha,(0,\theta_1)}^{(-\alpha,\{\theta=\theta_1\})}\le C||f_s^\p||_{1,\alpha,(0,\theta_1)}^{(-\alpha,\{\theta=\theta_1\})}.
        		\end{align}
        		Combining \eqref{Ten2} and \eqref{Ten5}, we obtain
        		\begin{align}
        		\label{Ten4}
        		||(b)||_{1,\alpha,(0,\theta_1)}^{(-\alpha,\{\theta=\theta_1\})}\le  C 
        		||f_s^\p||_{1,\alpha,(0,\theta_1)}^{(-\alpha,(0,\theta_1))}+C\sigma.
        		\end{align}
        		
        		From the facts that $\pt_\theta T_{en,f}(0)=0$, $(a)\in C^{1,\alpha}_{(-\alpha,\{\theta=\theta_1\})}((0,\theta_1))$ and $(b)\in C^{1,\alpha}_{(-\alpha,\{\theta=\theta_1\})}((0,\theta_1))$ (obtained from \eqref{Ten1} and \eqref{Ten4}), we see that 
        		$T_{en,f}\in C^{1,\alpha}_{(-\alpha,\pt\Gam_f)}(\Gam_f)$ (see Lemma \ref{lemaxi}). 
        		Then $||T_{en,f}||_{1,\alpha,\Gam_f}^{(-\alpha,\pt\Gam_f)}$ can be estimated by estimating $||T_{en,f}||_{1,\alpha,(0,\theta_1)}^{(-\alpha,\{\theta=\theta_1\})}$. 
        		By this fact, \eqref{Ten1} and \eqref{Ten4}, we have \eqref{Tenesti}.
        		This finishes the proof. 
        	\end{proof}
        \end{lem}
        From (A$^\p$), (B$^\p$), the Pseudo Free Boundary Problem is naturally derived. 
        We explain 
        this below. 
        
        For a given $(\rho_-,\bu_-,p_-,p_{ex})$, find $(\Psi\be_{\vphi},A,T)$ satisfying  (A$^\p$), (B$^\p$) using an iteration scheme for a fixed boundary problem (for example, in a fixed domain $\n_{f(0)+f_s}^+$, 
        for a given  $\Psi\be_{\vphi}$, solve (B$^\p$), substitute the resultant $A$ and $T$ and the previously given $\Psi\be_\vphi$ into 
        the right-hand sides of (A$^\p$), obtain a new $\Psi\be_{\vphi}$ by solving this (A$^\p$) and show that a new $\Psi\be_{\vphi}$ is equal to the given $\Psi\be_{\vphi}$ using a fixed point argument). Then since $(\Psi\be_{\vphi},A,T)$ we find in this way does not satisfy 
        the fifth equation of (A$^\p$) 
        in general, 
        this iteration scheme does not give 
        a subsonic solution of \eqref{eqV}-\eqref{eqS} satisfying \eqref{compc} in general. 
        From the facts that the entropy at a point on a shock location in the subsonic side is conserved along the streamline passing through that point and the entropy of the downstream subsonic solution of a radial transonic shock solution in a divergent nozzle on a shock location 
        monotonically increases as a shock location moves toward the exit 
        (see Lemma \ref{Slem}), we see that we can find 
        $(\Psi\be_{\vphi},A,T)$ satisfying  the fifth equation of (A$^\p$) 
        by varying $S$ on $\Gam_{ex}$ by adjusting $f(0)$.  
        From this fact, Problem \ref{pb3} is derived.


        (A$^\p$) and (B$^\p$) are of the form of 
        one linear boundary value problem for a singular elliptic equation (this will be seen in the next subsection) and two initial value problems of a transport equation whose coefficient is an axisymmetric and divergence-free vector field, respectively. We will study 
        these problems, seperately,
        in \S \ref{subelliptic} and \S \ref{subtrans}. 

        \subsection{
        	Linear boundary value problem for 
        	a singular elliptic equation} 
        \label{subelliptic} 
        
        Fix the right-hand sides of the first and fourth equation in (A$^\p$) with the fifth equation in (A$^\p$) satisfied. Then we obtain 
        \begin{align}\label{EE}
        &\grad\times\left(\frac{1}{\rho_0^+}(1+\frac{{u_0^+}\be_r\otimes u_0^+\be_r}{{c_0^+}^2-{u_0^+}^2})\grad\times( \Psi\be_{\vphi})\right)={\bm F}\qd\tx{in}\qd \n_f^+,\\
        \label{EEbc}
        &\Psi\be_{\vphi}=\begin{cases}
        h_1\be_{\vphi}\qd\tx{on} \qd\Gam_f,\\
        \frac{f(\theta_1)h_1(f(\theta_1),\theta_1)}{r}\be_{\vphi} 
        \qd\tx{on}\qd \Gam_w^+,\\
        h_2\be_{\vphi}\qd\tx{on} \qd\Gam_{ex}
        \end{cases}
        \end{align}
        where $f$, ${\bm F}$ and $h_i\be_{\vphi}$ for $i=1,2$ are functions given in Lemma \ref{lemSE}. 
        Since \eqref{EE} is expressed as
        \begin{multline}\label{curlformsig}
        \bigg(-\frac{1}{\rho_0^+}\left(\Delta\Psi-\frac{\Psi}{r^2\sin^2\theta}\right)+\frac{\pt_r\rho_0^+}{{\rho_0^+}^2r}\pt_r(r\Psi)\\-\frac{{u_0^+}^2}{\rho_0^+({c_0^+}^2-{u_0^+}^2)r^2}\left(\frac{1}{\sin\theta}\pt_{\theta}(\sin\theta\pt_{\theta}\Psi)
        -\frac{\Psi}{\sin^2\theta}\right)\bigg)\be_{\vphi}={\bm F},
        \end{multline}
        \eqref{EE}, \eqref{EEbc} is a linear boundary value problem for a singular equation as a problem for $\Psi$. 
        Thus, the standard elliptic theorems cannot be applied to 
        this problem as a problem for $\Psi$. 
        We resolve the singularity issue in \eqref{EE}, \eqref{EEbc} by dealing with  \eqref{EE}, \eqref{EEbc} as a boundary value problem for an elliptic system. 

        The following is the main result in this subsection. 
        \begin{lem}\label{lemSE} 
        	Let $\alpha\in(\frac{2}{3},1)$. Suppose that $f$ is as in Lemma \ref{lemrholinear} and satisfy $f^\p(\theta_1)=0$.
        	Also, suppose that ${\bm F}:\n_f^+\ra \R^3$ is a function in $C^{\alpha}_{(1-\alpha,\Gam_w^+)}(\n_f^+)$ having the form
        	\begin{align}\label{Fform}
        	{\bm F}=\sum_i A^i \pt_r B^i \be_{\vphi}+\sum_i C^i\pt_{\theta}D^i \be_{\vphi}+E\frac{\pt_\theta(F\sin\theta)}{\sin\theta}\be_{\vphi}
        	\end{align}
        	where $A^i$, $B^i$, $C^i$, $D^i$, $E$ and $F$ are axisymmetric functions satisfying
        	\begin{align}\label{regularityF}
        	&A^i\in C^{1,\alpha}_{(-\alpha,\Gam_w^+)}(\n_f^+),\; B^i\be_\theta\in C^{1,\alpha}_{(-\alpha,\Gam_w^+)}(\n_f^+), \;
        	C^i\in C^{1,\alpha}_{(-\alpha,\Gam_w^+)}(\n_f^+),
        	\\& D^i\in C^{1,\alpha}_{(-\alpha,\Gam_w^+)}(\n_f^+),\; E\be_{\vphi}\in C^{1,\alpha}
        	(\ol{\n_f^+})\;\tx{and}\; F\be_{\vphi}\in C^{1,\alpha}
        	(\ol{\n_f^+}).\nonumber
        	\end{align}
        	Finally, suppose that $h_1\be_{\vphi}:\Gam_f\ra \R^3$ and $h_2\be_{\vphi}:\Gam_{ex}\ra \R^3$ are axisymmetric functions in $C^{2,\alpha}_{(-1-\alpha,\pt\Gam_f)}(\Gam_f)$ and $C^{2,\alpha}_{(-1-\alpha,\pt\Gam_{ex})}(\Gam_{ex})$, respectively, and satisfy $\frac{f(\theta_1)h_1(f(\theta_1),\theta_1)}{r_1}=h_2(r_1,\theta_1)$.
        	Then the boundary value problem \eqref{EE}, \eqref{EEbc}
        	has a unique axisymmetric $C^{2,\alpha}_{(-1-\alpha,\Gam_w^+)}(\n_f^+)$ solution $\Psi\be_{\vphi}$. Furthermore, the solution $\Psi\be_{\vphi}$ satisfies 
        	\begin{align}\label{Westimate}
        	||\Psi\be_{\vphi}||_{2,\alpha,\n_f^+}^{(-1-\alpha,\Gam_w^+)}
        	\le C\left(||{\bm F}||_{\alpha,\n_f^+}^{(1-\alpha,\Gam_w^+)}+\sum_{i=1,2,3}F_i
        	+||h_1\be_{\vphi}||_{2,\alpha,\Gam_f}^{(-1-\alpha,\pt \Gam_f)}+||h_2\be_{\vphi}||_{2,\alpha,\Gam_{ex}}^{(-1-\alpha,\pt \Gam_{ex})}\right)
        	\end{align}
        	where $C$ is a positive constant depending only on $(\rho_0^+,u_0^+,p_0^+)$, $\gam$, $r_s$, $r_1$, $\theta_1$ and $\alpha$, and 
        	\begin{align*}
        	&F_1=\sum_i
        	||A^i||_{W^{1,3}(\n_f^+)\cap L^\infty(\n_f^+)}||B^i\be_{\theta}||_{\alpha,\n_f^+}\\
        	&F_2=\sum_i||C^i||_{W^{1,3}(\n_f^+)\cap L^\infty(\n_f^+)} ||D^i||_{\alpha,\n_f^+},\\
        	&F_3=||E\frac{\pt_\theta(F\sin\theta)}{\sin\theta}\be_{\vphi}||_{L^q(\n_f^+)}
        	\end{align*}
        	for $q=\frac{3}{1-\alpha}$ with 
        	$||\cdot||_{W^{1,3}(\n_f^+)\cap L^\infty(\n_f^+)}:=
        	||\cdot||_{W^{1,3}(\n_f^+)}+||\cdot||_{L^\infty(\n_f^+)}$. 
        \end{lem}
        \begin{remark}
        	The form of ${\bm F}$ given in \eqref{Fform} is obtained from ${\bm F}_2(\Psi\be_\vphi,A,T)$ in \eqref{Wlinear}. This form will be used in the proof of Lemma \ref{lem1alpha}. 
        \end{remark}
        To avoid the singularity issue in \eqref{EE}, \eqref{EEbc}, we deal with \eqref{EE}, \eqref{EEbc} as a boundary value problem for a vector equation. 
        From $\grad\times(\grad \times(\Psi\be_{\vphi}))=-\Delta (\Psi\be_{\vphi})$, we expected that \eqref{EE} can be transformed into a form of an elliptic system. 
        We, motivated by the work in \cite{MR3739930}, thought that if \eqref{EE} can be transformed into a solvable elliptic system form, then the unique existence and regularity of solutions of \eqref{EE}, \eqref{EEbc} can be obtained by obtaining those of solutions of the elliptic system form of \eqref{EE}, \eqref{EEbc} as a boundary value problem for an elliptic system. 
        
        For this argument to hold, it is needed to find a solvable elliptic system form of \eqref{EE}. For computational convenience to find such a form and for our later argument (reflection argument in the proof of Lemma \ref{lemalpha} and Lemma \ref{lem1alpha}), 
        we use the following tensor notation. 
        
        \textbf{Tensor notation}
        
        Let ${\bm a}\otimes {\bm b}={\bm a}{\bm b}^T$
        for ${\bm a}$, ${\bm b}\in \R^3$. Then ${\bm a}\otimes {\bm b}$ is a linear map from $\R^3$ to $\R^3$ and any linear map from $\R^3$ to $\R^3$ can be represented using this operator. 
        This notation can be extended so that using the extension of this operator, we can represent any linear map from $\R^{3\times3}$ to $\R^{3\times3}$. 
        For any ${\bm a}$, ${\bm b}$, ${\bm c}$, ${\bm d}\in \R^3$, let 
        ${\bm a}\otimes {\bm b}\otimes {\bm c}\otimes {\bm d}$ 
        be an operator satisfying
        \begin{align}\label{abcddef}
        ({\bm a}\otimes {\bm b}\otimes {\bm c}\otimes {\bm d})({\bm e}\otimes {\bm f})=({\bm d}\cdot {\bm e})({\bm c}\cdot {\bm f}){\bm a}\otimes{\bm b}
        \end{align}
        where ${\bm e}$, ${\bm f}\in \R^3$. Then ${\bm a}\otimes {\bm b}\otimes {\bm c}\otimes {\bm d}$ is a linear map from $\R^{3\times3}$ to $\R^{3\times3}$ and any linear map from $\R^{3\times3}$ to $\R^{3\times3}$ can be represented using this operator. 

        By direct computation done by using the above tensor notation, 
        we found the following form of \eqref{EE}
        \begin{multline}\label{Psielliptic}
        \Div \left(\frac{{c_0^+}^2}{\rho_0^+({c_0^+}^2-{u_0^+}^2)}\left({\bm I}-\frac{ {u_0^+}^2}{{c_0^+}^2}(\mc I\otimes \be_r\otimes\be_r \otimes \mc I )\right)D(\Psi \be_{\vphi})\right)\\-\frac{\pt_r\rho_0^+}{{\rho_0^+}^2r}\Psi \be_{\vphi}=-{\bm F}\qdin \n_f^+
        \end{multline}
        where ${\bm I}$ is the identity map from $\R^{3\times3}$ to $\R^{3\times3}$ and $\mc I\otimes \be_r\otimes\be_r \otimes \mc I$ is a linear map from $\R^{3\times3}$ to $\R^{3\times3}$ satisfing 
        \begin{align*}
        (\mc I\otimes \be_r\otimes \be_r\otimes \mc I)({\bm a}\otimes{\bm b})=({\bm b}\cdot \be_r){\bm a}\otimes \be_r
        \end{align*}
        for any ${\bm a}$, ${\bm b}\in R^3$
        (see the definition of $\mc I\otimes {\bm a}\otimes {\bm b} \otimes \mc I$ for any ${\bm a}$, ${\bm b}\in \R^3$ in \eqref{IItensor}). 
        By $M_0^+<1$ in $\n_f^+$ and the boundedness of $(\rho_0^+,u_0^+,p_0^+)$ in $\n_f^+$ for $\n_f^+\subset \n_{r_s-\delta_1}^+$, 
        there exist positive constants  $\mu$  and $\mc M$ such that 
        \begin{align}\label{legendre}
         \mu|{\bm \xi}|^2\le
        \sum_{\alpha,\beta,i,j=1}^3 \left[\frac{{c_0^+}^2}{\rho_0^+({c_0^+}^2-{u_0^+}^2)}\left({\bm I}-\frac{ {u_0^+}^2}{{c_0^+}^2}(\mc I\otimes \be_r\otimes\be_r \otimes \mc I )\right)\right]_{ij}^{\alpha\beta}\xi^i_\alpha\xi^j_\beta
        \le\mc M|{\bm \xi}|^2 
        \qdin\n_f^+
        \end{align}
        for any ${\bm \xi}\in \R^{3\times 3}$ where $|{\bm \xi}|=\sqrt{\sum_{i,j=1}^3|\xi_j^i|^2}$ with ${\bm \xi}=[\xi_j^i]$. 
        And by $\pt_r \rho_0^+>0$ in $\n_f^+$ for $\n_f^+\subset \n_{r_s-\delta_1}^+$ (see Remark \ref{rmkrhop}), 
        $$\frac{\pt_r\rho_0^+}{{\rho_0^+}^2r}> 0\qdin \n_f^+.$$
        Hence, \eqref{Psielliptic} is a form of a solvable elliptic system for a dirichlet boundary condition. 

        We obtain the unique existence and regularity of solutions of \eqref{EE}, \eqref{EEbc} by obtaining those of solutions of \eqref{Psielliptic}, \eqref{EEbc} as a boundary value problem for an elliptic system. 
        The result of the unique existence and regularity of solutions of \eqref{Psielliptic}, \eqref{EEbc} as a boundary value problem for an elliptic system is given in the following lemma. 
        \begin{lem}\label{lemLES} Under the 
        	assumptions as in Proposition \ref{lemSE}, the boundary value problem
        	\begin{align}\label{LES}
        	&\Div \left(\frac{{c_0^+}^2}{\rho_0^+({c_0^+}^2-{u_0^+}^2)}\left({\bm I}-\frac{ {u_0^+}^2}{{c_0^+}^2}(\mc I\otimes \be_r\otimes\be_r \otimes \mc I )\right)D\bU\right)-\frac{\pt_r\rho_0^+}{{\rho_0^+}^2r}\bU=-{\bm F}\qdin \n_f^+,\\
        	\label{LESbc}
        	&{\bU}=\begin{cases}
        	h_1\be_{\vphi}\qd\tx{on} \qd\Gam_f,\\
        	\frac{f(\theta_1)h_1(f(\theta_1),\theta_1)}{r}\be_{\vphi}
        	\qd\tx{on}\qd \Gam_w^+,\\
        	h_2\be_{\vphi}\qd\tx{on} \qd\Gam_{ex},
        	\end{cases}
        	\end{align}
        	has a unique $C^{2,\alpha}_{(-1-\alpha,\Gam_w^+)}(\n_f^+)$ solution $\bU$, and this solution $\bU$ satisfies
        	\begin{align}\label{Uestimate}
        	||\bU||_{2,\alpha,\n_f^+}^{(-1-\alpha,\Gam_w^+)}
        	\le \underbrace{C\left(||{\bm F}||_{\alpha,\n_f^+}^{(1-\alpha,\Gam_w^+)}+\sum_{i=1,2,3}F_i
        	+||h_1\be_{\vphi}||_{2,\alpha,\Gam_f}^{(-1-\alpha,\pt \Gam_f)}+||h_2\be_{\vphi}||_{2,\alpha,\Gam_{ex}}^{(-1-\alpha,\pt \Gam_{ex})}\right)}_{=:C^*}
        	\end{align}
        	where $C$ is a positive constant depending on $(\rho_0^+,u_0^+,p_0^+)$, $\gam$, $r_s$, $r_1$, $\theta_1$ and $\alpha$ and $F_i$ for $i=1,2,3$ are constants given in Lemma \ref{lem1alpha}. This solution $\bU$ is of the form $\Psi(r,\theta)\be_{\vphi}$.
        \end{lem}
        One can see that 
        that Lemma \ref{lemSE} is obtained from Lemma \ref{lemLES}. To prove Lemma \ref{lemSE}, in the remainder of this subsection, we prove Lemma \ref{lemLES}.  
        
        Transform \eqref{LES}, \eqref{LESbc} into the following ${\bm 0}$ boundary problem: 
        \begin{align}\label{veqn}
        &\Div({\bm A}D \bU^\sharp)-d \bU^\sharp=-{\bm F}-\Div({\bm A}D {\bm h})+ d{\bm h}
        (=:{\bm F}^\sharp)
        \qdin \n_{f}^+,\\
        \label{vbdry}
        &\bU^\sharp=0\qdon \pt\n_{f}^+
        \end{align}
        where $\bU^\sharp:=\bU-{\bm h}$, 
        ${\bm A}=\frac{{c_0^+}^2}{\rho_0^+({c_0^+}^2-{u_0^+}^2)}\left({\bm I}-\frac{ {u_0^+}^2}{{c_0^+}^2}(\mc I\otimes \be_r\otimes\be_r \otimes \mc I )\right),$
        $d:=\frac{\pt_r\rho_0^+}{{\rho_0^+}^2r}$ and
        $ {\bm h}:=\frac{(r-f(\theta))\frac{r_1}{r} h_2\be_{\vphi}+(r_1-r)\frac{f(\theta)}{r} h_1\be_{\vphi}}{r_1-f(\theta)}$.
        One can see that Lemma \ref{lemLES} can be proved by showing \eqref{veqn}, \eqref{vbdry}
        has a unique weak solution, 
        the weak solution of  \eqref{veqn}, \eqref{vbdry} is in $C^{2,\alpha}_{(-1-\alpha,\Gam_w^+)}(\n_f^+)$ and the 
        $C^{2,\alpha}_{(-1-\alpha,\Gam_w^+)}(\n_f^+)$ solution of \eqref{veqn}, \eqref{vbdry} is of the form $\Psi(r,\theta)\be_\vphi$. 
        Hereafter, we prove these statements.  
        
        
        We first prove the unique existence of weak solution of \eqref{veqn}, \eqref{vbdry}.

        
        
        \begin{lem}\label{lemH1}
        	Under the assumptions as in Proposition \ref{lemSE}, the boundary value problem
        	\eqref{veqn}, \eqref{vbdry} has a unique weak solution $\bU^\sharp\in H_0^1(\n_f^+)$. Furthermore,  
        	$\bU^\sharp$ satisfies
        	$$||\bU^\sharp||_{H^1(\n_{f}^+)}\le C||{\bm F}^\sharp||_{L^2(\n_{f}^+)} $$
        	where $C$ is a positive constant depending on $(\rho_0^+,u_0^+,p_0^+)$, $\gam$ and $\n_{f}^+$. 
        	\begin{proof}
        		Write \eqref{veqn}, \eqref{vbdry} in the form 
        		\begin{align}\label{intveqn}
        		(B[\bU^\sharp,{\bm\xi}]:=)\int_{\n_{f}^+}{\bm A}D\bU^\sharp D{\bm \xi}+d\bU^\sharp {\bm \xi} =\int_{\n_{f}^+}{\bm F}^\sharp{\bm \xi}(=:<{\bm F}^\sharp,{\bm \xi}>)
        		\end{align}
        		for all ${\bm \xi}\in H_0^1(\n_{f}^+)$.
        		Then $B$ is a bilinear map satisfying $B[\bU^\sharp,{\bm \xi}]\le C ||\bU^\sharp||_{H^1(\n_{f}^+)} ||{\bm \xi}||_{H^1(\n_{f}^+)}$ for a constant $C>0$ 
        		and 
        		$$\underline{\mu}||\bU^\sharp||_{H^1(\n_{f}^+)}^2\le B[\bU^\sharp,\bU^\sharp]$$
        		where $\underline{\mu}=\min( \mu,\min_{\n_f^+}d)>0$ with $\mu$ given in \eqref{legendre}. And by $h_1 \be_{\vphi}\in C^{2,\alpha}_{(-1-\alpha,\pt \Gam_f)}(\Gam_f)$, $h_2 \be_{\vphi}\in C^{2,\alpha}_{(-1-\alpha,\pt \Gam_{ex})}(\Gam_{ex})$ and ${\bm F}\in C^{\alpha}_{(1-\alpha,\Gam_w^+)}(\n_f^+)$ for $\alpha\in (\frac 2 3,1)$, ${\bm F}^\sharp\in L^2(\n_{f}^+)$ and thus $<{\bm F}^\sharp,{\bm \xi}>$ is a bounded linear functional on $H_0^1(\n_f^+)$. 
        		With these facts, we apply the Lax-Milgram Theorem to \eqref{intveqn}. Then we obtain that
        		there exists a unique $\bU^\sharp\in H_0^1(\n_f^+)$ such that \eqref{intveqn} holds 
        		for all ${\bm \xi}\in H_0^1(\n_{f}^+)$. This finishes the proof. 
        	\end{proof}
        \end{lem}
        We next prove that this weak solution is in $C^{2,\alpha}_{(-1-\alpha,\Gam_w^+)}(\n_f^+)$. For this, we prove that the weak solution of \eqref{veqn}, \eqref{vbdry} is in 
         $C^\beta(\ol{\n_f^+})$ for any $\beta\in (0,1)$ and $C^{1,\alpha}(\ol{\n_f^+})$. 
        \begin{lem}\label{lemalpha}
        	Under the assumptions as in Proposition \ref{lemSE}, let $\bU^\sharp$ be a weak solution of the boundary value problem \eqref{veqn}, \eqref{vbdry}. 
        	Then for any $\beta\in (0,1)$, 
        	\begin{align*}
        	||\bU^\sharp||_{\beta,\n_{f}^+}\le C\left(||{\bm F}^\sharp||_{L^p(\n_{f}^+)}+||\bU^\sharp||_{H^1(\n_f^+)}\right) 
        	\end{align*}
        	for $p=\frac{3}{2-\beta}$ where $C$ is a positive constant depending on  $\mu$, 
        	$\mc M$, $\tau$, $||d||_{L^3(\n_f^+)}$ and $\n_f^+$, and $\tau$ is the modulus of continuity of ${\bm A}$ in $\n_f^+$ given as
        	\begin{align}
        	\label{tau}\tau(t)=\mathop{\sup_{{\rm x},{\rm y}\in \n_f^+,}}_{|{\rm x}-{\rm y}|\le t}(\sum|A^{\alpha \beta}_{ij}({\rm x})-A^{\alpha \beta}_{ij}({\rm y})|^2)^{\frac 1 2}.
        	\end{align}
        \end{lem}
        \begin{lem}\label{lem1alpha}
        	Under the assumptions as in Proposition \ref{lemSE}, let $\bU^\sharp$ be a weak solution of \eqref{veqn}, \eqref{vbdry}. 
        	Then 
        	\begin{align*}
        	||\bU^\sharp||_{1,\alpha,\n_f^+}\le C\left(\sum_{i=1,2,3}F_i
        	+||{\bm h}||_{1,\alpha,\n_f^+ }+||\bU^\sharp||_{H^1(\n_f^+)}+||{\bm F}^\sharp||_{L^3(\n_f^+)}\right)
        	\end{align*}
        	where 
        	$C$ is a positive constant depending on 
        	$\mu$, $\mc M$, $||{\bm A}||_{\alpha,\n_f^+}$, $||d||_{L^q(\n_f^+)}$ with $q=\frac{3}{1-\alpha}$ and $\n_f^+$ and $F_i$ for $i=1,2,3$ are constants given in Lemma \ref{lem1alpha}. 
        \end{lem}
        We will prove Lemma \ref{lemalpha} and Lemma \ref{lem1alpha} using 
        the method of freezing the coefficients (Korn's device of freezing the coefficients) (see \cite[Chapter 3]{MR717034}). Since $\n_f^+$ is a Lipshitz domain, 
        $\bU^\sharp\in C^\beta(\ol{\n_f^+})$ and $\bU^\sharp\in C^{1,\alpha}(\ol{\n_f^+})$
        can be proved by showing that (i) there are positive constants $C$ and $R$ such that 
        \begin{align}\label{igalpha}
        \int_{D_{t}({\rm x}_0)}|D\bU^\sharp|^2\le C {t}^{3-2+2\beta}\qd\tx{for any}\qd 0<t<R
        \end{align}
        for all ${\rm x}_0\in \ol{\n_{f}^+}$,
        and (ii) there are positive constants $C$ and $R$ such that 
        \begin{align}\label{ig1alpha}
        \int_{D_{t}({\rm x}_0)} |D\bU^\sharp-(D\bU^\sharp)_{{\rm x}_0,{t}}|^2\le C{t}^{3+2\alpha}\qd\tx{for any}\qd 0<{t}< R
        \end{align}
        for all ${\rm x}_0\in \ol{\n_{f}^+}$ where $D_{t}({\rm x}_0):=B_{t}({\rm x}_0)\cap \n_{f}^+$ with $B_t({\rm x}_0):=\{{\rm x}\in \R^3:|{\rm x}-{\rm x}_0|<t\}$ and $(D\bU^\sharp)_{{\rm x}_0,{t}}:=\frac{1}{|D_{t}({\rm x}_0)|}\int_{D_{t}({\rm x}_0)}D\bU^\sharp$. We prove 
        (i) and (ii) 
        by obtaining \eqref{igalpha} and \eqref{ig1alpha} at each point ${\rm x}_0$ in $\ol{\n_f^+}$ for $C$ and $R$ independent of ${\rm x}_0$ 
        using the method of freezing the coefficients. 
        When we do this, there exists some difficulty. 
        For the case of ${\rm x}_0\in \n_f^+$ or $\Gam_f\cup \Gam_w^+\cup \Gam_{ex}$, we can obtain the integral estimates for the fixed coefficients equation using the Cacciopolli inequality 
        and the quotient difference method, and obtain \eqref{igalpha} and \eqref{ig1alpha} at $x_0\in \n_f^+$ or $\Gam_f\cup \Gam_w^+\cup \Gam_{ex}$ using these estimates and the method of freezing the coefficients (see \cite[Chapter 6]{MR3887613}).  
        But 
        for the case of ${\rm x}_0\in \ol{\Gam_f}\cap\ol{\Gam_w^+}$ or $\ol{\Gam_w^+}\cap \ol{\Gam_{ex}}$, 
        we cannot obtain 
        the integral estimates for the fixed coefficients equation using the Cacciopolli inequality 
        and the quotient difference method. 
        Thus, we cannot obtain \eqref{igalpha} and \eqref{ig1alpha} at  ${\rm x}_0\in \ol{\Gam_f}\cap\ol{\Gam_w^+}$ or $\ol{\Gam_w^+}\cap \ol{\Gam_{ex}}$ using 
        the standard method of freezing the coefficients. 
        We resolve this difficulty by developing
        some reflection argument that holds for a linear boundary value problem on a Lipschitz  domain whose all corners are perpendicular for an elliptic system whose the domain part of principal coefficients is diagonal with respect to the coordinate systems representing the walls near the corners of the domain. 
        This will be seen in the proof of Lemma \ref{lemalpha} and Lemma \ref{lem1alpha}. 

        Hereafter, we use the following notation:
        \begin{align*}
        &\Div_{\rtp}:=(\be_1\pt_r +\be_2\pt_\theta +\be_3\pt_\vphi )\cdot,\qd D_{\rtp}:=\be_1\pt_r + \be_2\pt_\theta+ \be_3\pt_\vphi\\
        &\mathbb T:\;\tx{a one dimensional torus with period}\;2\pi,\\
        &\n_{a}^{+,*}:=\{\rtp\in \R^3\;|\;a<r<r_1,0<\theta<\theta_1, \vphi\in \mathbb T\},\\
        &\Gam_{a}^{*}:=\{\rtp\in \R^3\;|\;r=a,0<\theta<\theta_1, \vphi\in \mathbb T\},\\
        &\Gam_{w,a}^{+,*}:=\{\rtp\in \R^3\;|\;a<r<r_1,\theta=\theta_1, \vphi\in \mathbb T\}\qd\tx{for}\qd 0<a<r_1\\
        &\Gam_w^{+,*}:=\{\rtp\in \R^3\;|\;f(\theta_1)<r<r_1,\theta=\theta_1, \vphi\in \mathbb T\},\\
        &\Gam_{ex}^*:=\{\rtp\in \R^3\;|\;r=r_1,0<\theta<\theta_1, \vphi\in \mathbb T\} \\
        &{\mf r}, {\mf t}:\; \tx{a radius of a ball in the spherical coordinate system},\\
        &B_{\mf r}^*({\rm x}^*):=\{(r,\theta,\vphi)\in \R^+\times [0,\pi]\times\ \mathbb T:|r-r^*|^2+|\theta-\theta^*|^2+|\vphi-\vphi^*|^2<{\mf r}^2\},\\
        &D_{\mf r}^*({\rm x}^*):=B_{\mf r}^*({\rm x}^*)\cap \n_{f(\theta_1)}^{+,*}
        \qd\tx{for}\qd {\rm x}^*=(r^*,\theta^*,\vphi^*)\in \R^+\times (0,\pi)\times\ \mathbb T.
        \end{align*}
        To prove Lemma \ref{lemalpha} and Lemma \ref{lem1alpha}, we prove the following lemma.
        \begin{lem}\label{lemfixedcoeffi}
        	Let ${\rm x}_0^*= (f(\theta_1),\theta_1,\vphi_0)$ for some $\vphi_0\in \mathbb T$. Let $0<\mf r< \min(\theta_1,\pi,r_1-f(\theta_1))$. Suppose that $\bW\in H^1(D_{\mf r}^*({\rm x}_0^*))$ is a weak solution of
        	\begin{align}\label{fixedcoeffeqn}
        	&\Div_{\rtp}\left( \left.\bigg(\frac{r^2\sin\theta}{\rho_0^+}\right|_{r=f(\theta_1),\theta=\theta_1}\mc I\otimes \be_1\otimes \be_1\otimes \mc I\right.\\\nonumber &+\left.\frac{\sin\theta}{\rho_0^+}(\frac{{c_0^+}^2}{{c_0^+}^2-{u_0^+}^2})\right|_{r=f(\theta_1),\theta=\theta_1}\mc I\otimes \be_2\otimes\be_2\otimes \mc I\\
        	&\left.
        	+\left.\frac{1}{\sin\theta\rho_0^+}(\frac{{c_0^+}^2}{{c_0^+}^2-{u_0^+}^2})\right|_{r=f(\theta_1),\theta=\theta_1}\mc I\otimes \be_3\otimes \be_3 \otimes \mc I \bigg) D_{\rtp} \bW \right)=0\qdin D_{\mf r}^*({\rm x}_0^*)\nonumber,\\
        	\label{fixedcoeffbc}
        	&\bW={\bf 0}  \qdon \pt D_{\mf r}^*({\rm x}_0^*)\cap (\ol{\Gam_{f(\theta_1)}^{*}}\cup\ol{\Gam_{w,f(\theta_1)}^{+,*}}).
        	\end{align}
        	Then for any $\mf t$ 
        	such that $0<\mf t\le \mf r$, there hold
        	\begin{align}\label{Wintesti}
            \int_{D_{\mf t}^*({\rm x}_0^*)}|D_{\rtp} \bW|^2\le C\left(\frac{\mf t}{\mf r}\right)^3\int_{D_{\mf r}^*({\rm x}_0^*)}|D_{\rtp} \bW|^2,
            \end{align}
            and
            \begin{align}
            \label{Wintesti2}
        	\int_{D_{\mf t}^*({\rm x}_0^*)}|D_{\rtp}\bW-(D_{\rtp}\bW)^*_{{\rm x}_0^*,\mf t}|^2\le C\left(\frac{\mf t}{\mf r}\right)^{5}\int_{D_{\mf r}^*({\rm x}_0^*)}|D_{\rtp}\bW-(D_{\rtp}\bW)^*_{{\rm x}_0^*,\mf r}|^2
        	\end{align}
        	where $C$ is a positive constant depending on 
        	$\mu$, $\mc M$, $f(\theta_1)$ and $\theta_1$ 
        	and $(D_{\rtp}\bW)^*_{{\rm x}_0^*,\mf r}:=\frac{1}{|D_{\mf r}^*({\rm x}_0^*)|}\int_{D_{\mf r}^*({\rm x}_0^*)}D_{\rtp}\bW$. 
        	\begin{proof}
        		The result is obtained by using the reflection argument. 
        		
        		Extend \eqref{fixedcoeffeqn} 
        		in $B_{\mf r}^*({\rm x}_0^*)$:
        		\begin{align}\label{fixedextendedeqn}
        		&\Div_{\rtp}\left( \left.\bigg(\frac{r^2\sin\theta}{\rho_0^+}\right|_{r=f(\theta_1),\theta=\theta_1}\right.\mc I\otimes \be_1\otimes \be_1\otimes\mc I \\
        		&\nonumber+\left.\frac{\sin\theta}{\rho_0^+}(\frac{{c_0^+}^2}{{c_0^+}^2-{u_0^+}^2})\right|_{r=f(\theta_1),\theta=\theta_1}\mc I\otimes \be_2\otimes \be_2 \otimes\mc I\\
        		&\left.
        		+\left.\frac{1}{\sin\theta\rho_0^+}(\frac{{c_0^+}^2}{{c_0^+}^2-{u_0^+}^2})\right|_{r=f(\theta_1),\theta=\theta_1}\mc I\otimes \be_3\otimes \be_3 \otimes\mc I\bigg) D_{\rtp} \bW \right)=0\qdin B_{\mf r}^*({\rm x}_0^*).\nonumber
        		\end{align}
        	    Extend a weak solution $\bW$  of \eqref{fixedcoeffeqn}, \eqref{fixedcoeffbc} in $B_{\mf r}^*({\rm x}_0^*)$:
        	    \begin{align*}
        	    	\bW_{ext}=\begin{cases}{\bm W}(r, \theta,\vphi)\qdin   B_{\mf r}^*({\rm x}_0^*)\cap \{r\ge f(\theta_1),\theta\le \theta_1\} \\
        	    		-{\bm W}(2f(\theta_1)-r,\theta,\vphi)\qdin   B_{\mf r}^*({\rm x}_0^*)\cap \{r< f(\theta_1),\theta\le \theta_1\}\\
        	    		-{\bm W}(r,2\theta_1-\theta,\vphi)\qdin  B_{\mf r}^*({\rm x}_0^*)\cap \{r\ge f(\theta_1),\theta>\theta_1\}\\
        	    		{\bm W}(2f(\theta_1)-r,2\theta_1-\theta,\vphi)\qdin  B_{\mf r}^*({\rm x}_0^*)\cap \{r< f(\theta_1),\theta> \theta_1\}
        	    	\end{cases}
        	    \end{align*}
                Then ${\bW}_{ext}\in H^1(B_{\mf r}^*({\rm x}_0^*))$. 
        		There exists a unique weak solution of \eqref{fixedextendedeqn}, 
        		\begin{align}\label{fixedextendedeqnbc}
        		\bW=\bW_{ext}\qdon \pt B_{\mf r}^*({\rm x}_0^*). 
        		\end{align}
        		%
        		We denote the weak solution of \eqref{fixedextendedeqn}, \eqref{fixedextendedeqnbc} by $\ol{\bW}$. 
        		
        		One can check that $-\ol{\bm W}(2f(\theta_1)-r,\theta,\vphi)$ and $-\ol{\bm W}(r,2\theta_1-\theta,\vphi)$ are also weak solutions of \eqref{fixedextendedeqn}, \eqref{fixedextendedeqnbc}. 
        		By this fact and the uniqueness of weak solutions of \eqref{fixedextendedeqn}, \eqref{fixedextendedeqnbc}, 
        		\begin{align}\label{reflect}
        		\ol{\bm W}(r, \theta,\vphi)=
        		-\ol{\bm W}(2f(\theta_1)-r,\theta,\vphi)=-\ol{\bm W}(r,2\theta_1-\theta,\vphi) \qdin B_{\mf r}^*({\rm x}_0^*).
        		\end{align}
        		From \eqref{reflect}, we have $\ol{\bm W}=0$ on $\pt D_{\mf r}^*({\rm x}_0^*)\cap (\ol{\Gam_{f(\theta_1)}^{*}}\cup\ol{\Gam_{w,f(\theta_1)}^{+,*}})$. 
        		By this fact and the fact that $\ol{\bW}$ is a unique weak solution of \eqref{fixedextendedeqn}, \eqref{fixedextendedeqnbc}, $\ol{\bm W}$ is a weak solution of \eqref{fixedcoeffeqn},
        		\begin{align*}
        		\ol{\bm W}=\begin{cases}
        		\bW\qdon\pt D_{\mf r}^*({\rm x}_0^*)\setminus (\ol{\Gam_{f(\theta_1)}^{*}}\cup\ol{\Gam_{w,f(\theta_1)}^{+,*}}),\\
        		0\qdon \pt D_{\mf r}^*({\rm x}_0^*)\cap (\ol{\Gam_{f(\theta_1)}^{*}}\cup\ol{\Gam_{w,f(\theta_1)}^{+,*}}).
        		\end{cases}
        		\end{align*}
        		By the uniqueness of weak solutions of \eqref{fixedcoeffeqn} satisfying ${\bm W}=\bW_{ext}$ on $\pt D_{\mf r}^*({\rm x}_0^*)\setminus (\ol{\Gam_{f(\theta_1)}^{*}}\cup\ol{\Gam_{w,f(\theta_1)}^{+,*}})$ and ${\bm W}=0$ on $\pt D_{\mf r}^*({\rm x}_0^*)\cap (\ol{\Gam_{f(\theta_1)}^{*}}\cup\ol{\Gam_{w,f(\theta_1)}^{+,*}})$, we have
        		\begin{align}
        		\label{WbarW}\ol\bW=\bW \qdin D_{\mf r}^*({\rm x}_0^*).
        		\end{align} 
        		
        	    By \cite[Theorem 2.1, Chapter 3]{MR717034}, 
        	    $\ol{\bW}$ satisfies 
        		for any $0<\mf t\le \mf r$,  
        		\begin{align}\label{igrowthesti}
        		\int_{B_{\mf t}^*({\rm x}_0^*)}|D_{\rtp} \ol{\bm W}
        		|^2\le C\left(\frac{\mf t}{\mf r}\right)^3\int_{B_{\mf r}^*({\rm x}_0^*)}|D_{\rtp} \ol{\bm W}
        		|^2
        		\end{align}
        		and
        		\begin{align}
        		&\label{igrowthesti1}
        		\int_{B_{\mf t}^*({\rm x}_0^*)}|D_{\rtp}\ol{\bm W}
        		-(D_{\rtp}\ol{\bm W}
        		)^{**}_{\til {\rm x}_0^{*},\mf t}|^2
        		\\
        		&\qquad\qquad\qquad\qquad\qquad\le C\left(\frac{\mf t}{\mf r}\right)^{5}\int_{B_{\mf r}^*({\rm x}_0^*)}|D_{\rtp}\ol{\bm W}
        		-(D_{\rtp}\ol{\bm W}
        		)^{**}_{\til {\rm x}_0^{*},\mf r}|^2\nonumber
        		\end{align}
        		where $C$ is a positive constant depending on 
        		$\mu$,  $\mc M$, $f(\theta_1)$ and $\theta_1$ and $(D_{\rtp}\ol{\bm W}
        		)^{**}_{{\rm x}_0^*,\mf r}:=\frac{1}{|B_{\mf r}^*({\rm x}_0^*)|}\int_{B_{\mf r}^*({\rm x}_0^*)}D_{\rtp}\ol{\bm W}
        		$. 
        		By \eqref{reflect} and \eqref{WbarW}, we obtain from \eqref{igrowthesti} and \eqref{igrowthesti1} for any $0<\mf t\le \mf r$,  
        		\begin{align*}
        		4\int_{D_{\mf t}^*({\rm x}_0^*)}|D_{\rtp} \bW|^2\le 4C\left(\frac{\mf t}{\mf r}\right)^3\int_{D_{\mf r}^*({\rm x}_0^*)}|D_{\rtp} \bW|^2
        		\end{align*}
        		and
        		\begin{align*}
        		&4\int_{D_{\mf t}^*({\rm x}_0^*)}|D_{\rtp}\bW-(D_{\rtp}\bW)^*_{{\rm x}_0^*,\mf t}|^2
        		\\
        		&\qquad\qquad\qquad\qquad\qquad\le 4C\left(\frac{\mf t}{\mf r}\right)^{5}\int_{D_{\mf r}^*({\rm x}_0^*)}|D_{\rtp}\bW-(D_{\rtp}\bW)^*_{{\rm x}_0^*,\mf r}|^2
        		\end{align*}
        		where we used 
        		$$(D_{\rtp}\ol{\bm W})^{**}_{{\rm x}_0^*,\mf r}=\frac{1}{|B_{\mf r}^*({\rm x}_0^*)|}\int_{B_{\mf r}^*({\rm x}_0^*)}D_{\rtp}\ol{\bm W}=\frac{1}{|D_{\mf r}^*({\rm x}_0^*)|}\int_{D_{\mf r}^*({\rm x}_0^*)}D_{\rtp}\bW.$$ 
        		This finishes the proof.  
        	\end{proof}
        \end{lem}
        The following Corollary is obtained from Lemma \ref{lemfixedcoeffi} in the same way that Corollary 3.11 is obtained from Lemma 3.10 in \cite{MR2777537}. We omit the proof.  
        \begin{cor}\label{corWpert}
        	Suppose that $\bW$ is as in Lemma \ref{lemfixedcoeffi}. Let $\til\bU^*$ be any function in $H^1(D_{\mf r}^*({\rm x}_0^*))$ for $0<\mf r< \min(\theta_1,\pi,r_1-f(\theta_1))$.
        	Then 
        	for any $\mf t$ and $\mf r$ such that $0<\mf t\le \mf r< \min(\theta_1,\pi,r_1-f(\theta_1))$, there hold
        	\begin{align*}
        	\int_{D_{\mf t}^*({\rm x}_0^*)}|D_{\rtp}\til\bU^*|^2
        	\le C\left( \left(\frac{\mf t}{\mf r}\right)^3\int_{D_{\mf r}^*({\rm x}_0^*)}|D_{\rtp}\til\bU^*|^2+\int_{D_{\mf r}^*({\rm x}_0^*)}|D_{\rtp}(\til\bU^*-{\bf W})|^2\right)
        	\end{align*}
        	and 
        	\begin{multline*}
        	\int_{D_{\mf t}^*({\rm x}_0^*)}|D_{\rtp}\til\bU^*-(D_{\rtp}\til\bU^*)^*_{{\rm x}_0^*,\mf t}|^2\\
        	\le C\left(\left(\frac{\mf t}{\mf r}\right)^5\int_{D_{\mf r}^*({\rm x}_0^*)}|D_{\rtp}\til\bU^*-(D_{\rtp}\til\bU^*)^*_{{\rm x}_0^*,\mf r}|^2\right.\\
        	\left.+\int_{D_{\mf r}^*({\rm x}_0^*)}|D_{\rtp}(\til\bU^*-{\bf W})|^2\right)
        	\end{multline*}
        	where  $C$ is a positive constant depending on 
        	$\mu$, $\mc M$, $f(\theta_1)$ and $\theta_1$.
        \end{cor}
    
        We first prove Lemma \ref{lemalpha}. 
        \begin{proof}[Proof of Lemma \ref{lemalpha}]
        	We prove Lemma \ref{lemalpha} by proving that 
        	\eqref{igalpha} holds 
        	for all ${\rm x}_0\in \ol{\n_{f}^+}$ for some positive constants $C$ and $R$. 

        	1. Transform \eqref{veqn}, \eqref{vbdry} into a problem in $\n_{f(\theta_1)}^+$. 
        	
        	
        	Define a map 
        	\begin{align}\label{iota}
        	\Pi_{ab}^*(r,\theta,\vphi):=(\frac{r_1-b}{r_1-a}(r-a)+b,\theta,\vphi)
        	\end{align}
        	for $0<a,b<r_1$. Then $\Pi_{ab}^*(r,\theta,\vphi)$ maps $(a,r_1)\times (0,\theta_1)\times \mathbb T$ to $(b,r_1)\times (0,\theta_1)\times \mathbb T$. 
        	$\Pi^*_{ab}$ naturally induces a map from $\n_a^+$ to $\n_b^+$ as a map between two cartesian coordinate systems. 
        	Denote this map by $\Pi_{ab}$. 
        	Using $\Pi:=\Pi_{f(\theta)f(\theta_1)}$, we transform the boundary value problem \eqref{veqn}, \eqref{vbdry} 
        	into the following boundary value problem
        	\begin{align}\label{yeqn}
        	&\Div_{{\rm y}} (\til{\bm A}D_{\rm y} \til \bU^\sharp)-\til d \til \bU^\sharp=\til{\bm F}^\sharp\qdin \n_{f(\theta_1)}^+,\\
        	\label{ybdry}
        	&\til\bU^\sharp={\bf 0}\qdon \pt \n_{f(\theta_1)}^+
        	\end{align} 
        	where $\til {\bm A}:=\frac{1}{\det(\frac{\pt \Pi}{\pt {\rm x}})}(\frac{\pt {\rm \Pi}}{\pt {\rm x}})^T(
        	{\bm A}\circ \Pi^{-1})(\frac{\pt{\rm \Pi}}{\pt {\rm x}})$, $\til d:=\frac{d\circ \Pi^{-1}}{\det(\frac{\pt {\rm \Pi}}{\pt {\rm x}})}$, $\til \bU^\sharp:=\bU^\sharp\circ \Pi^{-1}$,
        	$\til{\bm F}^\sharp:=\frac{ {\bm F}^\sharp\circ \Pi^{-1}}{\det(\frac{\pt  \Pi}{\pt {\rm x}})}$ and  ${\rm x}$ and ${\rm y}$ are the cartesian coordinate systems for $\n_f^+$ and
        	$\n_{f(\theta_1)}^+$, respectively. 
        	
        	2. Transform 
        	the weak formulation of 
        	\eqref{yeqn}, \eqref{ybdry} near ${\rm x}_0\in \ol{\Gam_{f(\theta_1)}}\cap \ol{\Gam_w^+}$ 
        	into the weak formulation of the spherical coordinate representation of \eqref{yeqn}, \eqref{ybdry}. 
        	
        	
        	Write \eqref{yeqn}, \eqref{ybdry} in the form
        	\begin{align}\label{weakform}
        	\int_{\n_{f(\theta_1)}^+}\til {\bm A}D_{{\rm y}}\til \bU^\sharp D_{{\rm y}}{\bm{\xi}}+\til d \til \bU^\sharp {\bm \xi}=-\int_{\n_{f(\theta_1)}^+}\til{\bm F}^\sharp {\bm \xi}
        	\end{align}
        	for all ${\bm \xi}\in H_0^1(\n_{f(\theta_1)}^+)$. Let $(\til r,\til \theta,\til \vphi)$ be the spherical coordinate system for ${\rm y}$ and $\Xi$ be the map from ${\rm y}$ to $(\til r,\til \theta,\til \vphi)$. Choose any ${\rm x}_0^*:= (f(\theta_1),\theta_1,\vphi_0)$ for some $\vphi_0\in \mathbb T$ and
        	set ${\bm \xi}=0$ outside of $\Xi^{-1}(D_{{\mf r}}^*({\rm x}_0^*))$ for $0<\mf r<\min(\theta_1,\pi,r_1-f(\theta_1))$ in \eqref{weakform}. Then we obtain
        	\begin{align*}
        	\int_{\Xi^{-1}(D_{\mf r}^*({\rm x}_0^*))}\til {\bm A}D_{{\rm y}}\til \bU^\sharp D_{{\rm y}}{\bm{\xi}}+\til d \til \bU^\sharp {\bm \xi}=-\int_{\Xi^{-1}(D_{\mf r}^*({\rm x}_0^*))}\til{\bm F}^\sharp {\bm \xi}
        	\end{align*}
        	for all ${\bm \xi}\in H_0^1(\Xi^{-1}(D_{\mf r}^*({\rm x}_0^*)))$. 
        	Using $\Xi$, transform this equation. Then we have
        	\begin{align}\label{A0trans}
        	\int_{D_{\mf r}^*({\rm x}_0^*)}\frac{1}{\det \til M}\til M^T\til{\bm A}\circ\Xi^{-1}\til MD_{(\til r,\til \theta,\til\vphi)}\til\bU^*D_{(\til r,\til \theta,\til\vphi)}{\bm \xi}+\til d^*\til\bU^*{\bm \xi}=-\int_{D_{r}^*({\rm x}_0^*)}\til{\bm F}^*{\bm \xi}
        	\end{align}
        	for all ${\bm \xi}\in H_0^1(D_{\mf r}^*({\rm x}_0^*))$ where  $\til \bU^*=\til \bU^\sharp\circ\Xi^{-1}$, $\til d^*=\frac{\til d\circ\Xi^{-1}}{\det \til M}$, $\til {\bm F}^*=\frac{\til{\bm F}^\sharp\circ\Xi^{-1}}{\det \til M}$
        	and $\til M=\frac{\pt \Xi 
        	}{\pt y}$. 
        
            3. Obtain \eqref{igalpha} at ${\rm x}_0\in\ol{\Gam_f}\cap\ol{\Gam_w^+}$. 
            
            Rewrite \eqref{A0trans} as 
        	\begin{multline*}
        	\int_{D_{\mf r}^*({\rm x}_0^*)}\frac{1}{\det \til M}\til M^T{\bm A}\circ\Xi^{-1}\til MD_{(\til r, \til \theta, \til \vphi)} \til\bU^*D_{(\til r, \til \theta, \til \vphi)}{\bm \xi}\\
        	=\int_{D_{\mf r}^*({\rm x}_0^*)} (\frac{1}{\det \til M}\til M^T{\bm A}\circ\Xi^{-1}\til M-\frac{1}{\det \til M}\til M^T\til{\bm A}\circ\Xi^{-1}\til M)D_{(\til r, \til \theta, \til \vphi)}\til\bU^*D_{(\til r, \til \theta, \til \vphi)}{\bm \xi}-\til d^*\til\bU^*{\bm \xi}-\til{\bm F}^*{\bm \xi}
        	\end{multline*}
        	for all ${\bm \xi}\in H_0^1(D_{\mf r}^*({\rm x}_0^*))$. Fix the principal coefficients of the left-hand side of the resultant equation at ${\rm x}_0^*$. Then we obtain
        	\begin{multline}\label{A0trans2}
        	\int_{D_{\mf r}^*({\rm x}_0^*)}(\frac{1}{\det \til M}\til M^T{\bm A}\circ\Xi^{-1}\til M) ({\rm x}_0^*)D_{(\til r, \til \theta, \til \vphi)} \til\bU^*D_{(\til r, \til \theta, \til \vphi)}{\bm \xi}\\
        	=\int_{D_{\mf r}^*({\rm x}_0^*)}((\frac{1}{\det \til M}\til M^T{\bm A}\circ\Xi^{-1}\til M) ({\rm x}_0^*)-\frac{1}{\det \til M}\til M^T{\bm A}\circ\Xi^{-1}\til M)D_{(\til r, \til \theta, \til \vphi)} \til\bU^*D_{(\til r, \til \theta, \til \vphi)}{\bm \xi}
        	\\+ (\frac{1}{\det \til M}\til M^T{\bm A}\circ\Xi^{-1}\til M-\frac{1}{\det \til M}\til M^T\til{\bm A}\circ\Xi^{-1}\til M)D_{(\til r, \til \theta, \til \vphi)}\til\bU^*D_{(\til r, \til \theta, \til \vphi)}{\bm \xi}-\til d^*\til\bU^*{\bm \xi}-\til{\bm F}^*{\bm \xi}
        	\end{multline}
        	for all ${\bm \xi}\in H_0^1(D_{\mf r}^*(\til{\rm x}^*_0))$. Using the argument in Appendix, it can be checked that 
        	$(\frac{1}{\det \til M}\til M^T{\bm A}\circ\Xi^{-1}\til M) ({\rm x}_0^*)$  is equal to the principal coefficients of the equation in \eqref{fixedcoeffeqn}. 
        	
        	Let $\bW$ be the weak solution of
        	\begin{align}\label{fixed}
        	&\Div_{(\til r, \til \theta, \til \vphi)}((\frac{1}{\det\til M}\til M^T{\bm A}\circ \Xi^{-1}\til M)({\rm x}_0^*) D_{(\til r, \til \theta, \til \vphi)}\bW)=0\qdin D_{\mf r}^{*}({\rm x}_0^*),\\
        	&\bW=\label{fixedbc}
        	\begin{cases}
        	\til\bU^*\qdon \pt D_{\mf r}^*({\rm x}_0^*)\cap \n_{f(\theta_1)}^{+,*},   \\
        	{\bf 0}  \qdon \pt D_{\mf r}^*({\rm x}_0^*)\cap (\ol{\Gam_{f(\theta_1)}^{*}}\cup\ol{\Gam_{w,f(\theta_1)}^{+,*}}). 
        	\end{cases}
        	\end{align}
        	Subtract the weak formulation of \eqref{fixed}, \eqref{fixedbc} from  \eqref{A0trans2} and then take ${\bm \xi}={\bm V}$ to the resultant equation.
        	Then we obtain
        	\begin{multline*}
        	\int_{D_{\mf r}^*({\rm x}_0^*)}(\frac{1}{\det \til M}\til M^T{\bm A}\circ\Xi^{-1}\til M) ({\rm x}_0^*)D_{(\til r, \til \theta, \til \vphi)} {\bm V}D_{(\til r, \til \theta, \til \vphi)}{\bm V}\\
        	=\int_{D_{\mf r}^*({\rm x}_0^*)}((\frac{1}{\det \til M}\til M^T{\bm A}\circ\Xi^{-1}\til M) ({\rm x}_0^*)-\frac{1}{\det \til M}\til M^T{\bm A}\circ\Xi^{-1}\til M)D_{(\til r, \til \theta, \til \vphi)} \til\bU^*D_{(\til r, \til \theta, \til \vphi)}{\bm V}
        	\\+ (\frac{1}{\det \til M}\til M^T{\bm A}\circ\Xi^{-1}\til M-\frac{1}{\det \til M}\til M^T\til{\bm A}\circ\Xi^{-1}\til M)D_{(\til r, \til \theta, \til \vphi)}\til\bU^*D_{(\til r, \til \theta, \til \vphi)}{\bm V}-\til d^*\til\bU^*{\bm V}-\til{\bm F}^*{\bm V}
        	\end{multline*}
        	where ${\bm V}:=\til \bU^*-\bW\in H_0^1(D_{\mf r}^*({\rm x}_0^*))$. 
        	Using the Sobolev 
        	and H\"older inequality, we obtain from this equation
        	\begin{multline}\label{calphaenergy}
        	\int_{D_{\mf r}^*({\rm x}_0^*)}|D_{(\til r, \til \theta, \til \vphi)}{\bm V}|^2\le C\left( (\tau_1^2(\mf r)+\tau_2^2(\mf r)
        	)\int_{D_{\mf r}^*({\rm x}_0^*)}|D_{(\til r, \til \theta, \til \vphi)}\til \bU^*|^2\right.\\\left.+(\int_{D_{\mf r}^*({\rm x}_0^*)}|\til d^*|^3)^{\frac{2}{3}}\int_{D_r^*({\rm x}_0^*)}|\til \bU^*|^2+(\int_{D_{\mf r}^*({\rm x}_0^*)}|\til{\bm F}^*|^{\frac{6}{5}})^{\frac{5}{3}}\right)
        	\end{multline}
        	where 
        	$$\tau_1(\mf r)=\mathop{\sup_{{\rm x}^*,{\rm y}^*\in \n_f^{+,*}}}_{|{\rm x}^*-{\rm y}^*|\le {\mf r}}|(\frac{1}{\det \til M}\til M^T{\bm A}\circ\Xi^{-1}\til M) ({\rm x}^*)-(\frac{1}{\det \til M}\til M^T{\bm A}\circ\Xi^{-1}\til M) ({\rm y}^*)|$$ 
        	and 
        	$$\tau_2(\mf r):=\sup_{|\theta-\theta_1|\le \mf r} \left\{ |f(\theta)-f(\theta_1)|+|f^\p(\theta)|\right\}$$
        	(note that $f^\p(\theta_1)=0$). 
        	By Corollary \ref{corWpert} and \eqref{calphaenergy}, we have 
        	for any $0<\mf t\le \mf r$,
        	\begin{multline*}
        	\int_{D_{\mf t}^*({\rm x}_0^*)}|D_{(\til r, \til \theta, \til \vphi)}\til\bU^*|^2\le C
        	\Bigg(\left(\left(\frac{\mf t}{\mf r}\right)^3+\tau_1^2(\mf r)+\tau_2^2(\mf r)\right)\int_{D_{\mf r}^*({\rm x}_0^*)}|D_{(\til r, \til \theta, \til \vphi)}\til\bU^*|^2\\+(\int_{D_{\mf r}^*({\rm x}_0^*)}|\til d^*|^3)^{\frac{2}{3}}\int_{D_{\mf r}^*({\rm x}_0^*)}|\til \bU^*|^2+(\int_{D_{\mf r}^*({\rm x}_0^*)}|\til{\bm F}^*|^{\frac{6}{5}})^{\frac{5}{3}}\Bigg).
        	\end{multline*}
        	Using the H\"older and Poincar\'e inequality,
        	we get from this inequality
        	\begin{multline}\label{DF}
        	\int_{D_{\mf t}^*({\rm x}_0^*)}|D_{(\til r, \til \theta, \til \vphi)}\til\bU^*|^2\le C
        	\left(\left(\left(\frac{\mf t}{\mf r}\right)^3+\tau_1^2(\mf r)+\tau_2^2(\mf r)\right)\int_{D_{\mf r}^*({\rm x}_0^*)}|D_{(\til r, \til \theta, \til \vphi)}\til\bU^*|^2\right.\\\left.+(\int_{D_{\mf r}^*({\rm x}_0^*)}|\til d^*|^3)^{\frac{2}{3}}{\mf r}^2\int_{D_{\mf r}^*({\rm x}_0^*)}|D_{(\til r, \til \theta, \til \vphi)}\til \bU^*|^2+(\int_{D_{\mf r}^*({\rm x}_0^*)}|\til {\bm F}^*|^p)^{\frac{2}{p}}{\mf r}^{3-2+2\beta}\right)
        	\end{multline}
        	where $\beta=2-\frac{3}{p}\in (0,1)$ if $p\in (\frac{3}{2},3)$. 

        	Depending on the value of $\beta$, we consider two cases. 
        	
        	Case 1: $3-2+2\beta\le 2$. 
        	
        	Apply Lemma 2.1 in \cite[Chapter 3]{MR717034} to \eqref{DF}. Then we obtain that there exists $R_1\in (0,\min(\theta_1,\pi,r_1-f(\theta_1)))$ such that for any $0<\mf r\le R_1$,
        	\begin{align}\label{cbeta}
        	\int_{D_{\mf r}^*({\rm x}_0^*)}|D_{(\til r, \til \theta, \til \vphi)}\til\bU^*|^2&\le C{\mf r}^{3-2+2\beta}\left(\int_{D_{R_1}^*({\rm x}_0^*)}|D_{(\til r, \til \theta, \til \vphi)}\til\bU^*|^2+(\int_{D_{R_1}^*({\rm x}_0^*)}|\til{\bm F}^*|^p)^\frac{2}{p}\right).
        	\end{align}
        	
        	Case 2: $3-2+2\beta>2$. 
        	
        	As in the Case 1, we apply Lemma 2.1 in \cite[Chapter 3]{MR717034} to \eqref{DF}. Then we obtain that 
        	there exists $R_2\in (0,\min(\theta_1,\pi,r_1-f(\theta_1)))$ such that for any $0<\mf r\le R_2$,
        	\begin{align*}
        	\int_{D_{\mf r}^*({\rm x}_0^*)}|D_{(\til r, \til \theta, \til \vphi)}\til\bU^*|^2\le C{\mf r}^2\left(\int_{D_{R_2}^*({\rm x}_0^*)}|D_{(\til r, \til \theta, \til \vphi)}\til\bU^*|^2+(\int_{D_{R_2}^*({\rm x}_0^*)}|\til{\bm F}^*|^p)^\frac{2}{p}\right).
        	\end{align*}
        	Substitute this into \eqref{DF}. After then apply Lemma 2.1 in \cite[Chapter 3]{MR717034} again. Then we obtain that there exists $R_3\le R_2$ such that for any $0<\mf r\le R_3$, \eqref{cbeta}  holds with $R_1$ replaced by $R_2$. 
        	
        	From this result, we obtain that for any $\beta\in (0,1)$, there exists $R>0$ such that 
        	\begin{align*}
        	\int_{D_{ t}({\rm x}_0)
        	}|D
        	\bU^\sharp|^2\le C{t}^{3-2+2\beta}
        	\left(||\bU^\sharp||_{H^1(\n_f^+)}^2+||{\bm F}^\sharp||_{L^p(\n_f^+)}^2\right)\qd\tx{for any $0<t<R$}
        	\end{align*}
            for ${\rm x}_0\in\ol{\Gam_f}\cap\ol{\Gam_w^+}$.  
        	Since ${\bm F}^\sharp\in C^{\alpha}_{(1-\alpha,\Gam_{w}^+)}(\n_{f}^{+})$ 
        	for $\alpha\in (\frac 2 3,1)$, 
        	${\bm F}^\sharp\in L^3(\n_{f}^+)$. This implies that ${\bm F}^\sharp\in L^p(\n_{f}^+)$ for any $\beta\in (0,1)$ where $p=\frac{3}{2-\beta}$.
        	Hence, we obtain \eqref{igalpha} with $C$ replaced by $C_1=C(||\bU^\sharp||_{H^1(\n_f^+)}^2+||{\bm F}^\sharp||_{L^p(\n_f^+)}^2)$ for ${\rm x}_0\in\ol{\Gam_f}\cap\ol{\Gam_w^+}$.
        	
        	4. 
        	When ${\rm x}_0$ is in $\ol{\Gam_w^+}\cap\ol{\Gam_{ex}}$, 
        	we obtain \eqref{igalpha} with $C$ replaced by $C_1$ using similar arguments without the process of transforming $\n_f^+$ into $\n_{f(\theta_1)}^+$.  
        	When ${\rm x}_0$ is in $\n_f^+$ or $
        	\Gam_f\cup\Gam_w^+\cup\Gam_{ex}$ and far away from the corners $\ol{\Gam_f}\cap\ol{\Gam_w^+}$ and $ \ol{\Gam_w^+}\cap\ol{\Gam_{ex}}$, we obtain \eqref{igalpha} with $C$ replaced by $C_1$ using the standard method of freezing the coefficients. When 
        	${\rm x}_0$ is in $\n_f^+$ or $
        	\Gam_f\cup\Gam_w^+\cup\Gam_{ex}$ and near $\ol{\Gam_f}\cap\ol{\Gam_w^+}$ or $ \ol{\Gam_w^+}\cap\ol{\Gam_{ex}}$, we obtain \eqref{igalpha} with $C$ replaced by $C_1$ using the arguments in the proof of \cite[Theorem 5.21]{MR3099262} and arguments similar to the ones in Step 1-Step 3 above.  
        	Combining these results, we obtain that there exist a positive constant $R$ such that $\bU^\sharp$ satisfies \eqref{igalpha} for all ${\rm x}_0\in \ol{\n_f^+}$ for $C=C_1$. This finishes the proof. 
        \end{proof}
        Next, we prove Lemma \ref{lem1alpha}. 
        We prove Lemma \ref{lem1alpha} using the method of freezing the coefficients and the reflection arguments in the proof of Lemma \ref{lemalpha}. 
        When we do this, there exists some problem:
        since  ${\bm F}^\sharp$ is not in $L^p(\n_f^+)$ for $q=\frac{3}{1-\alpha}$ 
        nor has the form $\Div {\bm G}$ with ${\bm G}\in C^{\alpha}(\ol{\n_f^+})$, we cannot get the power of $t$ required in \eqref{ig1alpha} from the integral estimate of ${\bm F}^\sharp$ directly. 
        We 
        obtain this power by delivering $\theta$-derivatives imposed on some functions in ${\bm F}^\sharp$ to the functions multiplied to those functions in the integral form of ${\bm F}^\sharp$ using integration by parts and estimating the resultant integral form of ${\bm F}^\sharp$.  
        To make our argument clear, we present the detailed proof. 
        \begin{proof}[Proof of Lemma \ref{lem1alpha}]
        	Using \eqref{Fform}, write \eqref{veqn}, \eqref{vbdry} in the form
        	\begin{align*}
        	\int_{\n_f^+}{\bm A}D \bU^\sharp D{\bm \xi} +d\bU^\sharp{\bm \xi}=
        	&\int_{\n_f^+}\sum_i A^i\pt_r (B^i-B^i({\rm x}_0)) \xi_\vphi+\sum_i C^i \pt_\theta(D^i-D^i({\rm x}_0))\xi_\vphi\\
        	&+E\frac{\pt_\theta(F\sin\theta)}{\sin\theta}\xi_\vphi
        	+\Div({\bm A}D{\bm h}
        	){\bm \xi}-d{\bm h}{\bm \xi}
        	\end{align*}
        	for all ${\bm \xi}\in H_0^1(\n_f^+)$ where $\xi_\vphi={\bm \xi}\cdot \be_\vphi$. Using integration by parts, we change this equation into 
        	\begin{align*}
        	\int_{\n_f^+}{\bm A}D \bU^\sharp D{\bm \xi} +d\bU^\sharp{\bm \xi}=
        	&\int_{\n_f^+}-\sum_i\left(\pt_r A^i (B^i-B^i({\rm x}_0)) \xi_\vphi+A^i(B^i-B^i({\rm x}_0))\frac{1}{r^2}\pt_r (r^2\xi_\vphi)\right)\\
        	&-\sum_i\left(\pt_\theta C^i (D^i-D^i({\rm x}_0))\xi_\vphi+C^i(D^i-D^i({\rm x}_0))\frac{1}{\sin\theta}\pt_\theta(\xi_\vphi\sin\theta)\right)
        	\\
        	&
        	+E\frac{\pt_\theta(F\sin\theta)}{\sin\theta}\xi_\vphi-{\bm A}D{\bm h}D{\bm \xi}
        	-d{\bm h}{\bm \xi}
        	\end{align*}
        	for all ${\bm \xi}\in H_0^1(\n_f^+)$. Using $\Pi$ defined in Step 1 in the proof of Lemma \ref{lemalpha}, transform this equation. Then we obtain
        	\begin{align*}
        	\int_{\n_{f(\theta_1)}^+}\til {\bm A}D \til \bU^\sharp D{\bm \xi} +\til d\til \bU^\sharp{\bm \xi}=
        	\int_{\n_{f(\theta_1)}^+}(a) \frac{1}{\det(\frac{d{\rm y}}{ d{\rm x}})}
        	-\til{\bm A}D\til{\bm h}D{\bm \xi}
        	-\til d\til {\bm h}{\bm \xi}
        	\end{align*}
        	for all ${\bm \xi}\in H_0^1(\n_{f(\theta_1)}^+)$
        	where 
        	\begin{align*}
        	(a)=&-\sum_i 
        	\left(\frac{\pt \til r}{\pt r}\pt_{\til r}\til A^i(\til B^i-B^i({\rm x}_0)) \xi_\vphi+\til A^i(\til B^i-B^i({\rm x}_0))\frac{1}{r^2}\frac{\pt \til r}{\pt r}\pt_{\til r}(r^2\xi_\vphi)
        	\right)
        	\\
        	&-\sum_i\bigg(
        	(\frac{\pt \til r}{\pt \theta}\pt_{\til r}+\pt_{\til \theta})\til C^i
        	(\til D^i-D^i({\rm x}_0))\xi_\vphi
        	+\til C^i(\til D^i-D^i({\rm x}_0))\frac{1}{\sin\til \theta}(\frac{\pt \til r}{\pt \theta}\pt_{\til r}+\pt_{\til \theta})(\xi_\vphi\sin\til \theta)\bigg)
        	\\
        	&+\til E\frac{1}{\sin\til\theta}(\frac{\pt \til r}{\pt \theta}\pt_{\til r}+\pt_{\til \theta})(\til F\sin\til \theta)\xi_\vphi,
        	\end{align*}
        	$\til {\bm A}$, $\til d$ and $\til \bU^\sharp$ are functions given below \eqref{ybdry}, 
        	\begin{align*}
        	&\til A^{i}:=A^i\circ \Pi^{-1},\; \til B^i:=B^i \circ \Pi^{-1},\; \til C^{i}:=C^i\circ\Pi^{-1},\\
        	&\til D^i=D^i\circ\Pi^{-1},\;\til E:=E\circ \Pi^{-1}, \til F:=F\circ \Pi^{-1},\; \til {\bm h}={\bm h}\circ \Pi^{-1},
        	\end{align*}
        	${\rm x}$ is the cartesian coordinate representing $\n_f^+$, ${\rm y}=\Pi({\rm x})$ and
        	$(r,\theta,\vphi)$ and $(\til r,\til \theta,\til \vphi)$ are the spherical coordinate systems for ${\rm x}$ and ${\rm y}$, respectively. 
        	As we did in Step 2 in the proof of Lemma \ref{lemalpha}, set  ${\bm \xi}=0$ outside of $\Xi^{-1}(D_{{\mf r}}^*({\rm x}_0^*))$ for $0<\mf r<\min(\theta_1,\pi,r_1-f(\theta_1))$ where ${\rm x}_0^*= (f(\theta_1),\theta_1,\vphi_0)$ for some $\vphi_0\in \mathbb T$ and then transform this equation  using $\Xi$. Then we obtain	
        	\begin{align*}
        	&\int_{D_{\mf r}^*({\rm x}_0^*)}\frac{1}{\det \til M}\til M^T\til{\bm A}\circ\Xi^{-1}\til MD_{(\til r,\til \theta,\til\vphi)}\til\bU^*D_{(\til r,\til \theta,\til\vphi)}{\bm \xi}+\til d^*\til\bU^*{\bm \xi}\\
        	&=\int_{D_{\mf r}^*({\rm x}_0^*)}(a) \frac{\til r^2\sin\til \theta}{\det(\frac{d{\rm y}}{ d{\rm x}})}
        	\\
        	&-\left(\frac{1}{\det \til M}\til M^T \til{\bm A}\circ \Xi^{-1}\til MD_{(\til r, \til \theta,\til \vphi)} \til {\bm h}^*-(\frac{1}{\det \til M}\til M^T \til{\bm A}\circ \Xi^{-1}\til MD_{(\til r, \til \theta,\til \vphi)} \til {\bm h}^*)({\rm x}_0^*)\right)D_{(\til r, \til \theta,\til \vphi)}{\bm \xi}
        	-\til d^*\til {\bm h}^*{\bm \xi}
        	\end{align*}
        	for all ${\bm \xi}\in H_0^1(D_{\mf r}^*({\rm x}_0^*))$ where  $\til \bU^*$, $\til d^*$, 
        	and $\til M$ are functions defined below \eqref{A0trans}
            and 
            we used $\Div_{(\til r, \til \theta,\til \vphi)} \left( (\frac{1}{\det \til M}\til M^T \til{\bm A}\circ \Xi^{-1}\til MD_{(\til r, \til \theta,\til \vphi)} \til {\bm h}^*)({\rm x}_0^*)\right)=0$.
            Fix the principal coefficients of the left-hand side of the above equation. Then we get
            \begin{align*}
            &\int_{D_{\mf r}^*({\rm x}_0^*)}(\frac{1}{\det \til M}\til M^T{\bm A}\circ\Xi^{-1}\til M)({\rm x}_0^*)D_{(\til r,\til \theta,\til\vphi)}\til\bU^*D_{(\til r,\til \theta,\til\vphi)}{\bm \xi}
            \\
            &=\int_{D_{\mf r}^*({\rm x}_0^*)}((\frac{1}{\det \til M}\til M^T{\bm A}\circ\Xi^{-1}\til M) ({\rm x}_0^*)-\frac{1}{\det \til M}\til M^T{\bm A}\circ\Xi^{-1}\til M)D_{(\til r,\til \theta,\til\vphi)} \til\bU^*D_{(\til r,\til \theta,\til\vphi)}{\bm \xi}\\
            &+(\frac{1}{\det \til M}\til M^T{\bm A}\circ\Xi^{-1}\til M-\frac{1}{\det \til M}\til M^T\til{\bm A}\circ\Xi^{-1}\til M)D_{(\til r,\til \theta,\til\vphi)}\til\bU^*D_{(\til r,\til \theta,\til\vphi)}{\bm \xi}-\til d^*\til\bU^*{\bm \xi}\\
            &+(a) \frac{\til r^2\sin\til \theta}{\det(\frac{d{\rm y}}{ d{\rm x}})}
            \\
            &-\left(\frac{1}{\det \til M}\til M^T \til{\bm A}\circ \Xi^{-1}\til MD_{(\til r, \til \theta,\til \vphi)} \til {\bm h}^*-(\frac{1}{\det \til M}\til M^T \til{\bm A}\circ \Xi^{-1}\til MD_{(\til r, \til \theta,\til \vphi)} \til {\bm h}^*)({\rm x}_0^*)\right)D_{(\til r, \til \theta,\til \vphi)}{\bm \xi}
            -\til d^*\til {\bm h}^*{\bm \xi}
            \end{align*}
            for all ${\bm \xi}\in H_0^1(D_{\mf r}^*({\rm x}_0^*))$. 
            
            Let ${\bm W}$ be the weak solution of  \eqref{fixed}, \eqref{fixedbc}.
            As we did in Step 3 in the proof of Lemma \ref{lemalpha}, 
            subtracting the weak formulation of \eqref{fixed}, \eqref{fixedbc} from  the above equation. And then
            take ${\bm \xi}={\bm V}$ where ${\bm V}=\til {\bm U}^*-\bW$ to the resultant equation. Then we have 
            \begin{align*}
            &\int_{D_{\mf r}^*({\rm x}_0^*)}(\frac{1}{\det \til M}\til M^T{\bm A}\circ\Xi^{-1}\til M)({\rm x}_0^*)D_{(\til r,\til \theta,\til\vphi)}{\bm V}D_{(\til r,\til \theta,\til\vphi)}{\bm V}
            \\
            &=\int_{D_{\mf r}^*({\rm x}_0^*)}((\frac{1}{\det \til M}\til M^T{\bm A}\circ\Xi^{-1}\til M) ({\rm x}_0^*)-\frac{1}{\det \til M}\til M^T{\bm A}\circ\Xi^{-1}\til M)D_{(\til r,\til \theta,\til\vphi)} \til\bU^*D_{(\til r,\til \theta,\til\vphi)}{\bm V}\\
            &+(\frac{1}{\det \til M}\til M^T{\bm A}\circ\Xi^{-1}\til M-\frac{1}{\det \til M}\til M^T\til{\bm A}\circ\Xi^{-1}\til M)D_{(\til r,\til \theta,\til\vphi)}\til\bU^*D_{(\til r,\til \theta,\til\vphi)}{\bm V}-\til d^*\til\bU^*{\bm V}\\
            &-\Bigg(\sum_i 
            \left(\frac{\pt \til r}{\pt r}\pt_{\til r}\til A^i(\til B^i-B^i({\rm x}_0)) V_\vphi+\til A^i(\til B^i-B^i({\rm x}_0))\frac{1}{r^2}\frac{\pt \til r}{\pt r}\pt_{\til r} (r^2V_\vphi)
            \right)
            \\
            &-\sum_i\bigg(
            (\frac{\pt \til r}{\pt \theta}\pt_{\til r}+\pt_{\til \theta})\til C^i
            (\til D^i-D^i({\rm x}_0))V_\vphi
            +\til C^i(\til D^i-D^i({\rm x}_0))\frac{1}{\sin\til \theta}(\frac{\pt \til r}{\pt \theta}\pt_{\til r}+\pt_{\til \theta})(V_\vphi\sin\til \theta)\bigg)
            \\
            &+\til E\frac{1}{\sin\til\theta}(\frac{\pt \til r}{\pt \theta}\pt_{\til r}+\pt_{\til \theta})(\til F\sin\til \theta)V_\vphi\Bigg) \frac{\til r^2\sin\til \theta}{\det(\frac{d{\rm y}}{ d{\rm x}})}\\
            &-\left(\frac{1}{\det \til M}\til M^T \til{\bm A}\circ \Xi^{-1}\til MD_{(\til r, \til \theta,\til \vphi)} \til {\bm h}^*-(\frac{1}{\det \til M}\til M^T \til{\bm A}\circ \Xi^{-1}\til MD_{(\til r, \til \theta,\til \vphi)} \til {\bm h}^*)({\rm x}_0^*)\right)D_{(\til r, \til \theta,\til \vphi)}{\bm V}-\til d^*\til {\bm h}^*{\bm V}
            \end{align*}
            where $V_\vphi={\bm V}\cdot \be_{\vphi}$. 
            Using the Sobolev and  H\"older inequality and the facts that ${\bm A}\in C^{\alpha}(\ol{\n_f^+})$, $f\in C^{1,\alpha}(\ol{\Lambda})$, $\bm h\in C^{1,\alpha}(\ol{\n_f^+})$ and
            \begin{align*}
            \int_{D_{\mf r}^*({\rm x}_0^*)}|\frac{1}{r^2}\pt_{\til r}(r^2V_\vphi)|^2,\;
            \int_{D_{\mf r}^*({\rm x}_0^*)}|\frac{1}{\sin\til\theta}\pt_{\til\theta}(V_\vphi\sin\til\theta)|^2&\le C\int_{D_{\mf r}^*({\rm x}_0^*)}|D(V_\vphi\be_{\vphi})|^2 \qd\tx{(See \eqref{DPsi})}\\
            &\le C\int_{D_{\mf r}^*({\rm x}_0^*)}|D_{(\til r, \til \theta,\til\vphi)}{\bm V}|^2
            \end{align*}
            we obtain from the above equation
            \begin{multline}\label{Vfinal}
            \int_{D_{\mf r}^*({\rm x}_0^*)}|D_{\rtp}{\bm V}|^2\le C \left( {\mf r}^{2\alpha}\int_{D_r^*({\rm x}_0^*)}|D_{\rtp}\til\bU^*|^2\right.\\\left.+(\int_{D_{\mf r}^*({\rm x}_0^*)}|\til d^*|^q)^{\frac{2}{q}}{\mf r}^{2\alpha}\int_{D_{\mf r}^*({\rm x}_0^*)}|\til\bU^*|^2
            +{\mf r}^{3+2\alpha}F^\flat\right)
            \end{multline}
            for $q=\frac{3}{1-\alpha}$ where
            \begin{align*}
            F^\flat=&\sum_i ||B^i\be_{\theta}||_{\alpha,\n_f^+}^2||A^i||^2_{W^{1,3}(\n_f^+)\cap L^\infty(\n_f^+)} 
            +\sum_i
            ||D^{i}||^2_{\alpha,\n_f^+}||C^{i}||^2_{W^{1,3}(\n_f^+)\cap L^\infty(\n_f^+)}\\
            &+
            ||E\frac{\pt_\theta(F\sin\theta)}{\sin\theta}\be_{\vphi}||_{L^q(\n_f^+)}^2+|| {\bm h}||^2_{1,\alpha,\n_f^+}
            \end{align*}
            with
            $||\cdot||_{W^{1,3}(\Om)\cap L^\infty(\Om)}:=||\cdot||_{W^{1,3}(\Om)}+||\cdot||_{L^\infty(\Om)}$.
            By Corollory \ref{corWpert} and \eqref{Vfinal}, 
            we have 
            for any $0<{\mf t}\le {\mf r}$,
            \begin{multline}\label{UWdifa}
            \int_{D_{\mf t}^*({\rm x}_0^*)}|D_{(\til r,\til\theta,\til\vphi)}\til\bU^*-(D_{(\til r,\til\theta,\til\vphi)}\til\bU^*)^*_{{\rm x}_0^*,\mf t}|^2\\
            \le C\left(\left(\frac{\mf t}{\mf r}\right)^5\int_{D_{\mf r}^*(\til {\rm x}_0)}|D_{(\til r,\til\theta,\til\vphi)}\til\bU^*-(D_{(\til r,\til\theta,\til\vphi)}\til\bU^*)^*_{{\rm x}_0^*,\mf r}|^2\right.\\
            \left.+{\mf r}^{2\alpha}\int_{D_{\mf r}^*({\rm x}_0^*)}|D_{(\til r,\til\theta,\til\vphi)}\til\bU^*|^2+{\mf r}^{2\alpha}\int_{D_{\mf r}^*( {\rm x}_0^*)}|\til\bU^*|^2
            +{\mf r}^{3+2\alpha}F^\flat
            \right).
            \end{multline}
        	
        	In the proof of Lemma \ref{lemalpha}, we showed that for any $\eps\in(0,1)$, there exists $R_4>0$ such that 
        	\begin{align*}
        	\int_{D_{\mf r}^*({\rm x}_0^*)}|D_{(\til r,\til\theta,\til\vphi)}\til\bU^*|^2\le C{\mf r}^{3-2\eps}(||\bU^\sharp||_{H^1(\n_f^+)}^2+||{\bm F}^\sharp||_{L^3(\n_f^+)}^2)\qd\tx{for any $0<\mf r\le R_4$}.
        	\end{align*}
        	Using this inequality and $\bU^\sharp\in C^0(\ol{\n_f^+})$ obtained from Lemma \ref{lemalpha}, we apply Lemma 2.1 in \cite[Chapter 3]{MR717034} to \eqref{UWdifa}. Then we obtain 
        	\begin{multline}\label{c1alphae}\int_{D_{\mf r}^*({\rm x}_0^*)}|D_{(\til r,\til\theta,\til\vphi)}\til\bU^*-(D_{(\til r,\til\theta,\til\vphi)}\til\bU^*)^*_{{\rm x}_0^*,{\mf r}}|^2\\
        	\le C(||\bU^\sharp||_{H^1(\n_f^+)}^2+||{\bm F}^\sharp||_{L^3(\n_f^+)}^2+{F}^\flat){\mf r}^{3+2\alpha-2\eps} \qd\tx{for any $0<\mf r<R_5$}
        	\end{multline}
        	for a constant $R_5>0$. 
        	
        	When ${\rm x}_0^*$ is in $\ol{\Gam_w^{+.*}}\cap \ol{\Gam_{ex}^*}$, we obtain \eqref{c1alphae} using similar argument without the process of transforming $\n_f^+$ to $\n_{f(\theta_1)}^+$. 
        	When ${\rm x}_0^*$ is in $\n_f^{+,*}\cap \{\theta\ge\frac{ \theta_1}{3}\}$ or $(\Gam_f^*\cup\Gam_w^{+,*}\cup\Gam_{ex}^*)\cap \{\theta\ge\frac{ \theta_1}{3}\}$ and far away from the corners $\ol{\Gam_f^*}\cap\ol{\Gam_w^{+,*}}$ and $\ol{\Gam_w^{+,*}}\cap\ol{\Gam_{ex}^*}$, we obtain \eqref{c1alphae} using the standard method of freezing coefficients to the spherical coordinate representation of \eqref{veqn}, \eqref{vbdry} with integration by parts argument above. 
        	When ${\rm x}_0^*$ is in $\n_f^{+,*}\cap \{\theta\ge\frac{ \theta_1}{3}\}$ or $
        	\Gam_f^*\cup\Gam_w^{+,*}\cup\Gam_{ex}^*\cap \{\theta\ge\frac{ \theta_1}{3}\}$ and near $\ol{\Gam_f^*}\cap\ol{\Gam_w^{+_*}}$ or $ \ol{\Gam_w^{+,*}}\cap\ol{\Gam_{ex}^*}$, we obtain \eqref{c1alphae} using the arguments in 
        	\cite[Theorem 5.21]{MR3099262} and arguments similar to the ones above. Then we obtain 
        	$D\bU^\sharp\in C^{\alpha-\eps}([f(\theta),r_1]\times[\frac{\theta_1}{3},\theta_1]\times \mathbb T)$. 
        	When ${\rm x}_0$ is in $\n_f^+\cap \{\theta\le\frac{ 2\theta_1}{3}\}$ or $(\Gam_f\cup\Gam_w^+\cup\Gam_{ex})\cap \{\theta\le\frac{2 \theta_1}{3}\}$,  we obtain \eqref{ig1alpha} with $\alpha$ and $C$ replaced by $\alpha-\eps$ for any $\eps\in (0,1)$ and $C(||\bU^\sharp||_{H^1(\n_f^+)}^2+||{\bm F}^\sharp||_{L^3(\n_f^+)}^2+{F}^\flat)$ using the standard method of freezing coefficients to \eqref{veqn}, \eqref{vbdry} with integration by parts argument above. 
        	Note that if we estimate the integral form of ${\bm F}^\sharp$ using integration by parts argument above, 
        	then there is no singularity issue. From this result, we obtain $D\bU^\sharp\in C^{\alpha-\eps}(\ol{\n_f^+}\cap \{\theta\le \frac{2\theta_1}{3}\})$. Combining these two regularity results for $D\bU^\sharp$, we obtain $D\bU^\sharp\in C^{\alpha-\eps}(\ol{\n_f^+})$.
        	
        	Using the regularity result for $D\bU^\sharp$ and $\bU^\sharp\in C^0(\ol{\n_f^+})$, we obtain from \eqref{UWdifa}
        	\begin{multline}\label{VVV}
        	\int_{D_{\mf t}^*({\rm x}_0^*)}|D_{(\til r,\til \theta,\til \vphi)}\til\bU^*-(D_{(\til r,\til \theta,\til \vphi)}\til\bU^*)^*_{{\rm x}_0^*,\mf t}|^2\\
        	\le C\bigg(\left(\frac{\mf t}{\mf r}\right)^5\int_{D_{\mf r}^*({\rm x}_0^*)}|D_{(\til r,\til \theta,\til \vphi)}\til\bU^*-(D_{(\til r,\til \theta,\til \vphi)}\til\bU^*)^*_{{\rm x}_0^*,\mf r}|^2\\
        	+{\mf r}^{3+2\alpha}\left(||\bU^\sharp||_{H^1(\n_f^+)}^2+||{\bm F}^\sharp||_{L^3(\n_f^+)}^2
        	+{F}^\flat\right)
        	\bigg)
        	\end{multline}
        	(here ${\rm x}_0^*\in \ol{\Gam_f^*}\cap\ol{\Gam_w^{+,*}}$). 
        	Apply Lemma 2.1 in \cite[Chapter 3]{MR717034} to \eqref{VVV}. Then we have 
        	\begin{multline}\label{mmm}
        	\int_{D_{\mf r}^*({\rm x}_0^*)}|D_{(\til r,\til \theta,\til \vphi)}\til\bU^*-(D_{(\til r,\til \theta,\til \vphi)}\til\bU^*)^*_{{\rm x}_0^*,{\mf r}}|^2\\
        	\le C(||\bU^\sharp||_{H^1(\n_f^+)}^2+||{\bm F}^\sharp||_{L^3(\n_f^+)}^2+{F}^\flat){\mf r}^{3+2\alpha}\qd\tx{for any $0<\mf r<R_6$}
        	\end{multline}
        	for a constant $R_6>0$. Using the arguments used when we obtained $D\bU\in C^{\alpha-\eps}(\ol{\n_f^+})$, we obtain \eqref{mmm} at ${\rm x}_0^*\in \ol{\n_f^{+,*}}\cap\{\theta\ge \frac{\theta_1}{3}\}$ and \eqref{ig1alpha} with $C$ replaced by $C(||\bU^\sharp||_{H^1(\n_f^+)}^2+||{\bm F}^\sharp||_{L^3(\n_f^+)}^2+{F}^\flat)$ at ${\rm x}_0\in \ol{\n_f^+}\cap \{\theta\le\frac{2\theta_1}{3}\}$. 
        	From this result, we obtain the desired result. 
        	This finishes the proof. 
        \end{proof}
        Using the scailing argument given in the proof of Proposition 3.1 in \cite{MR3458161} with the results in Theorem 5.21 in \cite{MR3099262} and Lemma \ref{lem1alpha}, we can obtain the following result. %
        We omit the proof. 
        \begin{lem}\label{lem2alpha}
        	Under the assumption as in Lemma \ref{lemSE}, let $\bU^\sharp$ be a weak solution of \eqref{veqn}, \eqref{vbdry}. Then $\bU^\sharp\in C^{2,\alpha}_{(-1-\alpha,\Gam_w^+)}(\n_f^+)$. Furthermore,
        	$\bU^\sharp$ satisfies 
        	\begin{align*}
        	||\bU^\sharp||_{2,\alpha,\n_f^+}^{(-1-\alpha,\Gam_w^+)}\le C(||{\bm F}^\sharp||_{\alpha,\n_f^+}^{(1-\alpha,\Gam_w^+)}+||\bU^\sharp||_{1,\alpha,\n_f^+})
        	\end{align*}
        	where $C$ is a positive constant depending on $(\rho_0^+,u_0^+,p_0^+)$, $\gam$, $\n_f^+$ and $\alpha$.
        \end{lem}
        Finally, we prove that the
        $C^{2,\alpha}_{(-1-\alpha,\Gam_w^+)}(\n_f^+)$ solution of \eqref{veqn}, \eqref{vbdry} is of the form $\Psi(r,\theta)\be_{\vphi}$. This statement is proved using the argument as
        in  the proof of Proposition 3.3 in \cite{MR3739930} {(Method II)}. Although the arguments to prove this statement are almost same with those in  the proof of Proposition 3.3 in \cite{MR3739930} {(Method II)}, 
        since \eqref{veqn}, \eqref{vbdry} is different from the problem in Proposition 3.3 in \cite{MR3739930} and similar arguments will be used later in the proof of Lemma \ref{lemlegendre}, 
        we present the detailed proof. 
        \begin{lem}\label{lemform}
        	Under the assumption as in Lemma \ref{lemSE}, 
        	the $C^{2,\alpha}_{(-1-\alpha,\Gam_w^+)}(\n_f^+)$ solution of \eqref{veqn}, \eqref{vbdry} is of the form $\Psi(r,\theta)\be_{\vphi}$. 
        \end{lem}
        \begin{proof}
        	Let $\bU=U_r\be_r+U_\theta\be_\theta+U_\vphi\be_\vphi$ be the 
        	$C^{2,\alpha}_{(-1-\alpha,\Gam_w^+)}(\n_f^+)$ solution 
        	of \eqref{veqn}, \eqref{vbdry}. 
        	Then, $(U_r,U_\theta,U_\vphi)$  is in $(C^{2,\alpha}_{(-1-\alpha,\{\theta=\theta_1\})}(\n_{f}^{+,*}))^3$ satisfying
        	\begin{align}\label{Usphenorm}||U_r||_{2,\alpha,\n_{f}^{+,*}}^{(-1-\alpha,\{\theta=\theta_1\})},\; ||U_\theta||_{2,\alpha,\n_{f}^{+,*}}^{(-1-\alpha,\{\theta=\theta_1\})},\; ||U_\vphi||_{2,\alpha,\n_{f}^{+,*}}^{(-1-\alpha,\{\theta=\theta_1\})}\le C||\bU||_{2,\alpha,\n_f^+}^{(-1-\alpha,\Gam_w^+)}
        	\le CC^*,
        	\end{align}
        	where $C$ and $C^*$ are positive constants depending on $\n_f^+$ and $\alpha$ and given in \eqref{Uestimate}, respectively,  
        	and satisfies
        	the following spherical coordinate representation of \eqref{veqn}, \eqref{vbdry} 
        	\begin{align}\label{expsystem}
        	&\begin{cases}
        	\left(\frac{{c_0^+}^2}{\rho_0^+({c_0^+}^2-{u_0^+}^2)}(\Delta U_r-\frac{2U_r}{r^2}-\frac{2}{r^2\sin\theta}\pt_\theta(U_\theta\sin\theta)-\frac{2}{r^2\sin\theta}\pt_\vphi U_\vphi)\right.\\
        	\left.\qd\qd\qd\qd\qd\qd\qd\qd-\frac{ {u_0^+}^2}{\rho_0^+({c_0^+}^2-{u_0^+}^2)}\frac{1}{r^2}\pt_r (r^2\pt_r U_r)-\frac{\pt_r \rho_0^+}{{\rho_0^+}^2}\frac{\pt_r (r U_r)}{r}\right)=0,\\
        	\left(\frac{{c_0^+}^2}{\rho_0^+({c_0^+}^2-{u_0^+}^2)}(\Delta U_\theta-\frac{U_\theta}{r^2\sin^2\theta}+\frac{2}{r^2}\pt_\theta U_r-\frac{2\cos\theta}{r^2\sin^2\theta}\pt_\vphi U_\vphi)\right.\\
        	\left.\qd\qd\qd\qd\qd\qd\qd\qd-\frac{ {u_0^+}^2}{\rho_0^+({c_0^+}^2-{u_0^+}^2)}\frac{1}{r^2}\pt_r (r^2 \pt_r U_\theta)-\frac{\pt_r\rho_0^+}{{\rho_0^+}^2}\frac{\pt_r (r U_\theta)}{r}\right)=0,\\
        	\left( \frac{{c_0^+}^2}{\rho_0^+({c_0^+}^2-{u_0^+}^2)}(\Delta U_\vphi-\frac{U_\vphi}{r^2\sin^2\theta}+\frac{2}{r^2\sin\theta}\pt_\vphi U_r+\frac{2\cos\theta}{r^2\sin^2\theta}\pt_\vphi U_\theta)\right.\\
        	\left.\qd\qd\qd\qd\qd\qd\qd\qd-\frac{ {u_0^+}^2}{\rho_0^+({c_0^+}^2-{u_0^+}^2)} \frac{1}{r^2}\pt_r (r^2\pt_r U_\vphi) -\frac{\pt_r \rho_0^+}{{\rho_0^+}^2}\pt_r (r U_\vphi)\right)=- \mc F(r,\theta)
        	\end{cases}\;\tx{in $\n_f^{+,*}$},\\
        	\label{expsystembc}
        	&(U_r,U_\theta,U_\vphi)=
        	(0,0,0) \qd\tx{on}\qd \Gam_f^*,\Gam_w^{+,*},\Gam_{ex}^*,
        	\end{align}
        	where $\mc F={\bm F}^\sharp\cdot \be_{\vphi}$ and 
        	$\Delta U_k=\frac{1}{r^2}\pt_r(r^2\pt_r U_k)+\frac{1}{r^2\sin\theta}\pt_\theta(\sin\theta\pt_\theta U_k)+\frac{1}{r^2\sin^2\theta}\pt_\vphi^2 U_k$
        	for $k=r,\theta,\vphi$.
        	
        	Define 
        	$$U_k^n:=\frac{1}{2^n}\sum_{k=0}^{2^n-1}U_k(r,\theta,\vphi+\frac{2\pi k}{2^n})$$
        	for $k=r,\theta,\vphi$. 
        	Then by the definition of $U_k^n$ for $k=r,\theta,\vphi$ and \eqref{Usphenorm},
        	\begin{align}\label{Ure}
        	||U_r^n||_{2,\alpha,\n_{f}^{+,*}}^{(-1-\alpha,\{\theta=\theta_1\})},\; ||U_\theta^n||_{2,\alpha,\n_{f}^{+,*}}^{(-1-\alpha,\{\theta=\theta_1\})},\; ||U_\vphi^n||_{2,\alpha,\n_{f}^{+,*}}^{(-1-\alpha,\{\theta=\theta_1\})}\le CC^*.
        	\end{align}
            By \eqref{Ure} and  $C^{2,\alpha}_{(-1-\alpha,\{\theta=\theta_1\})}(\n_f^{+,*})\Subset C^{2,\alpha}_{(-1-\frac{\alpha}{2},\{\theta=\theta_1\})}(\n_f^{+,*})$, 
            there exists a subsequence $(U_r^{n_k},$ $U_\theta^{n_k},U_\vphi^{n_k})$ of $(U_r^{n},U_\theta^{n},U_\vphi^{n})$ such that  $(U_r^{n_k},U_\theta^{n_k},U_\vphi^{n_k})$ converges in $C^{2,\alpha}_{(-1-\frac{\alpha}{2},\{\theta=\theta_1\})}(\n_f^{+,*})$ as $n_k\ra \infty$. 
            Denote its limit 
            by $(U_r^*,U_\theta^*,U_\vphi^*)$.
            Then 
            $(U_r^*,U_\theta^*,U_\vphi^*)$ is independent of $\vphi$ 
            and $(U_r^*,U_\theta^*,U_\vphi^*)\in (C^{2,\alpha}_{(-1-\alpha,\{\theta=\theta_1\})}(\n_f^{+,*}))^3$. 
           
        	Since the coefficients of \eqref{expsystem}, $\mc F$ and the boundary conditions in \eqref{expsystembc} are independent of $\vphi$, $(U_r^n,U_\theta^n,U_\vphi^n)$ satisfies \eqref{expsystem}, \eqref{expsystembc} for all $n\in \mathbb N\cup \{0\}$. 
        	By this fact and the definition of $(U_r^*,U_\theta^*,U_\vphi^*)$, 
            $(U_r^*,U_\theta^*,U_\vphi^*)$ satisfies \eqref{expsystem}, \eqref{expsystembc}. 
            Since $(U_r^*,U_\theta^*,U_\vphi^*)$ is independent of $\vphi$, \eqref{expsystem} satisfied by $(U_r^*,U_\theta^*,U_\vphi^*)$ is given as 
            \begin{align}\label{expsystemred}
            \begin{cases}
            \left(\frac{{c_0^+}^2}{\rho_0^+({c_0^+}^2-{u_0^+}^2)}(\frac{1}{r^2}\pt_r(r^2\pt_r U_r^*)+\frac{1}{r^2\sin\theta}\pt_\theta(\sin\theta\pt_\theta U_r^*)-\frac{2U_r^*}{r^2}-\frac{2}{r^2\sin\theta}\pt_\theta(U_\theta^*\sin\theta))\right.\\
            \left.\qd\qd\qd\qd\qd\qd\qd\qd-\frac{ {u_0^+}^2}{\rho_0^+({c_0^+}^2-{u_0^+}^2)}\frac{1}{r^2}\pt_r (r^2\pt_r U_r^*)-\frac{\pt_r \rho_0^+}{{\rho_0^+}^2}\frac{\pt_r (r U_r^*)}{r}\right)=0,\\
            \left(\frac{{c_0^+}^2}{\rho_0^+({c_0^+}^2-{u_0^+}^2)}(\frac{1}{r^2}\pt_r(r^2\pt_r U_\theta^*)+\frac{1}{r^2\sin\theta}\pt_\theta(\sin\theta\pt_\theta U_\theta^*)-\frac{U_\theta^*}{r^2\sin^2\theta}+\frac{2}{r^2}\pt_\theta U_r^*)\right.\\
            \left.\qd\qd\qd\qd\qd\qd\qd\qd-\frac{ {u_0^+}^2}{\rho_0^+({c_0^+}^2-{u_0^+}^2)}\frac{1}{r^2}\pt_r (r^2 \pt_r U_\theta^*)-\frac{\pt_r\rho_0^+}{{\rho_0^+}^2}\frac{\pt_r (r U_\theta^*)}{r}\right)=0,\\
            \left( \frac{{c_0^+}^2}{\rho_0^+({c_0^+}^2-{u_0^+}^2)}(\frac{1}{r^2}\pt_r(r^2\pt_r U_\vphi^*)+\frac{1}{r^2\sin\theta}\pt_\theta(\sin\theta\pt_\theta U_\vphi^*)-\frac{U_\vphi^*}{r^2\sin^2\theta})\right.\\
            \left.\qd\qd\qd\qd\qd\qd\qd\qd-\frac{ {u_0^+}^2}{\rho_0^+({c_0^+}^2-{u_0^+}^2)} \frac{1}{r^2}\pt_r (r^2\pt_r U_\vphi^*) -\frac{\pt_r \rho_0^+}{{\rho_0^+}^2}\pt_r (r U_\vphi^*)\right)=- \mc F(r,\theta)
            \end{cases}\; \tx{in $\n_f^{+,*}$}. 
            \end{align}
            Note that the first and second equation of \eqref{expsystemred} is not coupled with the third equation of \eqref{expsystemred}. 
            
            Let $\Psi=U_\vphi^*$. Using the third equation of \eqref{expsystemred} and the facts that 
            $U_\vphi^*\in C^{2,\alpha}_{(-1-\alpha,\{\theta=\theta_1\})}(\n_f^{+,*})$ and  ${\bm F}^\sharp$ is an axisymmetric function in  $C^{\alpha}_{(1-\alpha,\Gam_w^+)}(\n_f^+)$, it can be checked that $\Psi=\pt_{\theta\theta}\Psi=0$ on $\theta=0$. By this fact and $U_\vphi^*\in C^{2,\alpha}_{(-1-\alpha,\{\theta=\theta_1\})}(\n_f^{+,*})$, we have that $\Psi\be_{\vphi}\in C^{2,\alpha}_{(-1-\alpha,\Gam_w^+)}(\n_f^+)$.  
            From \eqref{expsystemred} and \eqref{expsystembc}, one can see that $(0,0,U_\vphi^*)$ is a solution of \eqref{expsystem}, \eqref{expsystembc}. 
            Combining this fact with the fact that
            $\Psi\be_{\vphi}\in C^{2,\alpha}_{(-1-\alpha,\Gam_w^+)}(\n_f^+)$, we have that  $\Psi\be_\vphi$ is a $C^{2,\alpha}_{(-1-\alpha,\Gam_w^+)}(\n_f^+)$ solution of \eqref{veqn}, \eqref{vbdry}.  By Lemma \ref{lemH1}, a $C^{2,\alpha}_{(-1-\alpha,\Gam_w^+)}(\n_f^+)$ solution of \eqref{veqn}, \eqref{vbdry} is unique. 
            Therefore, $\bU=\Psi\be_{\vphi}$. This finishes the proof. 
        \end{proof}

        \subsection{Initial value problem of a transport equation with an axisymmetric divergence-free coefficient}\label{subtrans}
        
        The initial value problems of a transport equation 
        in (B$^\p$) 
        are of the form 
        \begin{align}
        \label{transeqn}&\grad\times ((\Phi_0^++\Psi)\be_\vphi)\cdot \grad Q=0\qd\tx{in}\qd \n_f^+,\\
        \label{transbdry} &Q=Q_{en}\qd\tx{on}\qd \Gam_f
        \end{align}
        where $\Psi\be_\vphi:\n_f^+\ra \R^3$ and $Q_{en}:\Gam_f\ra \R$ are axisymmetric functions. %
        Here, \eqref{transeqn} is a transport equation whose coefficient
        is an axisymmetric and divergence-free vector field.  Thus, 
        the stream function of the coefficient vector field of \eqref{transeqn} can be defined (see \eqref{V}). 
        We 
        find a solution of \eqref{transeqn}, \eqref{transbdry} and obtain the 
        regularity and uniqueness of solutions of \eqref{transeqn}, \eqref{transbdry}
        using the stream function of the coefficient vector field of \eqref{transeqn} and the solution expression given by using the stream function in the following lemma. 

         
        \begin{lem}\label{lemTrans}
        	Suppose that $f$ is as in Lemma \ref{lemrholinear}.  Let $\delta_6$ be a positive constant such that 
        	for such $f$, if  
        	$||\grad\times(\Psi\be_\vphi)||_{0,0,\n_f^+}\le \delta_6,$
        	then 
        	\begin{multline}\label{transcondition}
        	\grad\times ((\Phi_0^++\Psi)\be_\vphi)\cdot \be_r>c^*\qdin \ol{\n_f^+}\qd\tx{and}\qd
        	\grad\times((\Phi_0^++\Psi)\be_\vphi)\cdot {\bm\nu}_f >c^*\qdon \ol{\Gam_f}
        	\end{multline}
        	where $c^*$ is a positive constant depending on $\rho_0^+$, $u_0^+$, $r_s$, $r_1$, $\delta_1$ and $\delta_6$, and  ${\bm{\nu}}_f$ is the unit normal vector on $\Gam_f$ pointing toward $\n_f^+$.
        	Suppose  that $\Psi\be_{\vphi}:\n_f^+\ra \R^3$ is an axisymmetric function in $C^{1,\alpha}(\ol{\n_f^+})$ satisfying 
        	\begin{align}
        	\label{Psislip}\Psi=\frac{f(\theta_1)\Psi(f(\theta_1),\theta_1)}{r}\qdon\ol{\Gam_w^+}
        	\end{align}
        	and
        	\begin{align}\label{Phisize}
        	||\grad\times(\Psi\be_\vphi)||_{0,\alpha,\n_f^+}\le \delta_6. 
        	\end{align}
        	Suppose finally that $Q_{en}:\Gam_f\ra \R$ is an axisymmetric function in $C^{1,\alpha}_{(-\alpha,\pt\Gam_f)}(\Gam_f)$.  
        	Then the problem \eqref{transeqn}, \eqref{transbdry}
        	has a unique axisymmetric $C^0(\ol{\n_f^+})\cap C^1(\n_f^+)$ solution 
        	\begin{align}
        	\label{Qsol}Q=Q_{en}(\mc L) \qdin \n_f^+
        	\end{align}
        	where $\mc L=k^{-1}\circ V$ with $V=2\pi r \sin\theta (\Phi_0^++\Psi)$ and
        	$k(\theta)=V(f(\theta),\theta)$. Furthermore, this solution $Q$ satisfies
        	\begin{align*}
        	||Q||_{1,\alpha,\n_f^+}^{(-\alpha,\Gam_w^+)}\le C ||Q_{en}||_{1,\alpha,\Gam_f}^{(-\alpha,\pt \Gam_f)}
        	\end{align*}
        	where $C$ is a positive constant depending on $\rho_0^+$, $u_0^+$, $r_s$, $r_1$, $\theta_1$,  $\alpha$, $\delta_1$ and $\delta_6$.  
        \end{lem}
        \begin{remark}
    	The existence of $\delta_6$ is obtained from the facts that $\grad\times (\Phi_0^+\be_\vphi)\cdot \be_r =\rho_0^+u_0^+>0$ in $\ol{\n_{r_s-\delta_1}^+}$ and ${\bm \nu}_f\cdot \be_r>0$ on $\ol{\Gam_f}$ for $f$ given in Lemma \ref{lemrholinear}. 
        \end{remark}
        \begin{remark}
        	If $\Psi\be_\vphi$ satisfies \eqref{transcondition}, then $V=2\pi r \sin\theta(\Phi_0^++\Psi)$ satisfies $$\pt_\theta V>0\qdin\n_f^+\setminus\{\theta=0\}\qdand\pt_\theta (V(f(\theta),\theta))>0\qd\tx{for}\qd\theta\in (0,\theta_1).$$
        	This condition will be used to construct the stream surfaces of $\grad\times((\Phi_0^++\Psi)\be_{\vphi})$ in $\n_f^+$ in the proof of Lemma \ref{lemTrans}. 
        \end{remark}
        
        \begin{proof}[Proof of Lemma \ref{lemTrans}]
        	1. Construct the stream surfaces of $\grad\times((\Phi_0^++\Psi)\be_{\vphi})$ in $\ol{\n_f^+}$. 
        	
        	Let us define $V:=2\pi r \sin\theta(\Phi_0^++\Psi)$.  By the fact that $\Psi\be_{\vphi}$ and $\Phi_0^+\be_{\vphi}$ are axisymmetric functions in $C^{1,\alpha}(\ol{\n_f^+})$, 
        	$V$ is an axisymmetric function in $C^{1,\alpha}(\ol{\n_f^+})$ (see Lemma \ref{lemaxi}).  
        	By \eqref{transcondition}, $\pt_\theta V>0$ in $\n_f^+\setminus\{x=y=0\}$.
        	Using these facts, 
        	we apply the implicit function theorem to $V$. Then we obtain that for any $(r^\sharp,\theta^\sharp,\vphi^\sharp)\in \n_f^{+,*}$, there exists a unique $C^1$ surface $\theta=h_{(r^\sharp,\theta^\sharp,\vphi^\sharp)}(r,\vphi)$ defined near $(r^\sharp,\vphi^\sharp)$ such that $\theta^\sharp=h_{(r^\sharp,\theta^\sharp,\vphi^\sharp)}(r^\sharp,\vphi^\sharp)$ and $V(r,h_{(r^\sharp,\theta^\sharp,\vphi^\sharp)}(r,\vphi),\vphi)=V(r^\sharp,\theta^\sharp,\vphi^\sharp)$. 
        	Since $V$ is axisymmetric, this surface is axisymmetric. We denote $\theta=h_{(r^\sharp,\theta^\sharp,\vphi^\sharp)}(r,\vphi)$ by $\theta=h_{(r^\sharp,\theta^\sharp)}(r)$.  
        	
        	By  $V_0^+=V_0^+(f(\theta_1),\theta_1)$ on $\ol{\Gam_w^{+}}$ and \eqref{Psislip},  $V=V(f(\theta_1),\theta_1)$ on $\ol{\Gam_w^{+}}$. 
        	By this fact, $V=0$ on $\ol{\n_f^{+}}\cap \{x=y=0\}$ and $\pt_\theta V>0$  in $\n_f^{+}\setminus\{x=y=0\}$, $\theta=h_{(r^\sharp,\theta^\sharp)}(r)$ is defined 
        	until it reaches $\Gam_f$ or $\Gam_{ex}$ 
        	not touching $\ol{\n_f^{+}}\cap \{x=y=0\}$ or $\ol{\Gam_w^{+}}$.
        	Note that by the facts that $\pt_\theta (V(f(\theta),\theta))>0$ for $\theta\in (0,\theta_1)$ and $\pt_\theta V(r_1,\theta)>0$ for $\theta\in (0,\theta_1)$ obtained from \eqref{transcondition}, 
        	the surface $\theta=h_{(r^\sharp,\theta^\sharp)}(r)$ intersects with $\Gam_f$ and $\Gam_{ex}$ once, respectively. 
        	Collect $\theta=h_{(r^\sharp,\theta^\sharp)}(r)$ for all $(r^\sharp,\theta^\sharp)\in (f(\theta^\sharp),r_1)\times (0,\theta_1)$, $\ol{\n_f^{+}}\cap \{x=y=0\}$ and $\ol{\Gam_w^{+}}$. Then we have the entire level surfaces of $V$ in $\ol{\n_f^+}$. By $\pt_\theta V>0$  in $\n_f^{+}\setminus\{x=y=0\}$, the values of $V$ on distinct level surfaces of $V$ in $\ol{\n_f^+}$ are different from each other.

        	2. 
        	Find a solution of \eqref{transeqn}, \eqref{transbdry}. 
        	
        	By \eqref{PhiV}, \eqref{transeqn} can be written as
        	\begin{align}
        	\label{Qform}\ngrad V\cdot \grad Q=0.
        	\end{align}
        	Using this form of \eqref{transeqn}, it can be checked that 1) if $Q$ is in $C^1(\n_f^+)$ and $Q=\tx{constant}$ on
        	any curve on any level surface of $V$ in $\n_f^+$ whose $\vphi$ argument is fixed (in the case when a level surface of $V$ is $\n_f^+\cap \{x=y=0\}$, $Q=\tx{constant}$ on  $\n_f^+\cap \{x=y=0\}$), then $Q$ is a solution of \eqref{transeqn} and that 2) if $Q$ is a $C^1$ solution of \eqref{transeqn}, then $Q=\tx{constant}$ on any curve on any level surface of $V$ in $\n_f^+$ whose $\vphi$ argument is fixed. 
        	Denote the $\theta$-argument of the intersection points of $\Gam_f$ and the level surface of $V$ in $\ol{\n_f^+}$ passing through ${\rm x}\in \ol{\n_f^+}$ by $\mc L({\rm x})$. 
        	Since each level surface of $V$ in $\ol{\n_f^+}$ intersects with $\ol{\Gam_f}$ where the value of $V$ on $\ol{\Gam_f}$ is equal to the value of $V$ on the level surface, $\mc L({\rm x})$ is given by
        	$$\mc L({\rm x}):=k^{-1}\circ V({\rm x})$$ 
        	where $k(\theta)=V(f(\theta),\theta)$.
        	
        	By 2), a $C^0(\ol{\n_f^+})\cap C^1(\n_f^+)$ solution of \eqref{transeqn}, \eqref{transbdry} has the form
        	\begin{align*}
        	Q(r,\theta,\vphi)=\begin{cases}
        	Q_{en}(\mc L(r,\theta),\vphi)\qd\tx{if $(r,\theta)\in [f(\theta),r_1]\times(0,\theta_1]$}\\
        	Q_{en}(\mc L(r,0),0)\qd\tx{if $(r,\theta)\in [f(0),r_1]\times\{0\}$.}
        	\end{cases}
        	\end{align*} 
        	Since $Q_{en}$ is axisymmetric, this can be written as
        	$$Q=Q_{en}(\mc L).$$
        	Let $Q=Q_{en}(\mc L)$. 
        	One can see that $Q$ is a constant  on any level surface of $V$ in $\ol{\n_f^+}$ and satisfies \eqref{transbdry}. 
            Thus, by 1), if $Q\in C^0(\ol{\n_f^+})\cap C^1(\n_f^+)$, then $Q$ is a $C^0(\ol{\n_f^+})\cap C^1(\n_f^+)$ solution of \eqref{transeqn}, \eqref{transbdry}. 
        	
        	3. Estimate $||Q||_{1,\alpha,\n_f^+}^{(-\alpha,\Gam_w^+)}$. 
        	
        	Since $\mc L\in C^0(\ol{\n_f^+})$ and $Q_{en}\in C^0(\ol{\Gam_f})$, $Q$ is in $C^0(\ol{\n_f^+})$. It directly follows from the definition of $Q$ 
        	$$||Q||_{0,0,\n_f^+}=||Q_{en}||_{0,0,\Gam_f}.$$ 
        	Thus, to estimate $||Q||_{1,\alpha,\n_f^+}^{(-\alpha,\Gam_w^+)}$, 
        	it is enough to estimate $||DQ||_{\alpha,\n_f^+}^{(1-\alpha,\Gam_w^+)}$. We estimate $||DQ||_{\alpha,\n_f^+}^{(1-\alpha,\Gam_w^+)}$. 
        	
        	By direct computation, 
        	\begin{align}\nonumber
        	DQ&=Q_{en}^\p(\mc L)D\mc L\\
        	&=Q_{en}^\p(\mc L)\frac{DV}{(\pt_r V(f(\vtheta),\vtheta)f^\p(\vtheta)+
        		\pt_\vtheta V(f(\vtheta),\vtheta))|_{\vtheta=\mc L}}. \nonumber
        	\end{align}
        	To estimate $||DQ||_{\alpha,\n_f^+}^{(1-\alpha,\Gam_w^+)}$, we 
        	estimate $||D\mc L||_{0,\alpha,\n_f^+}$.
        	Write $D\mc L$ as
        	\begin{align*}
        	\frac{\frac{DV}{2\pi  \sin\theta} }{\left.(\ngrad V)(f(\vtheta),\vtheta)\cdot {\bm \nu}_f(\vtheta)f^2(\vtheta)\sqrt{1+\left(\frac{f^\p(\vtheta)}{f(\vtheta)}\right)^2}\right|_{\vtheta=\mc L}}\cdot \frac{\sin\theta}{\sin \mc L}
        	\end{align*}
        	Here, we used the definition of $\ngrad V$ and the spherical coordinate expression of ${\bm \nu}_f$.  
        	Using  \eqref{f} and \eqref{transcondition}, it is easily seen that 
        	\begin{align}
        	\label{trae1}\sup_{\n_f^+}\left|\left.(\ngrad V)(f(\vtheta),\vtheta)\cdot {\bm \nu}_f(\vtheta)f^2(\vtheta)\sqrt{1+\left(\frac{f^\p(\vtheta)}{f(\vtheta)}\right)^2}\right|_{\vtheta=\mc L}\right|>c^{**}
        	\end{align}
        	where $c^{**}$ is a positive constant depending on $c^*$, $r_s$ and $\delta_1$. 
        	Using \eqref{f} and \eqref{Phisize}, it is also easily seen that 
        	\begin{align}
        	\label{trae2}||\frac{DV}{2\pi  \sin\theta}||_{0,\alpha,\n_f^+}\le C
        	\end{align}
        	and
        	\begin{align}
        	\label{trae3}||\left.(\ngrad V)(f(\vtheta),\vtheta)\cdot {\bm \nu}_f(\vtheta)f^2(\vtheta)\sqrt{1+\left(\frac{f^\p(\vtheta)}{f(\vtheta)}\right)^2}\right|_{\vtheta=\mc L}||_{0,\alpha,\n_f^+}\le C||D\mc L||_{0,0,\n_f^+}
        	\end{align} 
        	where $Cs$ are positive constants depending on $\rho_0^+$, $u_0^+$, $r_s$, $\alpha$, $\delta_1$ and $\delta_6$.
        	By these three estimates, one can see that to estimate $||D\mc L||_{0,\alpha,\n_f^+}$, it is enough to estimate $||\frac{\theta}{\mc L}||_{0,\alpha,\n_f^+}$. 
        	To estimate $||D\mc L||_{0,\alpha,\n_f^+}$, we 
        	prove the following claim. 
        	
        	{\bf\textit{Claim.}}
        		\textit{There exists a positive constant $C$ depending on $\rho_0^+$, $u_0^+$, $r_s$, $r_1$, $\theta_1$, $\alpha$, $\delta_1$ and $\delta_6$ such that
        		\begin{align}\label{thetaLalpha}||\frac{\theta}{\mc L}||_{0,\alpha,\n_f^+}\le C.\end{align}}

        	
           
            \begin{proof}[Proof of Claim]
            	To simplify our argument, we assume that $f=r_s$. 
            	
            	First, we estimate 
            	$||\frac{\theta}{\mc L}||_{0,0,\n_{r_s}^+}$. 
            	By the definition of $\mc L$,
            	$$V(r_s, \mc L(r,\theta))=V(r,\theta)$$
            	for all $(r,\theta)\in [r_s,r_1]\times [0,\theta_1]$. 
            	Write this as
            	\begin{align*}
            	\int_0^\theta (\rho_0^+u_0^++\grad\times(\Psi\be_\vphi)\cdot \be_r)(r,\xi) r^2\sin \xi d\xi=\int_0^{\mc L(r,\theta)}(\rho_0^+u_0^++\grad\times(\Psi\be_\vphi)\cdot \be_r)(r_s,\xi)r_s^2\sin \xi  d \xi. 
            	\end{align*}
            	By \eqref{transcondition}, \eqref{Phisize} and the first equation of \eqref{conservsub}, we have from this equation
            	\begin{align*}
            	\int_0^\theta 
            	c^*r^2\sin  \xi d\xi\le\int_0^{\mc L(r,\theta)}(m_0+\delta_6r_s^2)\sin  \xi d \xi. 
            	\end{align*}
            	Using $
            	\frac{\sin\theta_1}{\theta_1}\xi  \le \sin \xi  \le  \xi $ for $\xi \in [0,\theta_1]$, change $\sin \xi  $ in the integrands in the left and right-hand side of the above inequalities to $\frac{\sin\theta_1}{\theta_1}\xi $ and $\xi $, respectively, and then integrate the resultant terms. Then we obtain
            	$$\frac{\sin\theta_1}{\theta_1}c^*r^2\frac{\theta^2}{2}\le (m_0+\delta_6r_s^2)\frac{\mc L^2(r,\theta)}{2}.$$
            	From this inequality, we have
            	\begin{align*}
            	\sqrt{\frac{\sin\theta_1}{\theta_1}\frac{c^*r_s^2}{m_0+\delta_6r_s^2}}\theta\le \mc L(r,\theta).
            	\end{align*}
            	This holds for all $(r,\theta)\in [r_s,r_1]\times [0,\theta_1]$. Hence, \begin{align}
            	\label{c00}||\frac{\theta}{\mc L}||_{0,0,\n_{r_s}^+}\le C
            	\end{align}
            	where $C$ is a positive constant depending on $\rho_0^+$, $u_0^+$, $r_s$, $\theta_1$, $\delta_6$ and $c^*$.  
            	
            	Next, we estimate $[\frac{\theta}{\mc L}]_{0,\alpha,\n_{r_s}^+}$. 
            	By \eqref{c00}, we can obtain an estimate of $[\frac{\theta}{\mc L}]_{0,\alpha,\n_{r_s}^+}$ by obtaining
            	\begin{align}\label{qdif}
            	 \left|\frac{\theta}{\mc L(r,\theta)}-\frac{\theta^\p}{\mc L(r^\p,\theta^\p)}\right|\le C\eps^\alpha
            	\end{align}
            	for all $(r,\theta)$, $(r^\p,\theta^\p)\in [r_s,r_1]\times [0,\theta_1]$ satisfying $\eps\le \eps_0$ for a positive constant $\eps_0$ and a positive constant $C$ where $\eps:=\sqrt{|r^\p-r|^2+|\theta^\p-\theta|^2}$. 
            	We obtain this estimate. 
            	Hereafter, to simplify our notation, we denote $\mc L(r^\p,\theta^\p)$, $\mc L(r,\theta)$ and $\Phi_0^++\Psi$ by $\mc L^\p$, $\mc L$ and $\Phi$, respectively. 
            	
            	By the definition of $\mc L$,
            	$$V(r_s,\mc L^\p)-V(r_s,\mc L)=V(r^\p,\theta^\p)-V(r,\theta).$$
            	Write this as
            	\begin{multline}\label{Ldifference}
                \int_{\mc L}^{\mc L^\p}(\grad\times (\Phi\be_\vphi)\cdot \be_r )(r_s,\xi)r_s^2 \sin \xi   d \xi\\
                =\int_\theta^{\theta^\p}(\grad\times(\Phi\be_\vphi)\cdot \be_r)(r,\xi) r^2\sin \xi   d \xi+\int_r^{r^\p}(\grad\times(\Phi\be_\vphi))(t,\theta^\p)\cdot\be_\theta(\theta^\p) t sin\theta^\p d t.
                \end{multline}
            	Using \eqref{Ldifference}, we find an upper bound of $\mc L^\p$. 
            	Using \eqref{Phisize}, the fact that $\Phi_0^+\be_\vphi\in C^\infty(\ol{\n_{r_s}^+})$ and $\frac{\sin\theta_1}{\theta_1}\xi \le \sin \xi  \le  \xi$ for $\xi\in [0,\theta_1]$, we have for $\mc L^\p \ge \mc L$
            	\begin{align*}
            	&\int_{\mc L}^{\mc L^\p} \left((\grad\times(\Phi\be_\vphi)\cdot \be_r )(r_s,\mc L)-C
            	(\xi-\mc L)^\alpha \right) r_s^2 \frac{\sin\theta_1}{\theta_1 } \xi
            	d \xi
            	\\&\le (\grad\times(\Phi\be_\vphi)\cdot \be_r)(r,\theta) r^2\frac{{\theta^\p}^2-\theta^2}{2}+C
            	r^2\left(\frac{1}{\alpha+1}|\theta^\p-\theta|^{\alpha+1}\theta^\p-\frac{1}{(\alpha+1)(\alpha+2)}|\theta^\p-\theta|^{\alpha+2}\right)\\
            	&+(\grad\times(\Phi\be_\vphi)\cdot \be_\theta)(r,\theta^\p)\sin\theta^\p \frac{ {r^\p}^2-r^2}{2}+C
            	\sin\theta^\p\left(\frac{1}{\alpha+1} |{r^\p}-r|^{\alpha+1}r^\p-\frac{1}{(\alpha+1)(\alpha+2)}|r^\p-r|^{\alpha+2}\right)\\&(=:(u))
            	\end{align*}
            	where the right-hand side is an upper bound of the right-hand side of \eqref{Ldifference} and $C$ is a positive constant depending on $\rho_0^+$, $u_0^+$, $\n_{r_s}^+$, $\alpha$ and $\delta_6$. 
            	From this inequality, we get
            	\begin{multline*}
                (\grad\times(\Phi\be_\vphi)\cdot \be_r )(r_s,\mc L)\frac{\sin\theta_1}{\theta_1} r_s^2 \frac{{\mc L^\p}^2-\mc L^2}{2}\\-C
                r_s^2\frac{\sin\theta_1}{\theta_1}
                \left(\frac{1}{\alpha+1}(\mc L^\p-\mc L)^{\alpha+1}\mc L^\p-\frac{1}{(\alpha+1)(\alpha+2)}(\mc L^\p-\mc L)^{\alpha+2}\right)
            	\le (u).
            	\end{multline*}
            	Using this inequality, we have that for each $(r,\theta)\in [r_s,r_1]\times[0,\theta_1]$, there exists positive constants $\eps_{(r,\theta)}< 1$ and $C_{(r,\theta)}$ such that for any $(r^\p,\theta^\p)$ satisfying $\eps\le\eps_{(r,\theta)}$, 
            	\begin{align*}
            	\mc L^\p\le \theta^\p \left(\frac{\mc L}{\theta}+C_{(r,\theta)}\eps^\alpha\right).
            	\end{align*}
            	Note that when $(r,\theta)\in [r_s,r_1]\times \{0\}$, there exists a positive constant $\eps_{(r,\theta)}<1$ such that for any $(r^\p,\theta^\p)$ satisfying $\eps\le\eps_{(r,\theta)}$,
            	\begin{align*}
            	\mc L^\p\le \theta^\p \left(\sqrt\frac{r^2(\grad\times (\Phi\be_\vphi)\cdot \be_r )(r,0)}{r_s^2(\grad\times (\Phi\be_\vphi)\cdot \be_r )(r_s,0)}+C_{(r,\theta)}\eps^\alpha\right).
            	\end{align*}
            	Let $\ol\eps_0=\inf_{(r,\theta)\in [r_s,r_1]\times [0,\theta_1]}\eps_{(r,\theta)}$ and $C=\sup_{(r,\theta)\in [r_s,r_1]\times [0,\theta_1]}C_{(r,\theta)}$. Then by the above statement, 
            	\begin{align*}
            	\frac{\mc L^\p}{\theta^\p}\le \frac{\mc L}{\theta}+C\eps^\alpha
            	\end{align*}
            	for all $(r,\theta)$, $(r^\p,\theta^\p)$ satisfying $\eps\le \ol\eps_0$. This gives
            	\begin{align*}
            	\frac{\theta}{\mc L}-\frac{\theta^\p}{\mc L^\p}\le C\eps^\alpha
            	\end{align*}
            	for all $(r,\theta)$, $(r^\p,\theta^\p)$ satisfying $\eps\le \ol\eps_0$ for a positive constant $C$. 
            	
            	Similarly, we can obtain 
            	\begin{align*}
            	- C\eps^\alpha\le\frac{\theta}{\mc L}-\frac{\theta^\p}{\mc L^\p}
            	\end{align*}
            	for all $(r,\theta)$, $(r^\p,\theta^\p)$ satisfying $\eps\le \ul\eps_0$ for a positive constant $C$ and a positive constant $\ul\eps_0$. Let $\eps_0=\min(\ol\eps_0,\ul\eps_0)$. 
            	Choose $\eps_{(r,\theta)}$ so that it can depend on $\rho_0^+$, $u_0^+$, $\n_{r_s}^+$, $\alpha$ and $\delta_6$. Then $\ol{\eps}_0$ depends on $\rho_0^+$, $u_0^+$, $\n_{r_s}^+$, $\alpha$ and $\delta_6$. In the same way, we have  $\ul{\eps}_0$ depends on $\rho_0^+$, $u_0^+$, $\n_{r_s}^+$, $\alpha$ and $\delta_6$. Thus, $\eps_0$ depends on $\rho_0^+$, $u_0^+$, $\n_{r_s}^+$, $\alpha$ and $\delta_6$.
            	This finishes the proof of Claim. 
            \end{proof}
        By \eqref{trae1}-\eqref{thetaLalpha},
        we have
        $$||D\mc L||_{0,\alpha,\n_f^+}\le C$$
        for a positive constant $C$. 
        It can be easily shown that there exists a positive constant $C$ such that
        $$|\theta-\theta_1|\le C|\mc L(r,\theta)-\theta_1|$$
        for all $(r,\theta)\in [f(\theta),r_1]\times [0,\theta_1]$. Using these two estimates, estimate $||DQ||_{\alpha,\n_f^+}^{(1-\alpha,\Gam_w^+)}$. Then we obtain
        $$||DQ||_{\alpha,\n_f^+}^{(1-\alpha,\Gam_w^+)}\le C$$
        where $C$ is a positive constant depending on $\rho_0^+$, $u_0^+$, $r_s$, $r_1$, $\alpha$, $\delta_1$ and $\delta_6$.  
        
        4. By the result in Step 3, 
        $Q$ is a $C^0(\ol{\n_f^+})\cap C^1(\n_f^+)$ solution of \eqref{transeqn}, \eqref{transbdry}. 
        By 2), a $C^0(\ol{\n_f^+})\cap C^1(\n_f^+)$ solution of \eqref{transeqn}, \eqref{transbdry} is unique. This finishes the proof. 
        
        \end{proof}
     \subsection{Proof of Proposition \ref{proPseudofree}.}\label{subproofPseudoFree}
     Using the results in \S \ref{subelliptic} and \S \ref{subtrans}, 
     we prove Proposition \ref{proPseudofree}. 
    \begin{proof}[Proof of Proposition \ref{proPseudofree} (Existence)]
    	Suppose that $(\rho_-,\bu_-,p_-,p_{ex},f_s^\p)$ is as in 
    	Problem \ref{pb3} for $\alpha\in (\frac 2 3,1)$ and $\sigma\in (0,\ol\sigma_3]$ where $\ol\sigma_3$ is a positive constant to be determined later.
    	For the same $\alpha$ and $\sigma$ and $M_1>0$ to be determined later, define
    	\begin{align*}
    	\mc P(M_1)&=\{(f(0),\Psi\be_{\vphi})\in \R\times C^{2,\alpha}_{(-1-\alpha,\Gam_{w,r_s+f_s}^+)}(\n_{r_s+f_s}^+)\;|\;\\
    	&\Psi\be_{\vphi}=\frac{f(\theta_1)\Psi(f(\theta_1),\theta_1)}{r}\be_{\vphi}\;\tx{on}\;\Gam_{w,r_s+f_s}^+,
    	\;|f(0)-r_s|+||\Psi\be_{\vphi}||_{2,\alpha,\n_{r_s+f_s}^+}^{(-1-\alpha,\Gam_{w,r_s+f_s}^+)}\le M_1\sigma\}.
    	\end{align*}
    	By the definition of $\mc P(M_1)$, $\mc P(M_1)$ is a compact convex subset of $\R\times C^{2,\frac{\alpha}{2}}_{(-1-\frac{\alpha}{2},\Gam_{w,r_s+f_s}^+)}(\n_{r_s+f_s}^+)$. 
    	We will prove the existence part of Proposition \ref{proPseudofree}
    	by constructing a continuous map of $\mc P(M_1)$ into itself as a map from $\R\times C^{2,\frac{\alpha}{2}}_{(-1-\frac{\alpha}{2},\Gam_{w,r_s+f_s}^+)}(\n_{r_s+f_s}^+)$ to $\R\times C^{2,\frac{\alpha}{2}}_{(-1-\frac{\alpha}{2},\Gam_{w,r_s+f_s}^+)}(\n_{r_s+f_s}^+)$
    	and applying the Schauder fixed point theorem.

    	In this proof, $C$s and $C_i$ for $i=1,2,\ldots$ denote positive constants depending on the whole or a part of the data, $\delta_1$, $\delta_2$, $\delta_4$, $\delta_5$ and $\delta_6$ unless otherwise specified. Each $C$ in different situations differs from each other. 
    	
    	1. For a fixed $(f(0),\Psi\be_\vphi)$, solve  (B$^\p$). 
    	
    	Take $(f(0)^*,\til \Psi^*\be_{\vphi})\in \mc P(M_1)$. By the definition of $f_s$ and the assumption that $\fsp$ satisfies \eqref{fsestimate},  
    	\begin{align}
    	\label{c1}||f_s||_{2,\alpha,\Lambda}^{(-1-\alpha,\pt\Lambda)}\le C_1\sigma.
    	\end{align}
    	Choose $\ol{\sigma}_3=\min(\frac{\delta_1}{2M_1},\frac{\delta_1}{2C_1})
    	(=:\ol{\sigma}_3^{(1)})$ so that
    	\begin{align}\label{f0star}
    	|f(0)^*-r_s|\le \frac{\delta_1}{2}\qdand||f_s||_{2,\alpha,\Lambda}^{(-1-\alpha,\pt\Lambda)} \le \frac{\delta_1}{2}.
    	\end{align}
    	And then, using $\Pi_{f(0)^*+f_sr_s+f_s}$, 
    	extend $\til \Psi^*\be_{\vphi}$ to a function in $\n_{f(0)^*+f_s}^+$:
    	$$\Psi^*\be_{\vphi}:=\frac{W^*}{2\pi r \sin\theta}\be_{\vphi}$$ 
    	where $W^*:=\til W^*(\Pi_{f(0)^*+f_sr_s+f_s})$ with $\til W^*:=2\pi r \sin\theta \til \Psi^*$ and $\Pi_{ab}$ for $0<a,b<r_1$ is a map from $\n_a^+$ to $\n_b^+$ defined in Step 1 in the proof of Lemma \ref{lemalpha}.
    	By \eqref{f0star} and the fact that $||\til \Psi^*\be_{\vphi}||_{2,\alpha,\n_{r_s+f_s}^+}^{(-1-\alpha,\Gam_{w,r_s+f_s}^+)}\le M_1\sigma$,  
    	\begin{align}\label{Psistarestimate}
    	||\grad\times(\Psi^*\be_{\vphi})||_{1,\alpha,\n_{f(0)^*+f_s}^+}^{(-\alpha,\Gam_{w,f(0)^*+f_s}^+)}\le C_2 M_1\sigma.
    	\end{align}
    	Choose $\ol{\sigma}_3= \min (\frac{\delta_6}{C_2M_1},
    	\ol{\sigma}_3^{(1)})(=:\ol{\sigma}_3^{(2)})$ 
    	where $\delta_6$ is a positive constant given in Lemma \ref{lemTrans} so that $\Psi\be_{\vphi}=\Psi^*\be_{\vphi}$ satisfies \eqref{Phisize} for $f=f(0)^*+f_s$. 
        
        For $(f(0),\Psi\be_{\vphi})=(f(0)^*,\Psi^*\be_{\vphi})$, solve (B$^\p$) : solve
    	\begin{align}\label{trans1}
    	&\begin{cases}
    	\grad\times((\Phi_0^++\Psi^*)\be_{\vphi})\cdot \grad A=0\qdin \n_{f(0)^*+f_s}^+,\\ 
    	A=A_{en,f(0)^*+f_s}\qdon \Gam_{f(0)^*+f_s},
    	\end{cases}\\
    	\label{trans2}
    	&\begin{cases}
    	\grad\times((\Phi_0^++\Psi^*)\be_{\vphi})\cdot \grad T=0\qdin \n_{f(0)^*+f_s}^+,\\ 
    	T=T_{en,f(0)^*+f_s}\qdon \Gam_{f(0)^*+f_s}
    	\end{cases}
    	\end{align}	
    	where $A_{en,f(0)^*+f_s}$ and $T_{en,f(0)^*+f_s}$ are $A_{en,f(0)+f_s}$ and $T_{en,f(0)+f_s}$ given in \eqref{Aentdef} and \eqref{Tentdef}, respectively, for $f(0)=f(0)^*$.  
    	By $(\rho_-,\bu_-,p_-,\fsp)\in (C^{2,\alpha}(\ol{\n}))^3\times C^{1,\alpha}_{(-\alpha,\{\theta=\theta_1\}),0}((0,\theta_1))$,
    	$A_{en,f(0)^*+f_s}\in C^{2,\alpha}_{(-1-\alpha,\pt \Gam_{f(0)^*+f_s})}$ $(\Gam_{f(0)^*+f_s})$ and $T_{en,f(0)^*+f_s}\in C^{1,\alpha}_{(-\alpha,\pt\Gam_{f(0)^*+f_s})}(\Gam_{f(0)^*+f_s}).$
    	Take $\ol{\sigma}_3 =\min (\ol{\sigma}_3^{(2)},\delta_5)(=:\ol{\sigma}_3^{(3)})$ where $\delta_5$ is a positive constant given in Lemma \ref{LemTenestimate} so that $(\rho_-,\bu_-,p_-)$ satisfies \eqref{supesti} for $\sigma\le\delta_5$. 
    	Then by Lemma \ref{LemTenestimate}, \eqref{fsestimate} and the fact that $f(0)^*$ satisfies $|f(0)^*-r_s|\le M_1\sigma$, 
    	\begin{align}\label{tenesti}||T_{en,f^*}||_{1,\alpha,\Gam_{f^*}}^{(-\alpha,\pt \Gam_{f^*})}\le CM_1\sigma +C\sigma
    	\end{align}
    	where $f^*:=f^*(0)+f_s$. 
    	Apply Lemma \ref{lemTrans} to \eqref{trans1}, \eqref{trans2}. Then we obtain that each \eqref{trans1} and \eqref{trans2} has a unique axisymmetric $C^0(\ol{\n_f^+})\cap C^1(\n_f^+)$ solution
    	\begin{align*}A^*=A_{en,f(0)^*+f_s}(\mc L^*)\qdand T^*=T_{en,f(0)^*+f_s}(\mc L^*),
    	\end{align*}
    	respectively, where $\mc L^*$ is $\mc L$ defined in Lemma \ref{lemTrans} for $V=2\pi r \sin\theta (\Phi_0^++\Psi^*)$ and $f=f(0)^*+f_s$. Furthermore, we have that  
    	$T^*$ satisfies 
    	\begin{align}\label{Tstarestimate}
    	||T^*||_{1,\alpha,\n_{f^*}^+}^{(-\alpha,\Gam_{w,f^*}^+)}\le CM_1\sigma+C\sigma
    	\end{align}
    	where we used \eqref{tenesti}. 
    	Using the solution expression of $A^*$, estimate $\frac{A^*}{2\pi r \sin\theta}\be_{\vphi}$ in $C^{1,\alpha}(\ol{\n_{f^*}^+})$. 
    	Using \eqref{Aentdef}, 
    	write $\frac{A^*}{2\pi r \sin\theta}\be_{\vphi}$ as
    	\begin{align}\label{Aesti}
    	\frac{2\pi f^*(\mc L^*)\sin (\mc L^*)u_{-,\vphi}(\mc L^*)}{2\pi r \sin\theta}\be_{\vphi}
    	\end{align}
    	where $u_{-,\vphi}=\bu_-\cdot\be_\vphi$. 
    	Using arguments similar to the ones in the proof of Claim in Lemma \ref{lemTrans}, we can obtain 
        $$||\frac{\mc L^*}{\theta}||_{0,\alpha,(f^*(\theta),r_1)\times (0,\theta_1)}\le C.$$ 
        With this estimate, $||\mc L^*||_{1,\alpha,(f^*(\theta),r_1)\times (0,\theta_1)}\le C$, \eqref{supesti} 
        and \eqref{f0star}, we estimate \eqref{Aesti} in $C^{1,\alpha}(\ol{\n_{f^*}^+})$. Then we have
    	\begin{align}\label{Astarestimate}
    	||\frac{A^*}{2\pi r \sin\theta}\be_{\vphi}||_{1,\alpha,\n_{f^*}^+}
    	\le C\sigma.
    	\end{align}

    	2. By substituting an extension of $(\Psi^*\be_{\vphi},A^*,T^*)$ into nonlinear parts of (A$^\p$), 
    	obtain a linear problem having unknowns $(f(0),\Psi\be_{\vphi})$.
    	
    	Let $f(0)^\sharp$ be a point in $[r_s-\frac{\delta_1}{2},r_s+\frac{\delta_1}{2}]$. 
    	By the choice of $\ol{\sigma}_3$, 
    	\begin{align*}
    	||f^\sharp-r_s||_{2,\alpha,\Lambda}^{(-1-\alpha,\pt \Lambda)}
    	\le \delta_1
    	\end{align*}
    	where $f^\sharp:=f(0)^\sharp+f_s$. 
    	Extend $\Psi^*\be_{\vphi}$ $A^*$ and $T^*$ to functions in  $\n_{f^\sharp}^+$:
    	$$\Psi^{\sharp,*}\be_{\vphi}:=\frac{W^{\sharp,*}}{2\pi r \sin\theta}\be_{\vphi},\qd A^{\sharp,*}:=A^*(\Pi_{f^\sharp f^*})\qdand T^{\sharp,*}:=T^*(\Pi_{f^\sharp f^*})$$ where $W^{\sharp,*}=W^*(\Pi_{f^\sharp f^*})$.
    	By \eqref{Psistarestimate}, \eqref{Tstarestimate} and \eqref{Astarestimate},  
    	\begin{align}\label{Psisharpstare}
    	||\grad\times(\Psi^{\sharp,*}\be_{\vphi})||_{1,\alpha,\n_{f^\sharp}^+}^{(-\alpha,\Gam_{w,f^\sharp}^+)}+||\frac{A^{\sharp,*}}{2\pi r \sin\theta}\be_\vphi||_{1,\alpha,\n_{f^\sharp}^+}+||T^{\sharp,*}||_{1,\alpha,\n_{f^\sharp}^+}^{(-\alpha,\Gam_{w,f^\sharp}^+)}\le C_3M_1\sigma+C_4\sigma
    	\end{align}
    	for all $f^\sharp(0)\in [r_s-\frac{\delta_1}{2},r_s+\frac{\delta_1}{2}]$. 
    	Take $\ol\sigma_3= \min (\ol\sigma_3^{(3)},
    	\frac{\delta_4}{C_3M_1+C_4})(=:\ol\sigma_3^{(4)})$
    	where 
    	$\delta_4$ is a positive constant given in 
    	Lemma \ref{F1f1estimate} so that 
    	$\varrho(\grad\times((\Phi_0^++\Psi^{\sharp,*})\be_{\vphi}),\frac{A^{\sharp,*}}{2\pi r \sin\theta}\be_{\vphi},S_0^++T^{\sharp,*})$ is well-defined in $\n_{f^\sharp}^+$ 
    	and $ \varrho(\grad\times((\Phi_0^++\Psi^{\sharp,*})\be_{\vphi}),\frac{A^{\sharp,*}}{2\pi r \sin\theta}\be_{\vphi},S_0^++T^{\sharp,*})$ and $S_0^++T^{\sharp,*}$ are strictly positive in $\n_{f^\sharp}^+$ for all $f^\sharp(0)\in [r_s-\frac{\delta_1}{2},r_s+\frac{\delta_1}{2}]$. 
    	By substituting $f(0)^\sharp$ 
    	and $
    	(\Psi^{\sharp,*}\be_{\vphi},A^{\sharp,*},T^{\sharp,*})$ into the place of $f(0)$ in (A$^\p$) and $(\Psi\be_{\vphi},A,T)$ in ${\bm F}_1(\Psi\be_{\vphi},A,T)$, $\mf f_0(T,p_{ex})$ and $\mf f_1(\Psi\be_{\vphi},A,T)$ in (A$^\p$), 
    	we obtain
    	\begin{multline}\label{Psieqnsubsti}
    	\grad\times\left(\frac{1}{\rho_0^+}(1+\frac{{u_0^+}\be_r\otimes u_0^+\be_r}{{c_0^+}^2-{u_0^+}^2})\grad\times(\Psi\be_{\vphi})\right) \\=\frac{{\rho_0^+}^{\gam-1}}{(\gam-1)u_0^+}(1+\frac{\gam {u_0^+}^2}{{c_0^+}^2-{u_0^+}^2})\frac{\pt_\theta T}{r}\be_{\vphi}+{\bm F}_1(\Psi^{\sharp,*}\be_{\vphi},A^{\sharp,*},T^{\sharp,*})\qdin \n_{f^\sharp}^+,
    	\end{multline}
    	\begin{align}\label{Psiboundarysubsti}
    	\Psi\be_{\vphi}=\begin{cases}
    	(\Phi_--\Phi_0^-)\be_{\vphi}\qd\tx{on}\qd \Gam_{f^\sharp},\\
    	\frac{r_0(\Phi_--\Phi_0^-)(r_0,\theta_1)}{r}\be_{\vphi}\qdon \Gam_{w,f^\sharp}^+:=\Gam_w\cap \{r>f^\sharp\},\\
    	\left(\left.\frac{1}{r_1\sin\theta}\int_0^\theta \left(\mf f_0(T^{\sharp,*},p_{ex})
    	\right.\right.\right.\\
    	\left.\left.\qd\qd-\frac{\rho_0^+((\gam-1){u_0^+}^2+{c_0^+}^2)}{\gam(\gam-1)u_0^+S_0^+}T+\mf f_1(\Psi^{\sharp,*}\be_{\vphi},A^{\sharp,*},T^{\sharp,*}) \right)r_1^2\sin\xi d\xi\right)\be_{\vphi}\qdon \Gam_{ex},
    	\end{cases}
    	\end{align}
    	\begin{multline}\label{Psicompcsubsti} \frac{1}{r_1\sin\theta_1}\int_0^{\theta_1} \bigg(\mf f_0(T^{\sharp,*},p_{ex})
    	\\
    	\qquad\qquad\qd-\frac{\rho_0^+((\gam-1){u_0^+}^2+{c_0^+}^2)}{\gam(\gam-1)u_0^+S_0^+}T+\mf f_1(\Psi^{\sharp,*}\be_{\vphi},A^{\sharp,*},T^{\sharp,*}) \bigg)\bigg|_{r=r_1}r_1^2\sin\xi d\xi\\
    	=\frac{r_0(\Phi_--\Phi_0^-)(r_0,\theta_1)}{r_1}.
    	\end{multline}
        We let $f(0)^\sharp$ be an unknown in this problem 
    	and let $T$ in \eqref{Psieqnsubsti}-\eqref{Psicompcsubsti} be a solution of 
    	\begin{align}\label{transshockposition}
    	\begin{cases}
    	\grad\times( (\Phi_0^++\Psi^{\sharp,*})\be_{\vphi})\cdot \grad T=0\qdin \n_{f^\sharp}^+,
    	\\ T=T_{en,f^\sharp}\qdon \Gam_{f^\sharp}.
    	\end{cases}
    	\end{align}
    	Then since a solution $T$ of \eqref{transshockposition} is uniquely determined by $f(0)^\sharp$ (see Step 3), unknowns of \eqref{Psieqnsubsti}-\eqref{Psicompcsubsti} become $(f^\sharp(0),\Psi\be_\vphi)$. We denote an unknown 
    	$\Psi\be_{\vphi}$ of \eqref{Psieqnsubsti}-\eqref{Psicompcsubsti} by $\Psi^\sharp\be_{\vphi}$. 
    	
    	3.  Find $f(0)^\sharp$ using \eqref{Psicompcsubsti}. 
    	
    	By the definition of $\Psi^{\sharp,*}\be_{\vphi}$ and \eqref{PhiV}, the transport equation in \eqref{transshockposition} can be written as 
    	\begin{align*}
    	\ngrad (V_0^++W^{\sharp,*})\cdot\grad T=0\qdin \n_{f^\sharp}^+.
    	\end{align*}
    	From this form of the transport equation in \eqref{transshockposition}, we see that 
    	the stream surface of the vector field $\grad\times( (\Phi_0^++\Psi^{\sharp,*})\be_{\vphi})$ in $\n_{f^\sharp}^+$ is obtained by  stretching or contracting the 
    	stream surface of the vector field $\grad \times((\Phi_0^++\Psi^*)\be_{\vphi})$ in $\n_{f^*}^+$ in $r$-direction. 
    	Using this fact, we obtain that 
    	the solution of \eqref{transshockposition} is given by
    	$$T=T_{en,f^\sharp}(\mc L^\sharp)$$
    	where $\mc L^\sharp=\mc L^*(\Pi_{f^\sharp f^*})$. We denote this solution 
    	by $T^\sharp$. 
    	
    	By \eqref{Tentdef}, $T^\sharp$ is expressed as 
    	\begin{align*}
    	T^{\sharp}=\left(g\left(\left(\frac{\bu_-\cdot{\bm\nu_{f^\sharp}(\mc L^\sharp)}}{c_-}\right)^2\right)S_-\right)
    	(f(0)^\sharp+f_s(\mc L^\sharp),\mc L^\sharp)-(g({M_0^-}^2))(r_s)S_{in}.
    	\end{align*}
    	Substituting this expression of $T^\sharp$ into the place of $T$ in \eqref{Psicompcsubsti} using the fact that $\Pi_{f^\sharp f^*}(r_1,\theta)=(r_1,\theta)$, 
    	we obtain 
    	\begin{align}\label{PsiTcompc}
    	(L):=&\frac{1}{r_1\sin\theta_1} \int_0^{\theta_1}\left.\frac{\rho_0^+((\gam-1){u_0^+}^2+{c_0^+}^2)}{\gam(\gam-1)u_0^+S_0^+}\right|_{r=r_1}(a)_{1}
    	r_1^2\sin\xi d\xi\\
    	&=-\frac{r_0(\Phi_--\Phi_0^-)(r_0,\theta_1)}{r_1}+\frac{1}{r_1\sin\theta_1}\int_0^{\theta_1} \left(\mf f_0(T^*,p_{ex})\right.\nonumber\\
    	&-\frac{\rho_0^+((\gam-1){u_0^+}^2+{c_0^+}^2)}{\gam(\gam-1)u_0^+S_0^+}\left((a)_{2}+
    	\left.\left.(a)_{3}\right)+\mf f_1(\Psi^{*}\be_{\vphi},A^{*}, T^{*}) \right)\right|_{r=r_1}r_1^2\sin\xi d\xi=:(R)\nonumber
    	\end{align}
    	where 
    	\begin{align*}
    	&(a)_{1}=
    	(g({M_0^-}^2))(f(0)^\sharp+f_s(\mc L^{*}(r_1,\theta)))S_{in}-(g({M_0^-}^2))(r_s+f_s(\mc L^{*}(r_1,\theta)))S_{in},\\
    	&(a)_{2}=\left(g\left(\left(\frac{\bu_-\cdot{\bm\nu_{f^\sharp}(\mc L^{*}(r_1,\theta))}}{c_-}\right)^2\right)S_-\right)(f(0)^\sharp+f_s(\mc L^{*}(r_1,\theta)),\mc L^{*}(r_1,\theta))
    	\\&\qd\qd\qd\qd\qd\qd\qd\qd\qd\qd\qd\qd\qd\qd\qd\qd\qd\qd-(g({M_0^-}^2))(f(0)^\sharp+f_s(\mc L^{*}(r_1,\theta)))S_{in},\\
    	&(a)_{3}=(g({M_0^-}^2))(r_s+f_s(\mc L^{*}(r_1,\theta)))S_{in}-(g({M_0^-}^2))(r_s)S_{in}.
    	\end{align*}
    	We find $f(0)^\sharp$ satisfying \eqref{PsiTcompc}. For this, we  estimate $|(R)|$. 
    	
    	Estimate $|(R)|$:
    	With \eqref{supesti} and \eqref{fsestimate}, we estimate $(a)_2$, $(a)_3$ and $\frac{r_0(\Phi_--\Phi_0^-)(r_0,\theta_1)}{r_1}$. 
    	Then we obtain 
    	\begin{align}
        \nonumber
    	\sup_{\theta\in (0,\theta_1)}|(a)_{2}|
    	&\le C\sigma|f(0)^\sharp-r_s|+C\sigma\\
    		\label{a1esti}&\le C\frac{\delta_1}{2}\sigma+C\sigma
    	\end{align}
    	for all $f^\sharp(0)\in [r_s-\frac{\delta_1}{2},r_s+\frac{\delta_1}{2}]$,
    	\begin{align}
    	\label{a3esti}\sup_{\theta\in (0,\theta_1)}|(a)_{3}|\le C\sigma
    	\end{align}
        and
        \begin{align}
    	\label{phisupesti}|\frac{r_0(\Phi_--\Phi_0^-)(r_0,\theta_1)}{r_1}|\le C\sigma.
    	\end{align}
        With \eqref{pexesti}, we estimate $\mf f_0(T^*,p_{ex})(r_1,\theta)$. Then we have
    	\begin{align}\label{f0esti}
    	\sup_{\theta\in (0,\theta_1)}|\mf f_0(T^*,p_{ex})(r_1,\theta)|\le C \sigma.
    	\end{align} 
    	By Lemma \ref{F1f1estimate} and using \eqref{Psistarestimate}, \eqref{Tstarestimate} and \eqref{Astarestimate}, we estimate $\mf f_1(\Psi^{*}\be_{\vphi},A^{*},T^*)(r_1,\theta)$. Then we have 
    	\begin{align}\label{f1esti}
    	\sup_{\theta\in (0,\theta_1)}|\mf f_1(\Psi^{*}\be_{\vphi},A^{*},T^*)(r_1,\theta)|
    	\le C(M_1+1)^2\sigma^2.
    	\end{align} 
        With \eqref{a1esti}-\eqref{f1esti}, we estimate $|(R)|$. Then we get
    	\begin{align}\label{c0norm}
    	|(R)|\le C_5(M_1+1)^2\sigma^2
    	+C_6\sigma 
    	\end{align}
    	for all $f^\sharp(0)\in [r_s-\frac{\delta_1}{2},r_s+\frac{\delta_1}{2}]$.
    	
    	Then we find $f(0)^\sharp$. By Lemma \ref{Slem}, there exists a positve constant $\lambda$ such that
    	\begin{align*}
    	(L)^\p(f(0)^\sharp)\ge\lambda
    	\end{align*}
    	for all $f^\sharp(0)\in [r_s-\frac{\delta_1}{2},r_s+\frac{\delta_1}{2}]$.
    	By this fact and $(L)(r_s)=0$, 
    	\begin{align}\label{Lleft}
    	(L)(\frac{\delta_1}{2})\ge\frac{\delta_1\lambda}{2}\qdand (L)(-\frac{\delta_1}{2})\le -\frac{\delta_1\lambda}{2}.
    	\end{align}
    	Choose $\ol{\sigma}_3= \min(\ol{\sigma}_3^{(4)},
    	\frac{C_6}{C_5(M_1+1)^2},\frac{\delta_1\lambda}{8C_6})(=:\ol{\sigma}_3^{(5)})$.
    	Then since
    	\begin{align*}
        |(R)|\le \frac{\delta_1\lambda}{4}
    	\end{align*} 
    	for all $f(0)^\sharp\in [r_s-\frac{\delta_1}{2},r_s+\frac{\delta_1}{2}]$
    	and \eqref{Lleft} holds, by the intermediate value theorem, there exists $f(0)^\sharp$ satisfying \eqref{PsiTcompc} in $[r_s-\frac{\delta_1}{2},r_s+\frac{\delta_1}{2}]$.
    	Such $f(0)^\sharp$ is unique because $(L)$ is a monotone function of $f(0)^\sharp$. 
    	Since at $f(0)^\sharp$ where \eqref{PsiTcompc} holds there holds
    	\begin{align*}
    	\lambda |f(0)^\sharp-r_s|\le |(R)(f(0)^\sharp)|,
    	\end{align*}
    	there holds 
    	\begin{align}\label{f0estimate}
    	|f(0)^\sharp-r_s|\le C(M_1+1)^2\sigma^2+C\sigma.
    	\end{align} 
    	
    	
    	4. Find 
    	$\Psi^\sharp\be_{\vphi}$. 
    	
    	Fix $f(0)^\sharp$ obtained in Step 3. By \eqref{f0estimate}, Lemma \ref{LemTenestimate} and Lemma \ref{lemTrans}, $T^{\sharp}$ determined by $f(0)^\sharp$ satisfies
    	\begin{align}\label{Tsharpf}
    	||T^{\sharp}||_{1,\alpha,\n_{f^\sharp}^+}^{(-\alpha,\Gam_{w,f^\sharp}^+)}\le C(M_1+1)^2\sigma^2+C\sigma.
    	\end{align} 
    	Substitute this $T^{\sharp}$ 
    	into the place of $T$ in \eqref{Psieqnsubsti} and \eqref{Psiboundarysubsti}. Then we obtain 
    	\begin{align}\label{Psi1eqn}
    	&\grad\times\left(\frac{1}{\rho_0^+}(1+\frac{{u_0^+}\be_r\otimes u_0^+\be_r}{{c_0^+}^2-{u_0^+}^2})\grad\times(\Psi\be_{\vphi})\right) \\&\qd\qd\qd\qd\nonumber=\frac{{\rho_0^+}^{\gam-1}}{(\gam-1)u_0^+}(1+\frac{\gam {u_0^+}^2}{{c_0^+}^2-{u_0^+}^2})\frac{\pt_\theta T^\sharp}{r}\be_{\vphi}+{\bm F}_1(\Psi^{\sharp,*}\be_{\vphi},A^{\sharp,*},T^{\sharp,*})\qdin \n_{f^\sharp}^+,\\
    	&\Psi\be_{\vphi}=\begin{cases}\label{Psi1bc}
    	(\Phi_--\Phi_0^-)\be_{\vphi}\qd\tx{on}\qd \Gam_{f^\sharp},\\
    	\frac{r_0(\Phi_--\Phi_0^-)(r_0,\theta_1)}{r}\be_{\vphi}\qdon \Gam_{w,f^\sharp}^+,\\
    	\left(\frac{1}{r_1\sin\theta}\int_0^\theta \left(\mf f_0(T^{\sharp,*},p_{ex})\right.\right.\\
    	\qd\qd\left.\left.-\frac{\rho_0^+((\gam-1){u_0^+}^2+{c_0^+}^2)}{\gam(\gam-1)u_0^+S_0^+}T^\sharp+\mf f_1(\Psi^{\sharp,*}\be_{\vphi},A^{\sharp,*},T^{\sharp,*}) \right)r_1^2\sin\xi d\xi\right)\be_{\vphi}\qdon \Gam_{ex}.
    	\end{cases}
    	\end{align}
    	Since $f(0)^\sharp$ is chosen for $T=T^\sharp$ to satisfy \eqref{Psicompcsubsti}, \eqref{Psi1bc} is a continuous boundary condition.
    	We apply Lemma \ref{lemSE} to \eqref{Psi1eqn}, \eqref{Psi1bc} with \eqref{supesti}, \eqref{pexesti}, \eqref{fsestimate}, \eqref{Psisharpstare} and \eqref{Tsharpf}. 
    	Then we obtain that \eqref{Psi1eqn}, \eqref{Psi1bc} has a unique $C^{2,\alpha}_{(-1-\alpha,\Gam_{f^\sharp})}(\n_{f^\sharp}^+)$ solution $\Psi^\sharp 
    	\be_{\vphi}$ 
    	and this solution satisfies the estimate 
    	\begin{align}\label{Psistaresti}
    	||\Psi^\sharp
    	\be_{\vphi}||_{2,\alpha,\n_{f^\sharp}^+}^{(-1-\alpha,\Gam_{w,f^\sharp}^+)}\le C(M_1+1)^2\sigma^2+C\sigma. 
    	\end{align}
    	
    	5. 
    	Using $\Pi_{r_s+f_sf^\sharp}$, transform $\Psi^\sharp\be_{\vphi}$ 
    	into a function in $\n_{r_s+f_s}^+$: 
    	$$\til \Psi^\sharp\be_{\vphi}=\frac{\til W^\sharp}{2\pi r \sin\theta}\be_{\vphi}$$
    	where $\til W^\sharp=W^\sharp(\Pi_{r_s+f_sf^\sharp})$ with $W^\sharp=2\pi r \sin\theta \Psi^\sharp$. By \eqref{f0estimate} and \eqref{Psistaresti}, 
    	\begin{align*}
    	||\til \Psi^\sharp\be_{\vphi}||_{2,\alpha,\n_{r_s+f_s}^+}^{(-1-\alpha,\Gam_{w,r_s+f_s}^+)}\le C(M_1+1)^2\sigma^2+C\sigma.
    	\end{align*} 
    	Combining this with \eqref{f0estimate}, we have
    	\begin{align*}
    	|f(0)^\sharp-r_s|+||\til \Psi^\sharp\be_{\vphi}||_{2,\alpha,\n_{r_s+f_s}^+}^{(-1-\alpha,\Gam_{w,r_s+f_s}^+)}\le C_7(M_1+1)^2\sigma^2+C_8\sigma.
    	\end{align*} 
    	Take $M_1=2C_8$ and $\ol{\sigma}_3= \min(\ol{\sigma}_3^{(5)},\frac{C_8}{C_7(M_1+1)^2})(=:\ol{\sigma}_3^{(6)})$. 
    	And then define a map $\mc J$ from $\R\times C^{2,\frac{\alpha}{2}}_{(-1-\frac{\alpha}{2},\Gam_{w,r_s+f_s}^+)}(\n_{r_s+f_s}^+)$ to $\R\times C^{2,\frac{\alpha}{2}}_{(-1-\frac{\alpha}{2},\Gam_{w,r_s+f_s}^+)}(\n_{r_s+f_s}^+)$ by
    	$$\mc J(f(0)^*,\til \Psi^*\be_{\vphi})=(f(0)^\sharp, \til \Psi^\sharp\be_{\vphi}).$$
    	By the choice of $M_1$ and $\ol\sigma_3$, $\mc J$ is a map  of $\mc P(M_1)$ into itself. Using the standard argument, one can easily check that $\mc J$ is continuous. Thus, the Schauder fixed point theorem can be applied to $\mc J$. 
        We apply the Schauder fixed point theorem to $\mc J$. Then we obtain that there exists a fixed point $(f^\flat(0),\til\Psi^\flat\be_{\vphi})\in \mc P(M_1)$ of $\mc J$.

        One can see that if $(f(0)^*,\til \Psi^*\be_{\vphi})=(f(0)^\sharp, \til \Psi^\sharp\be_{\vphi})$, then $\Psi^*\be_{\vphi} =\Psi^{\sharp,*}\be_{\vphi}=\Psi^\sharp\be_{\vphi}$, 
    	$A^{\sharp,*}=A^*$ and $T^*=T^{\sharp,*}=T^\sharp$. 
    	From this fact, we see that $(f(0),\Phi\be_{\vphi},L,S)=(f(0)^\flat,(\Phi_0^++\Psi^\flat)\be_{\vphi},A^\flat, S_0^++T^\flat)$, where $\Psi^\flat\be_{\vphi}:=\frac{(2\pi r \sin\theta \til\Psi^\flat)(\Pi_{f(0)^\flat+f_s r_s+f_s})}{2\pi r \sin\theta}\be_{\vphi}$ and $A^\flat$ and $T^\flat$ are solutions of \eqref{trans1} and \eqref{trans2} for given $(f(0)^*,\Psi^*)=(f(0)^\flat,\Psi^\flat)$, respectively, 
    	is a solution of (A), (B).  
    	By $(f(0)^\flat, \til \Psi^\flat\be_{\vphi})\in \mc P(M_1)$, \eqref{Tstarestimate} and \eqref{Astarestimate}, 
    	\begin{align}\label{CC}
    	|f(0)^\flat-r_s|+||\grad\times(\Psi^\flat\be_{\vphi})||_{1,\alpha,\n_{f^\flat}^+}^{(-\alpha,\Gam_{w,f^\flat}^+)}+||\frac{A^\flat}{2\pi r \sin\theta}||_{1,\alpha,\n_{f^\flat}^+}
    	+||T^\flat||_{1,\alpha,\n_{f^\flat}^+}^{(-\alpha,\Gam_{w,f^\flat}^+)}\le C\sigma
    	\end{align} 
    	where $f^\flat:=f(0)^\flat+f_s$. 
    	One can easily see that there exists a positive constant $\delta_7$ such that for any $\n_f^+\subset\n_{r_s-\delta_1}^+$, if 
    	$(\grad\times(\Phi\be_{\vphi}),\frac{L}{2\pi r \sin\theta}\be_{\vphi},S,B_0)\in B^{(1)}_{\delta_7,\n_f^+}$ where $B^{(1)}_{\delta,\Om}$ for $\delta>0$ and $\Om\subset \R^3$ is a neighborhood of $(\grad\times(\Phi_0^+\be_\vphi),0,S_0^+,B_0)$ defined in Lemma \ref{lemRho},  then 
    	\eqref{subdef} holds in $\n_f^+$. Let $\delta_7$ be such a constant. 
    	Take $\ol{\sigma}_3=\min(\ol{\sigma}_3^{(6)},\frac{\delta_7}{C})$ where $C$ is $C$ in \eqref{CC}. Then $(f(0)^\flat,(\Phi_0^++\Psi^\flat)\be_{\vphi},A^\flat, S_0^++T^\flat)$ is a subsonic solution of  \eqref{eqV}-\eqref{eqS} in $\n_{f^\flat}^+$. Choose 
    	$\delta_1$, $\delta_2$, $\delta_4$, $\delta_5$, $\delta_6$ and $\delta_7$ so that they can depend on the data. Then 
    	$M_1$ and $\ol{\sigma}_3$ depend on the data. This finishes the proof. 
    \end{proof}
    %
    \begin{proof}[Proof of Proposition \ref{proPseudofree} (Uniqueness)] 
    	Let $\alpha\in (\frac 2 3,1)$. Let $\sigma_3$ be a positive constant  $\le \ol\sigma_3$ and to be determined later.
    	Suppose that 
    	there exist two solutions 
    	$(f_i(0), \Phi_i\be_{\vphi},L_i,S_i)$ for $i=1,2$ 
    	of Problem \ref{pb3} for $\sigma\le \sigma_3$ 
    	satisfying the estimate \eqref{Pseudoestimate}. 
    	
    	Let $(\Psi_i,A_i,T_i):=(\Phi_i-\Phi_0^+,L_i,S_i-S_0^+)$ for $i=1,2$.  
    	We will prove that there exists a positive constant $\ul\sigma_3$ such that if $\sigma_3=\ul\sigma_3$, then
    	$$(f_1(0), \Psi_1\be_{\vphi},A_1,T_1)=(f_2(0),\Psi_2\be_{\vphi}, A_2,T_2)$$ by constructing a contraction map in a low regularity space.

    	In this proof, $C$s denote positive constants depending on the data unless otherwise specified. Each $C$ in different situations differs from each other. 
    	
    	1. By subtracting (A$^\p$) satisfied by $(f_1(0), \Psi_1\be_{\vphi},A_1,T_1)$ from (A$^\p$) satisfied by $(f_2(0), \Psi_2\be_{\vphi},$ $A_2,T_2)$, obtain the equations that will give a contraction map. 
    	
    	Let $f_i:=f_i(0)+f_s$ for $i=1,2$. From  $$(\ngrad a)(\Pi_{f_2f_1})=N\ngrad (a(\Pi_{f_2f_1}))\qdand \grad\times (\frac{a\be_{\vphi}}{2\pi r \sin\theta})=\ngrad a$$
    	for an axisymmetric scalar function $a$, we can obtain 
    	\begin{multline}\label{transform}
    	(\grad\times(\Psi_1\be_{\vphi}))(\Pi_{f_2f_1})=N\grad \times( \til \Psi_1\be_{\vphi})\\
    	\tx{and}\qd
    	\left(\grad\times \left(\frac{A_1}{2\pi r \sin\theta}\be_{\vphi}\right)\right)(\Pi_{f_2f_1})=
    	N\grad\times \left(\frac{\til A_1}{2\pi r \sin\theta}\be_{\vphi}\right),
    	\end{multline}
    	where $\til \Psi_1:=\frac{\til W_1}{2\pi r \sin\theta}$ with $\til W_1:=W_1(\Pi_{f_2f_1})$ and $W_1:=2\pi r \sin\theta \Psi_1$, $\til A_1:=A_1(\Pi_{f_2f_1})$ and
    	$$N=\frac{r^2}{(\Pi_{f_2f_1}^r)^2}\be_r\otimes\be_r-\frac{\left(\pt_{\til \theta}\Pi_{f_1f_2}^{r}\right)(\Pi_{f_2f_1}) }{(\Pi_{f_2f_1}^r)^2}r\be_r\otimes\be_\theta+\frac{\left(\pt_{\til r}\Pi_{f_1f_2}^{r}\right)(\Pi_{f_2f_1}) }{\Pi_{f_2f_1}^r}r\be_\theta\otimes\be_\theta$$
    	with $\Pi_{f_2f_1}^r$ and $\Pi_{f_1f_2}^{r}$, $r$-components of $\Pi_{f_2f_1}^*$ and $\Pi_{f_1f_2}^*$, respectively (see the definition of $\Pi_{ab}^*$ in \eqref{iota}), $(\til r,\til \theta)$, $(r,\theta)$ coordinates for the cartesian coordinate for $\n_{f_1(0)+f_s}^+$, and $(r,\theta)=\Pi_{f_1f_2}^*(\til r,\til \theta)$. 
    	Using $\Pi_{f_2f_1}$ and \eqref{transform}, transform (A$^\p$) satisfied by $(f_1(0), \Psi_1\be_{\vphi},A_1,T_1)$. 
    	And then subtract the resultant equations from (A$^\p$) satisfied by $(f_2(0), \Psi_2\be_{\vphi},A_2,T_2)$. Then we obtain  
    	\begin{multline}\label{uniqueeqn}
    	\grad\times\left(\frac{1}{\rho_0^+}(1+\frac{{u_0^+}\be_r\otimes u_0^+\be_r}{{c_0^+}^2-{u_0^+}^2})\grad\times((\Psi_2-\til \Psi_1)\be_{\vphi})\right)\\ =\frac{{\rho_0^+}^{\gam-1}}{(\gam-1)u_0^+}(1+\frac{\gam {u_0^+}^2}{{c_0^+}^2-{u_0^+}^2})\frac{\pt_\theta (T_2-\til T_1)}{r}\be_{\vphi}+{\bm F}_3\qdin \n_{f_2(0)+f_s}^+
    	\end{multline}
    	\begin{align}\label{uniquebc}
    	(\Psi_2-\til\Psi_1)\be_{\vphi}=
    	\begin{cases}
    	(\Phi_--\Phi_0^-)\be_{\vphi}-\frac{\Pi_{f_2f_1}^r(\Phi_--\Phi_0^-)(\Pi_{f_2f_1})}{r}\be_{\vphi}(=:\ol{{\bm h}}_1)\qdon \Gam_{f_2(0)+f_s},\\
    	{\bm 0}\qdon\Gam_{w,f_2(0)+f_s}^+:=\Gam_w\cap \{r>f_2(0)+f_s\},\\
    	\bigg(\frac{1}{r_1\sin\theta} \int_0^\theta \bigg(\mf f_0(T_2,p_{ex})-\mf f_0(\til T_1,p_{ex})
    	-\frac{\rho_0^+((\gam-1){u_0^+}^2+{c_0^+}^2)}{\gam(\gam-1)u_0^+S_0^+}(T_2-\til T_1)
    	\\+\mf f_1(\Psi_2\be_{\vphi},A_2,T_2) -\mf f_1(\til \Psi_1\be_{\vphi},\til A_1,\til T_1)\bigg)r_1^2\sin\xi d\xi\bigg)\be_{\vphi}(=:\ol{{\bm h}}_2)\qdon \Gam_{ex},
    	\end{cases}
    	\end{align}
    	\begin{multline}\label{uniquecpt}
    	0=\frac{1}{r_1\sin\theta_1} \int_0^{\theta_1} \bigg(\mf f_0(T_2,p_{ex})-\mf f_0(\til T_1,p_{ex})
    	-\frac{\rho_0^+((\gam-1){u_0^+}^2+{c_0^+}^2)}{\gam(\gam-1)u_0^+S_0^+}(T_2-\til T_1)
    	\\
    	\left.+\mf f_1(\Psi_2\be_{\vphi},A_2,T_2) -\mf f_1(\til \Psi_1\be_{\vphi},\til A_1,\til T_1)\bigg)\right|_{r=r_1}r_1^2\sin\xi d\xi
    	\end{multline}
    	where 
    	\begin{align*}
    	{\bm F}_3=
    	&-\grad\times\left(\frac{1}{\rho_0^+}(1+\frac{{u_0^+}\be_r\otimes u_0^+\be_r}{{c_0^+}^2-{u_0^+}^2})\grad\times (\til \Psi_1\be_{\vphi})\right)\\
    	&+M\grad \times\left(\left(\frac{1}{\rho_0^+}(1+\frac{{u_0^+}\be_r\otimes u_0^+\be_r}{{c_0^+}^2-{u_0^+}^2})\right)(\Pi_{f_2f_1})N\grad\times( \til \Psi_1\be_{\vphi})\right)\\
    	&+\frac{{\rho_0^+}^{\gam-1}}{(\gam-1)u_0^+}(1+\frac{\gam {u_0^+}^2}{{c_0^+}^2-{u_0^+}^2})\frac{1}{r}\pt_\theta \til T_1  \be_{\vphi}\\
    	&-\left(\frac{{\rho_0^+}^{\gam-1}}{(\gam-1)u_0^+}(1+\frac{\gam {u_0^+}^2}{{c_0^+}^2-{u_0^+}^2})\frac{1}{r}\right)(\Pi_{f_2f_1})\left(\frac{\pt\Pi_{f_1f_2}^{r}}{\pt\til\theta}(\Pi_{f_2f_1})\pt_r  + \pt_\theta \right)\til T_1\be_{\vphi}\\
    	&+{\bm F}_1(\Psi_2\be_{\vphi},A_2,T_2)-\til {\bm F}_1(\til \Psi_1\be_{\vphi},\til A_1,\til T_1),
    	\end{align*}
    	$\til T_1=T_1(\Pi_{f_2f_1})$, 
    	$$M=\left(\pt_{\til r}\Pi_{f_1f_2}^{r}\right)(\Pi_{f_2f_1})\be_r\otimes\be_r+\frac{\left(\pt_{\til \theta}\Pi_{f_1f_2}^{r}\right)(\Pi_{f_2f_1})}{\Pi_{f_2f_1}^r}\be_\theta\otimes\be_r+ \frac{r}{\Pi_{f_2f_1}^r}\be_\theta\otimes\be_\theta+
    	\frac{r}{\Pi_{f_2f_1}^r}\be_{\vphi}\otimes\be_{\vphi}$$
    	and 
    	$\til{\bm F}_1$ is ${\bm F}_1$ changed by using the transformation $\Pi_{f_2f_1}$.
    	We will construct a contraction map using \eqref{uniqueeqn}-\eqref{uniquecpt}. For this, we estimate $||T_2-\til T_1||_{0,\beta,\n_{f_2}^+}$ for 
    	$\beta=1-\frac{3}{q}\in (0,\alpha)$ where $q\in (3,\frac{1}{1-\alpha})$ and $||(\frac{A_2}{2\pi r\sin\theta}-\frac{\til A_1}{2\pi r \sin\theta})\be_{\vphi}||_{1,0,\n_{f_2}^+}$.
    	
    	2. Estimate $||T_2-\til T_1||_{0,\beta,\n_{f_2}^+}$. 
    	
    	Since $T_i$ for $i=1,2$ are solutions of 
    	$$\grad \times((\Phi_0^++\Psi_i)\be_{\vphi})\cdot \grad T=0\qdin \n_{f_i(0)+f_s}^+,\qd T=T_{en,f_i(0)+f_s}\qdon \Gam_{f_i(0)+f_s}$$
    	for $i=1,2$, respectively, where $T_{en,f_i(0)+f_s}$ is $T_{en,f(0)+f_s}$ given in \eqref{Tentdef} for $f(0)=f_i(0)$, 
    	by Lemma \ref{lemTrans},
    	\begin{align}\label{Ti}
    	T_i=T_{en,f_i(0)+f_s}(\mc L_i)
    	\end{align}
    	for $i=1,2$ where $\mc L_i$ are $\mc L$ given in Lemma \ref{lemTrans} for $V=V_0^++ W_i(=:V_i)$ and $f=f_i(0)+f_s$ with $W_i:=2\pi r \sin\theta \Psi_i$. 
    	By \eqref{Tentdef}, \eqref{Ti} and the definition of $\til T_1$, $T_2-\til T_1$ can be written as 
    	\begin{align*}
    	\left(g\left(\left(\frac{\bu_-\cdot{\bm\nu}_{f_2}(\mc L_2)}{c_-}\right)^2\right)S_- \right)(f_2(\mc L_2),{\mc L}_2)
    	-\left(g\left(\left(\frac{\bu_-\cdot{\bm \nu}_{f_1}(\til{\mc L}_1)}{c_-}\right)^2\right)S_- \right)(f_1(\til {\mc L}_1),\til {\mc L}_1)
    	\end{align*}
    	where  
    	${\bm \nu}_{f_i}$ for $i=1,2$ are the unit normal vectors on $\Gam_{f_i}$ pointing toward $\n_{f_i }^+$, respectively, and $\til {\mc L}_1:=\mc L_1(\Pi_{f_2f_1})$. 
        This can be decomposed into 
        \begin{multline*}
        \left(g\left(\left(\frac{\bu_-\cdot{\bm\nu}_{f_2}(\mc L_2)}{c_-}\right)^2\right)S_- \right)(f_2(\mc L_2),{\mc L}_2)-\left(g\left(\left(\frac{\bu_-\cdot{\bm\nu}_{f_1}(\til{\mc L}_1)}{c_-}\right)^2\right)S_- \right)(f_2(\mc L_2),{\mc L}_2)\\
        +\left(g\left(\left(\frac{\bu_-\cdot{\bm\nu}_{f_1}(\til{\mc L}_1)}{c_-}\right)^2\right)S_- \right)(f_2(\mc L_2),{\mc L}_2)-\left(g\left(\left(\frac{\bu_-\cdot{\bm \nu}_{f_1}(\til{\mc L}_1)}{c_-}\right)^2\right)S_- \right)(f_1(\til {\mc L}_1),\til {\mc L}_1)\\
        =:(a)+(b).
        \end{multline*}
    	To estimate $T_2-\til T_1$ in $C^\beta(\ol{\n_{f_2}^+})$, we estimate $(a)$ and $(b)$ in $C^\beta(\ol{\n_{f_2}^+})$, respectively. 
    	To obtain an estimate of $(a)$ in $C^\beta(\ol{\n_{f_2}^+})$ and later to obtain an estimate of $||(\frac{A_2}{2\pi r\sin\theta}-\frac{\til A_1}{2\pi r \sin\theta})\be_{\vphi}||_{1,0,\n_{f_2}^+}$, 
    	we prove the following claim. 
    	
    	We take $\sigma_3=\min(\ol{\sigma}_3,\frac{\delta_6}{C})(=:\sigma_3^{(1)})$ where $C$ is $C$ in \eqref{Pseudoestimate} so that $\Psi\be_\vphi=\Psi_i\be_\vphi$ for $i=1,2$ satisfy \eqref{transcondition} for $f=f_i(0)+f_s$.  
    	
    	{\bf\textit{Claim.}} \textit{Let $h:[0,\theta_1]\ra\R$ be a function in $C^{1,\alpha}_{(-\alpha,\{\theta=\theta_1\})}((0,\theta_1))$.
    		There holds
    		\begin{align*}
    		||\int_0^1 h^\p(t\mc L_2+(1-t)\til{\mc L}_1)dt(\mc L_2-\til {\mc L}_1)||_{0,\beta,\n_{f_2}^+}
    		\le C||h||_{1,\beta,(0,\theta_1)}^{(-\alpha,\{\theta=\theta_1\})}||(\Psi_2-\til \Psi_1)\be_{\vphi}||_{1,\beta,\n_{f_2}^+}.
    		\end{align*}}
    		
    		\begin{proof}[Proof of Claim]
    			By the definitions of $\mc L_2$ and $\til{\mc L}_1$,
    			\begin{align*}
    			\mc L_2=k_2^{-1}\circ V_2\qdand \til{\mc L}_1=(k_1^{-1}\circ V_1)(\Pi_{f_2f_1})
    			\end{align*}
    			where $k_i(\theta)=V_i(f_i(\theta),\theta)$ for $i=1,2$. 
    			By $V_0^+(\Pi_{f_2f_1})=V_0^+$ and $\til W_1(f_2(\theta),\theta)=W_1(f_1(\theta),\theta)$ where $\til W_1$ is defined below \eqref{transform},  
    			$\til{\mc L}_1$ can be written as
    			$$\til{\mc L}_1=\til k_1^{-1}\circ \til V_1$$
    			where $\til V_1=V_0^++\til W_1$ and $\til k_1=\til V_1(f_2(\theta),\theta)$. Since $\til V_1(r,\theta)\in [0,V_-(r_0,\theta_1)]$ for $(r,\theta)\in [f_2(\theta),r_1]\times[0,\theta_1]$ and $V_2(f_2(\theta),\theta)\in [0,V_-(r_0,\theta_1)]$ for $\theta\in [0,\theta_1]$ where $V_-=2\pi r \sin\theta \Phi_-$, $k_2^{-1}\circ\til V_1$ is well-defined in $\ol{\n_{f_2}^+}$. With this fact, we write $\mc L_2-\til{\mc L}_1$ 
    			as
    			\begin{align*}
    			\mc L_2-\til{\mc L}_1&=(k_2^{-1}\circ V_2-k_2^{-1}\circ\til V_1)+(k_2^{-1}\circ\til V_1-\til k_1^{-1}\circ \til V_1)\\
    			&\nonumber=\int_0^1\frac{1}{k_2^\p(k_2^{-1}\circ (tV_2+(1-t)\til V_1))}dt (V_2-\til V_1)\\
    			&\nonumber\qd\qd\qd\qd+\int_0^1 \frac{1}{\til k_1^\p(\til k_1^{-1}\circ (t \til k_1(\vartheta)+(1-t)k_2(\vartheta)))} dt
    			(\til k_1(\vartheta)-k_2(\vartheta))
    			\end{align*}
    			where $\vartheta=k_2^{-1}\circ \til V_1$ and we used the fact that $k_2^{-1}\circ\til V_1-\til k_1^{-1}\circ \til V_1=\til k_1^{-1}\circ \til k_1(\vartheta)-\til k_1^{-1}\circ k_2(\vartheta)$. 
    			Substitute the above expression of $\mc L_2-\til{\mc L}_1$ 
    			into the place of $\mc L_2-\til{\mc L}_1$ in $\int_0^1 h^{\p}(t\mc L_2+(1-t)\til{\mc L}_1)dt(\mc L_2-\til{\mc L}_1)$.
    		    Then we have 
    			\begin{align}\nonumber
    			&\int_0^1 h^{\p}(t\mc L_2+(1-t)\til{\mc L}_1)dt(\mc L_2-\til{\mc L}_1)\\
    			&=\int_0^1 h^{\p}(t\mc L_2+(1-t)\til{\mc L}_1)dt\int_0^1\frac{1}{k_2^\p(k_2^{-1}\circ (tV_2+(1-t)\til V_1))}dt (V_2-\til V_1)\nonumber\\
    			&+\int_0^1 h^{\p}(t\mc L_2+(1-t)\til{\mc L}_1)dt
    			\int_0^1\frac{1}{\til k_1^\p(\til k_1^{-1}\circ (t \til k_1(\vartheta)+(1-t)k_2(\vartheta)))}dt (\til k_1(\vartheta)-k_2(\vartheta))\nonumber\\
    			&=:(c)+(d).\nonumber
    			\end{align}
    			To estimate $||\int_0^1 h^\p(t\mc L_2+(1-t)\til{\mc L}_1)dt(\mc L_2-\til {\mc L}_1)||_{0,\beta,\n_{f_2}^+}$,
    			we estimate $(c)$ and $(d)$ in $C^\beta(\ol{\n_{f_2}^+})$, respectively. 
    			
    			Estimate $(c)$ in $C^\beta(\ol{\n_{f_2}^+})$: 
    			With the definitions of $k_2$, 
    			$\Psi_2$ and $\til\Psi_1$, write $(c)$ as
    			\begin{multline}\label{e1}
    			(c)=\int_0^1 (\theta-\theta_1) h^{\p}(t\mc L_2+(1-t)\til{\mc L}_1)dt\\
    			\int_0^1\frac{r \sin\theta}{J\sin(k_2^{-1}\circ (tV_2+(1-t)\til V_1))}dt\int_0^1 \pt_{\theta}(\Psi_2-\til \Psi_1)(r,t\theta+(1-t)\theta_1) dt
    			\end{multline}
    			where $$
    			J:=\left.\left( f_2^2\sqrt{1+\left(\frac{f_2^\p}{f_2}\right)^2}\right)(\cdot)(\grad\times((\Phi_0^++\Psi_2)\be_{\vphi}))(f(\cdot),\cdot)\cdot{\bm \nu}_{f_2}(\cdot)\right|_{\cdot=k_2^{-1}\circ (tV_2+(1-t)\til V_1)}.\nonumber$$
    			By the choice of $\sigma_3$,  $\Psi\be_\vphi=\Psi_i\be_\vphi$ for $i=1,2$ satisfy \eqref{transcondition} for $f=f_i(0)+f_s$. 
    			By this fact and $||f_2-r_s||_{2,\alpha,\Lambda}^{(-1-\alpha,\pt\Lambda)}\le \delta_1$ and the fact that $k_2^{-1}\circ (tV_2+(1-t)\til V_1)$ maps $\ol{\n_{f_2}^+}$ to $[0,\theta_1]$, 
    			\begin{align}\label{j}
    			J>\til c^*\qdin \n_{f_2}^+
    			\end{align}
    			 for all $t\in [0,1]$ for some positive constant $\til c^*$. 
    			Using arguments similar to the ones in the proof of Claim in Lemam \ref{lemTrans}, we can obtain
    			\begin{align}
    			\label{Lthetadist1}
    			||\frac{\theta-\theta_1}{(t\mc L_2+(1-t)\til{\mc L}_1)-\theta_1}||_{0,\beta,(f_2(\theta),r_1)\times(0,\theta_1)}\le C
    			\end{align}
    			and
    			\begin{align}\label{k2alphabound}
    			||\frac{\theta}{k_2^{-1}\circ (tV_2+(1-t)\til V_1)}||_{0,\beta,(f_2(\theta),r_1)\times(0,\theta_1)}\le C
    			\end{align}
    			for any $t\in [0,1]$. 
    			With \eqref{j}-\eqref{k2alphabound}, 
    			$||ab||_{0,\beta,\Om}\le ||a||_{0,\beta,\Om}||b||_{0,0,\Om}+ ||a||_{0,0,\Om}||b||_{0,\beta,\Om}$, \eqref{Pseudoestimate} satisfied by $(f_i(0), \Phi_i\be_{\vphi},L_i,S_i)$ for $i=1,2$, 
    			$h^{\p}\in C^\alpha_{(1-\alpha,\{\theta=\theta_1\})}((0,\theta_1))$, 
    			\begin{align}\label{L2L1}
    			||\mc L_2||_{1,0,(f_2(\theta),r_1)\times (0,\theta_1)}\le C,\;  ||\til{\mc L}_1||_{1,0,(f_2(\theta),r_1)\times (0,\theta_1)}\le C
    			\end{align}
    			and $||k_2^{-1}\circ (tV_2+(1-t)\til V_1)||_{1,0,(f_2(\theta),r_1)\times (0,\theta_1)}\le C$, 
    			we estimate the right-hand side of \eqref{e1} in $C^\beta(\ol{\n_{f_2}^+})$. Then we obtain
    			\begin{align}\label{e2}
    			||(c)||_{0,\beta,\n_{f_2}^+}\le C||h||_{1,\beta,(0,\theta_1)}^{(-\alpha,\{\theta=\theta_1\})}||(\Psi_2-\til\Psi_1)\be_{\vphi}||_{1,\beta,\n_{f_2}^+}.
    			\end{align}

    			$(d)$ can be estimated in $C^\beta(\ol{\n_{f_2}^+})$ in a similar way. 
    			As we do this, $\til k_1$ and $k_2$ play the role of $V_2$ and $\til V_1$ in the estimate of $(e)$ in $C^\beta(\ol{\n_{f_2}^+})$ and $\vartheta$ is regarded as the argument of $\til k_1$ and $k_2$. 
    			We have
    			\begin{align}\label{f2}
    			||(d)||_{0,\beta,\n_{f_2}^+}\le C||h||_{1,\beta,(0,\theta_1)}^{(-\alpha,\{\theta=\theta_1\})}||(\Psi_2-\til\Psi_1)\be_{\vphi}||_{1,\beta,\n_{f_2}^+}.
    			\end{align}
    			
    			Combining \eqref{e2} and \eqref{f2}, we obtain the desired result. 
    		\end{proof}
        Note that if we change $\mc L_2-\til{\mc L}_1$ in the way that we changed $\mc L_2-\til{\mc L}_1$ in the estimate of $\int_0^1 h^{\p}(t\mc L_2+(1-t)\til{\mc L}_1)dt(\mc L_2-\til{\mc L}_1)$ in the proof of Claim and estimate the resultant terms in $C^\beta(\ol{\n_{f_2}^+})$ without changing
        $V_2-\til V_1$ and $\til k_1(\vartheta)-k_2(\vartheta)$ into $2\pi r \sin\theta(\theta-\theta_1)\int_0^1\pt_{\theta}(\Psi_2-\til \Psi_1)(r,t\theta+(1-t)\theta_1)dt$ and 
        $2\pi f_2(\vtheta)\sin\vtheta(\vartheta-\theta_1)\int_0^1\pt_{\theta}((\til \Psi_1-\Psi_2)(f_2(\theta),\theta))|_{\theta=t\vartheta+(1-t)\theta_1}dt$,
        then we obtain
        \begin{align}
         \label{L2mL1}
         ||\mc L_2-\til {\mc L}_1||_{0,\beta,\n_{f_2}^+}\le C ||(\Psi_2-\til \Psi_1)\be_{\vphi}||_{0,\beta,\n_{f_2}^+}.
        \end{align}

        With the above Claim, we estimate $(a)$ in $C^\beta(\ol{\n_{f_2}^+})$. Write $(a)$ as
    	\begin{multline}\label{(a)exp}
    	\int_0^1 \left(g^\p \left(\left(\frac{\bu_-}{c_-}\cdot (t{\bm\nu}_{f_2}(\mc L_2)+(1-t){\bm\nu}_{f_1}(\til{\mc L}_1))\right)^2\right)\right.\\
    	\left.2\left( \frac{\bu_-}{c_-}\cdot (t{\bm\nu}_{f_2}(\mc L_2)+(1-t){\bm\nu}_{f_1}(\til{\mc L}_1))\right)\frac{\bu_-}{c_-}S_-\right)
    	(f_2(\mc L_2),\mc L_2)dt
    	\cdot ({\bm\nu}_{f_2}(\mc L_2)-{\bm\nu}_{f_1}(\til{\mc L}_1)).
    	\end{multline}
    	By ${\bm\nu}_{f_i}=\frac{\be_r-\frac{f_s^\p}{f_i(0)+f_s}\be_\theta}{\sqrt{1+(\frac{f_s^\p}{f_i(0)+f_s})^2}}$ for $i=1,2$, ${\bm\nu}_{f_2}(\mc L_2)-{\bm\nu}_{f_1}(\til{\mc L}_1)$ can written as
    	\begin{align}\label{t2t1}
    	t_1\be_r+t_2\be_\theta
    	\end{align}
    	where 
    	\begin{multline*}
    	t_1=\int_0^1 \left.\grad_{(f(0),f_s,f_s^\p)}\left(\frac{1}{\sqrt{1+(\frac{f_s^\p}{f(0)+f_s})^2}}\right)\right|_{(f(0),f_s,\fsp)=\chi}dt\\\cdot(f_2(0)-f_1(0),f_s(\mc L_2)-f_s(\til{\mc L}_1),f_s^\p(\mc L_2)-f_s^\p(\til{\mc L}_1))
    	\end{multline*}
    	and
    	\begin{multline*}
    	t_2=\int_0^1 \left.\grad_{(f(0),f_s,f_s^\p)}
    	\left(\frac{-\frac{f_s^\p}{f(0)+f_s}}{\sqrt{1+(\frac{f_s^\p}{f(0)+f_s})^2}}\right)\right|_{(f(0),f_s,\fsp)=\chi}dt\\\cdot(f_2(0)-f_1(0),f_s(\mc L_2)-f_s(\til{\mc L}_1),f_s^\p(\mc L_2)-f_s^\p(\til{\mc L}_1))
    	\end{multline*}
    	with $\chi=(tf_2(0)+(1-t)f_1(0),t f_s(\mc L_2)+(1-t)f_s(\til{\mc L}_1),tf_s^\p(\mc L_2)+(1-t)f_s^\p(\til{\mc L}_1))$. 
    	To estimate \eqref{(a)exp} in $C^\beta(\ol{\n_{f_2}^+})$, we estimate $f_s(\mc L_2)-f_s(\til {\mc L}_1)$ and $f_s^\p(\mc L_2)-f_s^\p(\til{\mc L}_1)$ in $C^\beta(\ol{\n_{f_2}^+})$, respectively. 
        
        Express $f_s(\mc L_2)-f_s(\til {\mc L}_1)$ and $f_s^\p(\mc L_2)-f_s^\p(\til{\mc L}_1)$
    	as 
    	\begin{align}
    	\label{fffm}\int_0^1 f_s^\p(t \mc L_2+(1-t)\mc L_1)dt (\mc L_2-\til {\mc L}_1)
    	\end{align}
    	and
    	\begin{align}\label{fffp}
    	\int_0^1 f_s^{\p\p}(t\mc L_2+(1-t)\til{\mc L}_1)dt(\mc L_2-\til{\mc L}_1),
    	\end{align} 
    	respectively.
    	Since $\fsp \in C^\alpha([0,\theta_1])$, \eqref{fffm} can be estimated in $C^\beta(\ol{\n_{f_2}^+})$ directly. 
        With \eqref{fsestimate}, \eqref{L2L1} and 
        \eqref{L2mL1}, 
        we estimate \eqref{fffm} in $C^\beta(\ol{\n_{f_2}^+})$ directly. Then we obtain
    	\begin{align}\label{fff}
    	||f_s(\mc L_2)-f_s(\til {\mc L}_1)||_{0,\beta,\n_{f_2}^+}\le C \sigma ||(\Psi_2-\til \Psi_1)\be_{\vphi}||_{0,\beta,\n_{f_2}^+}.
    	\end{align}
    	Since $f_s^{\p\p}\in C^{\alpha}_{(1-\alpha,\{\theta=\theta_1\})}((0,\theta_1))$, we cannot estimate \eqref{fffp} in $C^\beta(\ol{\n_{f_2}^+})$  directly. 
    	With Claim and \eqref{fsestimate}, 
    	we estimate \eqref{fffp} in $C^\beta(\ol{\n_{f_2}^+})$. Then we obtain
        \begin{align}\label{ffff}
        ||f_s^\p(\mc L_2)-f_s^\p(\til{\mc L}_1)||_{0,\beta,\n_{f_2}^+}\le C\sigma||(\Psi_2-\til\Psi_1)\be_{\vphi}||_{1,\beta,\n_{f_2}^+}.
        \end{align}
        
        With \eqref{supesti}, \eqref{fsestimate}, \eqref{L2L1}, the expression of ${\bm\nu}_{f_2}(\mc L_2)-{\bm\nu}_{f_1}(\til{\mc L}_1)$ given in \eqref{t2t1}, \eqref{fff} and \eqref{ffff}, we estimate \eqref{(a)exp} in $C^{\beta}(\ol{\n_{f_2}^+})$. Then we have
        \begin{align}\label{(a)}
        ||(a)||_{0,\beta,\n_{f_2}^+}\le C\sigma|f_2(0)-f_1(0)|+C\sigma ||(\Psi_2-\til\Psi_1)\be_{\vphi}||_{1,\beta,\n_{f_2}^+}.
        \end{align}

    	Next, we estimate $(b)$ in $C^\beta(\ol{\n_{f_2}^+})$.
    	Divide $(b)$ into two parts:
    	\begin{align*}
    	&\left(g\left(\left(\frac{\bu_-\cdot{\bm\nu}_{f_1}(\til{\mc L}_1)}{c_-}\right)^2\right)S_- \right)(f_2(\mc L_2),{\mc L}_2)-\left(g\left(\left(\frac{\bu_-\cdot{\bm\nu}_{f_1}(\til{\mc L}_1)}{c_-}\right)^2\right)S_- \right)(f_2(\til{\mc L}_1),\til{\mc L}_1)\\&=:(b)_1
    	\end{align*}
    	and
    	\begin{align*}
    	&\left(g\left(\left(\frac{\bu_-\cdot{\bm\nu}_{f_1}(\til{\mc L}_1)}{c_-}\right)^2\right)S_- \right)(f_2(\til{\mc L}_1),\til{\mc L}_1)-\left(g\left(\left(\frac{\bu_-\cdot{\bm \nu}_{f_1}(\til{\mc L}_1)}{c_-}\right)^2\right)S_- \right)(f_1(\til {\mc L}_1),\til {\mc L}_1)\\
    	&=:(b)_2.
    	\end{align*}
    	Write $(b)_1$ and $(b)_2$ as
    	\begin{align*}
    	\int_0^1\pt_\theta\left. \left(g\left(\left(\frac{\bu_-\cdot{\bm\nu}_{f_1}(\til{\mc L}_1)}{c_-}\right)^2\right)S_- \right)(f_2(\theta),\theta)\right|_{\theta=t\mc L_2+(1-t)\til{\mc L}_1}dt (\mc L_2-\til {\mc L}_1)
    	\end{align*}
    	and
    	\begin{align*}
    	\int_0^1\pt_r \left.\left(g\left(\left(\frac{\bu_-\cdot{\bm\nu}_{f_1}(\til{\mc L}_1)}{c_-}\right)^2\right)S_- \right)(r,\til{\mc L}_1)\right|_{r=tf_2(\til{\mc L}_1)+(1-t)f_1(\til{\mc L}_1)}dt (f_2 (0)-f_1(0)).
    	\end{align*}
    	Integrands in both the expressions 
    	are in $C^\beta(\ol{\n_{f_2}^+})$. Thus, $(b)_1$ and $(b)_2$ can be estimated in $C^\beta(\ol{\n_{f_2}^+})$ directly. With \eqref{supesti}, \eqref{fsestimate}, \eqref{L2L1} and \eqref{L2mL1}, we estiamte $(b)_1$ and $(b)_2$ in $C^\beta(\ol{\n_{f_2}^+})$, respectively. Then we obtain
    	\begin{align}\label{(b)1}
    	||(b)_1||_{0,\beta,\n_{f_2}^+}\le C\sigma||(\Psi_2-\til \Psi_1)\be_{\vphi}||_{0,\beta,\n_{f_2}^+} 
    	\end{align} 
    	and
    	\begin{align}\label{(b)2}
    	||(b)_2||_{0,\beta,\n_{f_2}^+}\le C|f_2(0)-f_1(0)|.
    	\end{align}
    	Combining \eqref{(b)1} and \eqref{(b)2}, we have
    	\begin{align}\label{(b)}
    	||(b)||_{0,\beta,\n_{f_2}^+}\le C|f_2(0)-f_1(0)|+C\sigma ||(\Psi_2-\til \Psi_1)\be_{\vphi}||_{0,\beta,\n_{f_2}^+}.
    	\end{align}
    	
    	Finally, combine \eqref{(a)} and \eqref{(b)}. Then we have
    	\begin{align}\label{T2mtilT1}
    	||T_2-\til T_1||_{0,\beta,\n_{f_2}^+}\le C|f_2(0)-f_1(0)|+C\sigma ||(\Psi_2-\til \Psi_1)\be_{\vphi}||_{1,\beta,\n_{f_2}^+}.
    	\end{align}
    	
    	3. Estimate $||(\frac{A_2}{2\pi r\sin\theta}-\frac{\til A_1}{2\pi r \sin\theta})\be_{\vphi}||_{1,0,\n_{f_2}^+}$. 
    	
    	Since $A_i$ for $i=1,2$ are solutions of 
    	\begin{align*}
    	\grad \times((\Phi_0^++\Psi_i)\be_{\vphi})\cdot \grad A=0\qdin \n_{f_i(0)+f_s}^+,\qd A=A_{en,f_i(0)+f_s}\qdon \Gam_{f_i(0)+f_s}
    	\end{align*}
    	for $i=1,2$, respectively, where $A_{en,f_i}$ is $A_{en,f(0)+f_s}$ given in \eqref{Aentdef} for $f(0)=f_i(0)$, 
    	by Lemma \ref{lemTrans}, 
    	\begin{align*}
    	A_i=A_{en,f_i}(\mc L_i)
    	\end{align*}
    	for $i=1,2$. 
    	With these solution expressions, express $(\frac{A_2}{2\pi r\sin\theta}-\frac{\til A_1}{2\pi r \sin\theta})\be_{\vphi}$ as
    	\begin{align}\label{A2tilA1}
    	\frac{f_2(\mc L_2)\sin(\mc L_2)u_{-,\vphi}(f_2(\mc L_2),\mc L_2)-f_1(\til{\mc L}_1)\sin(\til{\mc L}_1)u_{-,\vphi}(f_1(\til{\mc L}_1),\til{\mc L}_1)}{r\sin\theta}\be_{\vphi}
    	\end{align}
    	where $u_{-,\vphi}=\bu_-\cdot \be_{\vphi}$. 
    	With \eqref{supesti}, \eqref{fsestimate}, $||\frac{\mc L_2}{\theta}||_{0,\beta,\n_{f_2}^+}\le C$ and $||\frac{\til{\mc L}_1}{\theta}||_{0,\beta,\n_{f_2}^+}\le C$ obtained using arguments similar to the ones in the proof of Claim in Lemma \ref{lemTrans},
    	$||\mc L_2||_{1,\beta,(f_2(\theta),r_1)\times(0,\theta_1)}\le C$, $||\til{ \mc L}_1||_{1,\beta,(f_2(\theta),r_1)\times(0,\theta_1)}\le C$ and Claim, we estimate \eqref{A2tilA1} in $C^\beta(\ol{\n_{f_2}^+})$. Then we obtain
    	\begin{multline}\label{A2mtilA1}
    	||(\frac{A_2}{2\pi r\sin\theta}-\frac{\til A_1}{2\pi r \sin\theta})\be_{\vphi}||_{1,0,\n_{f_2}^+}\le C\sigma|f_2(0)-f_1(0)|+C\sigma||(\Psi_2-\til\Psi_1)\be_{\vphi}||_{1,\beta,\n_{f_2}^+}.
    	\end{multline}	
    	
    	4. Estimate $|f_2(0)-f_1(0)|$.
    	
    	
    	Substitute $T_2-\til T_1=(a)+(b)_1+(b)_2$ 
    	into \eqref{uniquecpt}. Then we have
    	\begin{align}\label{substiuniquecpt}
    	&\frac{1}{r_1\sin\theta_1}\int_0^{\theta_1}\left.\frac{\rho_0^+((\gam-1){u_0^+}^2+{c_0^+}^2)}{\gam(\gam-1)u_0^+S_0^+}(b)_2 
    	\right|_{r=r_1}r_1^2\sin\xi d\xi\\
    	&\qd\qd\nonumber=\frac{1}{r_1\sin\theta_1}\int_0^{\theta_1}\left(\mf f_0(T_2,p_{ex})-\mf f_0(\til T_1,p_{ex})\right.
    	-\frac{\rho_0^+((\gam-1){u_0^+}^2+{c_0^+}^2)}{\gam(\gam-1)u_0^+S_0^+}
    	\left((a)+(b)_1\right)\\
    	&\qd\qd\qd\nonumber\left.\left.+\mf f_1(\Psi_2\be_{\vphi},A_2,T_2) -\mf f_1(\til \Psi_1\be_{\vphi},\til A_1,\til T_1)\right)\right|_{r=r_1}r_1^2\sin\xi d\xi.
    	\end{align}
    	Using \eqref{Pseudoestimate} satisfied by $(f_i(0), \Phi_i\be_{\vphi},L_i,S_i)$ for $i=1,2$, we estimate $\mf f_1(\Psi_2\be_{\vphi},A_2,T_2) -\mf f_1(\til \Psi_1\be_{\vphi},\til A_1,\til T_1)$ in $C^\beta(\ol{\Gam_{ex}})$. Then we obtain
    	\begin{multline*}
    	||\mf f_1(\Psi_2\be_{\vphi},A_2,T_2) -\mf f_1(\til \Psi_1\be_{\vphi},\til A_1,\til T_1)||_{0,\beta,\Gam_{ex}}\\\le C\sigma (||(\Psi_2-\til\Psi_1)\be_{\vphi}||_{1,\beta,\Gam_{ex}}+||(\frac{A_2}{2\pi r\sin\theta}-\frac{\til A_1}{2\pi r \sin\theta})\be_{\vphi}||_{0,\beta,\Gam_{ex}}+||T_2-\til T_1||_{0,\beta,\Gam_{ex}}).
    	\end{multline*}
    	With this estimate, \eqref{supesti}, \eqref{pexesti}, \eqref{fsestimate} and \eqref{Pseudoestimate} satisfied by $(f_i(0), \Phi_i\be_{\vphi},L_i,S_i)$ for $i=1,2$, \eqref{(a)}, \eqref{(b)1}, \eqref{T2mtilT1}, \eqref{A2mtilA1} and the fact that $(g({M_0^-}^2))^\p S_{in}$ is strictly positive in $[r_0,r_1]$ (see Lemma \ref{Slem}), we estimate $f_2(0)-f_1(0)$ in \eqref{substiuniquecpt}. Then we get 
    	\begin{align*}
    	|f_2(0)-f_1(0)|\le C_1\sigma|f_2(0)-f_1(0)|+ C\sigma ||(\Psi_2-\til\Psi_1)\be_{\vphi}||_{1,\beta,\n_{f_2}^+}
    	\end{align*}
    	where $C_1$ is a positive constant depending on the data. 
        Take $\sigma_3=\min (\frac{1}{2C_1},\sigma_3^{(1)})(=:\sigma_3^{(2)})$.
        Then we have
        \begin{align}\label{f20mf10}
        |f_2(0)-f_1(0)|\le  C\sigma ||(\Psi_2-\til\Psi_1)\be_{\vphi}||_{1,\beta,\n_{f_2}^+}.
        \end{align}

    	5. With the fact that $\Psi_2\be_{\vphi}$ and $\til \Psi_1\be_{\vphi}$ are in $C^{2,\alpha}_{(-1-\alpha,\Gam_{w,f_2}^+)}(\n_{f_2}^+)$, transform \eqref{uniqueeqn} into an elliptic system form (see \eqref{Psielliptic}). 
    	And then transform the 
    	resultant equation with \eqref{uniquebc} into the following ${\bm 0}$ boundary value problem
    	\begin{align}\label{uniqueeq1}
    	&\Div \left({\bm A}D((\Psi_2-\til \Psi_1)\be_{\vphi}-\ol{\bm h})\right)-d((\Psi_2-\til \Psi_1)\be_{\vphi}-\ol{\bm h})\\
    	&=\underbrace{-\frac{{\rho_0^+}^{\gam-1}}{(\gam-1)u_0^+}(1+\frac{\gam {u_0^+}^2}{{c_0^+}^2-{u_0^+}^2})\frac{\pt_\theta (T_2-\til T_1)}{r}\be_{\vphi}-{\bm F}_3}_{\ol{\bm F}}
    	-\Div \left({\bm A}D\ol{\bm h}\right)+d\ol{\bm h}
    	\qdin \n_{f_2(0)+f_s}^+,\nonumber\\
    	\label{uniquebc1}&(\Psi_2-\til \Psi_1)\be_{\vphi}-\ol{\bm h}=
    	{\bm 0}\qd\tx{on} \qd \pt \n_{f_2(0)+f_s}^+
    	\end{align}
    	where ${\bm A}$ and $d$ are a $(2,2)$ tensor and scalar function defined below \eqref{vbdry} and
    	$$\ol{\bm h}:=\frac{(r-f(\theta))\frac{r_1}{r} \ol{{\bm h}}_2+(r_1-r)\frac{f(\theta)}{r} \ol{{\bm h}}_1}{r_1-f(\theta)}.$$ 
    	Here, $\ol{\bm F}$ is of the form 
    	\begin{align*}
    	&\sum_i (A^i_2-A^i_1)\pt_r B_2^i\be_{\vphi}+\sum_i A_1^i\pt_r (B_2^i-B_1^i)\be_{\vphi}\\
    	&+\sum_i C_1^i \pt_{\theta}(D_2^i-D_1^i)\be_{\vphi}+(E_2-E_1)\frac{\pt_{\theta}(F_2\sin\theta)}{\sin\theta}\be_{\vphi}+E_1\frac{\pt_{\theta}((F_2-F_1)\sin\theta)}{\sin\theta}\be_{\vphi}
    	\end{align*}
    	where
    	\begin{align*}
    	&A^i_j\in C^{1,\alpha}_{(-\alpha,\Gam_w^+)}(\n_f^+),\; B^i_j\be_\theta\in C^{1,\alpha}_{(-\alpha,\Gam_w^+)}(\n_f^+), \;
    	C^i_j\in C^{1,\alpha}_{(-\alpha,\Gam_w^+)}(\n_f^+),
    	\\& D^i_j\in C^{1,\alpha}_{(-\alpha,\Gam_w^+)}(\n_f^+),\; E_j\be_{\vphi}\in C^{2,\alpha}_{(-1-\alpha,\Gam_w^+)}(\n_f^+)
    	\;\tx{and}\; F_j\be_{\vphi}\in C^{2,\alpha}_{(-1-\alpha,\Gam_w^+)}
    	(\n_f^+)\nonumber
    	\end{align*}
        for $j=1,2$.
    	With arguments similar to the ones in the proof of Lemma \ref{lem1alpha} using \eqref{supesti}, \eqref{pexesti}, \eqref{fsestimate}, \eqref{Pseudoestimate} satisfied by $(f_i(0), \Phi_i\be_{\vphi},L_i,S_i)$ for $i=1,2$, \eqref{f20mf10} and
    	\begin{align}\label{T2mtilT1F}
    	||T_2-\til T_1||_{0,\beta,\n_{f_2}^+}\le C\sigma||(\Psi_2-\til\Psi_1)\be_\vphi||_{1,\beta,\n_{f_2}^+}
    	\end{align}
    	and
    	\begin{align}
    	\label{A2mtilA1F}||(\frac{A_2}{2\pi r\sin\theta}-\frac{\til A_1}{2\pi r \sin\theta})\be_\vphi||_{1,0,\n_{f_2}^+}\le C\sigma||(\Psi_2-\til\Psi_1)\be_\vphi||_{1,\beta,\n_{f_2^+}}
    	\end{align} 
    	obtained from \eqref{T2mtilT1} and \eqref{A2mtilA1} using \eqref{f20mf10}, we estimate $(\Psi_2-\til \Psi_1)\be_\vphi$
    	in \eqref{uniqueeq1}, \eqref{uniquebc1} in $C^{1,\beta}(\ol{\n_{f_2}^+})$ (in this argument, we only change $\int_{\n_{f_2}^+}\sum_i A_1^i\pt_r (B_2^i-B_1^i)\be_{\vphi}{\bm \xi}$, 
    	$\int_{\n_{f_2}^+} \sum_iC_1^i\pt_{\theta}(D_2^i-D_1^i)\be_\vphi{\bm \xi}$ 
    	and $\int_{\n_{f_2}^+}\Div ({\bm A}D{\ol{\bm h}}){\bm \xi}$ 
    	using integration by parts). Then we obtain  	
    	\begin{align*}
    	||(\Psi_2-\til \Psi_1)\be_\vphi||_{1,\beta,\n_{f_2}^+}\le C\sigma ||(\Psi_2-\til \Psi_1)\be_\vphi||_{1,\beta,\n_{f_2}^+}.
    	\end{align*}
    	Take $\sigma_3=\min (\frac{1}{2C},\sigma_3^{(2)})(=:\ul\sigma_3)
    	$. Then we have $\Psi_2\be_\vphi=\til \Psi_1\be_\vphi$. 
    	Using this fact, we obtain from \eqref{f20mf10}, \eqref{T2mtilT1F} and \eqref{A2mtilA1F} $f_2(0)=f_1(0)$, $T_2=\til T_1$ and $A_2=\til A_1$. One can see that $\ul\sigma_3$ depends on the data. This finishes the proof.
    \end{proof}
    
    \section{Determination of a shape of a shock location}\label{secdetermineshock}
    In the previous section, for given $(\rho_-,\bu_-,p_-,p_{ex},\fsp)$ in a small perturbation of $(\rho_0^-,u_0^-\be_r,p_0^-$ $,p_c,0)$, we found  $(f(0),\Phi\be_\vphi,L,S)$ satisfying all the conditions in Problem \ref{pb2} except \eqref{rhtau}. 
    In this section, to finish the proof of Theorem \ref{mainthm}, 
    for given $(\rho_-,\bu_-,p_-,p_{ex})$ in a small perturbation of $(\rho_0^-,u_0^-\be_r,p_0^-,p_c)$ as in the previous section or in a much small perturbation of $(\rho_0^-,u_0^-\be_r,p_0^-,p_c)$ if necessary, we find 
    $\fsp$ in a small perturbation of $0$ as in the previous section such that a solution of Problem \ref{pb3} 
    for given $(\rho_-,\bu_-,p_-,p_{ex},\fsp)=(\rho_0^-,u_0^-\be_r,p_0^-,p_c,\fsp)$ 
    satisfies \eqref{rhtau}. 
    
    \subsection{Proof of Theorem \ref{mainthm} (Existence)}\label{subshockexist}
    For a constant $\sigma>0$, we define
    \begin{align*}
    &\mc B^{(1)}_{\sigma}:=\{(\rho_-,\bu_-,p_-)\in  (C^{2,\alpha}
    (\ol\n))^3\;|\;\\
    &\qd\qd\qd\qd||\rho_- -\rho_0^-||_{2,\alpha,\n}
    +||\bu_- -u_0^-\be_r||_{2,\alpha,\n}+||p_- -p_0^-||_{2,\alpha,\n}\le \sigma\} \\
    &\mc B^{(2)}_{\sigma}:=\{p_{ex}\in C^{1,\alpha}_{(-\alpha,\pt\Gam_{ex})}(\Gam_{ex})\;|\;	||p_{ex}-p_c||_{1,\alpha,\Gam_{ex}}^{(-\alpha,\pt\Gam_{ex})}\le \sigma \},\\
    &\mc B^{(3)}_{\sigma}:=\{h\in C^{1,\alpha}_{(-\alpha,\{\theta=\theta_1\}),0}((0,\theta_1))\;|\; ||h||_{1,\alpha,(0,\theta_1)}^{(-\alpha,\{\theta=\theta_1\})}\le\sigma \},\\
    &\mc B^{(4)}_{\sigma}:=\mc B^{(1)}_{\sigma}\times\mc B^{(2)}_{\sigma}\times \mc B^{(3)}_{\sigma}.
    \end{align*}
    For given $(\rho_-,\bu_-,p_-,p_{ex},\fsp)\in \mc B_{\sigma}^{(4)}$ for $\sigma\le \sigma_3$, 
    let $(f(0),\Phi\be_{\vphi},L,S)$ be a solution of Problem \ref{pb3} satisfying \eqref{Pseudoestimate} given in Proposition \ref{proPseudofree}. 
    We define 
    \begin{multline}\label{mcA}
    \mc A(\rho_-,\bu_-,p_-,p_{ex},\fsp)\\
    :=\left.\left(\frac{1}{\varrho(\grad\times (\Phi\be_\vphi),\frac{L}{2\pi r\sin\theta}\be_\vphi,S)}\grad\times (\Phi\be_\vphi)\cdot {\bm \tau}_{f(0)+f_s}-
    \bu_-\cdot {\bm \tau}_{f(0)+f_s}\right)\right|_{\Gam_{f(0)+f_s}}\\
    \circ \Pi_{r_s f(0)+f_s}
    \end{multline}
    where ${\bm \tau}_{f(0)+f_s}$ is the unit tangent vector on $\Gam_{f(0)+f_s}$ perpendicular to $\be_\vphi$ and satisfying $({\bm \tau}_{f(0)+f_s}\times \be_\vphi)\cdot {\bm \nu}_{f(0)+f_s}>0$ and ${\bm \nu}_{f(0)+f_s}$ is the unit normal vector field on $\Gam_{f(0)+f_s}$. Then $\mc A$ satisfies $\mc A(\rho_0^-,u_0^-\be_r,p_0^-,p_c,0)=0$ and 
    $\mc A$ is a map from $\mc B^{(4)}_{\sigma}$ to $\mc B^{(3)}_{C\sigma}$ where $C$ is a positive constant depending on the data. 
    If for given $(\rho_-,\bu_-,p_-,p_{ex})\in \mc B^{(1)}_\sigma\times \mc B^{(2)}_\sigma$ for $\sigma\le\sigma_3$, we find $\fsp\in \mc B^{(3)}_{C\sigma}$ for $C\sigma\le \sigma_3$ such that $\mc A(\rho_-,\bu_-,p_-,p_{ex},\fsp)=0$, then $(f(0)+f_s,\Phi\be_{\vphi},L,S)$ satisfies all the conditions in Problem \ref{pb2} and thus the existence part of Theorem \ref{mainthm} is proved. 
    We will find such $\fsp$ 
    using the weak implicit function theorem introduced in \cite{MR2824466}.
    To apply the weak implicit function theorem, we need to prove that $\mc A$ is continuous, 
    $\mc A$ is Fr\'echet differentiable at $(\rho_0^-,u_0^-\be_r,p_0^-,p_c,0)(=:\zeta_0)$ and the partial Fr\'echet derivative of $\mc A$ with respect to $\fsp$ at $\zeta_0$ is invertible. We will prove these in the following lemmas. 
    
    We first prove that $\mc A$ is continuous. 
    \begin{lem}\label{lemconti}
    	$\mc A$ is continuous in $\mc B^{(4)}_{\ul{\sigma}_3}$  for a positive constant $\ul{\sigma}_3\le \sigma_3$  in the sense that if $\zeta^{(k)}:=(\rho_-^{(k)},\bu_-^{(k)},p_-^{(k)},p_{ex}^{(k)},(\fsp)^{(k)})\in \mc B^{(4)}_{\ul{\sigma}_3}$ 
    	converges to $\zeta^{(\infty)}:=(\rho_-^{(\infty)},\bu_-^{(\infty)},p_-^{(\infty)},p_{ex}^{(\infty)},(\fsp)^{(\infty)})\in \mc B^{(4)}_{\ul{\sigma}_3}$ in $(C^{2,\frac{\alpha}{2}}(\ol{\n}))^3\times C^{1,\frac{\alpha}{2}}_{(-\frac{\alpha}{2},\pt \Gam_{ex})}(\Gam_{ex})\times C^{1,\frac{\alpha}{2}}_{(-\frac{\alpha}{2},\{\theta=\theta_1\}),0}((0,\theta_1))$, then $\mc A(\zeta^{(k)})$ converges to $\mc A(\zeta^{(\infty)})$ in $C^{1,\frac{\alpha}{2}}_{(-\frac{\alpha}{2},\{\theta=\theta_1\}),0}((0,\theta_1))$.
    \end{lem}
    \begin{proof}
    	The result is obtained by using the standard argument. 
    	
    	Let $\ul{\sigma}_3$ be a positive constant $\le\sigma_3$ and to be determined later. 
    	Let $\zeta^{(k)}$ for $k=1,2,\ldots$ be a sequence in $\mc B^{(4)}_{\ul{\sigma}_3}$ that converges to $\zeta^{(\infty)}\in \mc B^{(4)}_{\ul{\sigma}_3}$ in $(C^{2,\frac{\alpha}{2}}(\ol{\n}))^3\times C^{1,\frac{\alpha}{2}}_{(-\frac{\alpha}{2},\pt \Gam_{ex})}(\Gam_{ex})\times C^{1,\frac{\alpha}{2}}_{(-\frac{\alpha}{2},\{\theta=\theta_1\}),0}((0,\theta_1))$.  
    	By Proposition \ref{proPseudofree}, for each given $(\rho_-,\bu_-,p_-,p_{ex},\fsp)=\zeta^{(k)}$ for $k=1,2,\dots$ and $\infty$,  there exists a unique $\mc U_k:=(f(0)^{(k)}, \Phi^{(k)}\be_{\vphi},L^{(k)},S^{(k)})$ satisfying (A), (B) and the estimate \eqref{Pseudoestimate} with $\sigma$ replaced by $\ul{\sigma}_3$.  
    	Let $(\til\Phi^{(k)},\til L^{(k)},\til S^{(k)}):=(\Phi^{(k)},L^{(k)},S^{(k)})(\Pi_{r_sf(0)^{(k)}+f_s^{(k)}})$ for $k=1,2,\ldots$ and $\infty$ and
    	$\til {\mc U}_k:=(f(0)^{(k)},\til\Phi^{(k)}\be_{\vphi},\til L^{(k)},\til S^{(k)})$ for $k=1,2,\ldots$ and $\infty$. To prove that $\mc A(\zeta^{(k)})$ converges to $\mc A(\zeta^{(\infty)})$ in $C^{1,\frac{\alpha}{2}}_{(-\frac{\alpha}{2},\{\theta=\theta_1\}),0}((0,\theta_1))$, 
    	we show that 
    	$\til {\mc U}_k$ converges to $\til {\mc U}_{\infty}$ in $\R\times C^{2,\frac {\alpha}{2}}_{(-1-\frac{\alpha}{2},\Gam_{w,r_s}^+)}(\n_{r_s}^+)\times (C^{1,\frac{\alpha}{2}}_{(-\frac{\alpha}{2},\Gam_{w,r_s}^+)}(\n_{r_s}^+))^2$. 
    		
    	To do so, we show that if a subsequence of $\til {\mc U}_k$ converges in $\R\times C^{2,\frac {\alpha}{2}}_{(-1-\frac{\alpha}{2},\Gam_{w,r_s}^+)}(\n_{r_s}^+)\times (C^{1,\frac{\alpha}{2}}_{(-\frac{\alpha}{2},\Gam_{w,r_s}^+)}(\n_{r_s}^+))^2$, then it converges to $\til {\mc U}_{\infty}$. 
    	Assume that a subsequence $\til {\mc U}_{k_j}$ of $\til {\mc U}_k$ 
    	converges to a function
    	$\til {\mc U}^*=(f(0)^*,\til\Phi^*\be_{\vphi},\til L^*,\til S^*)$ in $\R\times C^{2,\frac {\alpha}{2}}_{(-1-\frac{\alpha}{2},\Gam_{w,r_s}^+)}(\n_{r_s}^+)\times (C^{1,\frac{\alpha}{2}}_{(-\frac{\alpha}{2},\Gam_{w,r_s}^+)}(\n_{r_s}^+))^2$. 
    	Let (A)$^{(k)}$ and (B)$^{(k)}$ be (A) and (B) satisfied by $\mc U_k$, respectively,  and let $\tilde{\tx{(A)}}^{(k)}$ and $\tilde{\tx{(B)}}^{(k)}$ be (A)$^{(k)}$ and (B)$^{(k)}$ transformed by using $\Pi_{r_sf(0)^{(k)}+f_s^{(k)}}$, respectively. Since
    	${\mc U}_{k_j}$ satisfies (A)$^{(k_j)}$, (B)$^{(k_j)}$, 
    	$\til {\mc U}_{k_j}$ satisfies  $\tilde{\tx{(A)}}^{(k_j)}$, $\tilde{\tx{(B)}}^{(k_j)}$. 
    	Take $j\ra \infty$ to $\tilde{\tx{(A)}}^{(k_j)}$, $\tilde{\tx{(B)}}^{(k_j)}$. Then since $\zeta^{(k)}\ra \zeta^{(\infty)}$ in $(C^{2,\frac{\alpha}{2}}(\ol{\n}))^3\times C^{1,\frac{\alpha}{2}}_{(-\frac{\alpha}{2},\pt \Gam_{ex})}(\Gam_{ex})\times C^{1,\frac{\alpha}{2}}_{(-\frac{\alpha}{2},\{\theta=\theta_1\}),0}((0,\theta_1))$ and $\til {\mc U}_{k_j}\ra\til {\mc U}^*$ in  $\R\times C^{2,\frac {\alpha}{2}}_{(-1-\frac{\alpha}{2},\Gam_{w,r_s}^+)}(\n_{r_s}^+)\times (C^{1,\frac{\alpha}{2}}_{(-\frac{\alpha}{2},\Gam_{w,r_s}^+)}(\n_{r_s}^+))^2$, 
    	we obtain  1) $\tilde{\tx{(A)}}^{(\infty)}$, $\tilde{\tx{(B)}}^{(\infty)}$ with $\til {\mc U}_{\infty}$ replaced by $\til {\mc U}^*$.
    	By the facts that ${\mc U}_{k}$ for $k=1,2,\ldots$ satisfy \eqref{Pseudoestimate} with $\sigma$ replaced by $\ul{\sigma}_3$ 
    	and 
    	$(\fsp)^{(k)}$ for $k=1,2,\ldots$ are in $\mc B^{(3)}_{\ul{\sigma}_3}$,
    	\begin{multline*}
    	|f(0)^{(k)}-r_s|+||\grad\times((\til \Phi^{(k)}-\til \Phi_0^{+,(k)})\be_{\vphi})||_{1,\alpha,\n_{r_s}^+}^{(-\alpha,\Gam_{w,r_s}^+)}\\+||\frac{\til L^{(k)}}{2\pi  \Pi_{r_sf(0)^{(k)}+f_s^{(k)}}^r 
    		\sin\theta}\be_{\vphi}||_{1,\alpha,\n_{r_s}^+}^{(-\alpha,\Gam_{w,r_s}^+)}+||\til S^{(k)}-\til S_0^{+,(k)}||_{1,\alpha,\n_{r_s}^+}^{(-\alpha,\Gam_{w,r_s}^+)}\le C\ul{\sigma}_3
    	\end{multline*}
    	where $(\til \Phi_0^{+,(k)},\til S_0^{+,(k)}):=(\Phi_0^+,S_0^+)(\Pi_{r_s f(0)^{(k)}+f_s^{(k)}})$ for $k=1,2,\ldots$  and $C$ is a positive constant depending on the data and independent of $k$. 
    	Using the facts that $(\fsp)^{(k)}\ra (\fsp)^{(\infty)}$ in  $C^{1,\frac{\alpha}{2}}_{(-\frac{\alpha}{2},\{\theta=\theta_1\}),0}((0,\theta_1))$ and $\til {\mc U}_{k_j}\ra\til {\mc U}^*$ in $\R\times C^{2,\frac {\alpha}{2}}_{(-1-\frac{\alpha}{2},\Gam_{w,r_s}^+)}(\n_{r_s}^+)\times (C^{1,\frac{\alpha}{2}}_{(-\frac{\alpha}{2},\Gam_{w,r_s}^+)}(\n_{r_s}^+))^2$, 
    	%
    	we obtain from this inequality
    	\begin{multline*}
    	|f(0)^{*}-r_s|+||\grad\times((\til \Phi^{*}-\til \Phi_0^{+,(\infty)})\be_{\vphi})||_{1,\alpha,\n_{r_s}^+}^{(-\alpha,\Gam_{w,r_s}^+)}\\+||\frac{\til L^{*}}{2\pi  \Pi_{r_sf(0)^{(\infty)}+f_s^{(\infty)}}^r 
    		\sin\theta}\be_{\vphi}||_{1,\alpha,\n_{r_s}^+}^{(-\alpha,\Gam_{w,r_s}^+)}+||\til S^{*}-\til S_0^{+,(\infty)}||_{1,\alpha,\n_{r_s}^+}^{(-\alpha,\Gam_{w,r_s}^+)}\le C\ul{\sigma}_3
    	\end{multline*}
    	where $(\til \Phi_0^{+,(\infty)},\til S_0^{+,(\infty)}):=(\Phi_0^+,S_0^+)(\Pi_{r_s f(0)^{(\infty)}+f_s^{(\infty)}})$.
    	From this inequality, we have 
    	\begin{multline*}
    	|f(0)^*-r_s|+||\grad\times( (\Phi^*-\Phi_0^+)\be_{\vphi})||_{1,\alpha,\n_{f(0)^*+f_s^{(\infty)}}^+}^{(-\alpha,\Gam_{w,f(0)^*+f_s^{(\infty)}}^+)}\\+||\frac{L^*}{2\pi r 
    		\sin\theta}\be_{\vphi}||_{1,\alpha,\n_{f(0)^*+f_s^{(\infty)}}^+}^{(-\alpha,\Gam_{w,f(0)^*+f_s^{(\infty)}}^+)}+|| S^*-S_0^+||_{1,\alpha,\n_{f(0)^*+f_s^{(\infty)}}^+}^{(-\alpha,\Gam_{w,f(0)^*+f_s^{(\infty)}}^+)}\le C_1\ul{\sigma}_3
    	\end{multline*}
    	where $(\Phi^*,L^*,S^*):=(\til \Phi^*,\til L^*,\til S^*)(\Pi_{f(0)^*+f_s^{(\infty)} r_s})$ and $C_1$ is a positive constant depending on the data. 
    	Take $\ul{\sigma}_3=\min(\frac{C\sigma_3}{C_1},\sigma_3)$ where $C$ is $C$ in \eqref{Pseudoestimate}. 
    	Then 2) $(f(0)^*,\Phi^*\be_\vphi,L^*,S^*)$ satisfies \eqref{Pseudoestimate} with $\sigma$ replaced by $\sigma_3$. By 1) and 2),  $(f(0)^*,\Phi^*\be_\vphi,L^*,S^*)$ satisfies (A), (B) for $(\rho_-,\bu_-,p_-,\fsp)=\zeta^{(\infty)}$ and \eqref{Pseudoestimate} with $\sigma$ replaced by $\sigma_3$. 
    	By Proposition \ref{proPseudofree},  $(f(0),\Phi\be_{\vphi},L,S)$ satisfying (A), (B) for $(\rho_-,\bu_-,p_-,\fsp)=\zeta^{(\infty)}$ and \eqref{Pseudoestimate} with $\sigma$ replaced by $\sigma_3$ is unique. Therefore, $(f(0)^*,\Phi^*\be_\vphi,L^*,S^*)=(f(0)^{(\infty)},\Phi^{(\infty)}\be_\vphi,L^{(\infty)},S^{(\infty)})$. From this, we have that  $\til {\mc U}^*=\til {\mc U}_\infty$. 
    	
    	
    	Using the fact that we showed above, we prove that $\til {\mc U}_k$ converges to $\til {\mc U}_{\infty}$ in $\R\times C^{2,\frac {\alpha}{2}}_{(-1-\frac{\alpha}{2},\Gam_{w,r_s}^+)}(\n_{r_s}^+)\times (C^{1,\frac{\alpha}{2}}_{(-\frac{\alpha}{2},\Gam_{w,r_s}^+)}(\n_{r_s}^+))^2$. 
    	By $\R\times C^{2,\alpha}_{(-1-\alpha,\Gam_{w,r_s}^+)}(\n_{r_s}^+)\times (C^{1,\alpha}_{(-\alpha,\Gam_{w,r_s}^+)}(\n_{r_s}^+))^2\Subset\R\times C^{2,\frac {\alpha}{2}}_{(-1-\frac{\alpha}{2},\Gam_{w,r_s}^+)}(\n_{r_s}^+)\times (C^{1,\frac{\alpha}{2}}_{(-\frac{\alpha}{2},\Gam_{w,r_s}^+)}(\n_{r_s}^+))^2$, every subsequence of $\til {\mc U}_k$ has a convergent subsequence in $\R\times C^{2,\frac {\alpha}{2}}_{(-1-\frac{\alpha}{2},\Gam_{w,r_s}^+)}(\n_{r_s}^+)\times (C^{1,\frac{\alpha}{2}}_{(-\frac{\alpha}{2},\Gam_{w,r_s}^+)}(\n_{r_s}^+))^2$. 
    	By the fact that we showed above, this convergent subsequence must  converge to $\til{\mc U}_\infty$. Thus, we have  $\til {\mc U}_k\ra \til {\mc U}_\infty$ in $\R\times C^{2,\frac {\alpha}{2}}_{(-1-\frac{\alpha}{2},\Gam_{w,r_s}^+)}(\n_{r_s}^+)\times (C^{1,\frac{\alpha}{2}}_{(-\frac{\alpha}{2},\Gam_{w,r_s}^+)}(\n_{r_s}^+))^2$. 
    	
    	Using this fact, we can conclude that $\mc A(\zeta^{(k)})\ra \mc A(\zeta^{(\infty)})$ in $C^{1,\frac{\alpha}{2}}_{(-\frac{\alpha}{2},\{\theta=\theta_1\}),0}((0,\theta_1))$ as $\zeta^{(k)}\in \mc B^{(4)}_{\ul{\sigma}_3}$ converges to $\zeta^{(\infty)}\in \mc B^{(4)}_{\ul{\sigma}_3}$ in $(C^{2,\frac{\alpha}{2}}(\ol{\n}))^3\times C^{1,\frac{\alpha}{2}}_{(-\frac{\alpha}{2},\pt \Gam_{ex})}(\Gam_{ex})\times C^{1,\frac{\alpha}{2}}_{(-\frac{\alpha}{2},\{\theta=\theta_1\}),0}((0,\theta_1))$.
    	This finishes the proof. 
    \end{proof}
    Next, we prove that $\mc A$ is Fr\'echet differentiable. 
    \begin{lem}\label{lemFrechetdiff}
    	(i) The mapping $\mc A$ defined in \eqref{mcA} is Fr\'echet differentiable at $\zeta_0:=(\rho_0^-,\bu_0^-,p_0^-,p_c,0)$ as a map from $(C^{1,\alpha}_{(-\alpha,\Gam_w)}(\n))^3\times C^{1,\alpha}_{(-\alpha,\pt \Gam_{ex})}(\Gam_{ex})\times C^{1,\alpha}_{(-\alpha,\{\theta=\theta_1\}),0}((0,\theta_1))$ to $C^{1,\alpha}_{(-\alpha,\{\theta=\theta_1\}),0}((0,\theta_1))$.\\
    	(ii) The partial Fr\'echet derivative of $\mc A$ with respect to $\fsp$ at $\zeta_0$  is 
    	given by
    	\begin{align}\label{Dfsp}
    	D_{\fsp}\mc A(\zeta_0) \til \fsp =
    	\left.\frac{1}{\rho_0^+}\grad\times(\til \Psi^{(\til \fsp)}\be_\vphi)\cdot \be_\theta\right|_{r=r_s} 
    	-(u_0^- -u_0^+)(r_s)\frac{\til \fsp}{r_s}
    	\end{align}
    	for $\til \fsp\in C^{1,\alpha}_{(-\alpha,\{\theta=\theta_1\}),0}((0,\theta_1))$ where $\til \Psi^{(\til \fsp)}\be_\vphi$ is a solution of 
    	\eqref{Wtileqnstar}, \eqref{Wtileqnstarbc} for given $\til \fsp$. 
    \end{lem}
    \begin{proof}
    	In this proof, we prove that ${\mc A}$ is Fr\'echet differentiable as a function of $\fsp$ at $0$ with the other variables fixed at $(\rho_0^-,u_0^-\be_r,p_0^-,p_c)$. 
    	The Fr\'echet differentiability of ${\mc A}$ as a function of $(\rho_-,\bu_-,p_-,p_{ex},\fsp)$ at $\zeta_0$ can be proved in a similar way. 

    	In this proof, $C$s denote positive constants depending on the data. Each $C$ in different inequalities differs from each other. 
    	
    	1. Let $(f(0),\Phi\be_\vphi,L,S)$ be a solution of Problem \ref{pb3} for $\sigma\in (0,\sigma_3]$ satisfying \eqref{Pseudoestimate}.
    	Let $(\Psi,A,T)=(\Phi-\Phi_0^+,L,S-S_0^+)$ and $(\til \Psi,\til A,\til T):=(\Psi,A,T)(\Pi_{r_s f(0)+f_s})$.
        Find the partial Fr\'echet derivatives of 
    	$f(0)$, $\til\Psi\be_{\vphi}$, $\til A$ and $\til T$ with respect to $\fsp$ at $\zeta_0$. 
    	
    	Let $\til f_s^\p$ be a function in $C^{1,\alpha}_{(-\alpha,\{\theta=\theta_1\}),0}((0,\theta_1))$ satisfying $||\til f_s^\p||_{1,\alpha,(0,\theta_1)}^{(-\alpha,\{\theta=\theta_1\})}=1$.
    	By Proposition \ref{proPseudofree}, for each $(\rho_-,\bu_-,p_-,p_{ex},f_s^\p)=(\rho_0^-,u_0^-\be_r,p_0^-,p_c,\eps\til f_s^\p)$ for $\eps\in[0,\sigma_3]$, 
    	there exists a unique $(f_\eps(0),\Phi_\eps\be_{\vphi},L_\eps,S_\eps)$ satisfying (A), (B) and the estimate \eqref{Pseudoestimate}. Define
    	$(\Psi_\eps,A_\eps,T_\eps):=(\Phi_\eps-\Phi_0^+,L_\eps,S_\eps-S_0^+)$. Since each $(f_\eps(0),\Phi_\eps\be_{\vphi},L_\eps,S_\eps)$ satisfies (A), (B) for $(\rho_-,\bu_-,p_-,p_{ex},f_s^\p)=(\rho_0^-,u_0^-\be_r,p_0^-,p_c,\eps\til f_s^\p)$, 
    	each  $(f_\eps(0),\Psi_\eps\be_{\vphi},A_\eps,T_\eps)$ satisfies (A$^\p$), (B$^\p$) for $(\rho_-,\bu_-,p_-,p_{ex},f_s^\p)=(\rho_0^-,u_0^-\be_r,p_0^-,p_c,\eps\til f_s^\p)$. 
    	Let 
    	\begin{align}
    	\label{Gatauf0}
    	&\til f(0)^{(\til f_s^\p)}:=\lim_{\eps\ra 0^+}\frac{f_\eps(0)-f_0(0)}{\eps},\qd
    	\tilde{\Psi}^{(\til f_s^\p)}\be_{\vphi}=\lim_{\eps\ra 0^+}\frac{\til \Psi_\eps\be_\vphi-\Psi_0\be_{\vphi}}{\eps},\\
    	&\nonumber\til A^{(\til f_s^\p)}:=\lim_{\eps\ra 0^+}\frac{\til A_\eps-A_0}{\eps}\qdand
    	\til T^{(\til f_s^\p)}:=\lim_{\eps\ra 0^+}\frac{\til T_\eps-T_0}{\eps}
    	\end{align}
    	where $\til f(0)^{(\til f_s^\p)}$, $\tilde{\Psi}^{(\til f_s^\p)}\be_{\vphi}$, $\til A^{(\til f_s^\p)}$ and $\til T^{(\til f_s^\p)}$ represent the G\^ateaux derivatives of $f(0)$, $\til\Psi\be_{\vphi}$, $\til A$ and $\til T$ in the direction of $\til \fsp$ at $\zeta_0$. 
    	To find the Fr\'echet derivative of $\mc A$ with respect to $\fsp$ at $\zeta_0$, we find $\til f(0)^{(\til f_s^\p)}$, $\tilde{\Psi}^{(\til f_s^\p)}\be_{\vphi}$, $\til A^{(\til f_s^\p)}$ and $\til T^{(\til f_s^\p)}$. 
    	
    	Transform (A$^\p$), (B$^\p$) satisfied by $(f_\eps (0),\Psi_\eps \be_{\vphi},A_\eps, T_\eps)$ into the equations in $\n_{r_s}^+$ or on a part of $\pt\n_{r_s}^+$ by using $\Pi_{r_sf_\eps}$ where $f_\eps:=f_\eps(0)+\eps \til f_s$ with $\til f_s=\int_0^\theta \til \fsp$. And then subtract the resultant equations from the same equations satisfied by $(f_0 (0),\Psi_0 \be_{\vphi},A_0, T_0)$. Then we obtain 
    	\begin{align}
    	&\label{Wlineareps0}\grad\times \left(\frac{1}{\rho_0^+}(1+\frac{{u_0^+}\be_r\otimes u_0^+\be_r}{{c_0^+}^2-{u_0^+}^2})\grad\times ((\til \Psi_\eps-\Psi_0)\be_{\vphi})\right)\\
    	&\nonumber\qd\qd\qd=\frac{{\rho_0^+}^{\gam-1}}{(\gam-1)u_0^+}(1+\frac{\gam {u_0^+}^2}{{c_0^+}^2-{u_0^+}^2})\frac{\pt_\theta(\til T_\eps-T_0)}{r}\be_\vphi+F_0+\til {\bm F}_1(\til \Psi_\eps\be_{\vphi},\til A_\eps, \til T_\eps)\qdin \n_{r_s}^+,\\
    	&\label{Wfeps0}(\til \Psi_\eps-\Psi_0)\be_{\vphi}=
    	\begin{cases}
    	{\bm 0}\qdon\Gam_{r_s},\Gam_{w,r_s}^+,\\
    	\frac{1}{r_1\sin\theta}\int_0^\theta \bigg(- \frac{\rho_0^+((\gam-1){u_0^+}^2+{c_0^+}^2)}{\gam(\gam-1)u_0^+S_0^+}(\til T_\eps-T_0)\\
    	\qd\qd\qd\qd+\mf f_1(\til \Psi_\eps\be_{\vphi},\til A_\eps, \til T_\eps)\bigg)r_1^2\sin\xi d\xi\be_{\vphi}\qdon \Gam_{ex},
    	\end{cases}
    	\end{align}
    	\begin{align}
    	\label{compceps0}\frac{1}{r_1\sin\theta_1}\int_0^{\theta_1} \left.\left(- \frac{\rho_0^+((\gam-1){u_0^+}^2+{c_0^+}^2)}{\gam(\gam-1)u_0^+S_0^+}(\til T_\eps-T_0)
    	+\mf f_1(\til \Psi_\eps\be_{\vphi},\til A_\eps, \til T_\eps)\right)\right|_{r=r_1}r_1^2\sin\xi d\xi\be_{\vphi}={\bm 0},
    	\end{align}
    	\begin{align}\nonumber
    	\begin{cases}
    	(M_\eps^TM_\eps\grad\times((\til \Phi_0^++\til \Psi_\eps)\be_{\vphi})-\grad\times(\Phi_0^+\be_{\vphi}))\cdot\grad \til A_\eps+\grad\times(\Phi_0^+\be_{\vphi})\cdot\grad(\til A_\eps-A_0)=0\qdin\n_{r_s}^+,\\
    	\til A_\eps-A_0=0\qdon \Gam_{r_s}
    	\end{cases}
    	\end{align}
    	\begin{align}\nonumber
    	\begin{cases}
    	(M_\eps^TM_\eps\grad\times((\til \Phi_0^++\til \Psi_\eps)\be_{\vphi})-\grad\times(\Phi_0^+\be_{\vphi}))\cdot\grad \til T_\eps+\grad\times(\Phi_0^+\be_{\vphi})\cdot\grad(\til T_\eps-T_0)=0\qdin\n_{r_s}^+,\\
    	\til T_\eps-T_0=\left.g\left((M_0^-\be_r\cdot{\bm \nu_{f_\eps}})^2\right)\right|_{\Gam_{f_\eps}}S_{in}\circ \Pi_{r_sf_\eps}-(g({M_0^-}^2))(r_s)S_{in}\qdon \Gam_{r_s}
    	\end{cases}
    	\end{align}
    	where 
    	\begin{align*}
    	F_0=&\grad\times \left(\frac{1}{\rho_0^+}(1+\frac{{u_0^+}\be_r\otimes u_0^+\be_r}{{c_0^+}^2-{u_0^+}^2})\grad\times (\til \Psi_\eps\be_{\vphi})\right)\\
    	&-M_\eps\grad\times \left(\left(\frac{1}{\rho_0^+}(1+\frac{{u_0^+}\be_r\otimes u_0^+\be_r}{{c_0^+}^2-{u_0^+}^2})\right)(\Pi_{r_sf_\eps})M_\eps\grad\times (\til \Psi_\eps\be_{\vphi})\right)\\
    	&-\frac{{\rho_0^+}^{\gam-1}}{(\gam-1)u_0^+}(1+\frac{\gam {u_0^+}^2}{{c_0^+}^2-{u_0^+}^2})\frac{\pt_\theta\til T_\eps}{r}\be_\vphi\\
    	&+\left(
    	\frac{{\rho_0^+}^{\gam-1}}{(\gam-1)u_0^+}(1+\frac{\gam {u_0^+}^2}{{c_0^+}^2-{u_0^+}^2})\frac{1}{r}\right)(\Pi_{r_sf_\eps})\left(\pt_{\til\theta}\Pi_{f_\eps r_s}^r(\Pi_{r_sf_\eps})\pt_r\til T_\eps+\pt_\theta\til T_\eps\right)\be_\vphi,
    	\end{align*}
    	$(\til\Psi_\eps,\til A_\eps, \til T_\eps):=(\Psi_\eps,A_\eps,T_\eps)\circ\Pi_{r_sf_\eps}$,  $\til \Phi_0^+=\Phi_0^+\circ\Pi_{r_sf_\eps}$, $M_\eps=\left(\frac{\pt \Pi_{f_\eps r_s}}{\pt {\rm y}}\right)(\Pi_{r_sf_\eps})$, ${\rm y}$ is the cartesian coordinate system representing $\n_{f_\eps}^+$, 
    	$\til \theta$ is  $\theta$ coordinate for ${\rm y}$, 
    	$(r,\theta)$ is $(r,\theta)$ coordinates for  $\n_{r_s}^+$, 
    	and
    	$\Pi_{f_\eps r_s}^r$ is the $r$-component of  $\Pi_{f_\eps r_s}^*$, ${\bm \nu}_{f_\eps}$ is the unit normal vector field on $\Gam_{f_\eps}$ pointing toward $\n_{f_\eps}^+$ 
    	and $\til{\bm F}_1$ is ${\bm F}_1$ changed by using the transformation $\Pi_{f_\eps r_s}$. 
    	Divide the above equations by $\eps$ and formally take $\eps\ra 0^+$ using
    	\begin{align}\label{epsapproach}
    	&(f_\eps(0),\til \Psi_\eps,\til A_\eps,\til T_\eps)\ra (r_s,\Psi_0,A_0,T_0)=(r_s,0,0,0)\qd \tx{as}\qd \eps\ra 0^+,\\
    	&\eps\til f_s^\p\ra 0\qd \tx{as}\qd \eps\ra 0^+\nonumber
    	\end{align}
    	and \eqref{Gatauf0}. Then we have
    	\begin{align}
    	\label{Wtileqn}
    	\qd\qd\qd&\grad\times\left(\frac{1}{\rho_0^+}(1+\frac{{u_0^+}\be_r\otimes u_0^+\be_r}{{c_0^+}^2-{u_0^+}^2})\grad\times (\til \Psi^{(\til f_s^\p)}\be_{\vphi})\right) \\
    	&\nonumber\qd\qd\qd\qd\qd\qd\qd\qd\qd=\frac{{\rho_0^+}^{\gam-1}}{(\gam-1)u_0^+}(1+\frac{\gam {u_0^+}^2}{{c_0^+}^2-{u_0^+}^2})\frac{\pt_\theta \til T^{(\til f_s^\p)}}{r}\be_\vphi\qdin \n_{r_s}^+,\\
    	\label{Wtileqnbc}
    	&\til \Psi^{(\til f_s^\p)}\be_{\vphi}=
    	\begin{cases}
    	{\bm 0}\qdon \Gam_{r_s},\;\Gam_{w,r_s}^+,\\
    	-\frac{1}{r_1\sin\theta}\int_0^\theta \frac{\rho_0^+((\gam-1){u_0^+}^2+{c_0^+}^2)}{\gam(\gam-1)u_0^+S_0^+}\til T^{(\til f_s^\p)}r_1^2\sin\xi d\xi\be_{\vphi}\qdon \Gam_{ex},
    	\end{cases}
    	\end{align}
    	\begin{align}
    	\label{comptil}
    	&-\frac{1}{r_1\sin\theta_1}\int_0^{\theta_1}\left. \frac{\rho_0^+((\gam-1){u_0^+}^2+{c_0^+}^2)}{\gam(\gam-1)u_0^+S_0^+}\til T^{(\til f_s^\p)}\right|_{r=r_1}r_1^2\sin\xi d\xi\be_{\vphi}={\bm 0},\\
    	&\begin{cases}\label{Ltileqn}
    	\pt_r \til A^{(\til f_s^\p)}=0\qdin \n_{r_s}^+,\\
    	\til A^{(\til f_s^\p)}=0\qdon \Gam_{r_s},
    	\end{cases}\\
    	\label{Ttileqn}&\begin{cases}
    	\pt_r \til T^{(\til f_s^\p)}=0\qdin \n_{r_s}^+,\\
    	\til T^{(\til f_s^\p)}=(g({M_0^-}^2))^\p (r_s)S_{in}(\til f(0)^{(\til f_s^\p)}+\til f_s)\qdon \Gam_{r_s}.
    	\end{cases}
    	\end{align}
    	We solve this system 
    	for given $\til \fsp\in C^{1,\alpha}_{(-\alpha,\{\theta=\theta_1\}),0}((0,\theta_1))$. 
    	
    	By \eqref{Ltileqn} and \eqref{Ttileqn},
    	\begin{align}\label{Atil}
    	\til A^{(\til f_s^\p)}=0\qdin \n_{r_s}^+
    	\end{align}
    	and
    	\begin{align}\label{TT}\til T^{(\til f_s^\p)}=(g({M_0^-}^2))^\p (r_s)S_{in}(\til f(0)^{(\til f_s^\p)}+\til f_s).\end{align}
    	Substitute \eqref{TT} into \eqref{comptil}. And then find $\til f(0)^{(\til f_s^\p)}$ in the resultant equation. Then we obtain 
    	\begin{align}
    	\label{gatFAT}
    	\til f(0)^{(\til f_s^\p)}=\frac{-\int_0^{\theta_1}\til f_s\sin\zeta d\zeta}{\int_0^{\theta_1}\sin\zeta d\zeta}.
    	\end{align}
    	Substitute this into \eqref{TT} again. Then we have
    	\begin{align}
    	\label{Tlinear}\til T^{(\til f_s^\p)}=(g({M_0^-}^2))^\p(r_s) S_{in}(-\frac{\int_0^{\theta_1}\til f_s\sin\zeta d\zeta}{\int_0^{\theta_1}\sin\zeta d\zeta}+\til f_s)\qdin \n_{r_s}^+.
        \end{align}
    	Substituting $\til T^{(\til f_s^\p)}$ given in \eqref{Tlinear} into \eqref{Wtileqn} and \eqref{Wtileqnbc}, we get
    	\begin{align}\label{Wtileqnstar}
    	&\grad\times\left(\frac{1}{\rho_0^+}(1+\frac{{u_0^+}\be_r\otimes u_0^+\be_r}{{c_0^+}^2-{u_0^+}^2})\grad \times(\til \Psi^{(\til f_s^\p)}\be_{\vphi})\right) \\
    	&\nonumber\qd\qd\qd=\frac{{\rho_0^+}^{\gam-1}}{(\gam-1)u_0^+}(1+\frac{\gam {u_0^+}^2}{{c_0^+}^2-{u_0^+}^2})(g({M_0^-}^2))^\p(r_s) S_{in}\frac{\til f_s^\p}{r}\be_\vphi\qdin \n_{r_s}^+,\\
    	&\label{Wtileqnstarbc}\til \Psi^{(\til f_s^\p)}\be_{\vphi}=
    	\begin{cases}
    	{\bm 0}\qdon \Gam_{r_s},\;\Gam_{w,r_s}^+,\\
    	-\frac{1}{r_1\sin\theta}\int_0^\theta \frac{\rho_0^+((\gam-1){u_0^+}^2+{c_0^+}^2)}{\gam(\gam-1)u_0^+S_0^+}(g({M_0^-}^2))^\p(r_s) S_{in}\\
    	\qd\qd\qd\qd\qd\qd\qd\qd\qd(-\frac{\int_0^{\theta_1}\til f_s\sin\zeta d\zeta}{\int_0^{\theta_1}\sin\zeta d\zeta}+\til f_s)r_1^2\sin\xi d\xi\be_{\vphi}\qdon \Gam_{ex}.
    	\end{cases}
    	\end{align}
    	Note that since \eqref{Wtileqnstarbc} is a $C^0$ boundary condition, Lemma \ref{lemSE} can be applied to \eqref{Wtileqnstar}, \eqref{Wtileqnstarbc}. 
    	Using Lemma \ref{lemSE} and the arguments used to prove Lemma \ref{lem2alpha} (here we can obtain more higher regularity of solutions of \eqref{Wtileqnstar}, \eqref{Wtileqnstarbc} than that of solutions of \eqref{EE}, \eqref{EEbc} in Lemma \ref{lemSE} because $\til \fsp\in C^{1,\alpha}_{(-\alpha,\{\theta=\theta_1\})}((0,\theta_1))$), we obtain that there exists a unique  $C^{3,\alpha}_{(-1-\alpha,\Gam_{w,r_s}^+)}(\n_{r_s}^+)$ solution $\til \Psi^{(\til f_s^\p)}\be_{\vphi}$ of \eqref{Wtileqnstar}, \eqref{Wtileqnstarbc} and this solution satisfies 
    	\begin{align}\label{tilWestimate}
    	||\til \Psi^{(\til f_s^\p)}\be_{\vphi}||_{3,\alpha,\n_{r_s}^+}^{(-1-\alpha,\Gam_{w,r_s}^+)}\le C
    	\end{align}
    	for a positive constant $C$ independent of $\til f_s^\p$.  

    	Using \eqref{Pseudoestimate} satisfied by $(f_\eps(0),\Phi_\eps\be_{\vphi},L_\eps,S_\eps)$, \eqref{Wlineareps0}-\eqref{compceps0}, \eqref{Wtileqn}-\eqref{comptil}, \eqref{Atil}, \eqref{gatFAT}, \eqref{Tlinear}, the solution expressions of $\til A_\eps$ and $\til T_\eps$ obtained by solving (B$^\p$) for $(f(0),\Psi)=(f_\eps(0),\Psi_\eps)$ and $(\rho_-,\bu_-,p_-,p_{ex},f_s^\p)=(\rho_0^-,u_0^-\be_r,p_0^-,p_c,\eps\til f_s^\p)$ using Lemma \ref{lemTrans},
    	and Lemma \ref{lemSE}, we can obtain 
    	\begin{align}
    	&\nonumber|f_\eps(0)-f_0(0)-\eps\til f(0)^{(\til \fsp)}|\le C\eps^2,\\
    	&\label{PsiepsPsi0epsPsi0} ||\til\Psi_\eps-\Psi_0-\eps\til \Psi^{(\til \fsp)}||_{2,\alpha,\n_{r_s}^+}^{(-1-\alpha,\Gam_{w,r_s}^+)}\le C\eps^2,\\
    	&\nonumber||\til A_\eps-A_0-\eps\til A^{(\til \fsp)}||_{1,\alpha,\n_{r_s}^+}^{(-\alpha,\Gam_{w,r_s}^+)}=0,\\
    	&\nonumber||\til T_\eps-T_0-\eps\til T^{(\til \fsp)}||_{1,\alpha,\n_{r_s}^+}^{(-\alpha,\Gam_{w,r_s}^+)}\le C\eps^2
    	\end{align}
    	for all $\til \fsp\in C^{1,\alpha}_{(-\alpha,\{\theta=\theta_1\}),0}((0,\theta_1))$ satisfying $||\til f_s^\p||_{1,\alpha,(0,\theta_1)}^{(-\alpha,\{\theta=\theta_1\})}=1$ and $\eps \in [0,\sigma_3]$. 
    	By \eqref{Atil} and  \eqref{gatFAT}-\eqref{Wtileqnstarbc}, $\til f(0)^{(\til\fsp)}$, $\til \Psi^{(\til f_s^\p)}\be_{\vphi}$, $\til A^{(\til f_s^\p)}$ and $\til T^{(\til f_s^\p)}$ 
    	are bounded linear maps from $C^{1,\alpha}_{(-\alpha,\{\theta=\theta_1\}),0}((0,\theta_1))$ to $\R$, $C^{2,\alpha}_{(-1-\alpha,\Gam_{w,r_s}^+)}(\n_{r_s}^+)$, $C^{1,\alpha}_{(-\alpha,\Gam_{w,r_s}^+)}(\n_{r_s}^+)$ and $C^{1,\alpha}_{(-\alpha,\Gam_{w,r_s}^+)}(\n_{r_s}^+)$, respectively. From the above inequalities and this fact, we have that $f(0)$, $\til \Psi\be_{\vphi}$, $\til A$ and $\til T$ are Fr\'echet differentiable as a function of $\fsp$ at $0$ with the other variables fixed at $(\rho_0^-,u_0^-\be_r,p_0^-,p_c)$ and $\til f(0)^{(\til\fsp)}$, $\til \Psi^{(\til f_s^\p)}\be_{\vphi}$, $\til A^{(\til f_s^\p)}$ and $\til T^{(\til f_s^\p)}$ are the partial Fr\'echet derivatives of $f(0)$, $\til \Psi\be_{\vphi}$, $\til A$ and $\til T$ with respect to $\fsp$ at $\zeta_0$, respectively. 
    	
    	2. 
    	Show that ${\mc A}$ is Fr\'echet differentiable as a function of $\fsp$ at $0$ with the other variables fixed at $(\rho_0^-,u_0^-\be_r,p_0^-,p_c)$. 
    	
    	By definition, the G\^ateaux derivative of $\mc A$ in the direction $(0,0,0,0,\til f_s^\p)$ at $\zeta_0$ is given by
    	\begin{multline*}
    	\lim_{\eps\ra 0^+} \frac{\mc A(\rho_0^-,u_0^-\be_r,p_0^-,p_c,\eps \til f_s^\p)-\mc A(\rho_0^-,u_0^-\be_r,p_0^-,p_c,0)}{\eps}\\
    	=
    	\lim_{\eps\ra 0^+}\frac{1}{\eps}\left(\left.\frac{1}{\varrho(\grad\times((\Phi_0^++\Psi_\eps)\be_{\vphi}),\frac{A_\eps}{2\pi r\sin\theta}\be_\vphi,S_0^++T_\eps)}\right.
    	\grad\times( (\Phi_0^++\Psi_\eps)\be_{\vphi})\cdot{\bm \tau_{f_\eps}}\right|_{\Gam_{f_\eps}}\circ\Pi_{r_sf_\eps}\\    	
    	\\\left.
    	-\left.u_0^-\be_r\cdot{\bm \tau_{f_\eps}}\right|_{\Gam_{f_\eps}}\circ\Pi_{r_sf_\eps}-\frac{1}{\varrho(\grad \times(\Phi_0^+\be_{\vphi}),0,S_0^+)}\grad \times(\Phi_0^+\be_{\vphi})\cdot \be_\theta \right|_{\Gam_{r_s}}
    	\left.
    	+\left.
    	u_0^-\be_r\cdot \be_\theta\right|_{\Gam_{r_s}}\right)
    	\end{multline*}
    	where ${\bm \tau}_{f_\eps}$ is the unit tangent vector on $\Gam_{f_\eps}$ perpendicular to $\be_\vphi$ and satisfying $({\bm \tau}_{f_\eps}\times \be_\vphi)\cdot {\bm \nu}_{f_\eps}>0$ and ${\bm \nu}_{f_\eps}$ is the unit normal vector field on $\Gam_{f_\eps}$. 
    	As we did in Step 1, formally take $\eps\ra 0^+$ using \eqref{Gatauf0} and \eqref{epsapproach} to the right-hand side of the above equation. Then we obtain
    	\begin{align*}
    	\left.\frac{1}{\rho_0^+}\grad\times (\til \Psi^{(\til f_s^\p)}\be_{\vphi})\cdot \be_\theta\right|_{r=r_s}-(u_0^- -u_0^+)(r_s)\frac{\til f_s^\p}{r_s}.
    	\end{align*}
    	Define a map $L$ by 
    	\begin{align*}
    	L\til f_s^\p:=\left.\frac{1}{\rho_0^+}\grad\times (\til \Psi^{(\til f_s^\p)}\be_{\vphi}) \cdot \be_\theta\right|_{r=r_s}-(u_0^- -u_0^+)(r_s)\frac{\til f_s^\p}{r_s}
    	\end{align*}
    	for $\til \fsp\in C^{1,\alpha}_{(-\alpha,\{\theta=\theta_1\}),0}((0,\theta_1))$ satisfying $||\til f_s^\p||_{1,\alpha,(0,\theta_1)}^{(-\alpha,\{\theta=\theta_1\})}=1$
    	where $\til \Psi^{(\til f_s^\p)}\be_{\vphi}$ is a solution of \eqref{Wtileqnstar}, \eqref{Wtileqnstarbc} for given $\til \fsp \in C^{1,\alpha}_{(-\alpha,\{\theta=\theta_1\}),0}((0,\theta_1))$. 
    	Then $L$ is a bounded linear map from $C^{1,\alpha}_{(-\alpha,\{\theta=\theta_1\}),0}((0,\theta_1))$ to $C^{1,\alpha}_{(-\alpha,\{\theta=\theta_1\}),0}((0,\theta_1))$. 
    	Using \eqref{Pseudoestimate} satisfied by  $(f(0),\Phi_\eps\be_\vphi,L_\eps,S_\eps)$ and \eqref{PsiepsPsi0epsPsi0}, 
    	one can check that 
    	\begin{align}\label{oeestimate}
    	||\mc A(\rho_0^-,u_0^-\be_r,p_0^-,p_c,\eps \til f_s^\p)-\mc A(\rho_0^-,u_0^-\be_r,p_0^-,p_c,0)-\eps L\til f_s^\p||_{1,\alpha,(0,\theta_1)}^{(-\alpha,\{\theta=\theta_1\})}\le C\eps^2
    	\end{align}
    	for any $\til f_s^\p\in C^{1,\alpha}_{(-\alpha,\{\theta=\theta_1\}),0}((0,\theta_1))$ satisfying $||\til f_s^\p||_{1,\alpha,(0,\theta_1)}^{(-\alpha,\{\theta=\theta_1\})}=1$ and $\eps\in [0,\sigma_3]$. 
    	Therefore, ${\mc A}$ is Fr\'echet differentiable as a function of $\fsp$ at $0$ with the other variables fixed at $(\rho_0^-,u_0^-\be_r,p_0^-,p_c)$ and $L$ is the partial Fr\'echet derivative of $\mc A$ with respect to $\fsp$ at $\zeta_0$. This finishes the proof.   
    \end{proof}
    Finally, we prove that the partial Fr\'echet derivative of $\mc A$ with respect to $\fsp$ at $\zeta_0$ is invertible. 
    When we prove the invertibility of the partial Fr\'echet derivative of $\mc A$ with respect to $\fsp$ at $\zeta_0$, we use eigenfunction expansions of $\til\fsp$ and $\til \Psi^{(\til \fsp)}$. The eigenfunctions used to express $\til\fsp$ and $\til \Psi^{(\til \fsp)}$ are eigenfunctions of the following eigenvalue problem 
    \begin{align}\label{eigenprob}
    \begin{cases}\frac{1}{\sin\theta}\pt_{\theta}(\sin\theta\pt_{\theta}q)-\frac{q}{\sin^2\theta}=-\lambda q\qdin \theta\in(0,\theta_1),\\
    q=0\qdon \theta=0,\theta_1
    \end{cases}
    \end{align}
    that arises from $\theta$-part of the spherical coordinate representation of \eqref{Wtileqnstar}, \eqref{Wtileqnstarbc} 
    (note that this is the associated Legendre equation of type $m=1$ with a general domain that is a singular Sturm-Liouville problem). 
    To express $\til\fsp$ and $\til \Psi^{(\til \fsp)}$ using eigenfunctions of \eqref{eigenprob}, we need to prove 
    the orthogonal completeness of the set of eigenfunctions of \eqref{eigenprob}. We prove this in the following lemma.
    \begin{lem}\label{lemlegendre}
    	The eigenvalue problem \eqref{eigenprob} has infinitely countable eigenvalues $\lambda_j$ for $j=1,2,\ldots$ satisfying 
    	$\lambda_j\ra \infty$ as $j\ra \infty$ and $\lambda_j>0$. A set of eigenfunctions of \eqref{eigenprob} forms a complete orthorgonal set in $L^2((0,\theta_1),\sin\theta d\theta)$. 
    \end{lem}
    \begin{proof}
    	1. For given $f\in L^2((0,\theta_1),\sin\theta d\theta)$, we consider 
    	\begin{align}\label{qfproblem}
    	\begin{cases}\frac{1}{\sin\theta}\pt_{\theta}(\sin\theta\pt_{\theta}q)-\frac{q}{\sin^2\theta}=-f\qdin \theta\in(0,\theta_1),\\
    	q=0\qdon \theta=0,\theta_1.
    	\end{cases}
    	\end{align}
    	Write this equation in the form
    	\begin{align}\label{weakformqfproblem}
    	\int_0^{\theta_1}(\pt_\theta q\pt_\theta \xi+\frac{q\xi}{\sin^2\theta})\sin\theta d\theta=\int_0^{\theta_1}f\xi \sin\theta d\theta\qd
    	\end{align}
    	for all $\xi\in H_0^1((0,\theta_1),\sin\theta d\theta)$. 
    	Assume for a moment that there exists a unique $q\in H^1_0((0,\theta_1),\sin\theta d\theta)$ satisfying \eqref{weakformqfproblem} for all $\xi\in H_0^1((0,\theta_1),\sin\theta d\theta)$ 
    	and 
    	this $q$ satisfies 
    	\begin{align}
    	\label{qestimate}||q||_{H_0^1((0,\theta_1),\sin\theta d\theta)}\le C||f||_{L^2((0,\theta_1),\sin\theta d\theta)}
    	\end{align} 
    	for some positive constant $C$. 
    	Using this $q$, we define a map $\mc S: L^2((0,\theta_1),\sin\theta d\theta)\ra L^2((0,\theta_1),\sin\theta d\theta)$ by
    	$$\mc Sf=q.$$
    	Then $\mc S$ is a self-adjoint and compact linear operator. Hence, by the spectral theorem for compact self-adjoint operators, $\mc S$ has contable infinite eigenvalues $\mu_j$ satisfying $\mu_j\ra 0$ as $j\ra \infty$ and the set of eigenfunctions $q_j$ of $\mc S$ corresponding to $\mu_j$ forms a complete orthorgonal set in $L^2((0,\theta_1),\sin\theta d\theta)$. 
    	From this fact, we obtain that \eqref{eigenprob} has infinitely countable eigenvalues $\lambda_j\ra \infty$ as $j\ra \infty$ and the set of eigenfunctions of  \eqref{eigenprob} forms a complete orthorgonal set in $L^2((0,\theta_1),\sin\theta d\theta)$. 
    	
    	2.  Show that  there exists a unique $q\in H^1_0((0,\theta_1),\sin\theta d\theta)$ satisfying \eqref{weakformqfproblem} for all $\xi\in H_0^1((0,\theta_1),\sin\theta d\theta)$. 
    	
    	For given $f\in L^2((0,\theta_1),\sin\theta d\theta)$, we consider
    	\begin{align}\label{Uprobfvphi}
    	\int_{\mc D} \delta \bU \delta {\bm \xi} =\int_{\mc D} f(\theta)\be_{\vphi}{\bm \xi}
    	\end{align}
    	for all ${\bm \xi}\in H_0^1(\mc D)$
    	where $\bU:\mc D\ra \R^3$, ${\bm \xi}:\mc D\ra \R^3$, $\mc D:=\{(x,y,z)\in S^2:\frac{z}{\sqrt{x^2+y^2+z^2}}\ge \cos\theta_1\}$ and $\delta$ is the covariant derivative on $S^2$. By the Lax-Milgram theorem, there exists a unique $\bU \in H_0^1(\mc D)$ satisfying \eqref{Uprobfvphi} for all ${\bm \xi}\in H_0^1(\mc D)$ and this $\bU$ satisfies 
    	\begin{align}
    	\label{UH1}||\bU||_{H_0^1(\mc D)}\le C||f\be_\vphi||_{L^2(\mc D)}
    	\end{align}
    	for some positive constant $C$ depending on $\mc D$. Let $\bU$ be a function in $H_0^1(\mc D)$ satisfying \eqref{Uprobfvphi} for all ${\bm \xi}\in H_0^1(\mc D)$.
    	Using the standard argument, it can be shown that $\bU$ 
    	satisfies 
    	\begin{align*}
    	||\bU||_{H^2(\mc D)}\le C||f\be_\vphi||_{L^2(\mc D)}.
    	\end{align*}
    	From this fact, we see that $\bU$ 
    	satisfies
    	\begin{align*}
    	\int_{\mc D}\Delta_{S^2}\bU {\bm\xi}=-\int_{\mc D}f(\theta)\be_{\vphi}{\bm \xi}
    	\end{align*}
    	for all ${\bm \xi}\in H_0^1(\mc D)$.
    	Using this fact and the fact that the coefficients of $\Delta_{S^2}\bU$ in the spherical coordinate system are independent of $\vphi$, we apply arguments similar to the ones in the proof of Lemma \ref{lemform} to $\bU$ (here, we use the facts that a bounded sequence in $H^2(\mc D)$ contains a weakly convergent subsequence
    	and a $H^1_0(\mc D)$ function 
    	satisfying \eqref{Uprobfvphi} for all ${\bm \xi}\in H_0^1(\mc D)$ is unique). 
    	Then we have that $\bU$ 
    	only has the form $u_\vphi(\theta)\be_{\vphi}$. 
    	
    	One can see that  if $\bU=u_\vphi(\theta)\be_{\vphi}\in H^1_0(\mc D)$ satisfies \eqref{Uprobfvphi} for all ${\bm \xi}\in H_0^1(\mc D)$, then $u_\vphi\in H^1_0((0,\theta_1),\sin\theta d\theta)$ satisfies \eqref{weakformqfproblem} for all $\xi\in H_0^1((0,\theta_1),\sin\theta d\theta)$ and that if $u_\vphi\in H^1_0((0,\theta_1),\sin\theta d\theta)$ satisfies \eqref{weakformqfproblem} for all $\xi\in H_0^1((0,\theta_1),\sin\theta d\theta)$, then $\bU=u_\vphi(\theta)\be_{\vphi}\in H^1_0(\mc D)$ satisfies \eqref{Uprobfvphi} for all ${\bm \xi}\in H_0^1(\mc D)$ having the form $\xi(\theta) \be_{\vphi}$.
    	Using this fact, 
    	we can deduce that there exists a unique $q\in H^1_0((0,\theta_1),\sin\theta d\theta)$ satisfying \eqref{weakformqfproblem} for all $\xi\in H_0^1((0,\theta_1),\sin\theta d\theta)$. 
    	By \eqref{UH1}, this solution satisfies \eqref{qestimate}. 
    	
    	3. Show that eigenvalues $\lambda_j$ of \eqref{eigenprob} are positive. 
    	
    	If $q$ is an eigenfunction of $\mc S$ corresponding to an eigenvalue $\mu=\frac 1 \lambda$, then there holds
    	\begin{align}\label{qeigen}\int_{\mc D}\delta(q\be_\vphi)\delta{\bm \xi}=\int_{\mc D}\lambda q \be_\vphi {\bm \xi}
    	\end{align}
    	for all ${\bm \xi}\in H_0^1(\mc D)$ having the form $\xi(\theta) \be_{\vphi}$. 
    	Using the weak maximum principle, 
    	we can have that 
    	for $q\be_\vphi\in H_0^1(\mc D)$ to be a nonzero function satisfying \eqref{qeigen} for all ${\bm \xi}\in H_0^1(\mc D)$ having the form $\xi(\theta) \be_{\vphi}$, $\lambda$ must be positive. 
    	Hence, eigenvalues of \eqref{eigenprob} are positive.
    	This finishes the proof. 
    \end{proof}
    Then we prove the invertibility of the partial Fr\'echet derivative of $\mc A$ with respect to $\fsp$ at $\zeta_0$. 
    \begin{lem}\label{leminvertible}
    	The partial Fr\'echet derivative of $\mc A$ with respect to $\fsp$ at $\zeta_0$ 
    	given by \eqref{Dfsp} 
    	is an invertible map from $C^{1,\alpha}_{(-\alpha,\{\theta=\theta_1\})}((0,\theta_1))$ to $C^{1,\alpha}_{(-\alpha,\{\theta=\theta_1\})}((0,\theta_1))$. 
    \end{lem}
    \begin{proof}
    	By \eqref{tilWestimate} 
    	and $C^{2,\alpha}_{(-1-\alpha,\{\theta=\theta_1\})}((0,\theta_1))\Subset C^{1,\alpha}_{(-\alpha,\{\theta=\theta_1\})}((0,\theta_1))$, 
    	$D_{\fsp} {\mc A}(\zeta_0)$ given by \eqref{Dfsp} is of the form $c I-K$ where $c$ is a constant and $K$ is a compact linear map from $C^{1,\alpha}_{(-\alpha,\{\theta=\theta_1\}),0}((0,\theta_1))$ to $C^{1,\alpha}_{(-\alpha,\{\theta=\theta_1\}),0}((0,\theta_1))$. 
    	By the Fredholm alternative, this implies that if $\ker D_{\fsp}\mc A(\zeta_0)=\{0\}$, then $D_{\fsp}\mc A(\zeta_0)$  is invertible. 
    	In this proof, we show that $\ker D_{\fsp}\mc A(\zeta_0)=\{0\}$. Since it is obvious that $D_{\fsp}\mc A(\zeta_0)0=0$, we show that $D_{\fsp}\mc A(\zeta_0)\til \fsp=0$ only if $\til \fsp=0$. 
    	
    	1. 	Assume that for a nonzero $\til\fsp\in C^{1,\alpha}_{(-\alpha,\{\theta=\theta_1\}),0}((0,\theta_1))$, $D_{\fsp}\mc A(\zeta_0)\til \fsp=0$. 
    	Then there holds 
    	\begin{align}\label{dfsp}
    	\left.\frac{1}{\rho_0^+}
    	\grad\times(\Psi^{(\til \fsp)}\be_\vphi)\cdot\be_\theta\right|_{r=r_s}
    	-(u_0^- -u_0^+)(r_s)\frac{\til \fsp}{r_s}=0
    	\end{align}
        where $\til \Psi^{(\til \fsp)}\be_{\vphi}$ is the $C^{3,\alpha}_{(-1-\alpha,\Gam_{w,r_s}^+)}(\n_{r_s}^+)$ solution of 
    	\begin{align}
    	\label{Wtileqnspherical}
    	&-\frac{1}{\rho_0^+ r^2}\pt_r (r^2\pt_r \til\Psi^{(\til f_s^\p)})+\frac{\pt_r\rho_0^+}{{\rho_0^+}^2r}\pt_r(r\til\Psi^{(\til f_s^\p)})
    	-\frac{1}{\rho_0^+ r^2}\left(\frac{ {c_0^+}^2}{ {c_0^+}^2-{u_0^+}^2}\right)
    	\left(\frac{1}{\sin\theta}\pt_{\theta}(\sin\theta\pt_{\theta}\til\Psi^{(\til f_s^\p)})-\frac{\til\Psi^{(\til f_s^\p)}}{\sin^2\theta}\right)
    	\\&\nonumber
    	\qd\qd=\frac{{\rho_0^+}^{\gam-1}}{(\gam-1)u_0^+}\left(1+\frac{\gam {u_0^+}^2}{{c_0^+}^2-{u_0^+}^2}\right)\frac{\left(g( {M_0^-}^2)\right)^\p(r_s) S_{in}}{r}\til f_s^\p\qdin \n_{r_s}^{+,*},\\
    	\label{Wtileqnsphericalbc}
    	&\til \Psi^{(\til f_s^\p)}=\begin{cases}
    	0\qdon \Gam_{r_s}^*,\;\Gam_{w,r_s}^{+,*},\\
    	-\frac{1}{r_1\sin\theta}\int_0^\theta \frac{\rho_0^+((\gam-1){u_0^+}^2+{c_0^+}^2)}{\gam(\gam-1)u_0^+S_0^+}\left(g({M_0^-}^2)\right)^\p(r_s) S_{in}\\
    	\qd\qd\qd\qd\qd\qd\qd\qd\qd\qd\qd\qd\qd\qd\qd\qd
    	\left(-\frac{\int_0^{\theta_1}\til f_s\sin\zeta d\zeta}{\int_0^{\theta_1}\sin\zeta d\zeta}+\til f_s\right)r_1^2\sin\xi d\xi\qdon \Gam_{ex}^*,
    	\end{cases}
    	\end{align}
    	which is 
    	the spherical coordinate representation of \eqref{Wtileqnstar}, \eqref{Wtileqnstarbc}, 
    	for given $\til \fsp$. 
    	Using $
    	C^{1,\alpha}_{(-\alpha,\{\theta=\theta_1\})}((0,\theta_1)) 
    	\subset L^2((0,\theta_1),\sin\theta d\theta)$ and Lemma \ref{lemlegendre}, 
    	we express $\til \fsp$ as 
    	\begin{align}\label{fspexp}
    	\til \fsp=\sum_{j=1}^\infty c_j q_j
    	\end{align}
    	where $c_j$ are constants and $q_j$ are eigenfunctions of \eqref{eigenprob} corresponding to eigenvalues $\lambda_j$ of \eqref{eigenprob}. 
    	Define 
    	\begin{align}\label{fpsm}
    	\til f^\p_{s,m}:=\sum_{j=1}^m c_j q_j 
    	\end{align}
    	and $\til f_{s,m}:=\int_0^\theta \til f_{s,m}^\p$.
    	We consider  \eqref{Wtileqnspherical}, \eqref{Wtileqnsphericalbc} for given $\til \fsp=\til f_{s,m}^\p$ and $\til f_s=\til f_{s,m}$. 
    	
    	Using the fact that $q_j\in C^\infty([0,\theta_1])$ and $\frac{1}{\sin\theta}\pt_{\theta}(\sin\theta\pt_{\theta}q)-\frac{q}{\sin^2\theta}=\pt_{\theta}\left(\frac{\pt_{\theta}(q\sin\theta)}{\sin\theta}\right)$, we can have  
    	\begin{align}\label{fsm}
    	\til f_{s,m}=\sum_{j=1}^m -\frac{c_j}{\lambda_j}\left(\frac{\pt_{\theta}(q_j\sin\theta)}{\sin\theta}-\left.\frac{\pt_{\theta}(q_j\sin\theta)}{\sin\theta}\right|_{\theta=0}\right).
    	\end{align}
    	Here, $\left.\frac{\pt_{\theta}(q_j\sin\theta)}{\sin\theta}\right|_{\theta=0}$ is bounded because $q_j=0$ on $\theta=0$ and $q_j\in C^\infty([0,\theta_1])$. 
    	Substitute $\til \Psi_m=\sum_{j =1}^m p_j(r)q_j(\theta)$,
    	$\til f_{s,m}^\p$ given in \eqref{fpsm} and $\til f_{s,m}$ given in \eqref{fsm} into the places of $\til \Psi^{(\til \fsp)}$, $\til \fsp$ and $\til f_s$ in \eqref{Wtileqnspherical}, \eqref{Wtileqnsphericalbc}, respectively. Then we obtain 
    	\begin{align}
    	\label{rode}
    	(L_1:=)&-\frac{1}{r^2}(\frac{1}{\rho_0^+}r^2p_j^\p)^\p
    	+(\frac{1}{\rho_0^+r^2}(1+\frac{ {u_0^+}^2}{ {c_0^+}^2-{u_0^+}^2})\lambda_j+\frac{\pt_r\rho_0^+}{{\rho_0^+}^2r})p_j
    	\\
    	&\qd\qd\qd\qd\qd\qd=\frac{{\rho_0^+}^{\gam-1}}{(\gam-1)u_0^+}(1+\frac{\gam {u_0^+}^2}{{c_0^+}^2-{u_0^+}^2})\frac{\left(g( {M_0^-})^2\right)^\p(r_s) S_{in}}{r}c_j\qdin (r_s,r_1),\nonumber
    	\\
    	\label{rodebc}
    	&p_j=\begin{cases}
    	0\qdon r=r_s,\\
    	\frac{\rho_0^+((\gam-1){u_0^+}^2+{c_0^+}^2)(g( {M_0^-}^2))^\p(r_s)S_{in}r_1}{\gam(\gam-1)u_0^+S_0^+\lambda_j}c_j\qdon r=r_1
    	\end{cases}
    	\end{align}
    	for $j=1,\ldots,m$. One can see that $L_1$ is of the form  $L_0-c_0I$ where $L_0(=\frac{1}{r^2}(\frac{1}{\rho_0^+}r^2p_j^\p)^\p)$ is an invertible operator and  $c_0$ is a positive constant. Using the Fredholm alternative and the maximum principle, we can obtain that 
    	\eqref{rode}, \eqref{rodebc} for each $j$ has a unique $C^\infty([r_s,r_1])$ solution. 
    	Let $p_j$ be the solution of \eqref{rode}, \eqref{rodebc}. Then by the fact that $q_j\be_\vphi\in  C^\infty(\ol{\mc D})$ where $\mc D$ is a domain defined in the proof of Lemma \ref{lemlegendre}, we have that $\til \Psi_m \be_\vphi=\sum_{j=1}^m p_jq_j\be_\vphi$ is a $C^\infty(\ol{\n_{r_s}^+})$ solution of \eqref{Wtileqnstar}, \eqref{Wtileqnstarbc} for given 
    	$\til\fsp=\til f_{s,m}^\p$ and $\til f_s=\til f_{s,m}$. 
    	
    	2. Show that there exists a subsequence  $\til \Psi_{m_l}\be_{\vphi}$ of $\til \Psi_m\be_{\vphi}$ such that $\til\Psi_{m_l}\be_{\vphi}$ and $D(\til \Psi_{m_l}\be_{\vphi})$ weakly converge to $\til \Psi^{(\til \fsp)}\be_\vphi$ and $D(\til \Psi^{(\til \fsp)}\be_{\vphi})$ in $L^2(\n_{r_s}^+)$, respectively.

    	Since $\til \Psi_m \be_\vphi$ is a $C^\infty(\ol{\n_{r_s}^+})$ solution of \eqref{Wtileqnstar}, \eqref{Wtileqnstarbc} for given $\til\fsp=\til f_{s,m}^\p$ and $\til f_s=\til f_{s,m}$, \eqref{Wtileqnstar}, \eqref{Wtileqnstarbc} satisfied by $\til \Psi_m \be_\vphi$ can be transformed into the following boundary value problem for an elliptic system
    	\begin{multline}
    	\label{DivW}
    	\Div \left({\bm A}D(\til \Psi_m\be_{\vphi})\right)-d(\til \Psi_m\be_{\vphi})\\
    	=-\frac{{\rho_0^+}^{\gam-1}}{(\gam-1)u_0^+}(1+\frac{\gam {u_0^+}^2}{{c_0^+}^2-{u_0^+}^2})(g({M_0^-}^2))^\p(r_s)S_{in}\frac{\til f_{s,m}^\p}{r}\be_\vphi\qdin \n_{r_s}^+,
    	\end{multline}
    	where ${\bm A}$ and $d$ are a $(2,2)$-tensor and scalar function defined below \eqref{vbdry},
    	\begin{align}\label{DivWbc}
    	\til \Psi_m\be_{\vphi}=\begin{cases}
    	{\bm 0}\qdon \Gam_{r_s},\;\Gam_{w,r_s}^+,\\
    	-\frac{1}{r_1\sin\theta}\int_0^\theta \frac{\rho_0^+((\gam-1){u_0^+}^2+{c_0^+}^2)}{\gam(\gam-1)u_0^+S_0^+}(g({M_0^-}^2))^\p(r_s)S_{in}\\
    	\qd\qd\qd\qd\qd\qd(-\frac{\int_0^{\theta_1}\til f_{s,m}\sin\zeta d\zeta}{\int_0^{\theta_1}\sin\zeta d\zeta}+\til f_{s,m})r_1^2\sin\xi d\xi\be_{\vphi}(=:h_m\be_{\vphi})\qdon \Gam_{ex}.
    	\end{cases}
    	\end{align} 
    	Set ${\bm h}_m:=\frac{(r-r_s)r_1}{(r_1-r_s)r}h_m\be_\vphi.$ Transform \eqref{DivW}, \eqref{DivWbc} into a ${\bm 0}$ boundary value problem by substituting $\til \Psi_m^*\be_{\vphi}+{\bm h}_m$ into the place of $\til\Psi_m\be_{\vphi}$ in \eqref{DivW}, \eqref{DivWbc}. 
    	Write 
    	the resultant problem in the following form
    	\begin{multline}\label{weakW}
    	\int_{\n_{r_s}^+}{\bm A}D(\til \Psi_m^*\be_{\vphi}) D{\bm \xi}+d\til \Psi_m^*\be_{\vphi} {\bm \xi}\\
    	=\int_{\n_{r_s}^+}\frac{{\rho_0^+}^{\gam-1}}{(\gam-1)u_0^+}(1+\frac{\gam {u_0^+}^2}{{c_0^+}^2-{u_0^+}^2})(g({M_0^-}^2))^\p(r_s)S_{in}\frac{\til f_{s,m}^\p}{r}\be_\vphi{\bm \xi}\\
    	-{\bm A}D{\bm h}_m D{\bm \xi}-d{\bm h}_m{\bm \xi}
    	\end{multline}
    	for all ${\bm \xi}\in H_0^1(\n_{r_s}^+)$ where $\til \Psi_m^*\be_{\vphi}:=\til\Psi_m\be_{\vphi}-{\bm h}_m$. 
    	Using Lemma \ref{lemH1}, we obtain from \eqref{weakW}  
    	$$||\til \Psi_m^* \be_\vphi||_{H_0^1(\n_{r_s}^+)}\le C||\til f^\p_{s,m}\be_\vphi||_{L^2(\n_{r_s}^+)}.$$
    	Since $||\til f^\p_{s,m}\be_\vphi||_{L^2(\n_{r_s}^+)}\le ||\til \fsp\be_\vphi||_{L^2(\n_{r_s}^+)}$, $\til \Psi_m^*\be_{\vphi}$ is a bounded sequence in $H_0^1(\n_{r_s}^+)$.
    	Hence, there exists a subsequence $\til \Psi_{m_l}^*\be_{\vphi}$ of $\til \Psi_m^*\be_{\vphi}$ and some function $\til \Psi^*\be_{\vphi}\in H_0^1(\n_{r_s}^+)$ such that $\til \Psi_{m_l}^*\be_{\vphi}$ and $D(\til \Psi_{m_l}^*\be_{\vphi})$ weakly converge to  $\til \Psi^*\be_{\vphi}$ and $D(\til \Psi^*\be_{\vphi})$ in $L^2(\n_{r_s}^+)$, respectively.
    	Take $l\ra \infty$ to \eqref{weakW} for $m=m_l$. 
    	Then by 
    	$\til \Psi_{m_l}^*\be_{\vphi}\rightharpoonup\til \Psi^*\be_{\vphi}$  in $L^2((0,\theta_1),\sin\theta d\theta)$, $D(\til \Psi_{m_l}^*\be_{\vphi})\rightharpoonup  D(\til \Psi^*\be_{\vphi})$ in $L^2((0,\theta_1),\sin\theta d\theta)$ and $\til f_{s,m}^\p\ra \til f_s^\p$ in $L^2((0,\theta_1),\sin\theta d\theta)$, one has \eqref{weakW} with $\til f_{s,m}^\p$, $\til f_{s,m}$ and $\til \Psi_m^*$ replaced by $\til f_s^\p$, $\til f_s$ and $\til \Psi^*\be_{\vphi}$, respectively. 
    	Thus, $\til \Psi^*\be_{\vphi}$ is a $H_0^1(\n_{r_s}^+)$ function satisfying \eqref{weakW} with $\til f_{s,m}^\p$ and $\til f_{s,m}$ replaced by $\til f_s^\p$ and $\til f_s$, respectively, for all ${\bm \xi}\in H_0^1(\n_{r_s}^+)$. 
    	Let  ${\bm h}$ be ${\bm h}_m$ with $\til f_{s,m}^\p$ and $\til f_{s,m}$ replaced by $\til f_s^\p$ and $\til f_s$, respectively. 
    	One can see that $\til \Psi^{(\til \fsp)}\be_{\vphi}-{\bm h}$ is a $H_0^1(\n_{r_s}^+)$ function satisfying \eqref{weakW} with $\til f_{s,m}^\p$ and $\til f_{s,m}$ replaced by $\til f_s^\p$ and $\til f_s$, respectively, for all ${\bm \xi}\in H_0^1(\n_{r_s}^+)$. 
    	By Lemma \ref{lemH1}, 
    	a $H_0^1(\n_{r_s}^+)$ function satisfying \eqref{weakW} with $\til f_{s,m}^\p$ and $\til f_{s,m}$ replaced by $\til f_s^\p$ and $\til f_s$, respectively, for all  ${\bm \xi}\in H_0^1(\n_{r_s}^+)$ is unique. Hence, $\til \Psi^*\be_{\vphi}=\til \Psi^{(\til \fsp)}\be_{\vphi}-{\bm h}$. From this, we can conclude that $\til\Psi_{m_l}\be_{\vphi}$ and $D(\til \Psi_{m_l}\be_{\vphi})$ weakly converge to $\til \Psi^{(\til \fsp)}\be_\vphi$ and $D(\til \Psi^{(\til \fsp)}\be_{\vphi})$ in $L^2(\n_{r_s}^+)$, respectively.

    	3. Show that $D_{\fsp}\mc A(\zeta_0)\til \fsp=0$ only if $\til \fsp=0$. 
    	
    	Since $\til f_s^\p\neq 0$ by the assumption, there exists $k\in {\mathbb N}$ such that $c_k>0$ or $c_k<0$ in the expression of $\til\fsp$ in \eqref{fspexp}. Without loss of generality, assume that $c_k>0$ for some $k\in {\mathbb N}$. Then since $\lambda_k>0$ by Lemma \ref{lemlegendre}, $p_k$, that is, the solution of \eqref{rode}, \eqref{rodebc} for $j=k$,
    	satisfies $p_k(r_1)>0$. Using this fact, $p_k(r_s)=0$ and 
    	the form of \eqref{rode}, we can deduce that $p_k \ge 0$ in $[r_s,r_1]$. Thus, $p_k^\p(r_s)\ge 0$. 
    	
    	Write \eqref{dfsp} 
    	in the form 
    	\begin{align*}
    	\int_0^{\theta_1}\left(\left.-\frac{1}{\rho_0^+}\frac{1}{r}\pt_r (r\til \Psi^{(\til \fsp)})\right|_{r=r_s}-(u_0^- -u_0^+)(r_s)\frac{\til f_s^\p}{r_s}\right)\xi\sin\theta d \theta=0
    	\end{align*}
    	for all $\xi\in L^2((0,\theta_1),\sin\theta d\theta)$. 
    	Rewrite this as
    	\begin{multline}\label{L2}
    	\int_0^{\theta_1}\left(-\frac{1}{\rho_0^+}\sum_{j=1}^{m_l}\frac{\pt_r(rp_j)}{r}q_j-(u_0^- -u_0^+)(r_s)\sum_{j=1}^{m_l}\frac{c_jq_j}{r_s}\right.\\
    	\left.-\frac{1}{\rho_0^+}\frac{1}{r}\pt_r (r\til \Psi^{(\til \fsp)})\right|_{r=r_s}
    	-(u_0^- -u_0^+)(r_s)\frac{\til f_s^\p}{r_s}
    	\\
    	\left.-\left(-\frac{1}{\rho_0^+}\sum_{j=1}^{m_l}\frac{\pt_r(rp_j)}{r}q_j-(u_0^- -u_0^+)(r_s)\sum_{j=1}^{m_l}\frac{c_jq_j}{r_s}\right)\right)\xi\sin\theta d\theta=0
    	\end{multline}
    	for all $\xi\in L^2((0,\theta_1),\sin\theta d\theta)$. 
    	Since $p_k^\p(r_s)\ge 0$, $p_k(r_s)=0$, $(u_0^--u_0^+)(r_s)>0$, $\til\Psi_{m_l}\be_{\vphi}\rightharpoonup\til \Psi^{(\til \fsp)}\be_\vphi$ in $L^2(\n_{r_s}^+)$ and $D(\til \Psi_{m_l}\be_{\vphi})\rightharpoonup D(\til \Psi^{(\til \fsp)}\be_{\vphi})$ in $L^2(\n_{r_s}^+)$, for a sufficiently large $l$ such that $m_l\ge k$, if we take $\xi=q_k$, then the left-hand side of \eqref{L2} becomes a negative number (here we used the trace theorem). 
    	This contradicts to the assumption that $D_{\fsp}\mc A(\zeta_0)\til\fsp=0$. 
    	This finishes our proof. 
    \end{proof}
    Applying the weak implicit function theorem introduced in \cite{MR2824466} with the results in Lemma \ref{lemconti}, Lemma \ref{lemFrechetdiff} and Lemma \ref{leminvertible}, we obtain the result of the existence part of Theorem \ref{mainthm}. 
    \subsection{Proof of Theorem \ref{mainthm} (Uniqueness)}\label{subunique}
    Finally, we prove the uniqueness part of Theorem \ref{mainthm}. 
    \begin{proof}[Proof of Theorem \ref{mainthm} (Uniqueness)]
    	Let $\ol{\sigma}_1$ be a positive constant $\le \sigma_3$ obtained in the previous subsection such that if $\sigma\le \ol{\sigma}_1$, then Problem \ref{pb2} has a solution satisfying 
    	\begin{multline}
    	\label{mainestimate2}
    	|f(0)-r_s|+||f_s^\p||_{1,\alpha,(0,\theta_1)}^{(-\alpha,\{\theta=\theta_1\})}\\
    	+||\grad\times((\Phi-\Phi_0^+)\be_\vphi)||_{1,\alpha,\n_f^+}^{(-\alpha,\Gam_w^+)}+||\frac{L}{2\pi r \sin\theta}\be_\vphi||_{1,\alpha,\n_f^+}^{(-\alpha,\Gam_w^+)}+||S-S_0^+||_{1,\alpha,\n_f^+}^{(-\alpha,\Gam_w^+)}\le C  \sigma
    	\end{multline}	
        where $C$ is a positive constant depending on the data. 
    	Let $\sigma_1$ be a positive constant  
    	$\le\ol{\sigma}_1$ and to be determined later. 
    	Suppose that 
    	there exist two solutions 
    	$(f_i,\Phi_i\be_\vphi,L_i,S_i)$ for $i=1,2$ of Problem \ref{pb2} for $\sigma \le \sigma_1$ 
    	satisfing \eqref{mainestimate2}. 
    	
    	We will prove that there exists a positive constant $\ul{\sigma_1}\le \ol{\sigma}_1$ such that if $\sigma_1=\ul\sigma_1$, then 
    	\begin{align}\label{uniqueness}
    	(f_1,\Phi_1\be_\vphi,L_1,S_1)=(f_2,\Phi_2\be_\vphi,L_2,S_2).
    	\end{align}
    	In this proof, $C$s and $C_i$ for $i=1,2,\ldots$ denote positive constants depending on the data. Each $C$ in different inequalities differs from each other. In this proof,  when we estimate quantities, we will use all or a part of the conditions $(\rho_-,\bu_-,p_-)\in \mc B_{\sigma}^{(1)}$, $p_{ex}\in \mc B^{(3)}_\sigma$, \eqref{mainestimate2} satisfied by $(f_i,\Phi_i\be_\vphi,L_i,S_i)$ for $i=1,2$ for $\sigma\le \sigma_3$ without mentioning that we use these conditions.

    	Case 1. $f_{s,1}=f_{s,2}$ where $f_{s,i}:=f_i-f_i(0)$ for $i=1,2$. 
    	
    	By \eqref{mainestimate2} satisfied by $(f_i,\Phi_i\be_\vphi,L_i,S_i)$,
    	$$||f^\p_{s,i}||_{1,\alpha,(0,\theta_1)}^{(-\alpha,\{\theta=\theta_1\})}\le C_1\sigma$$
    	for $i=1,2$. 
    	Choose $\sigma_1=\min(\ol{\sigma}_1,\frac{C_3\sigma_3}{C_2},\frac{\sigma_3}{C_1})(=:\sigma_1^{(1)})$ where $C_2$ and $C_3$ are $C$ in \eqref{mainestimate2} and \eqref{Pseudoestimate}, respectively. 
    	Then $(\rho_-,\bu_-,p_-,p_{ex},\fsp)\in \mc B^{(4)}_{\sigma}$ for $\sigma\in (0,\sigma_3]$ and $(f_i(0),\Phi_i\be_{\vphi},L_i,S_i)$ satisfies \eqref{Pseudoestimate} for $f_s=f_{s,i}$, so the hypothesis in Proposition \ref{proPseudofree} is satisfied. If $f_{s,1}=f_{s,2}$, then by Proposition \ref{proPseudofree}, $(f_1(0),\Phi_1\be_\vphi,L_1,S_1)=(f_2(0),\Phi_2\be_\vphi,L_2,S_2)$. From this, we have \eqref{uniqueness}. 

    	Case 2. 
    	General case.
    	
    	1.  Let $(\Psi_i,A_i,T_i):=(\Phi_i-\Phi_0^+,L_i,S_i-S_0^+)$ and $(\til \Psi_i,\til A_i, \til T_i):=(\frac{\til W_i}{2\pi r \sin\theta},A_i(\Pi_{r_s f_i}),T_i(\Pi_{r_s f_i}))$ for $i=1,2$ with $\til W_i:=W_i(\Pi_{r_s f_i})$ and $W_i:=2\pi r \sin\theta \Psi_i$.
    	Show that 
    	\begin{multline}\label{substapri}
    	|f_2(0)-f_1(0)|+||(\til \Psi_2-\til \Psi_1)\be_\vphi||_{1,\beta,\n_{r_s}^+}\\+||\frac{\til A_2}{2\pi 
    		r\sin\theta}-\frac{\til A_1}{2\pi 
    		r\sin\theta}||_{0,\beta,\n_{r_s}^+}+||\til T_2-\til T_1||_{0,\beta,\n_{r_s}^+}\le C||f_{s,2}^\p-f_{s,1}^\p||_{0,\beta,(0,\theta_1)}
    	\end{multline}
    	where $\beta$ is a positive constant given in the proof of the uniqueness part of Proposition \ref{proPseudofree}.  
        
        By the assumption, 
        $(f(0),f_s,\Psi\be_{\vphi},L,T)=
        (f_i(0),f_{s,i},\Psi_i\be_{\vphi},A_i,T_i)$ for $i=1,2$ satisfy (A$^\p$). 
    	Transform (A$^\p$) satisfied by $(f_i(0),f_{s,i},\Psi_i\be_{\vphi},A_i,T_i)$ for $i=1,2$ 
    	into equations in $\n_{r_s}^+$ or on a part of $\pt\n_{r_s}^+$
    	by using $\Pi_{r_s f_i}$ in 
    	the way that we transformed (A$^\p$) satisfied by $(f_1(0),\Psi_1\be_{\vphi},A_1,T_1)$ 
    	in the proof of the uniqueness part of Proposition \ref{proPseudofree}. Then we obtain
    	\begin{align}\label{uniq1}
    	&M_i \grad\times\left(\frac{1}{\rho_0^+}(1+\frac{{u_0^+}\be_r\otimes u_0^+\be_r}{{c_0^+}^2-{u_0^+}^2})(\Pi_{r_s f_i})N_i\grad\times(\til\Psi_i\be_\vphi)\right)\\
    	&\qd\qd\qd\qd=\left(\frac{{\rho_0^+}^{\gam-1}}{(\gam-1)u_0^+}(1+\frac{\gam {u_0^+}^2}{{c_0^+}^2-{u_0^+}^2})\right)(\Pi_{r_s f_i})\frac{(\pt_{\til\theta_i}\Pi_{f_i r_s}^{*,r})(\Pi_{r_s f_i})\pt_r \til T_i+\pt_\theta\til T_i}{\Pi_{r_s f_i}^{*,r}}\be_\vphi\nonumber\\
    	&\qd\qd\qd\qd\qd\qd\qd\qd\qd\qd\qd\qd\qd\qd\qd\qd\qd\qd\qd\qd+\til {\bm F}_{1i}(\til\Psi_i\be_\vphi,\til A_i,\til T_i)\qdin \n_{r_s}^+,\nonumber
    	\\\label{uniq2}
    	&\til \Psi_i\be_\vphi=
    	\begin{cases}\frac{\Pi_{r_s f_i}^{*,r}
    		(\Phi_--\Phi_0^-)(\Pi_{r_s f_i})}{r}\be_{\vphi}\qdon \Gam_{r_s},\\
    	\frac{r_0(\Phi_--\Phi_0^-)(r_0,\theta_1)}{r 
    	}\be_{\vphi}\qdon \Gam_{w,r_s}^+:=\Gam_w\cap \{r>r_s\},\\
    	\left(\left.\frac{1}{r\sin\theta}\int_0^\theta \left(\mf f_0(\til T_i,p_{ex})\right.\right.\right.\\
    	\qd\left.\left.-\frac{\rho_0^+((\gam-1){u_0^+}^2+{c_0^+}^2)}{\gam(\gam-1)u_0^+S_0^+}\til T_i+\mf f_1(\til \Psi_i\be_\vphi,\til A_i,\til T_i) \right)r^2\sin\xi d\xi\right)\be_\vphi\qdon \Gam_{ex},
    	\end{cases}
    	\end{align}
    	\begin{multline}\label{uniq3}
    	\frac{1}{r\sin\theta}\int_0^\theta \left(\mf f_0(\til T_i,p_{ex})-\frac{\rho_0^+((\gam-1){u_0^+}^2+{c_0^+}^2)}{\gam(\gam-1)u_0^+S_0^+}\til T_i+\mf f_1(\til \Psi_i\be_\vphi,\til A_i,\til T_i) \right)r^2\sin\xi d\xi\\
    	=\frac{r_0(\Phi_--\Phi_0^-)(r_0,\theta_1)}{r_1}
    	\end{multline}
    	for $i=1,2$ where  $M_i=\left(\frac{\pt \Pi_{f_i r_s}}{\pt {\rm y}_i}\right)(\Pi_{r_s f_i})$ for $i=1,2$, 
    	$N_i=\frac{r^2}{(\Pi_{r_sf_i}^{*,r})^2}\be_r\otimes\be_r-\frac{\left(\pt_{\til \theta_i}\Pi_{f_ir_s}^{*,r}\right)(\Pi_{r_sf_i}) }{(\Pi_{r_sf_i}^{*,r})^2}r\be_r\otimes\be_\theta+\frac{\left(\pt_{\til r_i}\Pi_{f_i r_s}^{*,r}\right)(\Pi_{r_sf_i}) }{\Pi_{r_sf_i}^{*,r}}r\be_\theta\otimes\be_\theta$, 
    	${\rm y}_i$ for $i=1,2$ are the cartesian coordinate systems representing $\n_{f_i}^+$, respectively, $(\til r_i,\til \theta_i)$ are $(r,\theta)$ coordinates for ${\rm y}_i$, respectively,
    	$(r,\theta)=\Pi_{f_i r_s}^*(\til r_i,\til \theta_i)$,  
    	$\Pi_{r_s f_i}^{*,r}$ and $\Pi_{f_i r_s}^{*,r}$ are the $r$-components of $\Pi_{r_s f_i}^{*}$ and $\Pi_{f_i r_s}^{*}$, respectively, and $\til{\bm F}_{1i}$ for $i=1,2$ are ${\bm F}_1$ changed by using the transformation $\Pi_{r_sf_i}$ for $i=1,2$ using the relations
    	$
    	(\grad\times (\Psi_i\be_{\vphi}))(\Pi_{r_s f_i})=N_i\grad\times (\til \Psi_i
    	\be_{\vphi})$ and $\left(\grad\times\left(\frac{A}{2\pi r \sin\theta}\be_{\vphi}\right)\right)(\Pi_{r_s f_i})=N_i\grad\times \left(\frac{\til A_i}{2\pi r \sin\theta}\be_{\vphi}\right)$.
    	Subtract \eqref{uniq1}-\eqref{uniq3} for $i=1$ from the same equations for $i=2$. Then we have
    	\begin{align}\label{un1}
    	&\grad\times\left(\frac{1}{\rho_0^+}(1+\frac{{u_0^+}\be_r\otimes u_0^+\be_r}{{c_0^+}^2-{u_0^+}^2})\grad\times((\til\Psi_2-\til\Psi_1)\be_\vphi)\right)\\
    	&\nonumber\qd\qd\qd\qd\qd=
    	\frac{{\rho_0^+}^{\gam-1}}{(\gam-1)u_0^+}(1+\frac{\gam {u_0^+}^2}{{c_0^+}^2-{u_0^+}^2})\frac{\pt_\theta (\til T_2-\til T_1)}{r}\be_\vphi+A_2-A_1+B_2-B_1\\
    	\nonumber&\qd\qd\qd\qd\qd\qd\qd\qd\qd\qd\qd\qd+\til {\bm F}_{12}(\til\Psi_2\be_\vphi,\til A_2,\til T_2)-\til {\bm F}_{11}(\til\Psi_1\be_\vphi,\til A_1,\til T_1)(=:\til{\bm F})\qdin \n_{r_s}^+,\\
    	\label{un2}
    	&(\til \Psi_2-\til\Psi_1)\be_\vphi=
    	\begin{cases}
    	\left(\frac{\Pi_{r_s f_2}^{*,r}
    		(\Phi_--\Phi_0^-)(\Pi_{r_s f_2})-\Pi_{r_s f_1}^{*,r}
    		(\Phi_--\Phi_0^-)(\Pi_{r_s f_1})}{r}
    	\right)\be_{\vphi}(=:\til{\bm h}_1)\qdon \Gam_{r_s},\\
    	{\bm 0}
    	\qdon \Gam_{w,r_s}^+,\\
    	\left(\left.\frac{1}{r\sin\theta}\int_0^\theta \left(\mf f_0(\til T_2,p_{ex})-\mf f_0(\til T_1,p_{ex})\right.\right.\right.
    	-\frac{\rho_0^+((\gam-1){u_0^+}^2+{c_0^+}^2)}{\gam(\gam-1)u_0^+S_0^+}(\til T_2-\til T_1)\\
    	\qd\qd\left.\left.+\mf f_1(\til \Psi_2\be_\vphi,\til A_2,\til T_2)-\mf f_1(\til \Psi_1\be_\vphi,\til A_1,\til T_1) \right)r^2\sin\xi d\xi\right)\be_\vphi(=:\til{\bm h}_2)\qdon \Gam_{ex},
    	\end{cases}
    	\end{align}
    	\begin{multline}\label{un3}
    	\left.\frac{1}{r\sin\theta}\int_0^\theta \left(\mf f_0(\til T_2,p_{ex})-\mf f_0(\til T_1,p_{ex})\right.\right.-\frac{\rho_0^+((\gam-1){u_0^+}^2+{c_0^+}^2)}{\gam(\gam-1)u_0^+S_0^+}(\til T_2-\til T_1)\\
    	\left.+\mf f_1(\til \Psi_2\be_\vphi,\til A_2,\til T_2)-\mf f_1(\til \Psi_1\be_\vphi,\til A_1,\til T_1) \right)r^2\sin\xi d\xi=0
    	\end{multline}
    	where 
    	\begin{align*}
    	A_i=&\grad\times\left(\frac{1}{\rho_0^+}(1+\frac{{u_0^+}\be_r\otimes u_0^+\be_r}{{c_0^+}^2-{u_0^+}^2})\grad\times(\til\Psi_i\be_\vphi)\right)\\
    	&\qd\qd\qd\qd\qd\qd-M_i \grad\times\left(\frac{1}{\rho_0^+}(1+\frac{{u_0^+}\be_r\otimes u_0^+\be_r}{{c_0^+}^2-{u_0^+}^2})(\Pi_{r_sf_i})N_i\grad\times(\til\Psi_i\be_\vphi)\right)
    	\end{align*}
    	and 
    	\begin{align*}
    	B_i=&-\frac{{\rho_0^+}^{\gam-1}}{(\gam-1)u_0^+}(1+\frac{\gam {u_0^+}^2}{{c_0^+}^2-{u_0^+}^2})\frac{\pt_\theta \til T_i}{r}\be_\vphi\\
    	&+\left(\frac{{\rho_0^+}^{\gam-1}}{(\gam-1)u_0^+}(1+\frac{\gam {u_0^+}^2}{{c_0^+}^2-{u_0^+}^2})\right)(\Pi_{r_sf_i})\frac{(\pt_{\til\theta_i}\Pi_{f_i r_s}^{*,r})(\Pi_{r_sf_i})\pt_r \til T_i+\pt_\theta\til T_i}{\Pi_{r_sf_i}^r}\be_\vphi
    	\end{align*}
    	for $i=1,2$. 
    	We will estimate $||(\til \Psi_2-\til \Psi_1)\be_{\vphi}||_{1,\beta,\n_{r_s}^+}$ using \eqref{un1}-\eqref{un3}. 
    	For this, we estimate $	||\frac{\til A_2}{2\pi r\sin\theta}-\frac{\til A_1}{2\pi r\sin\theta}||_{1,0,\n_{r_s}^+}$ and $||\til T_2-\til T_1||_{0,\beta,\n_{r_s}^+}$. 
    	
    	Estimate $||\frac{\til A_2}{2\pi r\sin\theta}-\frac{\til A_1}{2\pi r\sin\theta}||_{0,\beta,\n_{r_s}^+}$ and $||\til T_2-\til T_1||_{0,\beta,\n_{r_s}^+}$: 
    	
    	Since $A_i$ and $T_i$ for $i=1,2$ are solutions of (B$^\p$) for $(\Psi,f(0),f_s)=(\Psi_i,f_i(0),f_{s,i})$, by Lemma \ref{lemTrans}, $A_i$ and $T_i$
    	are represented as 
    	$$A_i=2\pi f_i(\mc L_i) \sin\mc L_i u_{-,\vphi}(f_i(\mc L_i),\mc L_i)$$ and $$T_i=\left(g\left(\left(\frac{\bu_-\cdot{\bm \nu}_{f_i}(\mc L_i)
    	}{c_-}\right)^2\right)S_-\right)(f_i(\mc L_i),\mc L_i)-(g({M_0^-}^2))(r_s)S_{in}$$ where $u_{-,\vphi}=\bu_-\cdot \be_{\vphi}$, $\mc L_i$ for $i=1,2$ are $\mc L$ given in Lemma \ref{lemTrans} for $V=2\pi r \sin\theta (\Phi_0^++\Psi_i)$ and $f=f_i$, respectively,  and ${\bm \nu}_{f_i}$ for $i=1,2$ are the unit normal vectors on $\Gam_{f_i}$ pointing toward $\n_{f_i}^+$, respectively. Using these solution expressions, express $(\frac{\til A_2}{2\pi r\sin\theta}-\frac{\til A_1}{2\pi r\sin\theta})\be_\vphi$ and $\til T_2-\til T_1$ as  
    	\begin{align}\label{A2A1}
    	\left(\frac{2\pi f_2(\til{\mc L}_2) \sin\til{\mc L}_2 u_{\vphi,-}(f_2(\til{\mc L}_2),\til{\mc L}_2)}{2\pi r\sin\theta}-\frac{2\pi f_1(\til{\mc L}_1) \sin\til{\mc L}_1 u_{\vphi,-}(f_1(\til{\mc L}_1),\til{\mc L}_1)}{2\pi r\sin\theta}\right)\be_\vphi
    	\end{align}
    	and
    	\begin{align}\label{T2T1}
    	\left(g\left(\left(\frac{\bu_-\cdot{\bm \nu}_{f_2}(\til {\mc L}_2)
    	}{c_-}\right)^2\right)S_-\right)(f_2(\til {\mc L}_2),\til {\mc L}_2)-\left(g\left(\left(\frac{\bu_-\cdot{\bm \nu}_{f_1}(\til {\mc L}_1)
        }{c_-}\right)^2\right)S_-\right)(f_1(\til {\mc L}_1),\til {\mc L}_1),
    	\end{align}
    	respectively, where $\til{\mc L}_i:=\mc L_i(\Pi_{r_s f_i})$ for $i=1,2$.
    	Using arguments similar to the ones used to estimate $||(\frac{A_2}{2\pi r \sin\theta}-\frac{\til A_1}{2\pi r \sin\theta})\be_{\vphi}||_{1,0,\n_{f_2}^+}$  and $||T_2-\til T_1||_{0,\beta,\n_{f_2}^+}$ in the proof of the uniqueness part of Proposition \ref{proPseudofree}, we estimate \eqref{A2A1} and \eqref{T2T1} in $C^\beta(\ol{\n_{r_s}^+})$. Then we obtain
    	\begin{align}\label{A2A1esti}
    	||(\frac{\til A_2}{2\pi r\sin\theta}-\frac{\til A_1}{2\pi r\sin\theta})\be_{\vphi}||_{1,0,\n_{r_s}^+}\le C\sigma\left(||f_2-f_1||_{1,\beta,\Lambda}+||(\til \Psi_2-\til \Psi_1)\be_{\vphi}||_{1,\beta,\n_{r_s}^+}\right)
    	\end{align}
    	and 
    	\begin{align}\label{T2T1esti}
    	||\til T_2-\til T_1||_{0,\beta,\n_{r_s}^+}
    	\le C||f_2-f_1||_{1,\beta,\Lambda}+ C\sigma
    	||(\til \Psi_2-\til \Psi_1)\be_{\vphi}||_{1,\beta,\n_{r_s}^+}.
    	\end{align}
    	
    	Using these estimates, we estimate $||(\til \Psi_2-\til \Psi_1)\be_{\vphi}||_{1,\beta,\n_{r_s}^+}.$
    	Substitute \eqref{T2T1} into \eqref{un3}. And then using \eqref{A2A1esti} and \eqref{T2T1esti}, estimate $f_2(0)-f_1(0)$ in the resultant equation (see Step 4 in the proof of the uniqueness part of Proposition \ref{proPseudofree}). Then we obtain 
    	\begin{align*}
    	|f_2(0)-f_1(0)|\le C_4\sigma|f_2(0)-f_1(0)|+ C||f_{s,2}^\p-f_{s,1}^\p||_{0,\beta,(0,\theta_1)}
    	+C\sigma ||(\til \Psi_2-\til \Psi_1)\be_{\vphi}||_{1,\beta,\n_{r_s}^+}.
    	\end{align*}
    	Take $\sigma_1=\min(\sigma_1^{(1)}, \frac{1}{2C_4})(=:\sigma_1^{(2)})$.
    	Then we have
    	\begin{align}\label{f2f1esti}
    	|f_2(0)-f_1(0)|\le C||f_{s,2}^\p-f_{s,1}^\p||_{0,\beta,(0,\theta_1)}
    	+C\sigma ||(\til \Psi_2-\til \Psi_1)\be_{\vphi}||_{1,\beta,\n_{r_s}^+}.
    	\end{align}
    	Using \eqref{f2f1esti}, we obtain from \eqref{A2A1esti} and \eqref{T2T1esti}
    	\begin{align}\label{A2A1esti1}
    	||\frac{\til A_2}{2\pi r\sin\theta}-\frac{\til A_1}{2\pi r\sin\theta}||_{0,\beta,\n_{r_s}^+}\le C\sigma\left(||f_{s,2}^\p-f_{s,1}^\p||_{0,\beta,(0,\theta_1)}+||(\til \Psi_2-\til \Psi_1)\be_{\vphi}||_{0,\beta,\n_{r_s}^+}\right)
    	\end{align}
    	and 
    	\begin{align}\label{T2T1esti1}
    	||\til T_2-\til T_1||_{0,\beta,\n_{r_s}^+}
    	\le C||f_{s,2}^\p-f_{s,1}^\p||_{0,\beta,(0,\theta_1)}+ C\sigma
    	||(\til \Psi_2-\til \Psi_1)\be_{\vphi}||_{1,\beta,\n_{r_s}^+}.
    	\end{align}
    	Using these two estimates, \eqref{f2f1esti} and arguments similar to the ones in Step 5 in the proof of the uniqueness part of Proposition \ref{proPseudofree}, 
        we estimate $(\til \Psi_2-\til \Psi_1)\be_{\vphi}$ in \eqref{un1}, \eqref{un2} in $C^{1,\beta}(\ol{\n_{r_s}^+})$. Then we obtain 
    	\begin{align*}
    	||(\til \Psi_2-\til \Psi_1)\be_{\vphi}||_{1,\beta,\n_{r_s}^+}\le C_5\sigma||(\til \Psi_2-\til \Psi_1)\be_{\vphi}||_{1,\beta,\n_{r_s}^+}+C||f_{s,2}^\p-f_{s,1}^\p||_{0,\beta,(0,\theta_1)}.
    	\end{align*}
    	Take $\sigma_1=\min (\sigma_1^{(2)}, \frac{1}{2C_5})(=:\sigma_1^{(3)})$. 
    	Then we have
    	\begin{align}\label{Psi2Psi1esti}
    	||(\til \Psi_2-\til \Psi_1)\be_{\vphi}||_{1,\beta,\n_{r_s}^+}\le C||f_{s,2}^\p-f_{s,1}^\p||_{0,\beta,(0,\theta_1)}
    	\end{align}
    	Combining  \eqref{f2f1esti}, \eqref{A2A1esti1}, \eqref{T2T1esti1} and \eqref{Psi2Psi1esti}, we obtain \eqref{substapri}. 
    	
    	2. 
    	Using the arguments in Step 1 in the proof of Lemma \ref{lemFrechetdiff}, we can see that the system \eqref{Wtileqn}-\eqref{Ttileqn} has a unique solution for given $\til f_s^\p=f_{s,2}^\p-f_{s,1}^\p\in C_0^\beta([0,\theta_1]):=\{f\in C^\beta([0,\theta_1])\;|\; f^\p(0)=f^\p(\theta_1)=0\}$: 
    	\begin{align}
    	&\label{f0tilsol}f(0)^{(f_{s,2}^\p-f_{s,1}^\p)}=\frac{-\int_0^{\theta_1} (f_{s,2}-f_{s,1})\sin\zeta d\zeta}{\int_0^{\theta_1}\sin\zeta d\zeta},\\
    	&\label{Atilsol}\til A^{(f_{s,2}^\p-f_{s,1}^\p)}=0,\\
    	&\label{Ttilsol}\til T^{(f_{s,2}^\p-f_{s,1}^\p)}=(g({M_0^-}^2))^\p(r_s)S_{in}\left(\frac{-\int_0^{\theta_1} (f_{s,2}-f_{s,1})\sin\zeta d\zeta}{\int_0^{\theta_1}\sin\zeta d\zeta}+f_{s,2}- f_{s,1}\right)
    	\end{align}
    	and $\til \Psi^{(f_{s,2}^\p-f_{s,1}^\p)}\be_{\vphi}$, the unique $C^{2,\beta}_{(-1-\alpha,\Gam_{w,r_s}^+)}(\n_{r_s}^+)$ solution of \eqref{Wtileqnstar}, \eqref{Wtileqnstarbc} for given $\til \fsp=f_{s,2}^\p-f_{s,1}^\p\in C^\beta_0([0,\theta_1])$ (here we had  $\til \Psi^{(f_{s,2}^\p-f_{s,1}^\p)}\be_{\vphi}\in C^{2,\beta}_{(-1-\alpha,\Gam_{w,r_s}^+)}(\n_{r_s}^+)$ because $f_{s,2}^\p-f_{s,1}^\p\in C_0^\beta([0,\theta_1])$). 
    	Subtract \eqref{Wtileqn}-\eqref{comptil} for given $\til f_s^\p=f_{s,2}^\p-f_{s,1}^\p\in C_0^\beta([0,\theta_1])$ from \eqref{un1}-\eqref{un3}. Then we obtain 
    	\begin{multline}\label{unni1}
    	\grad\times\left(\frac{1}{\rho_0^+}(1+\frac{{u_0^+}\be_r\otimes u_0^+\be_r}{{c_0^+}^2-{u_0^+}^2})\grad\times((\til\Psi_2-\til\Psi_1-\til \Psi^{(f_{s,2}^\p-f_{s,1}^\p)})\be_\vphi)\right)\\
    	=
    	\frac{{\rho_0^+}^{\gam-1}}{(\gam-1)u_0^+}(1+\frac{\gam {u_0^+}^2}{{c_0^+}^2-{u_0^+}^2})\frac{\pt_\theta (\til T_2-\til T_1-\til T^{(f_{s,2}^\p-f_{s,1}^\p)})}{r}\be_\vphi\\
    	+A_2-A_1+B_2-B_1+\til {\bm F}_{12}(\til\Psi_2\be_\vphi,\til A_2,\til T_2)-\til {\bm F}_{11}(\til\Psi_1\be_\vphi,\til A_1,\til T_1)\qdin \n_{r_s}^+,
    	\end{multline}
    	\begin{multline}\label{unni2}
    	(\til \Psi_2-\til\Psi_1-\til \Psi^{(f_{s,2}^\p-f_{s,1}^\p)})\be_\vphi=
    	\begin{cases}
    	\left(\frac{\Pi_{r_s f_2}^{*,r}
    		(\Phi_--\Phi_0^-)(\Pi_{r_s f_2})-\Pi_{r_s f_1}^{*,r}
    		(\Phi_--\Phi_0^-)(\Pi_{r_s f_1})}{r}
    	\right)\be_{\vphi}\qdon \Gam_{r_s},\\
        {\bm 0} \qdon \Gam_{w,r_s}^+,\\
    	\Bigg(\frac{1}{r\sin\theta}\int_0^\theta \bigg(\mf f_0(\til T_2,p_{ex})-\mf f_0(\til T_1,p_{ex})\\
    	\qd\qd\qd\qd\qd\qd\qd
    	-\frac{\rho_0^+((\gam-1){u_0^+}^2+{c_0^+}^2)}{\gam(\gam-1)u_0^+S_0^+}(\til T_2-\til T_1-\til T^{(f_{s,2}^\p-f_{s,1}^\p)})\\
    	+\mf f_1(\til \Psi_2\be_\vphi,\til A_2,\til T_2)-\mf f_1(\til \Psi_1\be_\vphi,\til A_1,\til T_1) \bigg)r^2\sin\xi d\xi\Bigg)\be_\vphi\qdon \Gam_{ex},
    	\end{cases}
    	\end{multline}
    	\begin{multline}\label{unni3}
    	\frac{1}{r\sin\theta}\int_0^\theta \bigg(\mf f_0(\til T_2,p_{ex})-\mf f_0(\til T_1,p_{ex})
    	\\-\frac{\rho_0^+((\gam-1){u_0^+}^2+{c_0^+}^2)}{\gam(\gam-1)u_0^+S_0^+}(\til T_2-\til T_1-\til T^{(f_{s,2}^\p-f_{s,1}^\p)})\\
    	+\mf f_1(\til \Psi_2\be_\vphi,\til A_2,\til T_2)-\mf f_1(\til \Psi_1\be_\vphi,\til A_1,\til T_1) \bigg)r^2\sin\xi d\xi=0.
    	\end{multline}
    	Estimate $||\til \Psi_2-\til \Psi_1-\til \Psi^{(f_{s,2}^\p-f_{s,1}^\p)}||_{1,\beta,\n_{r_s}^+}$ using \eqref{unni1}-\eqref{unni3}.

    	Write $\til T^{(f_{s,2}^\p-f_{s,1}^\p)}$ as
    	\begin{align}
    	\label{T00}(g({M_0^-}^2))^\p(r_s)S_{in}\left(f(0)^{(f_{s,2}^\p-f_{s,1}^\p)}+f_{s,2}- f_{s,1}\right).
    	\end{align}
    	Substitute this $\til T^{(f_{s,2}^\p-f_{s,1}^\p)}$ and \eqref{T2T1} into the places of $\til T^{(f_{s,2}^\p-f_{s,1}^\p)}$ and $\til T_2-\til T_1$ in \eqref{unni3}, respectively. And then using \eqref{substapri}, estimate $f_2(0)-f_1(0)-f(0)^{(f_{s,2}^\p-f_{s,1}^\p)}$ in the resultant equation.
    	Then we obtain
    	\begin{align*}
    	|f_2(0)-f_1(0)-f(0)^{(f_{s,2}^\p-f_{s,1}^\p)}|\le C\sigma ||f_{s,2}^\p-f_{s,1}^\p||_{0,\beta,(0,\theta_1)}.
    	\end{align*}
    	Using this estimate and \eqref{substapri}, we estimate $\til T_2-\til T_1-\til T^{(f_{s,2}^\p-f_{s,1}^\p)}$ given by \eqref{T2T1} and \eqref{T00} in $C^{\beta}(\ol{\n_{r_s}^+})$. 
    	Then we obtain
    	\begin{align}\label{T1T2mmm}
    	||\til T_2-\til T_1-\til T^{(f_{s,2}^\p-f_{s,1}^\p)}||_{0,\beta,\n_{r_s}^+}\le C\sigma||f_{s,2}^\p-f_{s,1}^\p||_{0,\beta,(0,\theta_1)}. 
    	\end{align}
    	Using arguments similar to the ones in Step 5 in the proof of the uniqueness part of Proposition \ref{proPseudofree}  
    	with \eqref{substapri} and \eqref{T1T2mmm}, we estimate $\til \Psi_2-\til \Psi_1-\til \Psi^{(f_{s,2}^\p-f_{s,1}^\p)}$ in \eqref{unni1}, \eqref{unni2} in $C^{1,\beta}(\ol{\n_{r_s}^+})$. 
    	Then we have
    	\begin{align}\label{P2P1tilP}
    	||\til \Psi_2-\til \Psi_1-\til \Psi^{(f_{s,2}^\p-f_{s,1}^\p)}||_{1,\beta,\n_{r_s}^+}
    	\le C\sigma||f_{s,2}^\p-f_{s,1}^\p||_{0,\beta,(0,\theta_1)}.
    	\end{align}
    	
    	3. By the assumption, $(f_i,\Phi_i\be_\vphi,L_i,S_i)$ for $i=1,2$ satisfy \eqref{rhtau}. This implies that
    	$\mc A(\rho_-,\bu_-,p_-,p_{ex},f_{s,i}^\p)=0$
    	for $i=1,2$ where $\mc A$ is a map defined in \eqref{mcA}. 
    	Subtract these two equations. Then we have
    	$$\mc A(\rho_-,\bu_-,p_-,p_{ex},f_{s,2}^\p)-\mc A(\rho_-,\bu_-,p_-,p_{ex},f_{s,1}^\p)=0.$$
    	Write this as 
    	\begin{align}\label{Alinearize}
    	0= D_{\fsp} \mc A(\zeta_0)(f_{s,2}^\p-f_{s,1}^\p)+R
    	\end{align}
    	where
    	\begin{align*}
    	R=\mc A(\rho_-,\bu_-,p_-,p_{ex},f_{s,2}^\p)
    	-\mc A(\rho_-,\bu_-,p_-,p_{ex},f_{s,1}^\p) 
    	-D_{\fsp} \mc A(\zeta_0)(f_{s,2}^\p-f_{s,1}^\p)
    	\end{align*}
    	and $D_{\fsp} \mc A(\zeta_0)$ is a map given in \eqref{Dfsp} as a map from $C_0^\beta([0,\theta_1])$ to $C_0^\beta([0,\theta_1])$ (here $D_{\fsp} \mc A(\zeta_0)$ takes $C^\beta_0([0,\theta_1])$ functions and $\til \Psi^{(f_{s,2}^\p-f_{s,1}^\p)}\be_{\vphi}$ in $D_{\fsp} \mc A(\zeta_0)(f_{s,2}^\p-f_{s,1}^\p)$ is the $C^{2,\beta}_{(-1-\alpha,\Gam_{w,r_s}^+)}(\n_{r_s}^+)$ solution of \eqref{Wtileqnstar}, \eqref{Wtileqnstarbc} for given $\til \fsp=f_{s,2}^\p-f_{s,1}^\p\in C^\beta_0([0,\theta_1])$).

    	Using \eqref{substapri} and \eqref{P2P1tilP}, we can obtain
    	\begin{align}\label{Restimate}
    	||R||_{0,\beta,(0,\theta_1)}\le C\sigma ||f_{s,2}^\p-f_{s,1}^\p||_{0,\beta,(0,\theta_1)}.
    	\end{align}
    	In the way that we proved the invertiblity of $D_{\fsp} \mc A(\zeta_0)$ as a map from $C^{1,\alpha}_{(-\alpha,\{\theta=\theta_1\})}((0,\theta))$ to $C^{1,\alpha}_{(-\alpha,\{\theta=\theta_1\})}((0,\theta))$ in the the proof of Lemma \ref{leminvertible}, 
    	we can prove that  $D_{\fsp} \mc A(\zeta_0)$ is an invertible map as a map from $C_0^\beta([0,\theta_1])$ to $C_0^\beta([0,\theta_1])$. Using this fact and \eqref{Restimate}, we obtain from \eqref{Alinearize} 
    	\begin{align*}
    	||f_{s,2}^\p-f_{s,1}^\p||_{0,\beta,(0,\theta_1)}\le C_6\sigma ||f_{s,2}^\p-f_{s,1}^\p||_{0,\beta,(0,\theta_1)}. 
    	\end{align*}
    	Take $\sigma_1=\min(\sigma_1^{(3)},\frac{1}{2C_6})(=:\ul{\sigma}_1)$. Then we have 
    	$f_{s,2}^\p=f_{s,1}^\p$. One can see that $\ul\sigma_1$ depends on the data. This finishes the proof. 
    \end{proof}
    \section{Appendix}\label{secappen}
    In this section, we present some computations done by using the tensor notation given in \S \ref{subelliptic}. We explain how we transformed an elliptic system in the cartesian coordinate system into a system in the spherical coordinate system in the proof of Lemma \ref{lemalpha} and Lemma \ref{lem1alpha}, 
    and show that \eqref{Psielliptic} is equivalent to \eqref{EE}. 
    
    Let $(q_1,q_2,q_3)$ be an orthogonal coordinate system in $\R^3$.
    The unit vectors in this coordinate system in the direction of $q_i$ for $i=1,2,3$ are given as $\frac{1}{h_i}\frac{\pt {\bm x}}{\pt q_i}(=:\be_{q_i})$ for $i=1,2,3$ where ${\bm x}=x\be_1+y\be_2+z\be_3$ and $h_i:=|\frac{\pt {\bm x}}{\pt q_i}|$.  By  $\grad=\sum_{i=1}^3 \frac{\be_{q_i}}{h_i}\pt_{q_i}$,
    $\grad \bU$ where ${\bm U}: \R^3\ra \R^3$ can be written as
    \begin{align}\label{Ugra}
    \grad {\bm U}=\sum_{i=1}^3\frac{\pt_{q_i}{\bm U}}{h_i}\otimes \be_{q_i}.
    \end{align}
    Here, ${\bm a}\otimes {\bm b}$ for ${\bm a}, {\bm b}\in \R^3$ denotes ${\bm a}{\bm b}^T $. 
    By 
    \eqref{abcddef}, the multiplication of the tensor notation defined in  \S \ref{subelliptic} is defined by 
    \begin{align}
    \label{tensorproduct}
    ({\bm a}\otimes {\bm b}\otimes {\bm c}\otimes {\bm d}) ({\bm e}\otimes {\bm f}\otimes {\bm g}\otimes {\bm h})=({\bm d}\cdot{\bm e})({\bm c}\cdot {\bm f}) ({\bm a}\otimes {\bm b}\otimes {\bm g}\otimes {\bm h}) 
    \end{align}
    where ${\bm a}$, ${\bm b}$, ${\bm c}$, ${\bm d}$, ${\bm e}$, ${\bm f}$, ${\bm g}$, ${\bm h}\in \R^3$. 
    By \eqref{Ugra} and \eqref{tensorproduct}, $M$ in 
    $$\grad {\bm U}=M\grad_{(q_1,q_2,q_3)}{\bm U},$$
    where $\grad_{(q_1,q_2,q_3)}{\bm U}=\sum_{i=1}^3\pt_{q_i}{\bm U} \otimes \be_i$
    and $M=\frac{\pt (q_1,q_2,q_3)}{\pt (x,y,z) }$, can be expressed as 
    \begin{align}
    \label{M}M=\sum_{i=1}^3\frac{1}{h_i}\mc I\otimes \be_{q_i}\otimes \be_i\otimes \mc I.
    \end{align} 
    Here,  $\mc I\otimes {\bm a}\otimes {\bm b}\otimes \mc I$ for ${\bm a}$, ${\bm b}\in R^3$ is a linear map from $\R^{3\times3}$ to $\R^{3\times3}$ defined by
    \begin{align}
    \label{IItensor}\mc I\otimes {\bm a}\otimes {\bm b}\otimes \mc I:=\sum_{i=1}^3\be_i\otimes{\bm a}\otimes{\bm b}\otimes \be_i.
    \end{align}
    (Note that $\mc I\otimes {\bm a}\otimes {\bm b}\otimes \mc I$ satisfies
    \begin{align}\label{22tensoract}
    (\mc I\otimes {\bm a}\otimes {\bm b}\otimes \mc I)({\bm c}\otimes{\bm d})=({\bm b}\cdot {\bm d}){\bm c}\otimes {\bm a}
    \end{align}
    for ${\bm c}$, ${\bm d}\in \R^3$.)
    In the standard way, 
    we transform $\Div({\bm B}\grad{\bm U})=0$ in the cartesian coordinate system where ${\bm B}$ is a linear map from $\R^{3\times3}$ to $\R^{3\times3}$ into an equation in $(q_1,q_2,q_3)$-coordinate system.  
    Then we obtain 
    \begin{align}\label{systemtrans}
    \Div_{(q_1,q_2,q_3)}\left(\frac{1}{\det M}M^T {\bm B}\circ \chi^{-1}M\grad_{(q_1,q_2,q_3)}(\bU\circ \chi^{-1})\right)=0
    \end{align} 
    where $\Div_{(q_1,q_2,q_3)}:=\grad_{(q_1,q_2,q_3)}\cdot$, $\grad_{(q_1,q_2,q_3)}$ and $M$ are as above, 
    $\chi$ is the map from $(x,y,z)$ to $(q_1,q_2,q_3)$ and $\det M=\frac{1}{h_1h_2h_3}$.
    Using \eqref{tensorproduct} and 
    \eqref{M}, 
    we 
    compute $M^T{\bm B}\circ \chi^{-1} M$. 
    Then we have
    the explicit form of \eqref{systemtrans}. In this way, we obtain the explicit form of the spherical coordinate representation of \eqref{yeqn}.

    
    Next, we show that \eqref{Psielliptic} is equivalent to \eqref{EE}. 
    
    By \eqref{Ugra}, 
    \begin{align}\label{DPsi}
    \grad(\Psi\be_{\vphi})=\pt_r \Psi\be_{\vphi}\otimes \be_r+\frac{\pt_{\theta}\Psi}{r}\be_{\vphi}\otimes\be_{\theta}-\frac{\Psi}{r}\be_r\otimes\be_{\vphi}-\frac{\cos\theta}{r\sin\theta}\Psi\be_{\theta}\otimes\be_{\vphi}.
    \end{align}
    Using \eqref{22tensoract} and \eqref{DPsi},  we have 
    \begin{multline*}
    \frac{{c_0^+}^2}{\rho_0^+({c_0^+}^2-{u_0^+}^2)}\left({\bm I}-\frac{ {u_0^+}^2}{{c_0^+}^2}(\mc I\otimes \be_r\otimes\be_r \otimes \mc I )\right)\grad(\Psi \be_\vphi)
    \\=\frac{ {c_0^+}^2}{\rho_0^+({c_0^+}^2-{u_0^+}^2)}\grad(\Psi\be_{\vphi})-\frac{ {u_0^+}^2}{\rho_0^+({c_0^+}^2-{u_0^+}^2)}\pt_r\Psi\be_{\vphi}\otimes \be_r.
    \end{multline*}
    Using this relation and the relation 
    $\Div({\bm a}\otimes{\bm b})=\grad{\bm a}{\bm b}+{\bm a}\Div {\bm b}$, we compute the left-hand side of  \eqref{Psielliptic}. Then we get
    \begin{align}\label{compute}
    a(r)\Delta(\Psi\be_{\vphi})+\grad(\Psi\be_{\vphi})\grad a(r)-\grad\be_{\vphi}b(r)\pt_r\Psi\be_r-\be_{\vphi}\Div(b(r)\pt_r\Psi\be_r)-\frac{\pt_r\rho_0^+}{{\rho_0^+}^2r}\Psi \be_\vphi
    \end{align}
    where $a(r):=\frac{ {c_0^+}^2}{\rho_0^+({c_0^+}^2-{u_0^+}^2)}$ and $b(r):=\frac{ {u_0^+}^2}{\rho_0^+({c_0^+}^2-{u_0^+}^2)}$.  
    By direct computation, one can check that 
    $\grad\be_{\vphi}b(r)\pt_r\Psi\be_r=0$. Using this fact and \eqref{DPsi}, 
    we compute \eqref{compute}. Then we have
    \begin{align}\label{ab}
    \left(a(r)(\Delta \Psi -\frac{\Psi}{r^2\sin^2\theta})+a^\p(r)\pt_r\Psi -b(r)(\frac{1}{r^2}\pt_r(r^2\pt_r\Psi))-b^\p(r)\pt_r\Psi\right)\be_{\vphi}
    -\frac{\pt_r\rho_0^+}{{\rho_0^+}^2r}\Psi \be_\vphi.
    \end{align}
    By $a-b=\frac{1}{\rho_0^+}$, $(a-b)^\p=-\frac{\pt_r\rho_0^+}{{\rho_0^+}^2}$. Substitute this relation into \eqref{ab} and then express the resultant equation in the spherical coordinate system. Then we obtain
    \begin{align*}
    \left(\frac{1}{\rho_0^+r^2}\pt_r(r^2\pt_r \Psi)-\frac{\pt_r\rho_0^+}{{\rho_0^+}^2r}\pt_r(r\Psi)\right.
    \left.+\frac{ {c_0^+}^2}{\rho_0^+({c_0^+}^2-{u_0^+}^2)r^2}\left(\frac{1}{\sin\theta}\pt_{\theta}(\sin\theta\pt_{\theta}\Psi)-\frac{\Psi}{\sin^2\theta}\right)\right)\be_{\vphi}. 
    \end{align*}
    This after multiplied by $(-1)$ is equal to the left-hand side of \eqref{curlformsig}. 
    By this fact, \eqref{Psielliptic} is equivalent to \eqref{EE}. 

    \vspace{.25in}
    \noindent
    {\bf Acknowledgements:}
    The author would like to thank Myoungjean Bae for her advice in doing this work. 
    The research of Yong Park was supported in part by Samsung Science and Technology Foundation
    under Project Number SSTF-BA1502-02. 
This is a preprint of an article published in Archive for Rational Mechanics and Analysis. The final authenticated version is
available online at: https://doi.org/10.1007/s00205-021-01618-7
    
    \bibliography{transonicnozzles}
    \bibliographystyle{amsplain}
\end{document}